%% file: main.tex
\documentclass{article}
\usepackage{url, a4wide}
\usepackage{needspace}
\usepackage{amsthm,amsmath,amsfonts,amssymb}
\usepackage{natbib}
\usepackage{xcolor}
\usepackage{pdflscape}
\usepackage{hyperref}
\usepackage{algpseudocode}
\usepackage{mathrsfs}
\usepackage{multirow}
\usepackage{float}
\usepackage{graphicx}
\usepackage{epstopdf}
\usepackage{booktabs}
\usepackage{subfig, caption, multicol}
\usepackage[strings]{underscore}
\usepackage[section]{placeins}
\usepackage{afterpage}

\theoremstyle{plain}
\newtheorem{proposition}{Proposition}
\newtheorem{theorem}{Theorem}

\newtheorem{lemma}{Lemma}

\theoremstyle{definition}
\newtheorem{assumption}{Assumption}
\newtheorem{definition}{Definition}

\theoremstyle{remark}
\newtheorem{remark}{Remark}[section]
\newtheorem{example}{Example}

\newcommand{\Var}{\ensuremath{\mathrm{Var}}}
\newcommand{\Cov}{\ensuremath{\mathrm{Cov}}}
\newcommand{\Leb}{\ensuremath{\mathrm{Leb}}}
\newcommand{\Levy}{L\'{e}vy }

\allowdisplaybreaks

\title{Nonparametric estimation of trawl processes: Theory and applications}
\author{Orimar Sauri\thanks{
Department of Mathematical Sciences, Aalborg University,
Thomas Manns Vej 23, 
9220 Aalborg Øst,
Denmark.
\texttt{osauri@math.aau.dk}}\\
Aalborg University
\and
Almut E. D. Veraart\thanks{Department of Mathematics, Imperial College London,
180 Queen's Gate, London, SW7 2AZ, UK.
\texttt{a.veraart@imperial.ac.uk}}\\
Imperial College London}
\date{}

\begin{document}
\maketitle
\begin{abstract}
Trawl processes belong to the class of continuous-time, strictly stationary, infinitely divisible processes; they are defined as  L\'{e}vy bases evaluated over deterministic trawl sets. 		
This article presents the first nonparametric estimator of the trawl function characterising the trawl set and the  serial correlation of the process. Moreover, it establishes a detailed asymptotic theory for the proposed estimator, including  a law of large numbers and a central limit theorem for various  asymptotic relations between an in-fill and a long-span asymptotic regime.
In addition, it develops consistent estimators for both the asymptotic bias and variance, which are subsequently used for establishing feasible central limit theorems which can be applied to data. 
A  simulation study shows the good finite sample performance of the proposed estimators. 
The new methodology is applied to model misspecification testing,  forecasting high-frequency financial spread data from a limit order book and to estimating the busy-time distribution of a stochastic queue.
\end{abstract}


\section{Introduction}
\subsection{Overview}
This article develops the nonparametric estimation theory for \emph{trawl processes},
investigates the finite sample properties of the corresponding asymptotic results in a simulation study, and illustrates the new methodology in an empirical application to high-frequency financial data.
Trawl processes  are a class of 	
continuous-time, strictly stationary, infinitely divisible processes. 
A trawl process $X=(X_t)_{t\in \mathbb{R}}$ is  defined as a L\'{e}vy basis $L$ evaluated over a trawl set $A_t$, which is typically chosen as a subset of $\mathbb{R}^2$. I.e.~we set $X_t=L(A_t)$, for $t\in \mathbb{R}$; the mathematical details will be provided in the next section. 
What makes trawl processes great candidates for applications is the fact that their marginal (infinitely divisible) distribution and their serial dependence can be modelled independently of each other. 
What is more, they can be viewed as a universal class for modelling stationary stochastic processes with infinitely divisible marginal law, since for any infinitely divisible law, there exists a trawl process which has that particular law as its marginal distribution, see  \citet[p.~1535]{BN2011}. Combining this fact with the discussion in  	\citet[p.~1527]{WolpertTaqqu2005},  we can deduce that any stationary infinitely divisible process with a differentiable covariance function that is non-increasing and convex on $(0, \infty)$ can be written as a trawl process.
Moreover, the type of serial dependence is fully determined through the choice of the trawl set and allows for both short- and long-memory settings. 
Given the wide range of phenomena that can be modelled by trawl processes, there has been significant recent interest in this area. Therefore, we will briefly survey the related literature.

\subsection{Related work}
Trawl processes were first mentioned in \cite{BN2011}, motivated by the need for a flexible but analytically tractable class of stochastic processes that can be used in various applications, including the areas of turbulence and finance. 
\cite{WolpertTaqqu2005} considered stochastic processes in this realm under the name of ``upstairs representation'' in their study of the fractional Ornstein-Uhlenbeck process. What we call a trawl (set) was simply referred to as a ``geometric figure in a higher-dimensional space''. Furthermore, when restricted to integer values, trawl-type processes were mentioned by \cite{WolpertBrown2012} and referred to as ``random measure processes''. 
\cite{BNLSV2014} provide a systematic study of integer-valued trawl processes, describing their probabilistic properties, as well as simulation and inference methods.
This work was also extended to a multivariate framework by \cite{VERAART2019}. Recently, \cite{BLSV2021} developed a composite likelihood estimation theory for integer-valued trawl processes, as well as tools for model selection and forecasting.
All of these articles were motivated by applications to high-frequency financial data.
Related work also includes the article by \cite{SheYan}, which proposes a new modelling framework for high-frequency financial data using trawl processes, and the work by  \cite{Shephard2016}, which develops likelihood inference for the special case of trawl processes with an exponential trawl function.
Trawl processes also feature prominently in recent work on hierarchical models for extreme values using a peaks-over-threshold methodology.  
For example, \cite{Noven} propose using a latent trawl process
to model clusters in extremes of environmental time series data, and
\cite{CourgeauVeraart2022} derive the corresponding asymptotic theory for inference in such a model.
\cite{Bacro} propose a spatio-temporal model for extreme values based on extensions of trawls and successfully model extreme 
precipitation data. 
In a discrete-time setting, trawl processes were first introduced by \cite{DoukJakLopSurg18}, who also derived functional limit theorems for partial sums of discrete-time trawl processes. Motivated by these findings, various articles have focused on limit theorems for trawl processes (both in continuous and discrete time).
For example, \cite{GRAHOVAC2018235} study limit theorems to characterise the intermittency of the trawl process, and \cite{Paulau}
derives limit theorems for partial sums of linear processes with tapered innovations, including trawl processes.  
Furthermore, \cite{TalTre2019} study the long-term behaviour of integrated trawl processes with symmetric L\'{e}vy bases. 
Moreover, \cite{PPSV2021}  derive limit theorems for continuous-time trawl processes, including the asymptotic behaviour of partial sums of discretised trawl processes and the functional convergence in distribution of trawl processes. 
Statistical inference for trawl processes has so far focused on fully parametric settings and includes the use of the (generalised) method of moments, see \cite{BNLSV2014, SheYan, VERAART2019},
likelihood methods, see \cite{Noven, Bacro, CourgeauVeraart2022, Shephard2016, BLSV2021}, and spectral methods, see \cite{Doukhan2020}.

\subsection{Main contributions of this article}

This article develops the first  nonparametric estimation theory for trawl processes.
It presents an estimator for the trawl function which is based on the idea that the derivative of the autocovariance function  and the trawl function are proportional to each other.
It then establishes the consistency of the estimator under mild regularity conditions, using a double-asymptotic scheme based on both in-fill and long-span asymptotics. 
Next, we derive suitable central limit theorems for the proposed estimator, taking various asymptotic relations between the in-fill and the long-span asymptotic regimes into account.
The resulting central limit theorems are infeasible, in the sense that they typically feature an asymptotic bias and/or an asymptotic variance that is not directly observable.
Hence, we propose new estimators for both the asymptotic bias and variance and prove their consistency.
This enables us to present feasible central limit theorems that can be implemented in practice.
As an application of our theoretical results, we develop suitable estimators for slices of trawl sets, which are disjoint subsets obtained via intersection and set-difference operations applied to trawl sets. Such slice estimators can be applied in the construction of nonparametric predictors of trawl processes. 
Since our asymptotic theory depends on a delicate balance between an in-fill and a long-span asymptotic regime, we assess the finite-sample performance of our proposed methods in detailed simulation experiments.
Finally, we provide three case studies demonstrating the practical use of the new methodology: First, we show how it can be used for model selection and misspecification testing. Second, we apply the new methodology to modelling and forecasting high-frequency financial data, consisting of spreads between the best bid and ask prices of a limit order book. Third, we relate integer-valued trawl processes to stochastic queues and demonstrate how the associated busy time distribution can be estimated within our framework. 

This article is accompanied by two software releases, complementing the new statistical methods presented: 
The {\tt R} package {\tt ambit}, see \cite{ambit},
has been published on CRAN, collecting all the relevant implementations of the new methodology as well as detailed documentation. In addition, the code used in the empirical study is made publicly available on GitHub 
and archived on Zenodo, see \cite{emp-code-trawl}.

\subsection{Organisation of the article}

The remainder of this article is structured as follows.
Section \ref{sec:background} contains background material on the definition of trawl processes and presents the construction of the estimator for the trawl function.
Section \ref{sec:results} lays out the asymptotic theory, which is proven in the Supplementary Material, see  Section \ref{sec:proofs}.
Section \ref{sec:Sim} discusses some of the results of our simulation study designed to assess the finite sample performance of the proposed estimation theory, more details are available in  the Supplementary Material, see Sections \ref{asec:simsetup}, \ref{asec:consistency},  \ref{asec:asymGauss}, \ref{asec:slices}.
Finally, Section 
\ref{sec:Emp} applies the new estimation methodology to model misspecification testing, high-frequency financial data from a limit order book and the busy time distribution of a stochastic queue.

\section{Background and construction of the estimators}\label{sec:background}

\subsection{\Levy bases and infinite divisibility}

Let $\eta$ be a measure on $\mathcal{B}(\mathbb{R}^{d})$, the Borel
sets on $\mathbb{R}^{d}$, and let $\mathcal{B}_{b}^{\eta}(\mathbb{R}^{d}):=\{A\in\mathcal{B}(\mathbb{R}^{d}):\eta(A)<\infty\}.$
The family $L=\{L\left(A\right):A\in\mathcal{B}_{b}^{\eta}(\mathbb{R}^{d})\}$
of real-valued random variables (r.v.s) will be called a \textit{\Levy basis} if it is
an infinitely divisible (ID for short) independently scattered random
measure, that is, $L$ is $\sigma$-additive almost surely and such
that for any $A,B\in\mathcal{B}_{b}^{\eta}(\mathbb{R}^{d})$, $L(A)$
and $L(B)$ are ID r.v.s that are independent whenever $A\cap B=\emptyset$.
The cumulant function of an r.v. $\xi$, in case it exists, will be denoted
by $\mathcal{C}(z\ddagger\xi):=\log\mathbb{E}(e^{iu\xi})$. We will
say that $L$ is \textit{separable} with \textit{control measure}
$\eta$, if 
$ 	\mathcal{C}(z\ddagger L(A))=\eta(A)\psi(z)$, $A\in\mathcal{B}_{b}^{\eta}(\mathbb{R}^{d})$, $z\in\mathbb{R}$, 
where 
$ 
\psi(z):=i\gamma z-\frac{1}{2}b^{2}z^{2}+\int_{\mathbb{R}\backslash\{0\}}(e^{izx}-1-izx\mathbf{1}_{\left|x\right|\leq1})\nu(\mathrm{d}x)$, $z\in\mathbb{R}$ 
with $\gamma\in\mathbb{R},$ $b\geq0$ and $\nu$ is a \Levy measure,
i.e.~$\nu(\{0\})=0$ and $\int_{\mathbb{R}\backslash\{0\}}(1\land\left|x\right|^{2})\nu(\mathrm{d}x)<\infty$.
When $\eta=\Leb$, in which $\Leb$ represents the Lebesgue measure on
$\mathbb{R}^{d}$, $L$ is called \textit{homogeneous}. The ID r.v.
associated to the characteristic triplet $\left(\gamma,b,\nu\right)$
is called the \textit{\Levy seed} of $L$ and will be denoted by $L'$.
As usual, $\left(\gamma,b,\nu\right)$ will be called the characteristic
triplet of $L$ and $\psi$ its characteristic exponent.

\subsection{Trawl processes and some of their properties}
\begin{definition}	\label{def:trawldef}
	Let $L$ be a homogeneous \Levy basis on $\mathbb{R}^{2}$ with characteristic
	triplet $(\gamma,b,\nu)$. In addition, let $a:\mathbb{R}^{+}\rightarrow\mathbb{R}^{+}$
	be a continuous, non-increasing, integrable function and define 
	$
	A=\left\{ (r,y):r\leq0,0\leq y\leq a(-r)\right\}$.
	The process defined by
	$X_{t}:=L(A_{t})$, for $t\in\mathbb{R}$, 
	where $A_{t}:=A+(t,0)$ (the addition here is coordinate-wise), is called a \textit{trawl process}. From
	now on, we will refer to $A$ and $a$, as the  \textit{trawl set} and the  \textit{trawl
		function}, respectively. 
\end{definition}
It is well known that $X$ is strictly stationary
and, in the case when $L$ is square integrable with Levy seed $L'$, its autocovariance function is given by
\begin{equation}
	\varGamma_{X}(h):=\Cov(X_{t+h},X_{t})=\Var(L')\int_{\left|h\right|}^{\infty}a(u)\mathrm{d}u,\,\,\,h\in\mathbb{R}.\label{TrawlACF}
\end{equation}
This relation implies that $\varGamma_{X}$ uniquely characterizes $a$ within the class of trawl processes represented by $L$. More precisely,
if $X$ and $\tilde{X}$ are two trawl
processes associated with  $L$ with trawls functions $a$ and $\tilde{a}$,
respectively, then $a=\tilde{a}$ a.e.~if and only if
$\varGamma_{X}=\varGamma_{\tilde{X}}$.

For a detailed exposition on the basic properties of trawl processes
we refer to \cite{BNLSV2014} and \citet[Chapter 8]{BNBV2018-book}.
In what follows, we focus on the estimation of $a$ under the assumption
that 
$	\Var(L^{\prime})=b^{2}+\int_{\mathbb{R}}x^{2}\nu(dx)=1$.

\subsection{Estimator construction}
We aim to estimate the trawl function $a$ based on equidistant observations
of the trawl process $X$, say $(X_{i\Delta_{n}})_{i=0}^{n-1}$. From equation \eqref{TrawlACF}, we deduce the relation (valid when $\Var(L^{\prime})=1$)
\[
\varGamma_{X}^{\prime}=-a, 
\]
on $[0,\infty)$ (at $0$ we consider the right-sided derivative of $\varGamma_{X}$)  which motivates the construction of our nonparametric estimator of
$a$: We can use an estimator for $\varGamma_{X}$ and then approximate
its derivative. We do this via the sample autocovariance function.
Specifically, by letting 
\[
\hat{\varGamma}_{l}:=\frac{1}{n}\sum_{k=0}^{n-1-l}(X_{(l+k)\Delta_{n}}-\bar{X}_{n})(X_{k\Delta_{n}}-\bar{X}_{n}),\,\,\,l=0,\ldots,n-1,
\]
in which $\bar{X}_{n}=\frac{1}{n}\sum_{k=1}^{n}X_{(k-1)\Delta_{n}},$ we
propose to estimate $a(t)$, for $t>0$, by 
\begin{align}
	\hat{a}(t) & =-\Delta_{n}^{-1}\left[\hat{\varGamma}_{l+1}-\hat{\varGamma}_{l}\right],\,\,\,\text{if }\Delta_{n}l\leq t<(l+1)\Delta_{n},\label{eq:a-est}
\end{align}
while for $t=0$, we suggest estimation via the averaged realised variance:
\begin{equation}
	\hat{a}(0)=\frac{1}{2\Delta_{n}n}\sum_{k=0}^{n-2}(\delta_{k}X)^{2},\,\,\,\delta_{k}X:=X_{(k+1)\Delta_{n}}-X_{k\Delta_{n}}.\label{eq:defa0hat}
\end{equation}
Although estimating $a(0)$ via (\ref{eq:a-est}) with $l=0$ may be more natural, our choice brings several benefits, as we briefly outline below. Denote by $\check{a}(0)$ the right-hand side of (\ref{eq:a-est}) for $t=0$ and observe that
\begin{equation}
   \check{a}(0)=-\frac{1}{n\Delta_{n}}\sum_{k=0}^{n-2}X_{k\Delta_{n}}(X_{(k+1)\Delta_{n}}-X_{k\Delta_{n}})+\mathrm{O}_{\mathbb{P}}(1/(n\Delta_{n}))=\hat{a}(0)+\mathrm{O}_{\mathbb{P}}(1/(n\Delta_{n})).  \label{hata0_dec}
\end{equation}
This shows that by choosing $\hat{a}(0)$ over $\check{a}(0)$ we reduce bias and the estimation error due to the sample mean. Furthermore, as we will see later, when the process is Gaussian, the error produced by $\check{a}(0)$ does not fulfil any type of second-order limit theorem. The estimator proposed in (\ref{eq:defa0hat}) avoids this situation in practice and makes our estimation procedure more robust.
\section{Asymptotic theory}\label{sec:results}

\subsection{Consistency}
It is clear that the larger $n$, the better our approximation for $\varGamma_{X}$ becomes. A similar conclusion holds when $\Delta_{n}$ is small. In consequence, it is evident that for our estimator to be consistent, the following basic assumption on the in-fill and long-span sampling scheme must be satisfied:

\begin{assumption}\label{as:basicsamplingscheme} We assume that, as $n\uparrow\infty$,
	$\Delta_{n}\downarrow0$ and $n\Delta_{n}\rightarrow+\infty.$ 
\end{assumption}
Under this set-up, our estimator $\hat a(t)$ is consistent:
\begin{theorem}[Consistency]\label{propconsistency}
	Let Assumption \ref{as:basicsamplingscheme} hold and suppose that $\Var(L^{\prime})=1$ and $\mathbb{E}(\mid L^{\prime}\mid^{4})<\infty$.
	Then, for all $t\geq 0$, 
	$\hat a(t) \stackrel{\mathbb{P}}{\to} a(t)$.
\end{theorem}

\subsection{Asymptotic Gaussianity -- An infeasible result}
In this part, we present a CLT for $\hat a(t)$. Not surprisingly, stronger assumptions are required, which we will introduce next.

First, we refine the sampling scheme introduced in Assumption \ref{as:basicsamplingscheme}.
\begin{assumption}\label{as:samplingscheme} We assume that, as $n\uparrow\infty$,
	$\Delta_{n}\downarrow0$ and $n\Delta_{n}\rightarrow+\infty.$ Furthermore,
	$n\Delta_{n}^{3}\rightarrow\mu\in[0,\infty]$. \end{assumption}
Second, we  impose the following assumption on the trawl
function: 
\begin{assumption}\label{as:trawl} The trawl function admits the
	representation 
	$
	a(s)=\int_{s}^{\infty}\phi(y)dy$, for $s\geq0$, 
	where $\phi$ is a strictly positive function that is continuous on $[0,\infty)$ and such
	that $\phi(s)=\mathrm{O}(s^{-\alpha-1})$ as $s\rightarrow+\infty$
	for some $\alpha>1$.
	
\end{assumption}
\begin{example}\label{ex:partrawl}
	Various choices for the trawl function are possible, see e.g.~\cite{BNLSV2014} for many parametric examples. In this article, we will be focusing on two particular examples for our simulation study and empirical work:
	\begin{enumerate}
		\item Exponential trawl function: Here we choose $a(s)=c\exp(-\lambda s)$, for $c, \lambda >0$, $s\geq 0$. In this case, we have $\phi(s)=c \lambda \exp(-\lambda s)$, for $s\geq 0$, and Assumption \ref{as:trawl} is satisfied for any $\alpha>1$.
		\item SupGamma trawl function: Here we choose $a(s)=c\left(1+\frac{s}{\overline{\alpha}}\right)^{-H}$, for $c, \overline{\alpha} >0, H>1$, $s\geq 0$. In this case, we have $\phi(s)=c\frac{H}{\overline{\alpha}}\left(1+\frac{s}{\overline{\alpha}}\right)^{-H-1}$, for $s\geq 0$; i.e.~we can choose $\alpha=H$ in Assumption \ref{as:trawl}. Note that, for $H \in (1, 2]$, the resulting trawl process has long memory and for $H>2$ short memory. 
	\end{enumerate}
\end{example}
Finally, for the \Levy seed associated with the \Levy basis we make the following
moment assumption:
\begin{assumption}\label{as:seed} $	\Var(L^{\prime})=1$
	and for some $p_{0}>4$ we have that $\mathbb{E}(\mid L^{\prime}\mid^{p_{0}})<\infty$. \end{assumption} 
The following result fully describes the second-order asymptotic theory of our proposed
estimator under the preceding set-up. 

\begin{theorem}\label{thmCLTtrawlapprox} Let Assumptions \ref{as:samplingscheme},
	\ref{as:trawl} and \ref{as:seed} hold. Then, for all $t\geq0$ as
	$n\rightarrow\infty$, we have the following convergence results.
	\begin{enumerate}
		\item If $\mu=0$, then 
		\[
		\sqrt{n\Delta_{n}}\left(\hat{a}(t)-a(t)\right)\overset{d}{\rightarrow}N(0,\sigma_{a}^{2}(t)),
		\]
		where, for $c_{4}(L'):=\int x^{4}\nu(dx)$, we have 
		\begin{align*}
			\sigma_{a}^{2}(t)= & c_{4}(L')a(t)+2\left\{ \int_{0}^{\infty}a(s)^{2}ds+\int_{0}^{\infty}a(\left|t-s\right|)a(t+s)\mathrm{sign}(t-s)ds\right\} .
		\end{align*}
		\item If $0<\mu<\infty$, then 
		\[
		\sqrt{n\Delta_{n}}\left(\hat{a}(t)-a( \lfloor  t/\Delta_{n}\rfloor\Delta_{n})\right)\overset{d}{\rightarrow}\frac{1}{2}\sqrt{\mu}a^{\prime}(t)+N(0,\sigma_{a}^{2}(t)).
		\]
		\item If $\mu=\infty$, then
		\[
		\frac{1}{\Delta_{n}}\left(\hat{a}(t)-a(\lfloor t/\Delta_{n}\rfloor\Delta_{n})\right)\overset{\mathbb{P}}{\rightarrow}\frac{1}{2}a^{\prime}(t).
		\]
	\end{enumerate}
\end{theorem}
\begin{remark}\label{rmkdiscerror} The following remarks are in order:
	\begin{enumerate}

		\item The asymptotic bias appearing in the previous theorem comes from the discretisation error created when one approximates $a(t)$ by
		\begin{equation}
			\frac{1}{\Delta_{n}}\int_{\lfloor t/\Delta_{n}\rfloor\Delta_{n}}^{\lfloor t/\Delta_{n}\rfloor\Delta_{n}+\Delta_{n}}a(s)ds=a(\lfloor t/\Delta_{n}\rfloor\Delta_{n})+\Delta_{n}\frac{1}{2}a^\prime(\lfloor t/\Delta_{n}\rfloor\Delta_{n})+\mathrm{o}(\Delta_{n}).\label{discerror} 
		\end{equation} 
		Thus, under Assumption \ref{as:samplingscheme}, this discretisation error is negligible as soon as $n\Delta_{n}^3\rightarrow0$. When this is not the case, the quantity in (\ref{discerror}) no longer approximates $a(t)$ but $a(\lfloor t/\Delta_{n}\rfloor\Delta_{n})$, and when $n\Delta_{n}^3\rightarrow+\infty$, it will dominate the asymptotic behaviour of $\hat{a}(t)$.
		\item One can unify the results of the previous theorem by centring around the discretisation error discussed above. More precisely, within the framework of Theorem \ref{thmCLTtrawlapprox}, it holds that
		\[
		\sqrt{n\Delta_{n}}\left(\hat{a}(t)- \frac{1}{\Delta_{n}}\int_{\lfloor t/\Delta_{n}\rfloor\Delta_{n}}^{\lfloor t/\Delta_{n}\rfloor\Delta_{n}+\Delta_{n}}a(s)ds\right)\overset{d}{\rightarrow}N(0,\sigma_{a}^{2}(t)).
		\]
		This result is robust to whether $n\Delta_{n}^3$ is convergent or not.
        \item Let us give an intuitive interpretation of the three scenarios appearing in Theorem \ref{thmCLTtrawlapprox}. Case 1 can be regarded as a classical in-fill asymptotic setting, where we have a fast sampling frequency relative to the sample length. Here, we get a classical Gaussian limit with no bias. In the moderate asymptotic case, Case 2, a discretisation bias term appears. In Case 3, we essentially consider a long-span asymptotic result, where we sample slowly, and the bias dominates the asymptotics. 

  \item We note that the results of  Theorem \ref{thmCLTtrawlapprox} do not depend on the memory of $X$ which is unexpected. In the next subsection we give an intuitive explanation of why this is the case. 
	\end{enumerate}
	
\end{remark} 
 \subsection{Intuition behind the main results}\label{explan_results}

 As we already mentioned above, the fact that the limit theorems obtained in \ref{thmCLTtrawlapprox} do not depend on the memory of the process is surprising. Typically, the second-order  asymptotics of the sample autocovariance function (referred to as SACF hereafter)
are influenced by whether the underlying process has short- or long-range dependence. Let us try to clarify how the aforementioned case (and similar situations as the one in \cite{KoulSurgailis01} and \cite{Hsing99}) differs from our setup. Let $t>0$ and suppose for simplicity that $L$ has mean $0$. Set  $l_{n}=\lfloor t/\Delta_{n}\rfloor$. For $n$ large enough, we can decompose as in \eqref{hata0_dec}
		\begin{equation}
		  			\hat{a}(t)-a(t)=-\frac{1}{n\Delta_{n}}\sum_{k=l_{n}}^{n-2}[X_{(k-l_{n})\Delta_{n}}\delta_{k}X-\mathbb{E}(X_{(k-l_{n})\Delta_{n}}\delta_{k}X)]+\mathrm{O}_{\mathbb{P}}(1/n\Delta_{n})+\mathrm{O}_{\mathbb{P}}(\Delta_{n}),\label{decomp_ahat}  
		\end{equation}
where $\delta_{k}X=X_{(k+1)\Delta_{n}}-X_{k\Delta_{n}}$, and the last error term corresponds to the discretisation error described in Remark \ref{rmkdiscerror}. As we can see, the increments of $X$ appear in the quadratic functional rather than $X$ itself, which contrasts with the SACF.  We will now show that this distinction is precisely what addresses the memory of $X$.
\begin{figure}
    \centering
    \includegraphics[width=0.6\linewidth]{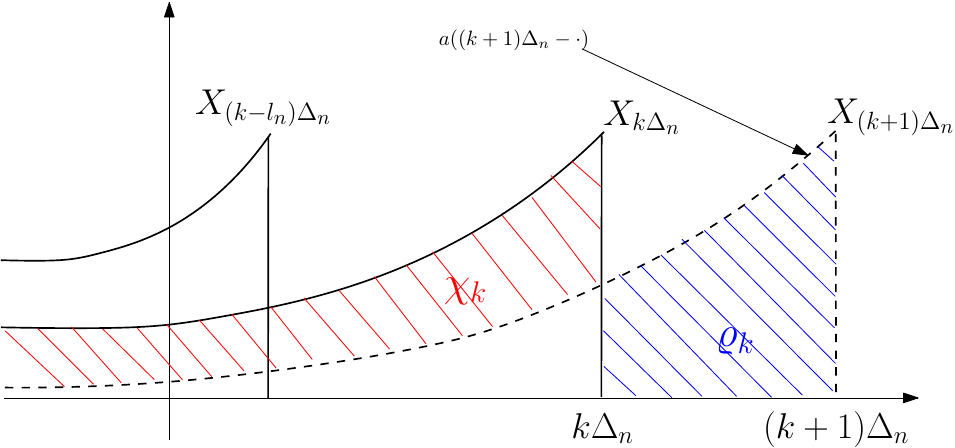}
    \caption{Graphical representation of the increments $X_{(k+1)\Delta_{n}}-X_{k\Delta_{n}}$.}
    \label{figincrementsX}
\end{figure}
By definition (see Figure \ref{figincrementsX})	
\[\delta_{k}X  =L(A_{(k+1)\Delta_{n}}\backslash A_{k\Delta_{n}})-L(A_{k\Delta_{n}}\backslash A_{(k+1)\Delta_{n}})
				 =:\varrho_{k}-\chi_{k}.\]
Furthermore, $\varrho_{k}$ is independent of $X_{(k-l_{n})\Delta_{n}}$. Thus, the estimation error equals (up to an $\mathrm{o}_{\mathbb{P}}(1)$ term)
			\[
\frac{1}{n\Delta_{n}}\sum_{k=l_{n}}^{n-2}[X_{(k-l_{n})\Delta_{n}}\chi_{k}-\mathbb{E}(X_{(k-l_{n})\Delta_{n}}\chi_{k})]-\frac{1}{n\Delta_{n}}\sum_{k=l_{n}}^{n-2}X_{(k-l_{n})\Delta_{n}}\varrho_{k}
				 =:S_{1}^{n}-S_{2}^{n}.
		\]
It is not difficult to verify  
that the array $\zeta_{k}^{n}=X_{(k-l_{n})\Delta_{n}}\varrho_{k}$ is a martingale difference. In consequence, we deduce, thanks to the stationarity of $X$ and the fact that $\mathbb{E}[L(B)^2]=\mathbb{E}[(L^{\prime})^2]\mathrm{Leb}(B)$, that as long as $L$ is squared integrable
\[\mathrm{Var}(S_{2}^{n})=\frac{1}{(n\Delta_n)^2}\sum_{k=l_{n}}^{n-2}\mathbb{E}[(\zeta_{k}^{n})^{2}]=\mathbb{E}[X_{0}^{2}]\mathbb{E}[\varrho_{0}^{2}]\frac{n-1-l_{n}}{(n\Delta_n)^2}=\mathrm{O}(1/(n\Delta_{n})).\]
Therefore, unlike in the classical case (SACF), the convergence rate of the statistic $S_{2}^{n}$ does not depend on the memory of $X$.  Similar arguments can be used to deduce that after normalising by the rate $\sqrt{n\Delta_n}$ and under appropriate moment conditions, $S_{2}^{n}$ converges to a Gaussian random variable, independently of whether $X$  has long memory or not.
$S_{1}^{n}$ can also be written as sums of martingale differences (although not under the same filtration as  $S_{2}^{n}$) but the representation is more involved. However, $\varrho_{k}$ and $\chi_{k}$ have the same distribution, so analogous computations as above can be used to check that, just as for $S_{2}^{n}$,  $\mathrm{Var}(S_{1}^{n})=\mathrm{O}(1/(n\Delta_{n}))$ and that $\sqrt{n\Delta_{n}}S_{1}^{n}$ converges in distribution towards a random variable with normal distribution and these properties are not constrained to the memory of $X$.

To support the results of Theorem \ref{thmCLTtrawlapprox} for the long memory case, we have plotted  a QQ-plot of the debiased and standardised error associated with $\hat{a}$ in Figure \ref{fig:CLTLM}.
\begin{figure}[htbp]
 \centering
	\captionsetup[subfigure]{aboveskip=-4pt,belowskip=-4pt}
 	\subfloat[	\label{fig:qq1}]{	\includegraphics[width=0.4\textwidth, height=3.5cm]{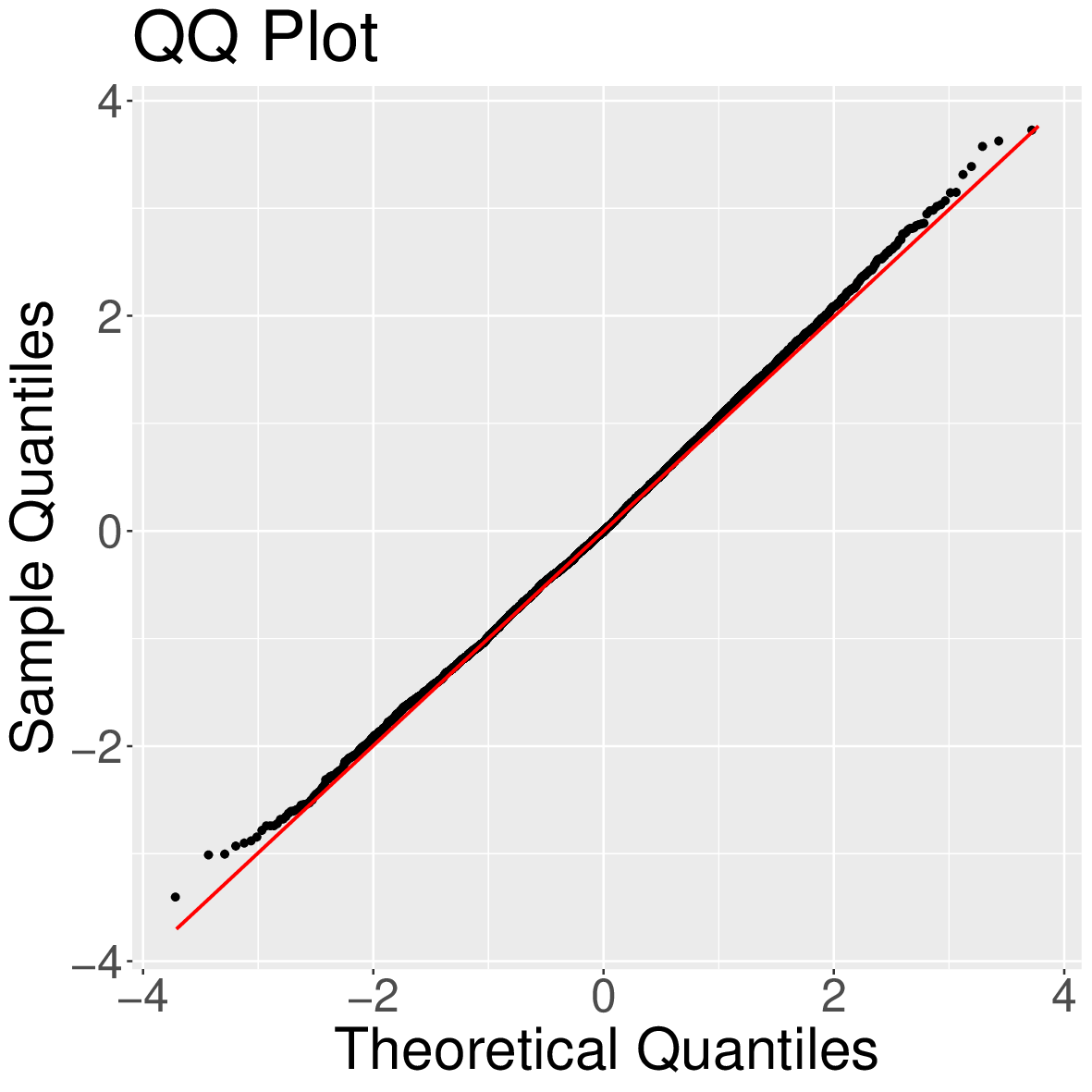} } 
	\subfloat[	\label{fig:qq2}]{	\includegraphics[width=0.4\textwidth, height=3.5cm]{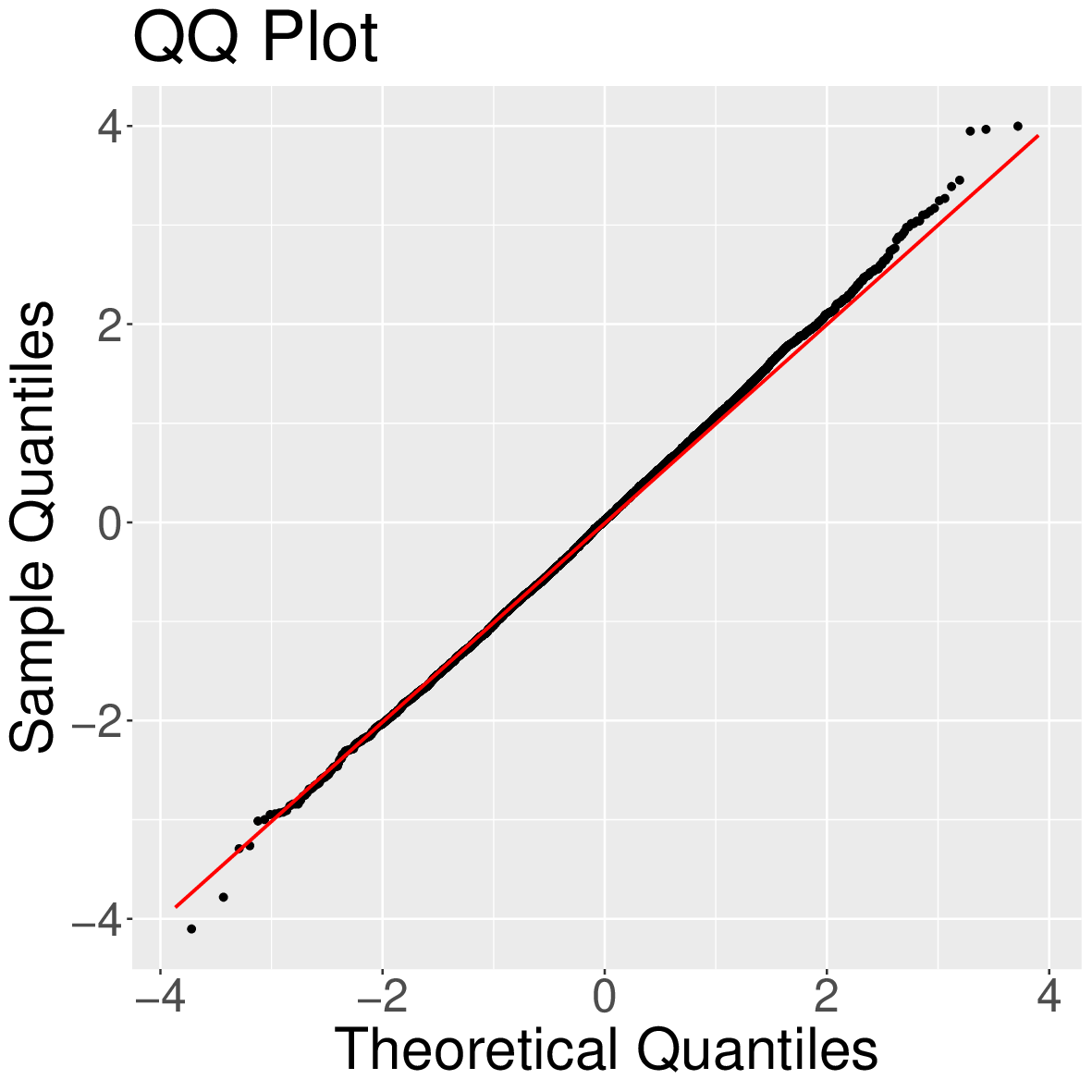} } 
	\caption{\it QQ plot of the debiased and standardised estimation error $\sqrt{n\Delta_{n}}\frac{\left(\hat{a}(1)-a(1)-\frac{1}{2}\Delta_n a'(1)\right)}{\sqrt{\sigma_a^2(1)}}$ based on $5000$ Monte Carlo repetitions. In our simulations, the L\'evy seed followed a standard normal distribution and we used a supGamma trawl $a(s)=(1+s)^{-H}$ with $H=1.5$. Furthermore, the distance between observations was set to $\Delta_n=n^{-1/3}$ with $n$ denoting the sample size, where $n=5000$ in Figure \ref{fig:qq1} and $n=10000$ in Figure \ref{fig:qq2}. 
}
	\label{fig:CLTLM}
\end{figure}

\subsubsection{The case when $t=0$}
The reader might have noticed that $\sigma_{a}^{2}(0)$ vanishes when
$L$ is purely Gaussian (i.e. $c_{4}(L')=0$), so a degenerate limit appears in the case $\mu<\infty$. To get an intuition of why this is the case, first, note that 
\[
\hat{a}(0)-a(0)=\underbrace{\frac{1}{2\Delta_{n}n}\sum_{k=0}^{n-2}\left\{ (\delta_{k}X)^{2}-\mathbb{E}((\delta_{k}X)^{2})\right\} }_{\text{Statistical Error}}+\underbrace{\mathrm{O}_{\mathbb{P}}(\Delta_{n})}_{\text{Discretisation Error}}.
\]
Now, according to \cite{BasseRosis16}, when $L$ is Gaussian, $X$
can be decomposed as the sum of a Brownian motion with parameter $a(0)$
plus a continuous process of finite variation. This implies that for
$\Delta_{n}$ small, the increments of $X$ are comparable with those
of a Brownian motion. Thus, thanks to the self-similar property of
the latter, we heuristically deduce that 
\[
\hat{a}(0)-a(0)\approx\mathrm{O}_{\mathbb{P}}(1/\sqrt{n})+\mathrm{O}_{\mathbb{P}}(\Delta_{n}).
\]
This formal approximation reveals that the asymptotic properties
of $\hat{a}(0)$ in the Gaussian case will depend on the limiting behaviour
of $n\Delta_{n}^{2}$. Therefore, our study for this particular case
is conducted under the following  assumption: 

\begin{assumption}\label{as:samplingschemeGaussiant0} We assume
	that, as $n\uparrow\infty$, $\Delta_{n}\downarrow0$ and $n\Delta_{n}\rightarrow+\infty.$
	Furthermore, $n\Delta_{n}^{2}\rightarrow\mu_{0}\in[0,\infty]$.
	
\end{assumption} 

In the next result, we will use the so-called quarticity, which is denoted as
\begin{equation}
	Q_{n}:=\frac{1}{2\Delta_{n}n}\sum_{k=0}^{n-2}(\delta_{k}X)^{4}.\label{quarticity}
\end{equation}

\begin{theorem}\label{thmCLTt0Gaussian} Let Assumptions \ref{as:samplingscheme} (with $\mu<\infty$),
	\ref{as:trawl} and \ref{as:seed} hold. As long as $c_{4}(L')>0$,
	it holds that
	\begin{align}\label{eq:a0gen}
		\sqrt{\frac{n\Delta_{n}}{Q_{n}}}\left(\hat{a}(0)-a(0)\right)\overset{d}{\rightarrow}N(0,1)+\frac{a^{\prime}(0)}{2}\sqrt{\frac{\mu}{c_{4}(L')a(0)}},\,\,\,\text{as }n\rightarrow\infty.
	\end{align}
	If $L$ is purely Gaussian, assume, in addition, that Assumption \ref{as:samplingschemeGaussiant0}
	holds. Then, we have the following limits as $n\rightarrow\infty$:
	\begin{enumerate}
		\item When $0\leq\mu_{0}<\infty$, it holds that 
		\begin{align}\label{eq:a0pureGauss}
			\sqrt{\frac{n\Delta_{n}}{Q_{n}}}\left(\hat{a}(0)-a(0)\right)\overset{d}{\rightarrow}N(0,1/3)+\frac{a^{\prime}(0)}{2}\sqrt{\frac{\mu_{0}}{6a(0)^{2}}}.
		\end{align}
		\item In the case when $\mu_{0}=\infty$, we have that 
		\[
		\frac{1}{\Delta_{n}}\left(\hat{a}(0)-a(0)\right)\overset{\mathbb{P}}{\rightarrow}\frac{1}{2}a^{\prime}(0).
		\]
	\end{enumerate}	
\end{theorem}

The central limit theorems presented above are, in general,  \emph{infeasible} in the sense that the asymptotic variance and asymptotic bias (if present) are not known a priori.
Hence, they need to be estimated as well to obtain a \emph{feasible} asymptotic theory which can be used in practice. 
Only in the cases of equations \eqref{eq:a0gen} and \eqref{eq:a0pureGauss}, when $\mu=0$ and $\mu_0=0$, the resulting limit theorems in  are feasible and can be implemented in practice.

We conclude this part with a remark.
\begin{remark}  It is very natural to ask about second-order asymptotics in Theorem \ref{thmCLTtrawlapprox} (\ref{thmCLTt0Gaussian}) when $n\Delta_{n}^3\rightarrow+\infty$ ($n\Delta_{n}^2\rightarrow+\infty$). It turns out that this will be dependent on the relationship between how fast $n\Delta_{n}^3$ ($n\Delta_{n}^2$) explodes and how smooth the trawl function is. Since the analysis in this situation is based on estimators for high order derivatives of $a$, we have decided to leave this problem for future research. Therefore, for the rest of this article, we concentrate on the case when $n\Delta_{n}^3\rightarrow\mu\in[0,\infty)$ ($n\Delta_{n}^2\rightarrow\mu_0\in[0,\infty)$).
\end{remark}

\subsection{Asymptotic Gaussianity -- Towards a feasible theory}

In order to be able to use the asymptotic results presented in Theorems
\ref{thmCLTtrawlapprox} and \ref{thmCLTt0Gaussian}, we need to replace
the asymptotic bias and variance (from
now on ABI and AVAR, respectively) by a suitable estimator. Our methodology (described
below) requires two additional tuning parameters  $K_{n}, N_{n}\in\mathbb{N}$. The tuning parameter $K_n$ features in the estimator of  ABI and  $N_n$ in the AVAR estimator.

\begin{assumption}\label{as:Knprop}
	The sequence $(K_n)$ taking values in $\mathbb{N}$ satisfies the following two conditions as $n\rightarrow\infty$:
	$i) \, K_{n}\uparrow+\infty$, and  $ii) \,K_{n}\Delta_{n}\rightarrow0$.
\end{assumption}

\begin{assumption}\label{as:Nnprop}
	The sequence $(N_n)$ taking values in $\mathbb{N}$ satisfies the following two conditions as $n\rightarrow\infty$:
	$i) \, N_{n}\Delta_{n}\rightarrow+\infty$, and $ii) \, N_{n}/n\rightarrow0$. 
\end{assumption} 
Note that the sequence $N_n$ in the previous assumption necessarily needs to diverge towards $+\infty$ due to the fact that $\Delta_n\rightarrow0$.

\subsubsection{Estimating the ABI}
Since the ABI involves the derivative of $a$, in this part, we propose
a procedure for estimating such a function. A natural estimator
within our framework is
\[\hat{a}^{\prime}_d(t)=\frac{1}{\Delta_{n}}[\hat{a}(t+\Delta_{n})-\hat{a}(t)].\]
As in the case of $\hat{a}(0)$, we can modify $\hat{a}^{\prime}_d$ in order to reduce the bias as well as the sample mean error. To do this, first note that by arguing as in \eqref{decomp_ahat}, $\hat{a}^{\prime}_d(t)$ equals (up to an $\mathrm{O}_{\mathbb{P}}(1/(n\Delta_{n}^2))$ term)
\[\frac{1}{n\Delta_{n}^2}\left(\sum_{k=l_{n}}^{n-2}X_{(k-l_{n})\Delta_{n}}\delta_{k}X-\sum_{k=l_{n}+1}^{n-2}X_{(k-l_{n}-1)\Delta_{n}}\delta_{k}X\right)
 =\frac{1}{n\Delta_{n}^{2}}\sum_{k=l_n+1}^{n-2}\delta_{k}X\cdot\delta_{k-l_n-1}X +\mathrm{O}_{\mathbb{P}}(1/(n\Delta_{n}^2)),\]
where $l_n=\lfloor t/\Delta_n\rfloor$. Based on this, we propose to estimate $a^{\prime}$ as
\begin{equation}
	\hat{a}^{\prime}(t)=\frac{1}{n\Delta_{n}^{2}}\sum_{k=l+1}^{n-2}\delta_{k}X\cdot\delta_{k-l-1}X,\,\,\,\text{if }\Delta_{n}l\leq t<(l+1)\Delta_{n}.\label{eq:aderdef}
\end{equation}
Observe that, unlike $\hat{a}^{\prime}_d(t)$, the computation of $\hat{a}^{\prime}(t)$ does not requires to demean our sample. In the Supplementary Material, Section \ref{sec:proofs},  
we show the following result.

\begin{proposition}\label{ABIconv} Suppose that Assumptions  \ref{as:samplingscheme}, \ref{as:trawl} and
	\ref{as:seed}
	hold. Let $\hat{a}^{\prime}(t)$ be as in (\ref{eq:aderdef}). Then, for all $t\geq0$,
	\[ \hat{a}^{\prime}(t)=a^{\prime}(t)+\mathrm{O}_{\mathbb{P}}\left(1/\sqrt{n\Delta_{n}^{2}}\right)+\mathrm{o}(1),\,\,\,\text{as }n\rightarrow\infty. \]\end{proposition}

Note that the ABI is only present when either $n\Delta_{n}^{3}\rightarrow\mu\in(0,\infty)$
or when the process is Gaussian and $n\Delta_{n}^{2}\rightarrow\mu_{0}\in(0,\infty)$.
In the former situation $\hat{a}^{\prime}(t)$ is consistent while in the latter is not.
Nevertheless, if we have that $n\Delta_{n}^{2}\rightarrow\mu_{0}\in(0,\infty)$,
we can, by means of Assumption \ref{as:Knprop}, construct a consistent estimator for $a^\prime$ by sampling ``less'' frequently as we describe next: Put
\[
\tilde{\Delta}_{n}:=K_{n}\Delta_{n};\,\,\,M_{n}:=\lfloor(n-1)/K_{n}\rfloor,
\]
where $K_{n}$ as in Assumption \ref{as:Knprop}. Note that by construction
$\tilde{\Delta}_{n}\rightarrow0$. Furthermore, since  $n\Delta_n\rightarrow\infty$, condition ii) in Assumption \ref{as:Knprop} further implies that $n/K_n\rightarrow\infty$. Hence, $M_{n}\rightarrow+\infty$
and $M_{n}\tilde{\Delta}_{n}\rightarrow+\infty$. In consequence, 
when $n\Delta_{n}^{2}\rightarrow\mu_{0}\in(0,\infty)$, the following sequence (constructed using the sub-sample $(X_{i\tilde{\Delta}_{n}})_{i=0,\ldots,M_n}$) 
\begin{equation}
	\tilde{a}^{\prime}(t)=\frac{1}{M_{n}\tilde{\Delta}_{n}^{2}}\sum_{k=l+1}^{M_{n}-1}\tilde{\delta}_{k}X\cdot\tilde{\delta}_{k-l-1}X,\,\,\,\text{if }\tilde{\Delta}_{n}l\leq t<(l+1)\tilde{\Delta}_{n},\label{eq:derivasubsample}
\end{equation}
where 
\[
\tilde{\delta}_{k}X:=X_{(k+1)\tilde{\Delta}_{n}}-X_{k\tilde{\Delta}_{n}},\,\,k=0,\ldots,M_{n}-1,
\]
satisfies, due to Proposition \ref{ABIconv}, that 
\[ 	\tilde{a}^{\prime}(t)=a^{\prime}(t)+\mathrm{O}_{\mathbb{P}}(1/\sqrt{K_n})+\mathrm{o}(1),\,\,\,\text{as }n\rightarrow\infty. \]

\subsubsection{Feasible theory for $t=0$}
\begin{theorem}\label{thm:CLTfeasible-1} Let Assumptions \ref{as:samplingscheme} (with $\mu<\infty$),
	\ref{as:trawl} and \ref{as:seed} hold. If $c_{4}(L')>0$,
	then 
	\[
	\sqrt{\frac{n\Delta_{n}}{Q_{n}}}\left(\hat{a}(0)-a(0)-\frac{1}{2}\Delta_{n}\hat{a}^{\prime}(0)\right)\overset{d}{\rightarrow}N(0,1).
	\]
	If $c_{4}(L')=0$, assume in addition that Assumptions \ref{as:samplingschemeGaussiant0} (with $\mu_{0}<\infty$)
	and \ref{as:Knprop} hold. Then, as $n\rightarrow\infty$ 
	\[
	\sqrt{\frac{n\Delta_{n}}{Q_{n}}}\left(\hat{a}(0)-a(0)-\frac{1}{2}\Delta_{n}\tilde{a}^{\prime}(0)\right)\overset{d}{\rightarrow}N(0,1/3),
	\]
	in which $\tilde{a}^{\prime}(0)$ is defined as in (\ref{eq:derivasubsample}).
	
\end{theorem}	
Although the previous theorem is feasible, \textit{valid inference on $a(0)$ using $\hat{a}(0)$ still depends on whether the data is Gaussian or not}. A fortiori, our proofs suggest the use of $Q_n$ to address this issue. As shown in  the Supplementary Material, Subsection \ref{sec_proof_t0}, under our assumptions, $Q_{n}\overset{\mathbb{P}}{\rightarrow}c_{4}(L')a(0)$. Thus, unless $X$ is trivial, $Q_n$ is asymptotically negligible if, and only if, $X$ is a Gaussian process, in which case $Q_{n}/\Delta_{n}\overset{\mathbb{P}}{\rightarrow}6a(0)^{2}$. Consequently, $T_{n}:=Q_{n}/(\Delta_{n}6\hat{a}(0)^{2})\overset{\mathbb{P}}{\rightarrow}1$ if, and only if, the process is Gaussian; otherwise, $T_{n}\overset{\mathbb{P}}{\rightarrow}+\infty$. Therefore, a large value for $T_n$ suggests that the data is unlikely to be normally distributed.
\begin{remark}
  To rigorously use $T_n$ for testing Gaussianity, a joint Central Limit Theorem for $\hat{a}(0)$ and $Q_n$ is needed. This requires additional analysis, which we postpone to future research.
\end{remark}

\subsubsection{Estimating the AVAR}
Next, we focus on estimating  the AVAR given by  
\[
\sigma_{a}^{2}(t)=c_{4}(L')a(t)+2\left\{ \int_{0}^{\infty}a(s)^{2}ds+\int_{0}^{t}a(t-s)a(t+s)ds-\int_{t}^{\infty}a(s-t)a(t+s)ds\right\},
\]
for $t>0$ (the case $t=0$ is already covered in Theorem \ref{thm:CLTfeasible-1}). 

For each of the four components of $\sigma_{a}^{2}(t)$, we construct an estimator  based on the observations
$(X_{i\Delta_{n}})_{i=0}^{n-1}$. The tuning parameter $N_n$ satisfying Assumption \ref{as:Nnprop} features in the estimates of the integrals of the trawl functions.
\begin{description}
	
	\item [{Step 1:}] Estimate $c_{4}(L')a(t)$: In  the proof
	of Theorem \ref{thmCLTt0Gaussian} we verify that
	$Q_{n}$ is consistent for $c_{4}(L')a(0)$. Thus, we can estimate $c_{4}(L')$ as 
	$\widehat{c_{4}(L')}:=Q_{n}/\hat{a}(0)>0$.
	We set 
	\[
	\hat{v}_{1}(t):=\widehat{c_{4}(L')}\hat{a}(t).
	\]

	\item [{Step 2:}] 
	We use a step-function approximation, i.e.: $2\int_{0}^{\infty}a(s)^{2}ds\approx2\int_{0}^{N_{n}\Delta_{n}}a(s)^{2}ds$
	is estimated by 
	\[
	\hat{v}_{2}(t):=2\sum_{l=0}^{N_{n}}\hat{a}^{2}(l\Delta_{n})\Delta_{n}.
	\]
	
	\item [{Step 3:}] For $i=1,\ldots,n-1$, $\Delta_{n}i\leq t<(i+1)\Delta_{n}$,
	we approximate $2\int_{0}^{t}a(t-s)a(t+s)ds\approx2\int_{0}^{i\Delta_{n}}a(i\Delta_{n}-s)a(i\Delta_{n}+s)ds$
	by 
	\[
	\hat{v}_{3}(t):=2\sum_{l=0}^{\min\{i,n-1-i\}}\hat{a}((i-l)\Delta_{n})\hat{a}((i+l)\Delta_{n})\Delta_{n}.
	\]
	\item [{Step 4:}] For $i=0,\ldots,n-1$,	  $\Delta_{n}i\leq t<(i+1)\Delta_{n}$,
	we approximate $-2\int_{t}^{\infty}a(s-t)a(t+s)ds\approx-2\int_{i\Delta_{n}}^{\infty}a(s-i\Delta_{n})a(i\Delta_{n}+s)ds\approx-2\int_{i\Delta_{n}}^{N_{n}\Delta_{n}}a(s-i\Delta_{n})a(i\Delta_{n}+s)ds$
	by 
	\[
	\hat{v}_{4}(t):=-2\sum_{l=i}^{N_{n}-i}\hat{a}((l-i)\Delta_{n})\hat{a}((i+l)\Delta_{n})\Delta_{n}.
	\]
	Note that as soon as $N_{n}-i<i$, we use the convention that the sum
	is set to 0. 
\end{description}
So, overall we have that, for $t>0$
\begin{align}
	\widehat{\sigma_{a}^{2}}(t):=\widehat{\sigma_{a}^{2}}(i\Delta_{n})=\hat{v}_{1}(t)+\hat{v}_{2}(t)+\hat{v}_{3}(t)+\hat{v}_{4}(t).
	\label{eq:sigma2a-est}
\end{align}
Next, we verify the consistency of $\widehat{\sigma_{a}^{2}}(t)$.
\begin{proposition}\label{AVARconv} Suppose that Assumptions \ref{as:samplingscheme} (with $\mu<\infty$),
	\ref{as:trawl}, \ref{as:seed}, and \ref{as:Nnprop} hold. Then, for all $t>0$, as $n\rightarrow\infty$, $	\widehat{\sigma_{a}^{2}}(t)\overset{\mathbb{P}}{\rightarrow}\sigma_{a}^{2}(t).$
\end{proposition}

\begin{remark}
	Note that we used a simple step function approximation of the integrals. Alternative numerical schemes, e.g.~a trapezoidal rule, could be used, but our simulation study revealed little difference in such choices and hence we will not further modify the proposed simple estimator.
\end{remark}
\subsubsection{Feasible theory for $t>0$}
Based on the results of Propositions \ref{ABIconv} and \ref{AVARconv}, we can now formulate feasible versions of Theorem \ref{thmCLTtrawlapprox}
as follows:

\begin{theorem}\label{thm:CLTfeasible} Let Assumptions \ref{as:samplingscheme} (with $\mu<\infty$),
	\ref{as:trawl} and \ref{as:seed} hold along with Assumption \ref{as:Nnprop}. Then, for all
	$t>0$ as $n\rightarrow\infty$, we have the following convergence results:
	\begin{enumerate}
		\item If $\mu=0$, then
		\[
		\sqrt{\frac{n\Delta_{n}}{\widehat{\sigma_{a}^{2}}(t)}}\left(\hat{a}(t)-a(t)-\frac{1}{2}\Delta_{n}\hat{a}^{\prime}(t)\right)\overset{d}{\rightarrow}N(0,1),
		\]
		in which $\hat{a}^{\prime}(t)$ is given by (\ref{eq:aderdef}) and $\widehat{\sigma_{a}^{2}}(t)$ as in \eqref{eq:sigma2a-est}.
		\item If $0<\mu<\infty$, then 
		\[
		\sqrt{\frac{n\Delta_{n}}{\widehat{\sigma_{a}^{2}}(t)}}\left(\hat{a}(t)-a(\lfloor t/\Delta_{n}\rfloor\Delta_{n})-\frac{1}{2}\Delta_{n}\hat{a}^{\prime}(t)\right)\overset{d}{\rightarrow}N(0,1).
		\]
	\end{enumerate}
\end{theorem}

\begin{remark}
 We would like to emphasise that the previous result does not require the consistency of $\hat{a}^{\prime}(t)$. Indeed, suppose that Assumption \ref{as:samplingscheme} holds. Then,
\begin{enumerate}
   \item If $n\Delta_{n}^{3}\rightarrow0$, $\hat{a}^{\prime}(t)$ does not converge in general. Nevertheless, from Proposition \ref{ABIconv}
\[\sqrt{n\Delta_{n}}\Delta_{n}\hat{a}^{\prime}(t)=\sqrt{n\Delta_{n}^{3}}a^{\prime}(t)+\mathrm{O}_{\mathbb{P}}(\sqrt{\Delta_{n}})+\mathrm{o}(1)=\mathrm{o}_{\mathbb{P}}(1).\]
This implies that, in this case, debiasing $\hat{a}$ is asymptotically insignificant. However, our simulation study demonstrates that it can still be advantageous in finite samples.
\item If $n\Delta_{n}^{3}\rightarrow\mu\in(0,\infty)$, it follows that $n\Delta_{n}^{2}\rightarrow+\infty$. Proposition \ref{ABIconv} guarantees the consistency of $\hat{a}^{\prime}(t)$.
\end{enumerate}

\end{remark}

\subsection{Application to slice estimation and nonparametric forecasting of trawl processes}\label{sec:TheoryForecasting}	
Let $\mathcal{F}_t = \sigma((X_s)_{s \leq t})$, $t\geq 0$,  and suppose we wish to forecast the trawl process, where  $h>0$ denotes the  forecast horizon. 
\cite{BNLSV2014} discussed that, due to the non-Markovianity of trawl processes, the  predictive distribution $X_{t+h}|\mathcal{F}_t$
is intractable and suggests approximating it by the conditional distribution
$X_{t+h}|X_t$ instead.
Recently, \cite{DL2022} showed the following result:
\begin{align}\label{eq:predictor}
	\mathbb{E}(X_{t+h}|X_t)=\frac{\Leb(A\cap A_h)}{\Leb(A)}X_t+\frac{\Leb(A\setminus A_h)}{\Leb(A)}\mathbb{E}(X_t),
\end{align}
see \cite{BLSV2021} for examples of integer-valued trawl processes. 

Again, consider $n$ observations $X_0, \ldots, X_{(n-1)\Delta_n}$. Suppose that we would like to predict at time $t=(n-1)\Delta_n$ for time $t+h$. 
Using our estimator for the trawl function, we can now define a predictor $\hat X_{t+h|t}$ as follows. 
\begin{align}\label{eq:predictor-est}
	\hat X_{t+h|t}:=	\frac{\widehat{\Leb(A\cap A_h)}}{\widehat{\Leb(A)}}X_t+\frac{\widehat{\Leb(A\setminus A_h)}}{\widehat{\Leb(A)}}\overline X,
\end{align}
where 	$\overline X= \frac{1}{n}\sum_{l=0}^{n-1}X_{l\Delta_n}$,
\begin{multline}\label{eq:stepfctapprox}
		\widehat{\Leb(A)}=	\sum_{l=0}^{n-1}\hat a(l\Delta_n)\Delta_n, \quad  
		%
		\widehat{	\Leb(A\cap A_h)} =\sum_{\substack{l\in \{0, \ldots, n-1\}:\\ l\Delta_n \geq h}} \hat a(l\Delta_n)\Delta_n,
		\\
		\widehat{\Leb(A\setminus A_h)}=\sum_{\substack{l\in \{0, \ldots, n-1\}:\\ l\Delta_n < h}} \hat a(l\Delta_n)\Delta_n.
\end{multline}
We can justify the choice of the estimators by the following consistency result whose proof follows from the fact that $X$ is mixing (due to Corollary 1 in \cite{RosZak96}) and that\\ $\mathbb{E}[(X_t-X_s)^2] \leq C\mid t-s\mid$, for $t,s\in \mathbb{R}$.
\begin{proposition}\label{prop:consistencyforecast}
	Let the assumptions of Theorem \ref{propconsistency} hold. Then, for all $h>0$, the statistics in (\ref{eq:stepfctapprox}) are consistent for $\Leb(A)$,  $\Leb(A\cap A_h)$ and $\Leb(A\setminus A_h)$, respectively.	
\end{proposition}

\begin{remark}
	Note that when estimating $\Leb(A)$ and the slices 	$\Leb(A\cap A_h)$ and $\Leb(A\setminus A_h)$, we only consider linear functionals of the trawl function as integrands. In particular, by the definition of $\hat{a}$ and the ergodicity of the process $X$ we deduce that
	\[ \widehat{\Leb(A\cap A_h)}=\hat{\varGamma}_{\lfloor h/\Delta_{n}\rfloor}+ \mathrm{o}_{\mathbb{P}}(1)=\varGamma_X(h)+ \mathrm{o}_{\mathbb{P}}(1).\]
	A similar relation holds for $\widehat{\Leb(A)}$ and $\widehat{\Leb(A\setminus A_h)}$. Here, the situation is simpler, and hence a tuning parameter is not required. In contrast, when estimating the AVAR, we had to consider non-linear functionals of $\hat{a}$. In this case, to the best of our knowledge, we cannot directly apply the Ergodic Theorem which forced us to introduce the tuning parameter $(N_n)$.

\end{remark}

When studying the expression for the conditional expectation  in \eqref{eq:predictor}, we note that we have the equivalent representation given by 
\begin{align*}
	\mathbb{E}(X_{t+h}|X_t)
		&= \mathrm{Cor}(X_t, X_{t+h})X_t+ (1-\mathrm{Cor}(X_t, X_{t+h}))\mathbb{E}(X_t).
\end{align*}
This suggests an alternative estimator of the form 
\begin{multline} \label{eq:acf-basedest}		\widehat{\Leb(A)}=\widehat{\Var(X)}=\hat \Gamma_0,\quad 
		%
		\widehat{	\Leb(A\cap A_h)} =\widehat{\mathrm{Cov}(X_0, X_h)}=\hat \Gamma_{\lfloor h/\Delta_n\rfloor},
		\\
		\widehat{\Leb(A\setminus A_h)}=\widehat{\Var(X)}-\widehat{\mathrm{Cov}(X_0, X_h)}
		=\hat \Gamma_0-\hat \Gamma_{\lfloor h/\Delta_n\rfloor},
\end{multline}
where $\widehat{\Var(X)}$ and $\widehat{\mathrm{Cov}(X_0, X_h)}$
denote the empirical variance and autocovariance, respectively.
Hence an alternative  predictor is given by 
\begin{align}\label{eq:predictor2}
	\hat X_{t+h|t}^{\mathrm{ACF}}:=	\frac{\hat \Gamma_{\lfloor h/\Delta_n\rfloor}}{\hat \Gamma_0}X_t+\left(1-\frac{\hat \Gamma_{\lfloor h/\Delta_n\rfloor}}{\hat \Gamma_0}\right)\overline X.
\end{align}
Although the two predictors are similar, they are not identical.
We will compare the two alternative slice ratio estimation methods in   the subsequent simulation study and compare the respective performance of both predictors in the empirical study.

\section{Simulation study}\label{sec:Sim} 
To assess the finite sample performance of our new methodology, we ran extensive simulation studies, which are described in detail in the Supplementary material, see \cite{SV2022-sup}.

We simulate trawl processes with three different marginal distributions (negative binomial, Gamma and Gaussian) and two choices of trawl functions (exponential and supGamma). The parameters are chosen such that all processes have the same mean and unit variance. While the exponential trawl function allows for a possibly relatively fast decay in the corresponding autocorrelation function, the parameters of the supGamma trawl are chosen such that the resulting process exhibits long memory, in the sense that the integral of its autocorrelation function equals infinity.

Throughout our simulation experiments, we assume that $\mu=0$ (and $\mu_0=0$) since this is possibly the situation of most practical relevance: In practice, we would typically observe data on a fixed frequency without being able to assess the relation between $\Delta_n$ and $n$. Hence, we vary both $\Delta_n$ and $n$ and evaluate the performance of the proposed estimators for various settings of the sampling scheme. 

\subsection{Consistency results}
\subsubsection{Trawl function estimation}
First of all, we consider the trawl function estimator $\widehat a(\cdot)$ and its bias-corrected version $\widehat a(\cdot)-0.5 \Delta_n \widehat a'(\cdot)$, which we evaluate for various points in time. We depict the boxplots of the two estimators over 1000 Monte Carlo runs for various choices of $\Delta_n$ and various choices of the sample size $n$, see Section \ref{asec:consistency} in the Supplementary Material. 
Here, we summarise our main findings for the case of a negative binomial marginal distribution with an exponential trawl and a supGamma trawl.
We observe that, for an exponential trawl, the estimation bias is generally small for both estimators, with the one for the bias-corrected version, as expected, being even smaller. 
Also, for a  supGamma trawl function with a parameter choice in the long-memory regime, the consistency results are generally good, with the bias-corrected version performing better than the original one. 
We also observe that the particular choice of $n$ has almost no impact on the size of the bias, but an increase in $n$ decreases the standard deviation of the estimators as expected.
Overall, we conclude that the proposed estimators work well in finite samples across different marginal distributions, short- and long-memory settings and different in-fill and long-span asymptotic settings.
In the long-memory case, we recommend to use the bias-corrected estimator, in particular for small times $t$.

\subsubsection{Consistency results of the slice (ratio) estimators}
Next, we assess
how well we can estimate the trawl set, its slices and the corresponding ratios given by
$\mathrm{\Leb}(A)$, $\mathrm{\Leb}(A\cap A_h)$, 
$\mathrm{\Leb}(A \setminus A_h)$, $\mathrm{\Leb}(A\cap A_h)/\mathrm{\Leb}(A)$ and  $\mathrm{\Leb}(A \setminus A_h)/\mathrm{\Leb}(A)$.
Detailed results are available in
Section \ref{asec:slices} in the Supplementary material. 
In this part of the simulation study, we considered the asymptotic setting closely related to our empirical study, by setting 
$\Delta_n=0.1$ and $n=5000$.
We considered three different methods for estimating these integrals of the trawl function: Using a Riemann approximation of the integrals based on the original trawl function estimator as in \eqref{eq:stepfctapprox}, using the Riemann approximation based on the bias-corrected trawl function estimator and using the empirical variance and autocovariance estimator presented in \eqref{eq:acf-basedest}. 
There are three main lessons we learned from this part of the simulation study: First, the trawl set, the slices and the corresponding ratios can be estimated with high precision. Second,  the original trawl function estimator performs better than its bias-corrected counterpart for the exponential trawl function, whereas the picture is reversed for the supGamma trawl function. 
Third, the empirical variance and autocovariance-based estimator and the trawl-function-based estimators perform very similarly.

\subsection{Asymptotic Gaussianity results}
We also investigated how closely
the infeasible statistic
$T^{IF}(t):=\sqrt{n\Delta_{n}}\left(\hat{a}(t)-a(t)\right)/\sqrt{\sigma_a^2(t)}$,
the feasible statistic without bias correction
$T^{F}:=\sqrt{n\Delta_{n}}\left(\hat{a}(t)-a(t)\right)/\sqrt{\hat \sigma_a^2(t)}$, 
and the feasible statistic with bias correction
$T^{F}_{\mathrm{bias \, corrected}}:=\sqrt{n\Delta_{n}}\left(\hat{a}(t)-a(t)-\frac{1}{2}\Delta_n \hat a'(t)\right)/\sqrt{\hat \sigma_a^2(t)}$, 
follow a standard normal distribution in finite samples.
Detailed results are available in Section \ref  {asec:asymGauss} in the Supplementary material.

While the consistency results reported above were generally very good throughout our simulation experiments, the picture for the asymptotic Gaussianity is slightly more mixed. 
Recall that for the consistency result, we only needed simultaneously that $\Delta_n$ is ``small'', $n$ is ``large'' and  $n \Delta_n$ is ``large''. 
In order to get asymptotic Gaussianity, we also need that $n \Delta_n^3$ is ``small'' (or  that $n \Delta_n^2$ is ``small'') when considering the case of $\mu=0$, $\mu_0=0$. 	 
For each of the three statistics, we compute the sample mean, standard deviation, and the 
90\%, 95\% and 99\% coverage probabilities over 1000 Monte Carlo runs. 
We describe the specific findings for each of the marginal distributions and the trawl function in detail in the Supplementary material. 
Here, let us summarise our main findings
for the case of a negative binomial marginal distribution with an exponential trawl and a supGamma trawl.
First, for $T^F$ we often observe a  finite sample bias, which is noticeably reduced when $T^{F}_{\mathrm{bias \, corrected}}$ is considered instead. Bias correction is particularly advisable for smaller values of time $t$. 
Second, it appears that the asymptotic variance estimator $\hat \sigma_a(\cdot)$ becomes less reliable when time increases; in such cases, we occasionally observe a deviation from the (asymptotic) unit standard deviation. This is not a surprise since the larger $t$, the less data is used to compute $\hat{v}_{4}(t)$.
Third, decreasing $\Delta_n$ does not necessarily lead to an improved finite sample performance, unless, as expected, $n$ is also increased to get the ``right'' balance between the in-fill and long-span asymptotics.
Fourth, in the purely Gaussian case, we observe that the feasible test statistics are not asymptotically standard Gaussian unless $\Delta_n$ is ``comparably small'' (in our case $\Delta_n=0.01$), $n$ is ``large'' and a bias correction is used.  These results are in line with our theoretical findings, which require stronger assumptions on the relation between $\Delta_n$ and $n$. 
Fifth, when comparing the performance of the test statistics in the short-and long-memory case, we generally get better results in the short-memory settings. For long-memory settings, we need rather small values for $\Delta_n$ combined with very large values for $n$ and still end up with poor finite sample performance when $t$ is large, whereas the results are better for smaller values of $t$. 
Finally, we recall that when estimating the asymptotic variance $\sigma_a^2(t)$ in the feasible statistics, we need to choose the tuning parameter $N_n$.  In our simulation experiments,  we found a low sensitivity to varying choices of the tuning parameters. We typically set the parameter such that a maximum number of observations can be used in the estimation. However, since the monotonic trawl function decreases to 0 for large time points, we found that the impact of including or omitting estimates of the trawl function which are close to 0 is negligible in practice. 
In summary, our simulation study has revealed that the trawl function and also integrals of the trawl function, such as the trawl set, its slices and corresponding ratios can be estimated with a very small bias. If one wants to use the asymptotic Gaussianity for constructing asymptotic confidence intervals, then it is important to check the relation between the interval width $\Delta_n$ of the observations and the number of observations $n$. While for short-memory trawl functions, the coverage probabilities appear well-aligned with the theoretical ones for small enough $\Delta_n$ and large enough $n$, there seemed to be greater deviations in the long-memory case.

\section{Applications}\label{sec:Emp}
We will now apply the new methodology to three distinct areas of application: model selection and model misspecification testing, high-frequency financial spread data, and stochastic queueing systems.

\subsection{Applications to model selection and model misspecification testing }
While existing parametric approaches assume a specific functional form for the trawl function, we demonstrate three important scenarios where nonparametric trawl function estimation is essential: (1) detecting complex structural features of the trawl function that are masked by integrals of the trawl function or the ACF (widely used in parametric inference), (2) model specification testing, and (3) hybrid modelling for processes with distinct short- and long-term behaviour.

First, without imposing specific structural assumptions on the trawl function $a$, the new nonparametric estimator can recover its shape directly from the data. This is particularly important when $a$ has a complex form.  As an example, Figure \ref{fig:trawl_fcts_ill} compares a piecewise exponential (pe) with an exponential (e) trawl function 
\begin{align}\label{eq:trawlspec}
    a_{\text{pe}}(x) = a(x)= \begin{cases}
5 e^{-0.1x} & \text{if } x < 2 \\
5 e^{-0.2} \cdot e^{-0.5(x-2)} & \text{if } x \geq 2
\end{cases},
&& a_{\text{e}}(x)=a(x)=5 e^{-0.3x}, \quad x\geq 0.
\end{align}
Crucially, changes in the slope of the autocorrelation function (ACF) correspond directly to changes in $a$, whereas the ACFs themselves, being the integrals of the corresponding trawl functions, are remarkably similar and difficult to distinguish, see Figure \ref{fig:trawl_fcts_ill}. In this case,  non-linear least squares estimation of the ACF selects the exponential specification over the correct piecewise-exponential one, leading to a misspecified trawl model.  This example demonstrates that parametric estimation based on integrated quantities, such as the ACF, can fail to detect important structural features that are evident in the trawl function itself. Nonparametric estimation targeting the trawl function avoids this problem. 
\begin{figure}[htbp]
\centering
\subfloat[Piecewise exp.~$a(x)$]{%
    \includegraphics[width=0.24\textwidth]{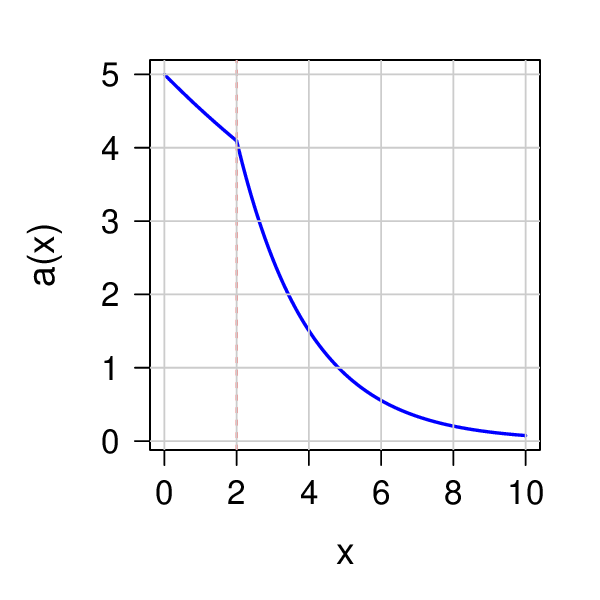}%
}\hfill
\subfloat[ACF $\rho(x)$ (piecewise exp.)]{%
    \includegraphics[width=0.24\textwidth]{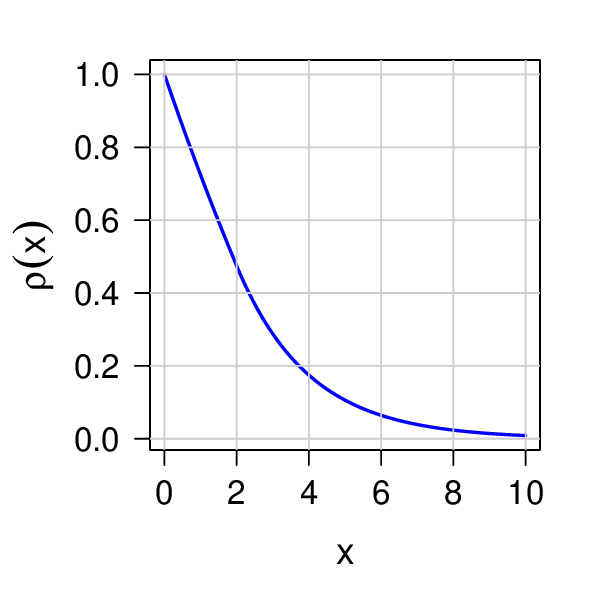}%
}\hfill
\subfloat[Simple exponential $a(x)$]{%
    \includegraphics[width=0.24\textwidth]{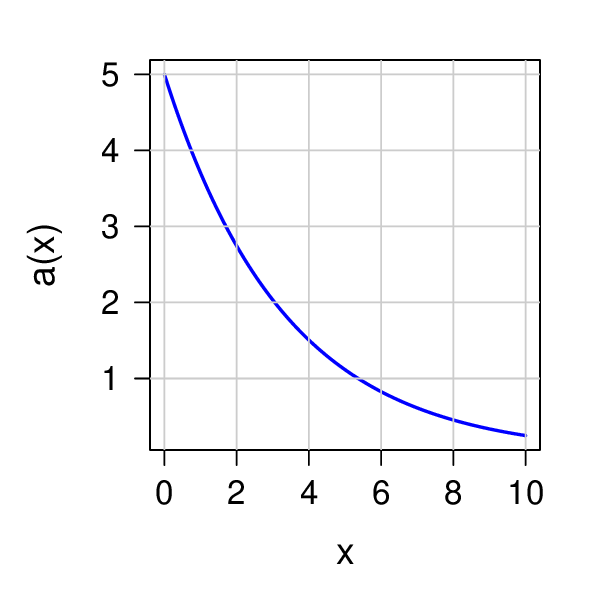}%
}\hfill
\subfloat[ACF $\rho(x)$ (exponential)]{%
    \includegraphics[width=0.24\textwidth]{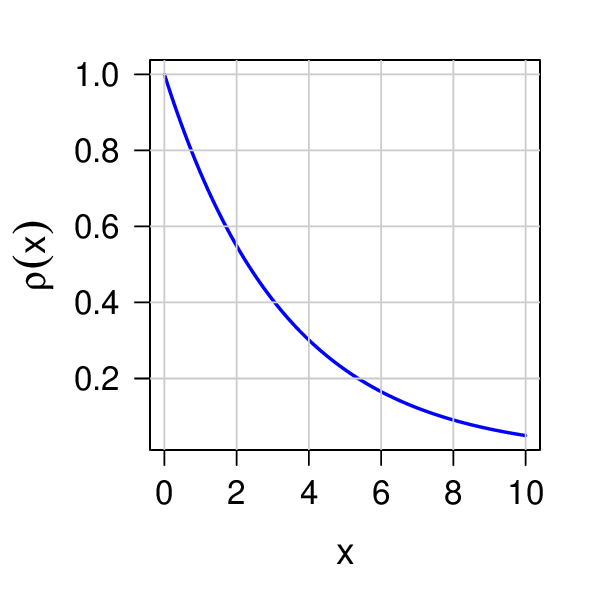}%
}
\caption{Comparison of two trawl functions $a(x)$, specified in \eqref{eq:trawlspec}, and their associated ACFs $\rho(x)=\Gamma(x)/\Gamma(0)$, where $\Gamma(x)=\int_x^{\infty}a(s)ds$. While the piecewise exponential specification exhibits a clear breakpoint, the corresponding ACFs are visually nearly indistinguishable. This illustrates why parametric ACF-based methods can fail to detect model misspecification.\label{fig:trawl_fcts_ill}}
\end{figure}

Second, even when the ultimate goal is to fit a parametric model, our nonparametric trawl estimates, together with their asymptotic confidence bounds, provide an essential diagnostic tool for model validation. The bias-corrected estimator $\tilde{a}(t) := \hat{a}(t) - \frac{1}{2}\Delta_n\hat{a}'(t)$ satisfies (under the assumption stated above)
$\sqrt{n\Delta_{n}}(\tilde{a}(t)-a(t)) \overset{d}{\rightarrow} N(0,\sigma_{a}^{2}(t))$. 
 This suggests the following three-step diagnostic procedure: First, estimate $a(t)$ nonparametrically with confidence bands $\tilde{a}(t) \pm z_{\alpha/2} \cdot \frac{\hat{\sigma}_a(t)}{\sqrt{n\Delta_n}}$; second, fit the parametric model of interest (e.g., exponential $a_{\text{par}}(t) = ce^{-\lambda t}$); third, overlay $a_{\text{par}}(t)$ on the nonparametric confidence bands for visual inspection. If the parametric specification consistently falls outside the confidence bounds, this indicates misspecification.
The ability to test at individual lags is particularly valuable because misspecification often manifests itself locally rather than globally. For a fixed lag $t \geq 0$ and a parametric model $a_{\text{par}}(t; \theta)$ (with e.g. $\theta=(c, \lambda)$ in the exponential trawl model), we can formally test 
$H_0: a(t) = a_{\text{par}}(t; \theta_0)$ versus $H_1: a(t) \neq a_{\text{par}}(t; \theta_0)$.
Under $H_0$, the test statistic
$T_n(t) := \frac{\sqrt{n\Delta_n}(\tilde{a}(t) - a_{\text{par}}(t; \hat{\theta}))}{\hat{\sigma}_a(t)} \overset{d}{\rightarrow} N(0, 1)$, 
where $\hat{\theta}$ is a $\sqrt{n\Delta_n}$-consistent estimator of $\theta_0$. We reject $H_0$ at level $\alpha$ if $|T_n(t)| > z_{\alpha/2}$.
This pointwise testing framework allows us to identify precisely which features of the trawl function are misspecified, for instance, whether the decay rate is incorrect at short lags, long lags, or specific intermediate ranges. Such information is invaluable for refining parametric specifications systematically rather than relying on trial-and-error model selection.
To illustrate the above framework, we simulate a Poisson trawl process with piecewise-exponential trawl function as specified in \eqref{eq:trawlspec}.

Figure \ref{fig:poisson_comparison} 
shows the true trawl function (red), the nonparametrically estimated trawl function (black), and the corresponding (pointwise) 95\% confidence bounds (blue). In addition, the misspecified exponential trawl function is estimated and displayed in orange, and the correctly specified trawl function is estimated and added as a purple line. 
 The fitted exponential curve does not consistently lie within the confidence bounds obtained from the nonparametric estimator, with departures occurring at specific lag regions.  We also add the corresponding values of the test statistics for the exponential and the piecewise exponential fit, indicating that an exponential trawl is misspecified for very short lags (1-4) and for lags 15-19, close to the breakpoint at lag 20. 
Since our test is pointwise in $t$, testing at multiple lags $t_1, \ldots, t_K$ requires multiple testing corrections (e.g., Bonferroni or FDR control). A detailed discussion of potential joint/functional estimation approaches is given in \cite{Sauri2025}, which could be further developed for the multiple testing problem that arises here.

\begin{figure}[htbp]
    \centering
       \subfloat[Comparison of nonparametric estimate (black dots) with 95\% confidence 
    bounds (blue), true trawl function (red), and fitted exponential (orange) and 
    piecewise exponential (purple) models.
        \label{fig:poisson_comparison_top}]{%
        \includegraphics[width=0.5\textwidth, height=4cm]{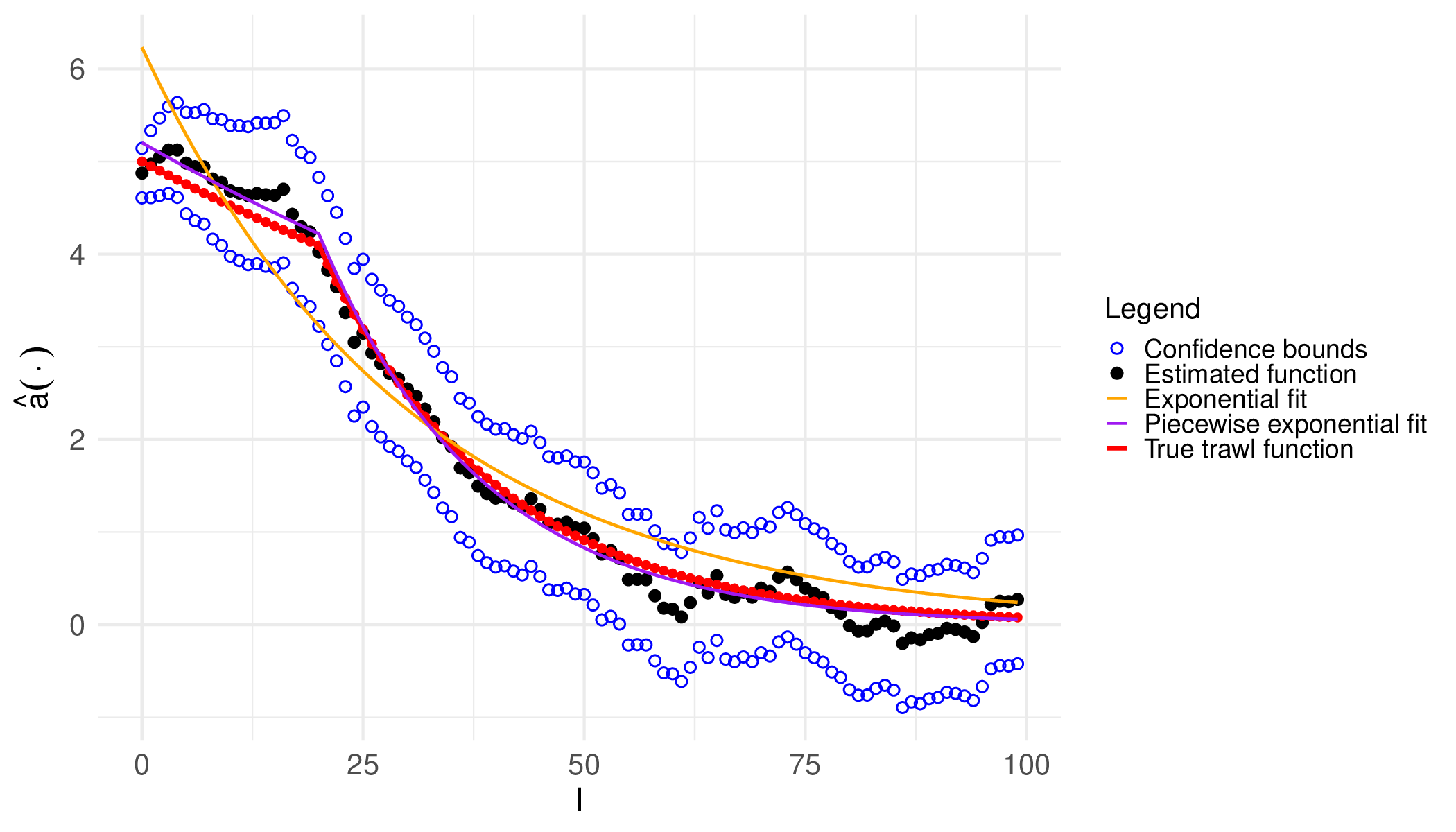}%
    }
    \hfill
    \subfloat[Exponential specification shows systematic rejections concentrated at very short lags and near 
    the structural breakpoint at lag 20 (time = 2.0).\label{fig:poisson_exponential}]{%
        \includegraphics[width=0.24\textwidth, height=4.5cm]{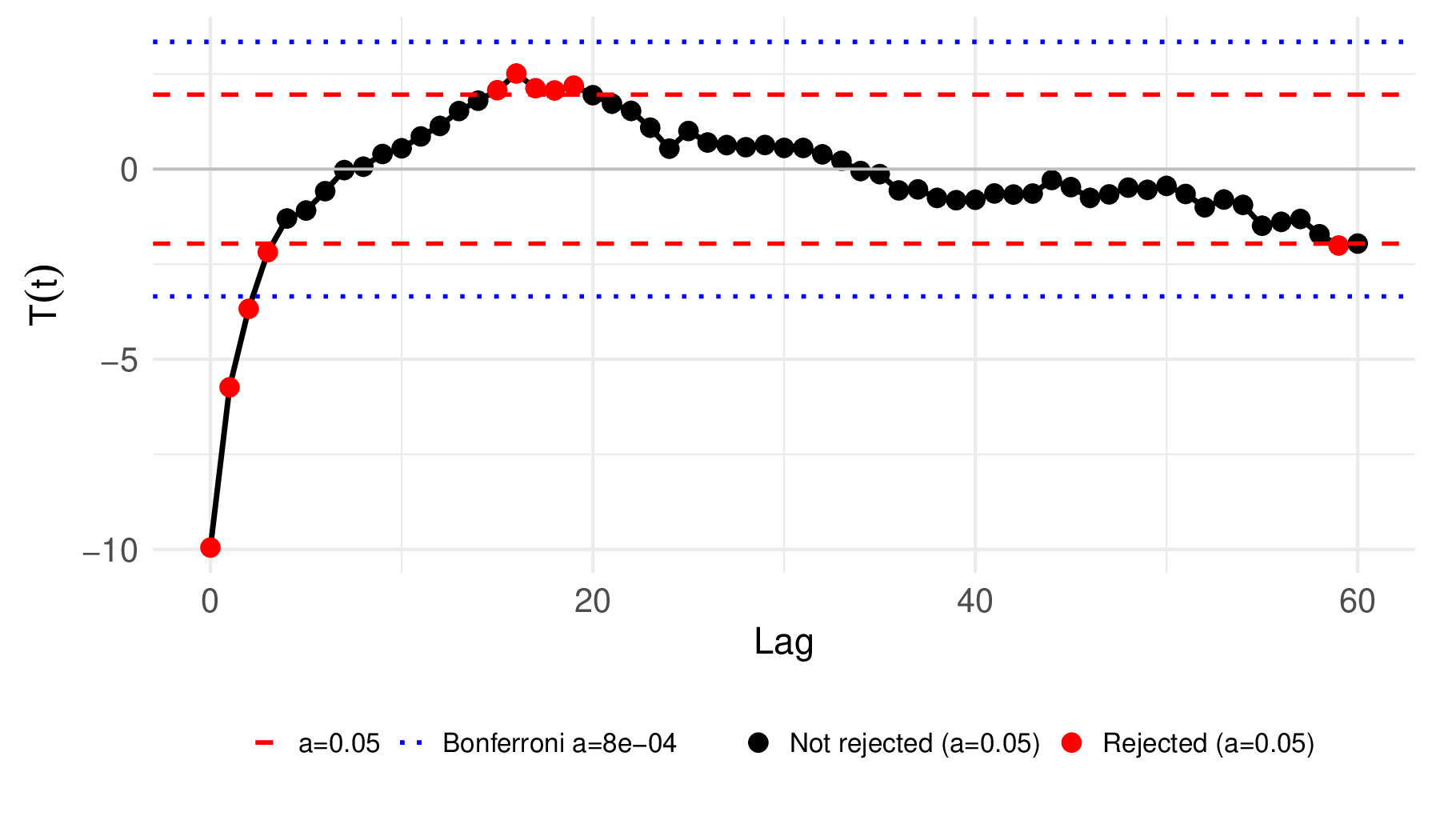}%
    }
    \hfill
    \subfloat[Piecewise exponential (true model) shows test statistics within expected 
    range ($\pm 2$) with no rejections, confirming appropriate test size.\label{fig:poisson_piecewise}]{%
        \includegraphics[width=0.24\textwidth, height=4.5cm]{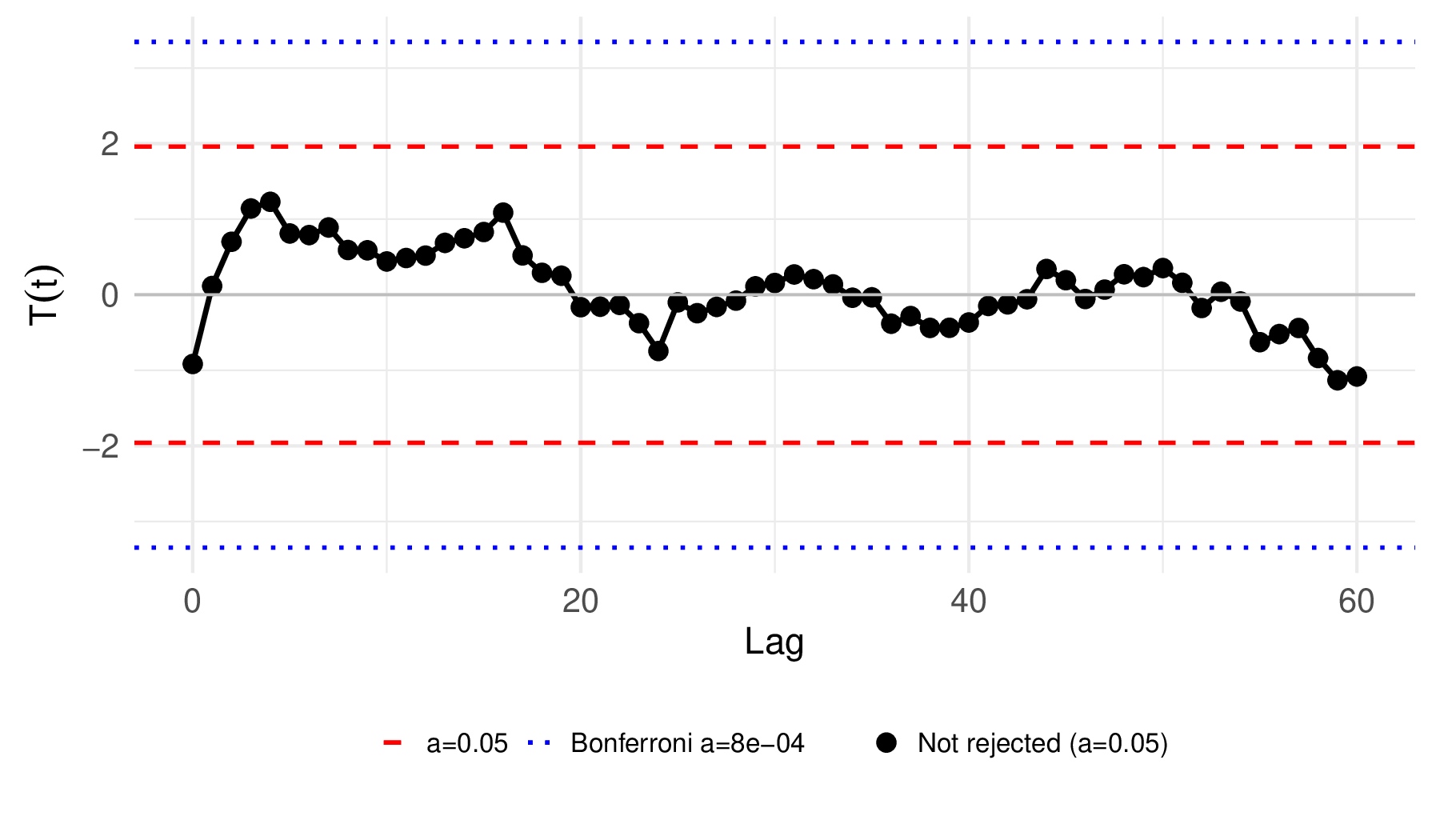}%
    }
    
    \caption{Trawl function estimation and specification testing for the Poisson trawl process with piecewise-exponential trawl function. 
    Test statistics for specification tests across lags $0:60$. Dashed lines indicate 
    $\alpha=0.05$ critical values; dotted lines show the Bonferroni-corrected thresholds 
    ($\alpha/61\approx0.0008$).}
    \label{fig:poisson_comparison}
\end{figure}

While parametric likelihood-based tests can detect misspecification globally, they typically cannot identify which aspects of the trawl function are misspecified. Our pointwise testing framework (Figures \ref{fig:poisson_exponential} and \ref{fig:poisson_piecewise}) provides this additional diagnostic information. Moreover, even when domain knowledge suggests a particular parametric family, our nonparametric approach provides model-free validation without requiring nested model comparisons or complex likelihood ratio tests.

Third, when the trawl function exhibits complex short-term behaviour, a hybrid approach can be used: the non-parametric estimator is applied for small time lags, while a simple parametric model, e.g.~an exponential one, is fitted to capture the tail behaviour of the trawl function. We can use the derived test statistics to determine a suitable break point to switch from the nonparametric estimator to a smooth, parametric tail model for the trawl function. 
We demonstrate the breakpoint determination for a Poisson  trawl process with piecewise exponential trawl in Figures \ref{fig:Poissonbreakpointa} and \ref{fig:Poissonbreakpointb}.

\begin{figure}
 \subfloat[Poisson trawl process: Identification of breakpoint for hybrid estimation.\label{fig:Poissonbreakpointa}]{%
        \includegraphics[width=0.48\textwidth, height=4cm]{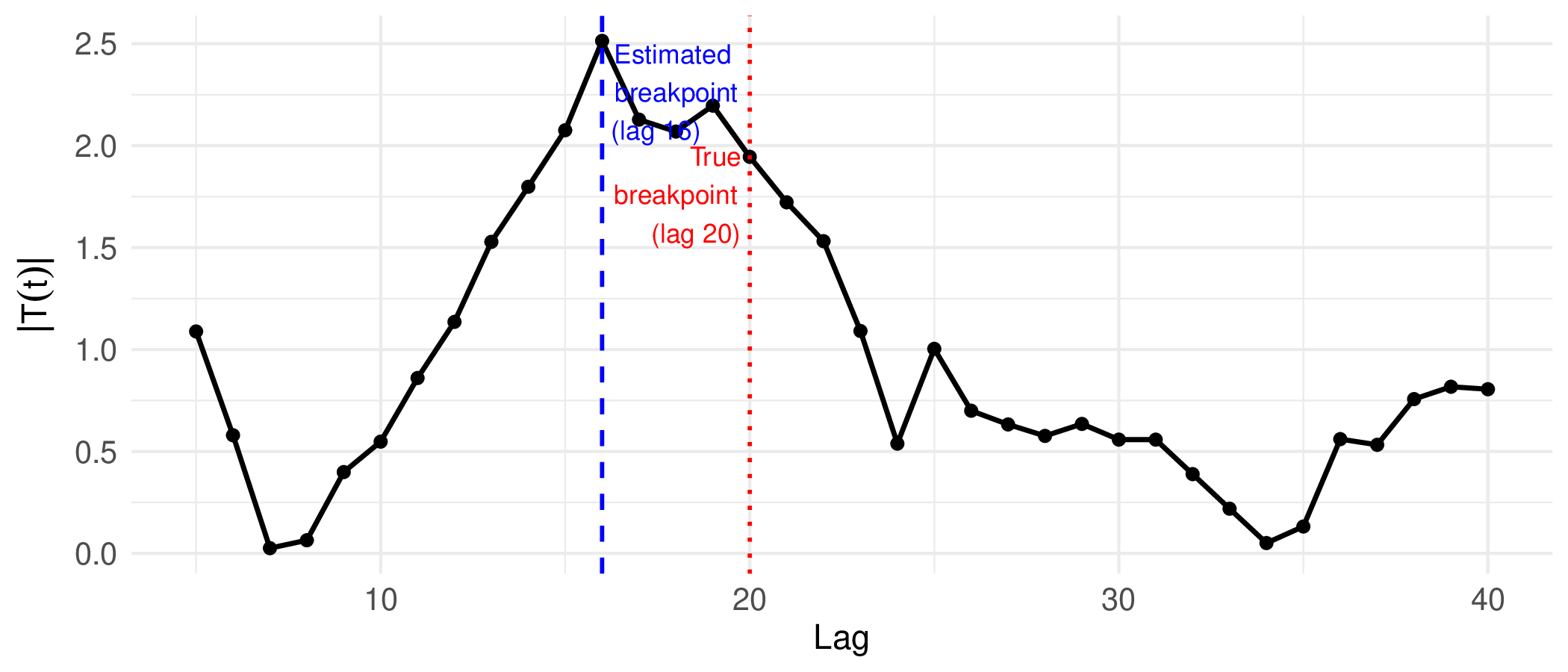}%
    }
    \subfloat[Poisson trawl process: Identification of breakpoint for hybrid estimation.\label{fig:Poissonbreakpointb}]{%
        \includegraphics[width=0.48\textwidth, height=4cm]{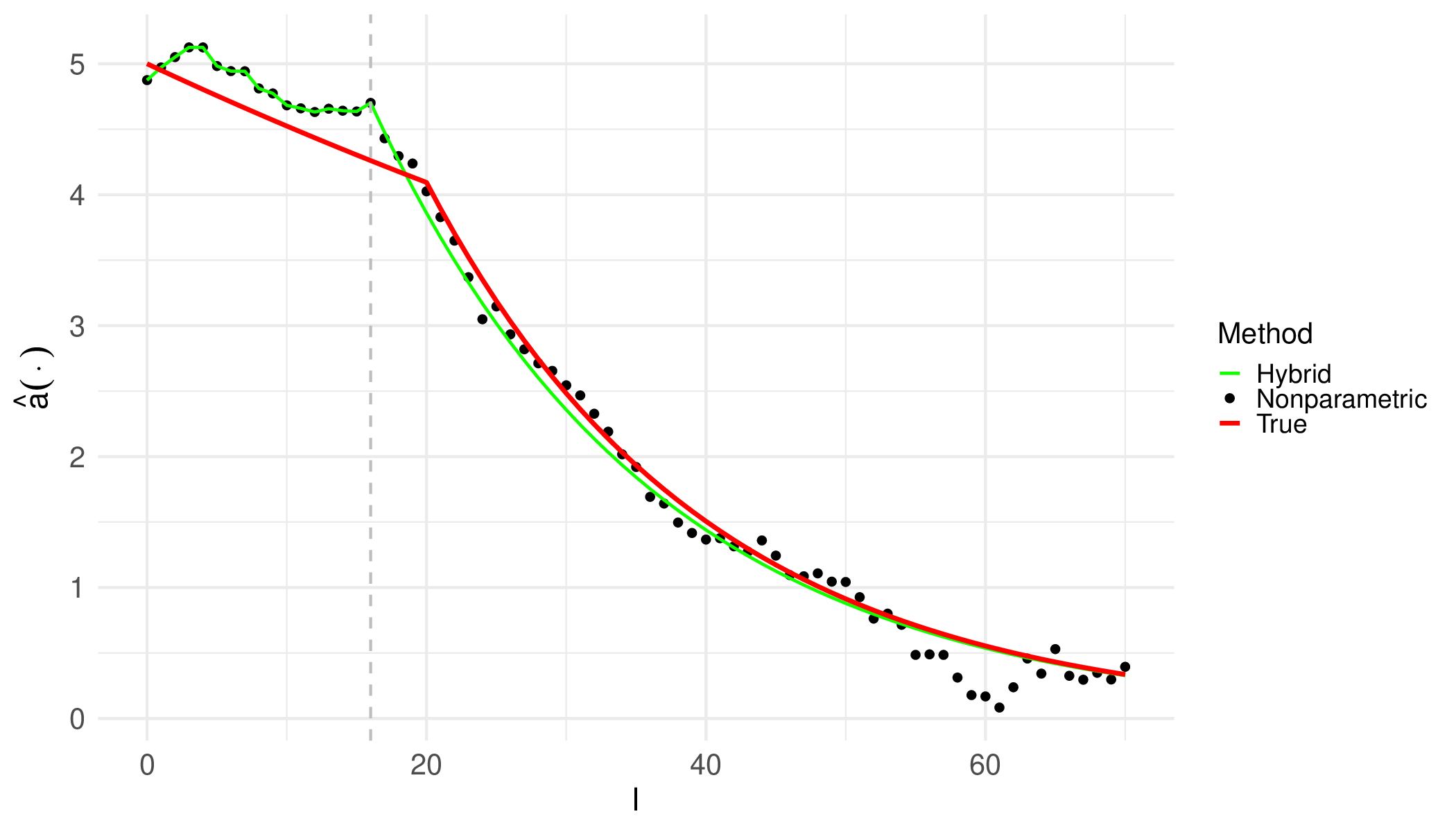}%
    }
    \caption{Using the maximum of the test statistic (comparing with an exponential trawl function) to identify the breakpoint in a hybrid trawl estimation. The hybrid estimator uses nonparametric estimates $\hat{a}(t)$ up to the detected breakpoint $\hat{c}_1$ (identified as the lag maximizing $|T(t)|$), then switches to an exponential tail $\hat{a}(\hat{c}_1) \exp(-\hat{\lambda}(t - \hat{c}_1))$ where $\hat{\lambda}$ is fitted via log-linear regression on the next 30 lags.}
\end{figure}

We now investigate how model (mis)specification affects forecasting performance. For this purpose, we simulate 10,000 observations each of a Poisson and a Gaussian trawl process with a piecewise-exponential trawl function, splitting each sample into 8,000 training and 2,000 test observations. The figures for the Gaussian case are relegated to the Supplementary Material.  
 We estimate a (misspecified) exponential trawl (orange), the piecewise-exponential trawl (purple), the nonparametrically estimated trawl (grey), the hybrid trawl function estimator (blue), and compute the true trawl (red). Based on these estimates, we compute the 1 to 10-step ahead forecasts and compute the out-of-sample RMSEs for each of the 10 lags, see Figure  \ref{fig:Poissonforecast}. 
We observe that the results for both the Poisson and Gaussian cases are very similar: When the misspecified trawl function is used, the RMSEs are largest. The hybrid and nonparametric estimator results in the smallest RMSEs (when ignoring the true trawl function) in the Poisson and in the Gaussian case, respectively. That said, we note that the non-parametric and hybrid estimators perform at least at par or better than the correctly specified model, and they, as expected,  outperform the misspecified model. 
Figure
\ref{fig:Poissonforecast-pen} 
illustrates the percentage increase of the RMSE if the misspecified exponential trawl function is used rather than the alternatives (hybrid, nonparametric, (estimated) piecewise-exponential, true trawl function (oracle)).

\begin{figure}[htbp]
    \centering
    \subfloat[Poisson: RMSE\label{fig:Poissonforecast}]{%
        \includegraphics[width=0.48\textwidth, height=5cm]{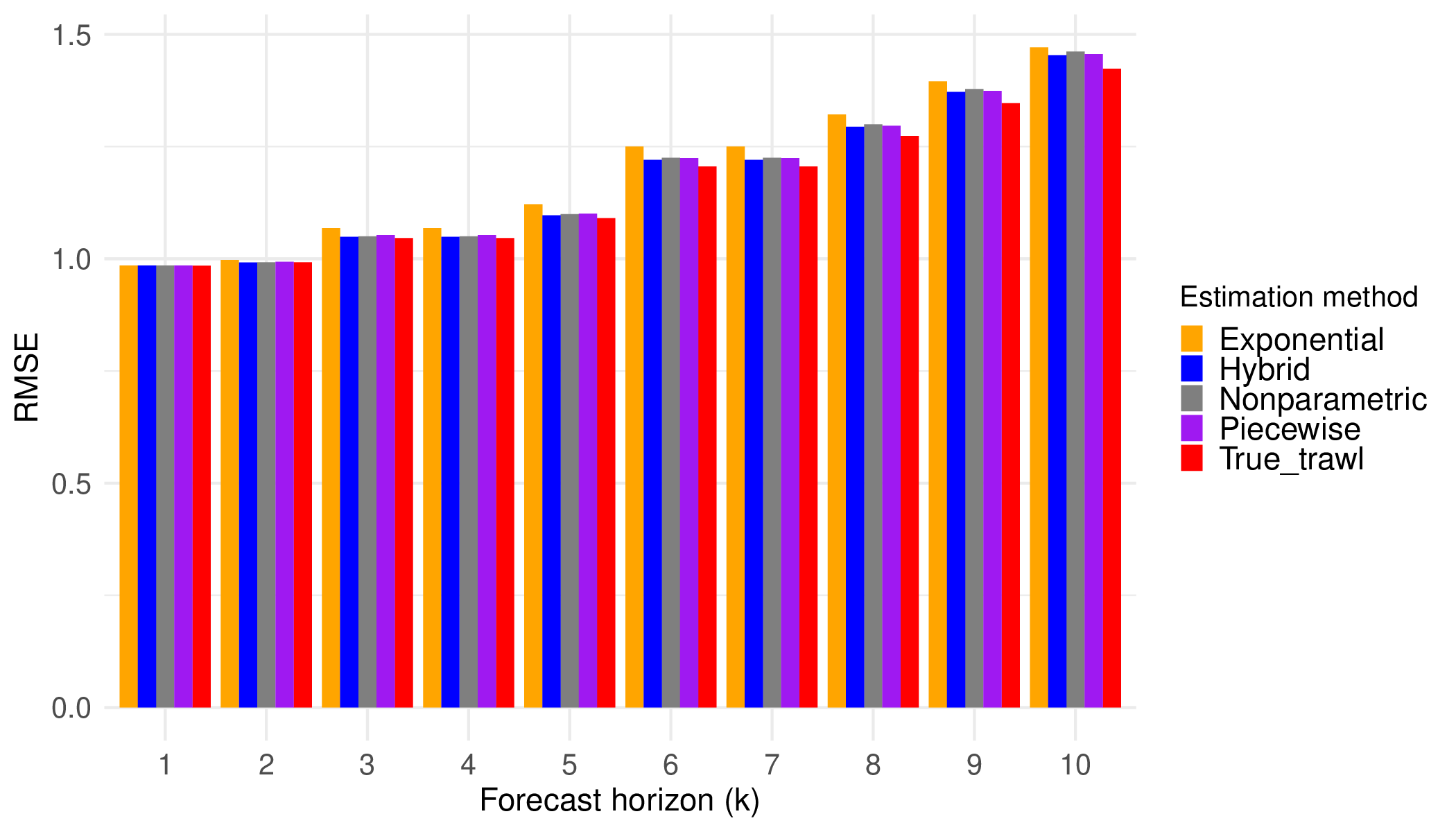}%
    }
      \hfill
    \subfloat[Poisson: \% penalty\label{fig:Poissonforecast-pen}]{%
       \includegraphics[width=0.48\textwidth, height=5cm]{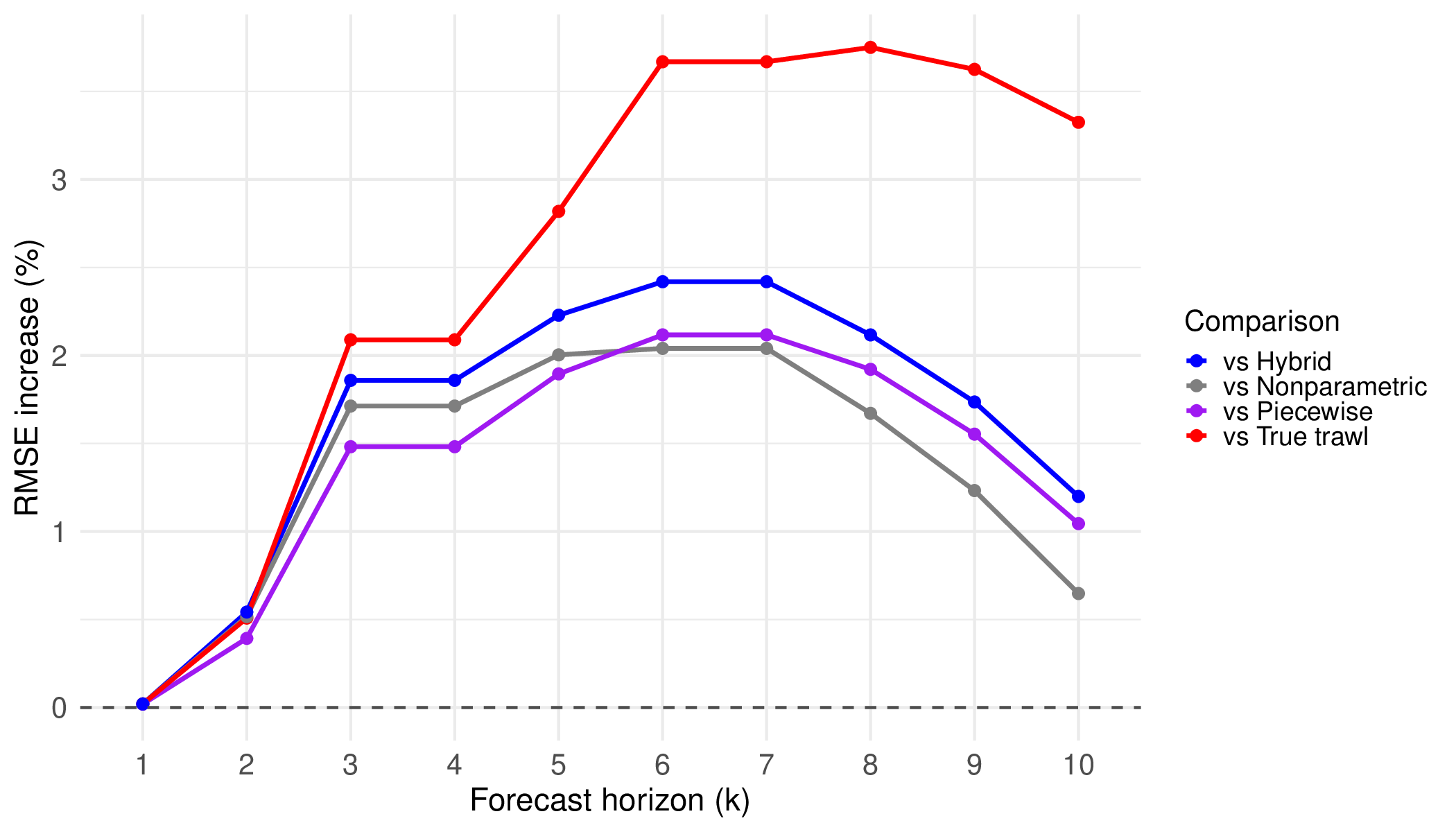}%
    }
    \hfill
     \caption{Forecast accuracy for the Poisson  trawl process based on 2000 out-of-sample forecasts. RMSE across methods (exponential, hybrid, nonparametric, piecewise-exponential, true) and percentage penalty of using a misspecified exponential specification.}
    \label{fig:forecast_comparison}
\end{figure}

\subsection{Forecasting high-frequency financial spread data}
We will now present an empirical illustration\footnote{The code which reproduces the stated results  is available on GitHub 
and archived in Zenodo, see \cite{emp-code-trawl}.
} of our new methodology in the context of modelling the bid-ask spread of equity prices.
There is an extensive literature on this topic, see, e.g., \cite{HS1997} and \cite{BSW2004}. 
\cite{BNLSV2014} proposed using integer-valued trawl processes for modelling spread data, and this work has been further extended in 	
\cite{BLSV2021}. 	
In the following, we will consider high-frequency data from the stocks of four companies: 
 Agilent Technologies Inc.  (ticker: A), Discover Financial Services (ticker: DFS), 
 Waters Corporation (ticker: WAT) and Waste Management (ticker: WM).
 The data has been downloaded from the LOBSTER website\footnote{ \url{https://lobsterdata.com/index.php}}, which provides reconstructed limit order book data on a subscription basis.
We consider time series for 2 years, from 01.09.2022 to 31.08.2024.    
The stocks are traded on the New York Stock Exchange, which is open from 9.30 AM to 4 PM.
 The data is available at a very high frequency, but to obtain equidistant data, we sample the observations with $\Delta_n= 1/12\approx 0.083 $ minutes (i.e.~5 seconds) time steps, using the previous tick approach.
 Also, for each day, we discard the first 60 minutes of data and consider the period from 10:30 AM to 4 PM only to obtain data which are not distorted by market opening effects and align better with our assumption of stationarity. 
 For each day, starting at 10:30 AM, we have 
 $n=3961$ observations of the spread denoted by $s_t$ for $t\in \{0, \Delta_n,\ldots, 3961\Delta_n\}$. We discard days with data inconsistencies, resulting in 458, 452, 415 and 437 days in total for A, DFS, WAT and WM, respectively.  
Since the minimum spread level in the data is one tick (one dollar cent), we work on this time series minus one, i.e.~on $x_t = s_t-1$.  An illustrative figure including the time series plot of the first of the 461 time series considered for stock A, the corresponding histogram, showing values between 0 and 27 and the autocorrelation function is provided in Figure \ref{fig:A_OnePath}.
\begin{figure}[htbp]
\centering
\captionsetup[subfigure]{aboveskip=-4pt, belowskip=-4pt}

\subfloat[Time series of A for one day]{%
  \includegraphics[scale=0.20]{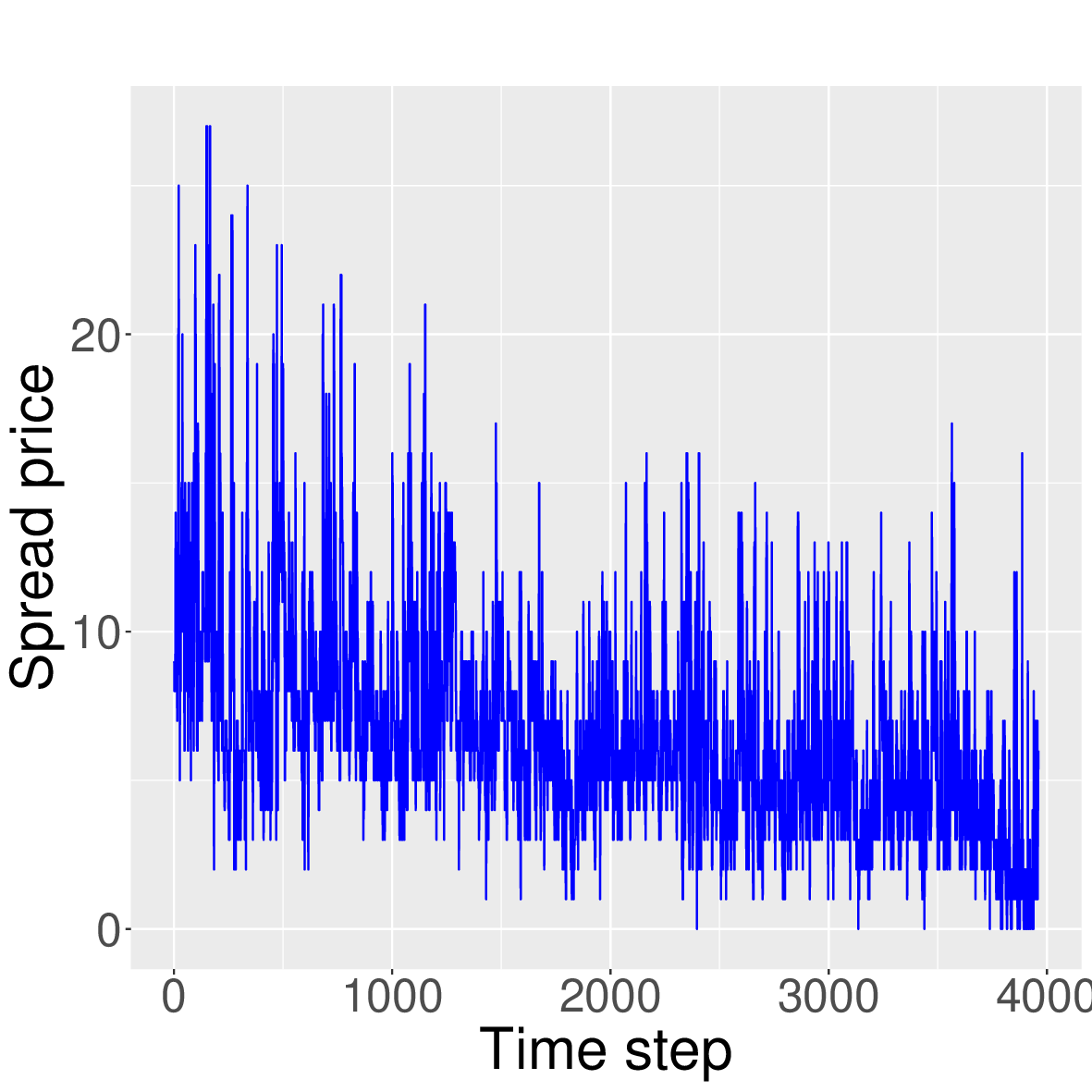}
}
\hfill
\subfloat[Histogram of the time series]{%
  \includegraphics[scale=0.20]{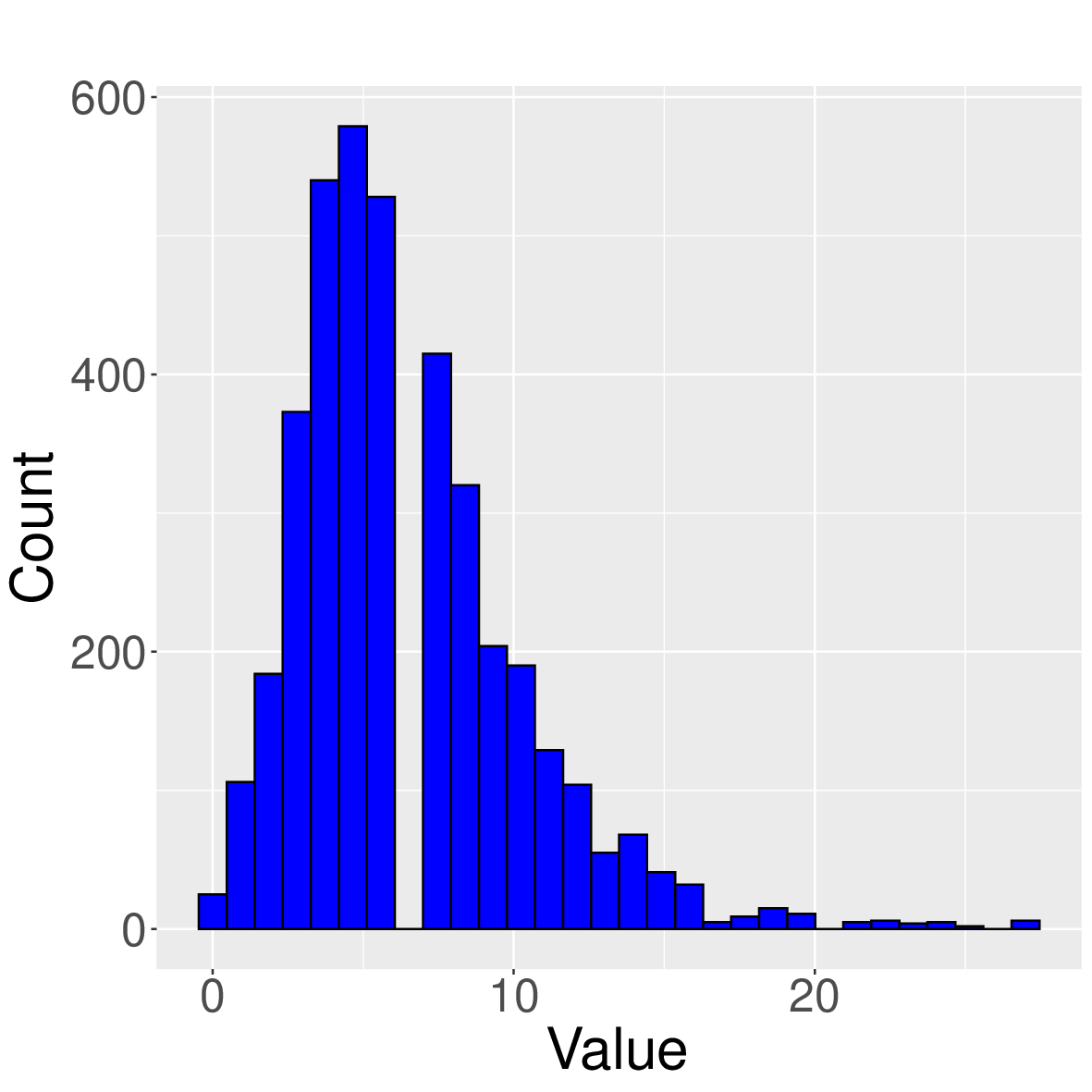}
}
\hfill
\subfloat[ACF of the time series]{%
  \includegraphics[scale=0.20]{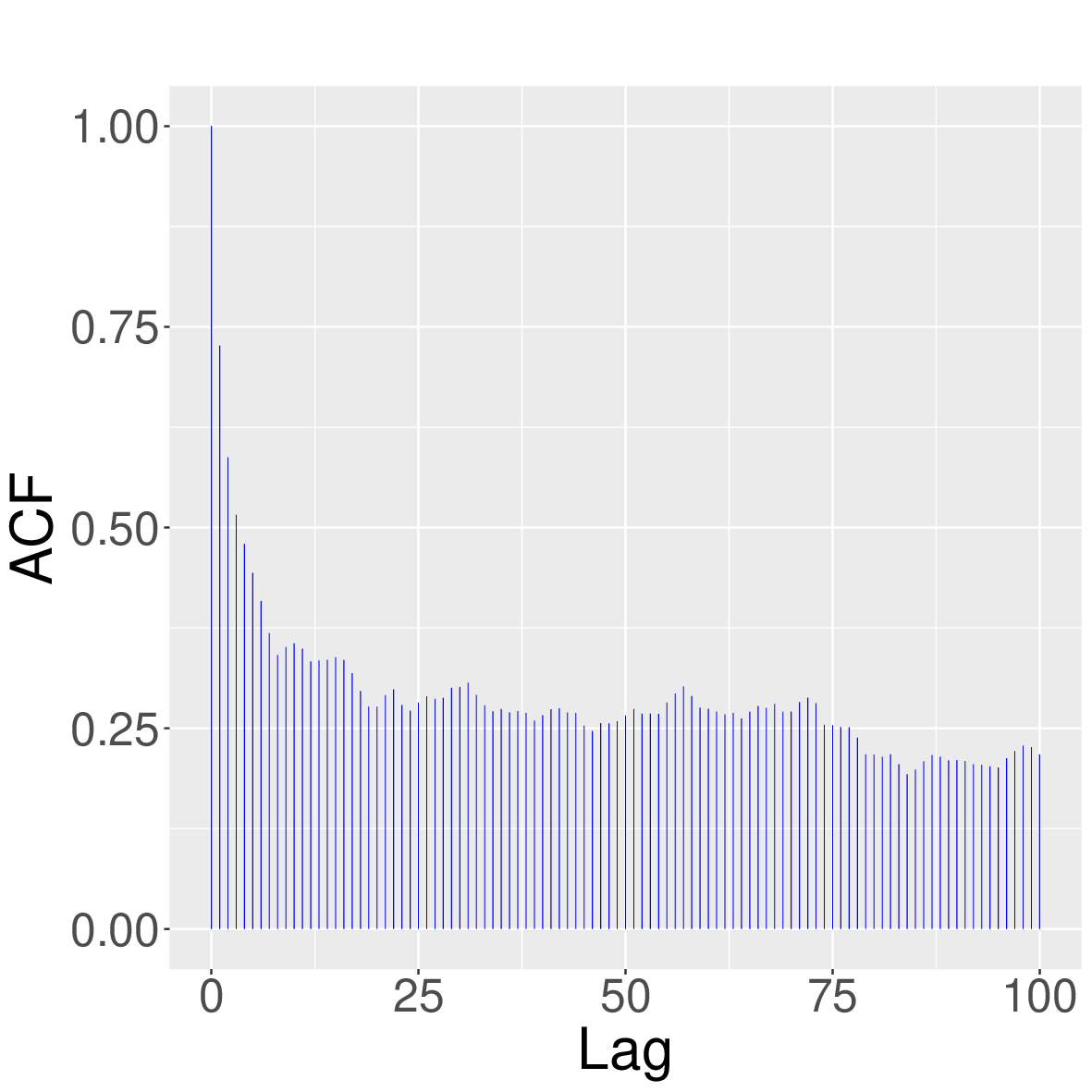}
}

	\caption{\it 
 The first picture shows the time series of A for the first day of the sample; the second figure shows the corresponding histogram highlighting that we are dealing with count data between 0 and 27. The third picture depicts the autocorrelation function of the time series.}
 \label{fig:A_OnePath}
\end{figure}

For each stock and each of the approx.~450 days, 
we estimate the trawl function using our nonparametric estimator and present the boxplots of the resulting trawl function estimates in the  figure on the left in Figures \ref{fig:Trawl_A}, \ref{fig:Trawl_DFS}, \ref{fig:Trawl_WAT} and \ref{fig:Trawl_WM}, the latter three figures are available in the Supplementary Material. 

We observe that for all four stocks, the trawl functions decrease to zero over the 10 lags of 5s intervals displayed, with DFS and WM decreasing relatively fast and A and WAT slightly more slowly.

Next, we apply our estimation method to forecasting trawl processes using the results established in Subsection \ref{sec:TheoryForecasting}. We will describe the results for the stock Agilent Technologies (ticker: A) in the following and relegate the relevant figures for the stocks DFS, WAT and WM to the supplementary material, see 
Figures \ref{fig:Trawl_DFS} --
 \ref{fig:ForecastRatios_WM}. Given that this is the first work considering nonparametric estimation methods for trawl processes, we are particularly interested in the relative performance of the two nonparametric methods we have proposed earlier, namely the trawl-function-based predictor defined in \eqref{eq:predictor-est} and the ACF-based predictor $\hat X_{t+h|t}^{\mathrm{ACF}}$. For each of the approx.~450 days, we split our data set into a training sample of length $n:=n_{training}:=3168$, i.e.~using 80\% of the observations,  
and then compute the 1-step, 2-step, $\ldots$, 10-step ahead forecasts.
Overall, for each day, we compute $3961-3168-20=773$ forecasts, which are 1 to 10 steps ahead. For each day and each of these 10 forecast horizons, we compute the mean squared error (MSE)  over the 773 forecasts. 
We re-estimate the model for each day and repeat the forecasting exercise. 

We compare the performance of the proposed predictor 
$\hat X_{t+h|t}$ with the nonparametric alternative based on the empirical autocorrelation function 	$\hat X_{t+h|t}^{\mathrm{ACF}}$. We plot the ratios of the corresponding MSEs in the first column of Figures \ref{fig:ForecastRatios_A}, \ref{fig:ForecastRatios_DFS}, \ref{fig:ForecastRatios_WAT} and \ref{fig:ForecastRatios_WM}.
We observe that the mass of the distribution of the MSE ratios is below 1, indicating smaller MSEs for the trawl-function-based predictor compared to the ACF-based predictor. 

In addition, we fit two relevant  parametric benchmark models for the trawl function: First, we consider an exponential trawl (Exp) given by  $a(x)=e^{-\lambda x}$, for $\lambda>0$ and $x\geq 0$. Here we have
the following expression for the set intersection weight in the forecast formula
$\Leb(A \cap A_h)/\Leb(A)=\exp(-\lambda h)$.
Often, such a simple exponential model leads to a too fast of a decay and, hence, we also consider a supGamma/long memory (LM) trawl function given by $a(x)=(1+x/\overline{\alpha})^{-H}$, for $\overline{\alpha}>0$, $H>1$ and $x\geq 0$. We note that the resulting trawl process has long memory if $H\in(1,2]$ and  short memory if $H>2$.
In this case, the set intersection weight in the forecast formula is given by 
$\frac{\Leb(A \cap A_h)}{\Leb(A)}=(1+\overline{\alpha} h)^{(1-H)}$.
We repeat the above forecasting exercise and compare the MSEs of the forecasts based on the  exponential and LM fit with the MSE based on the non-parametric trawl estimate. If the ratios are $<1$, then the non-parametric method performs better. The ration plots are provided in the 2nd and 3rd columns, for the Exp and LM fit, respectively, in Figures \ref{fig:ForecastRatios_A}, \ref{fig:ForecastRatios_DFS}, \ref{fig:ForecastRatios_WAT} and \ref{fig:ForecastRatios_WM}.
We also provide a comparison of the estimated set intersection weights $\Leb(A \cap A_h)/\Leb(A)$ based on the nonparameteric trawl, the acf-based method and the Exp and LM parameteric benchmarks in the second column of Figures \ref{fig:Trawl_A},
\ref{fig:Trawl_DFS},
\ref{fig:Trawl_WAT}
and
\ref{fig:Trawl_WM}.
In all cases, the exponential trawl does not fit the data well and its performance is clearly inferior compared to the other three methods. Interestingly, the non-parametric method and the LM-based method seem to perform comparably, with slight gains of the LM-based method for higher lags for the stock WAT only.
This demonstrates again the power of the non-parametric approach, which does not require any model selection and performs robustly throughout our analysis.

We might ask whether the differences in the MSE are statistically significant. To answer this question, we computed the p-values  of the Diebold-Mariano (DM) test, see \cite{DM1995}, of the null hypothesis of equal forecasting
performance between the two procedures, against the alternative that the 
trawl-function-based predictor provides superior forecasts compared to the ACF-based predictor.
For each of the four stocks, we adjust for the multiple testing problem of running approx.~$450\times10=4,500$ tests and compute the adjusted p-values using the Benjamini-Hochberg procedure. For each stock, we also compute the percentage of adjusted p-values falling below the 0.05 threshold, which are 92\%, 84\%, 98\% and 87\% for A, DFS, WAT and WM, respectively.
If we consider all four stocks and the ten forecast horizons jointly, we find that 90\% of adjusted p-values are below the 0.05 threshold.

Overall, we observe that, for the four stocks considered,  the MSEs tend to be generally smaller for all 
lags when the trawl-based predictor is used versus the ACF-based predictor. The same is true for the exponential trawl, whereas the LM trawl performs  on par with the nonparametric trawl method.  We conclude that, for our datasets,   it is advantageous to use the trawl-function-based predictor
for forecast horizons of 1-10 steps ahead. 
This ensures that the chosen forecasting method is at least on par with, if not superior to, its competitor.
We note that our study has focused on the relative performance of the two non-parametric methods proposed in this article and some relevant parametric trawl-based benchmark models. Further investigations are needed to 
assess their performance against alternative non-trawl-based forecasting methods.

\begin{figure}[htbp]
\centering
\captionsetup[subfigure]{aboveskip=-4pt, belowskip=-4pt}
\subfloat[Boxplots of estimated trawl function]{%
  \includegraphics[width=0.45\textwidth, height=3cm]{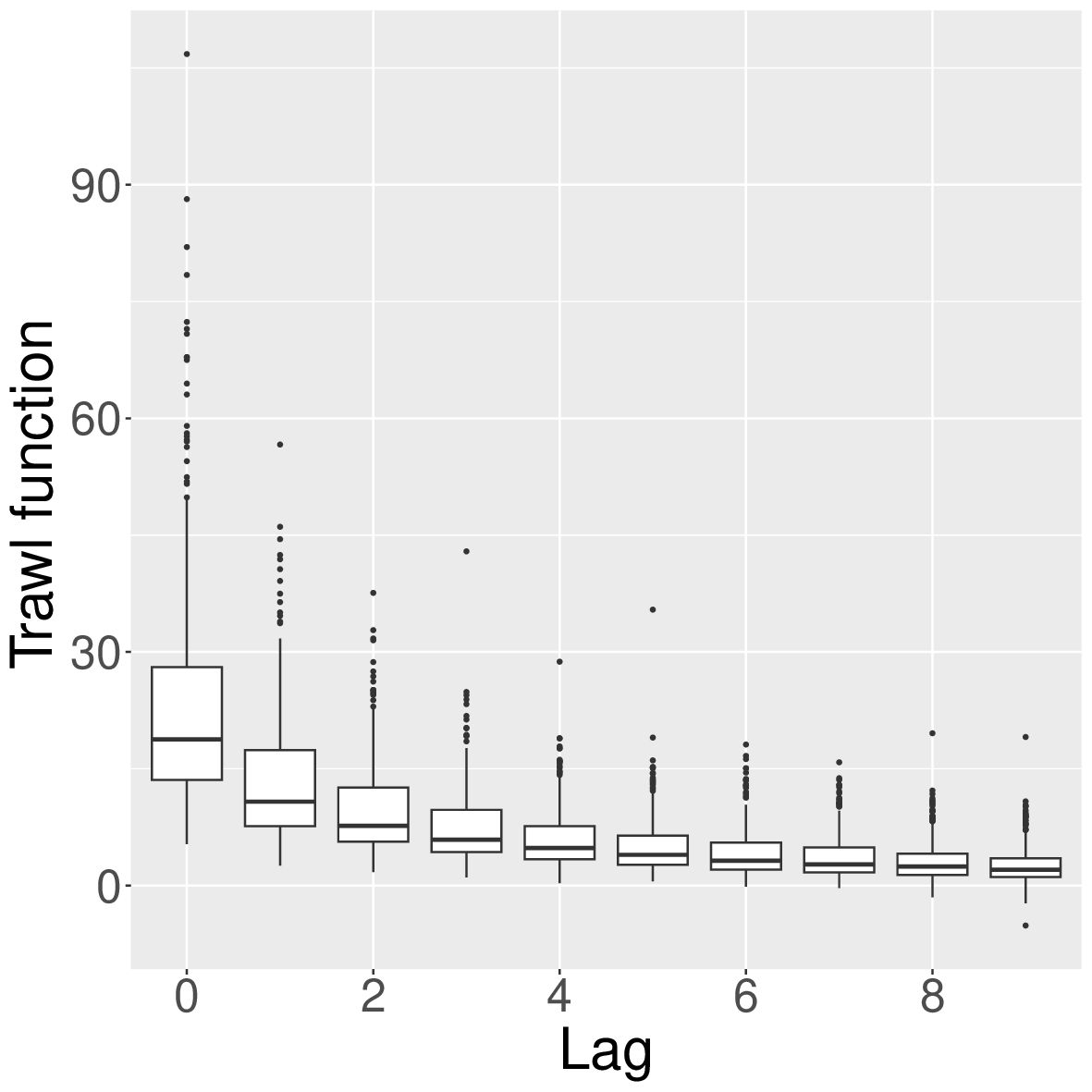}
}
\hfill
\subfloat[Estimated forecast weights]{%
  \includegraphics[width=0.45\textwidth, height=3cm]{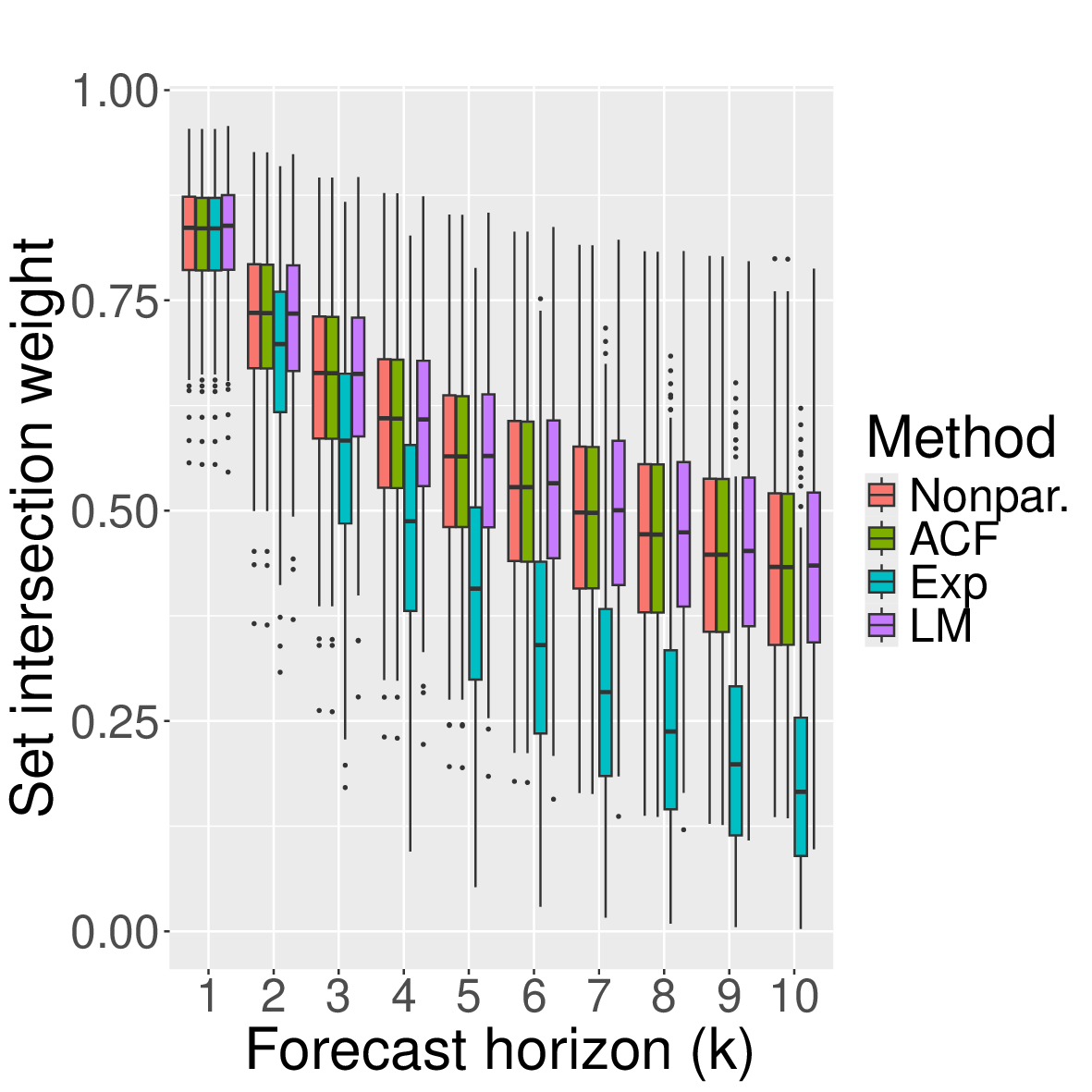}
}
	\caption{\it 
 The first picture shows the box plots of the estimated trawl functions for stock A. The second figure shows the estimated forecast weight $\Leb(A \cap A_h)/\Leb(A)$ using a nonparametric trawl (Nonparametric), and acf-based estimate (ACF), an estimated exponential trawl (Exp) and an estimated long memory/supGamma trawl (LM).}
 \label{fig:Trawl_A}
\end{figure}

\begin{figure}[htbp]
\centering
\captionsetup[subfigure]{aboveskip=-4pt, belowskip=-4pt}

\subfloat[Nonparametric versus ACF]{%
  \includegraphics[width=0.32\textwidth, height=3cm]{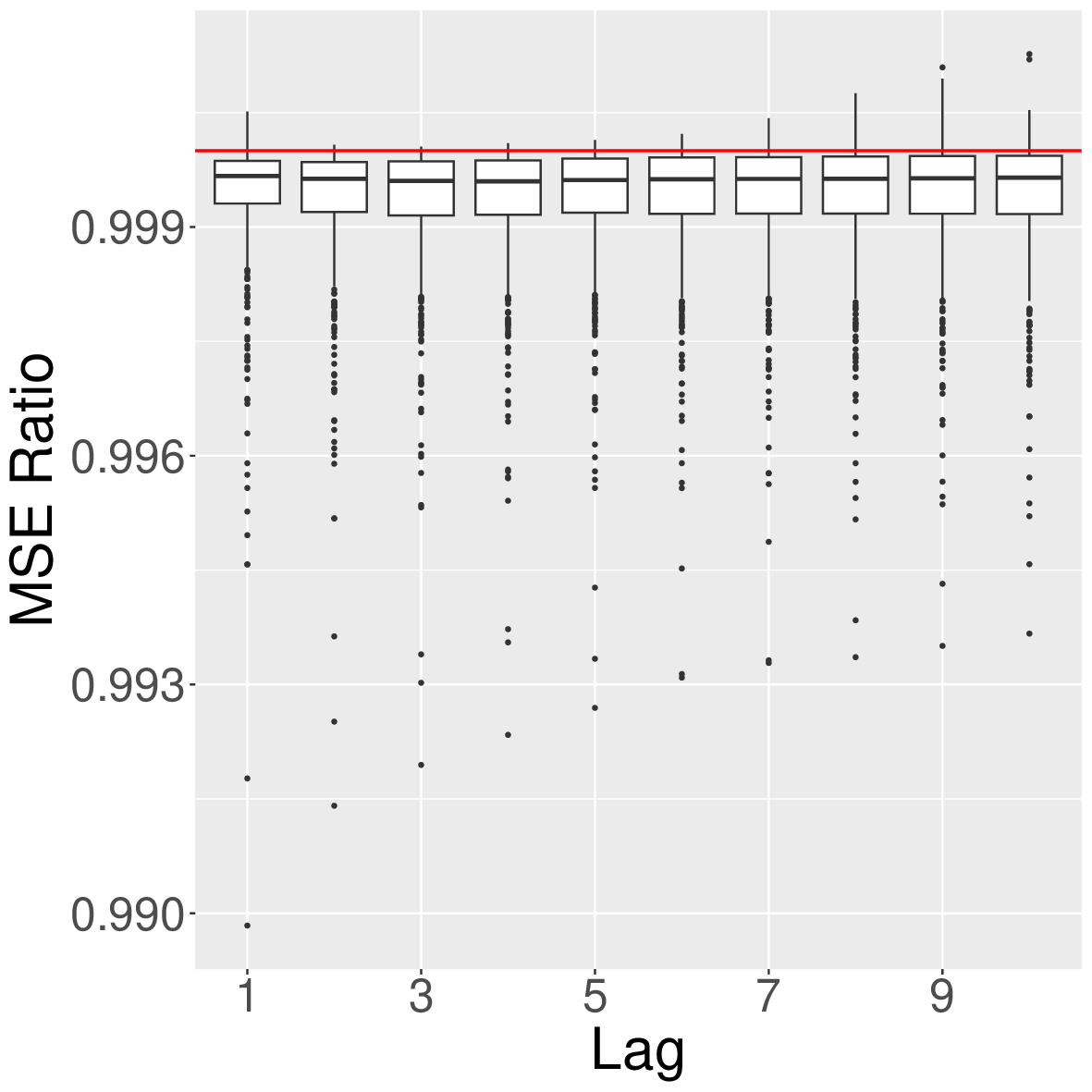}
}
\hfill
\subfloat[Nonparametric versus Exp]{%
  \includegraphics[width=0.32\textwidth, height=3cm]{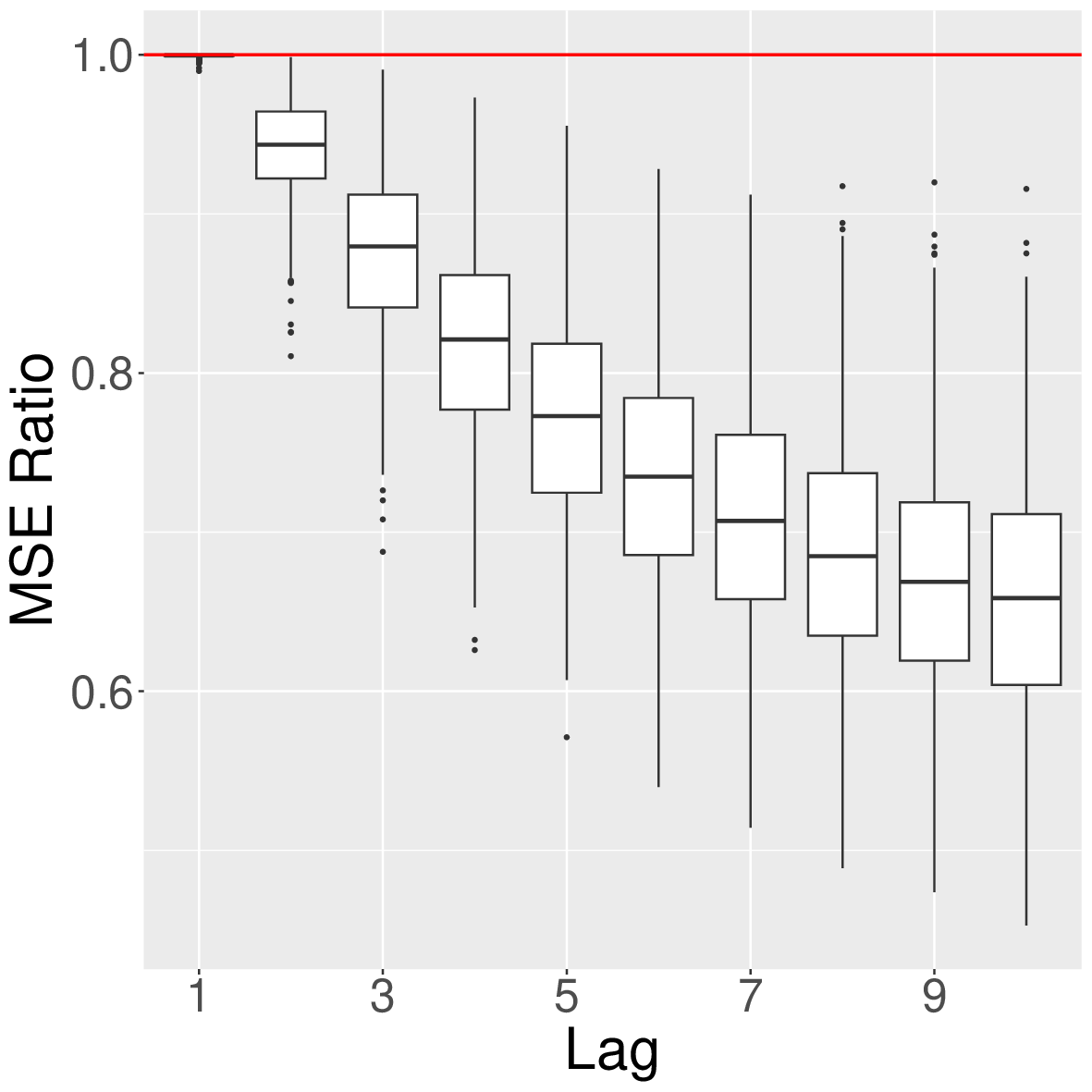}
}
\hfill
\subfloat[Nonparametric versus LM]{%
  \includegraphics[width=0.32\textwidth, height=3cm]{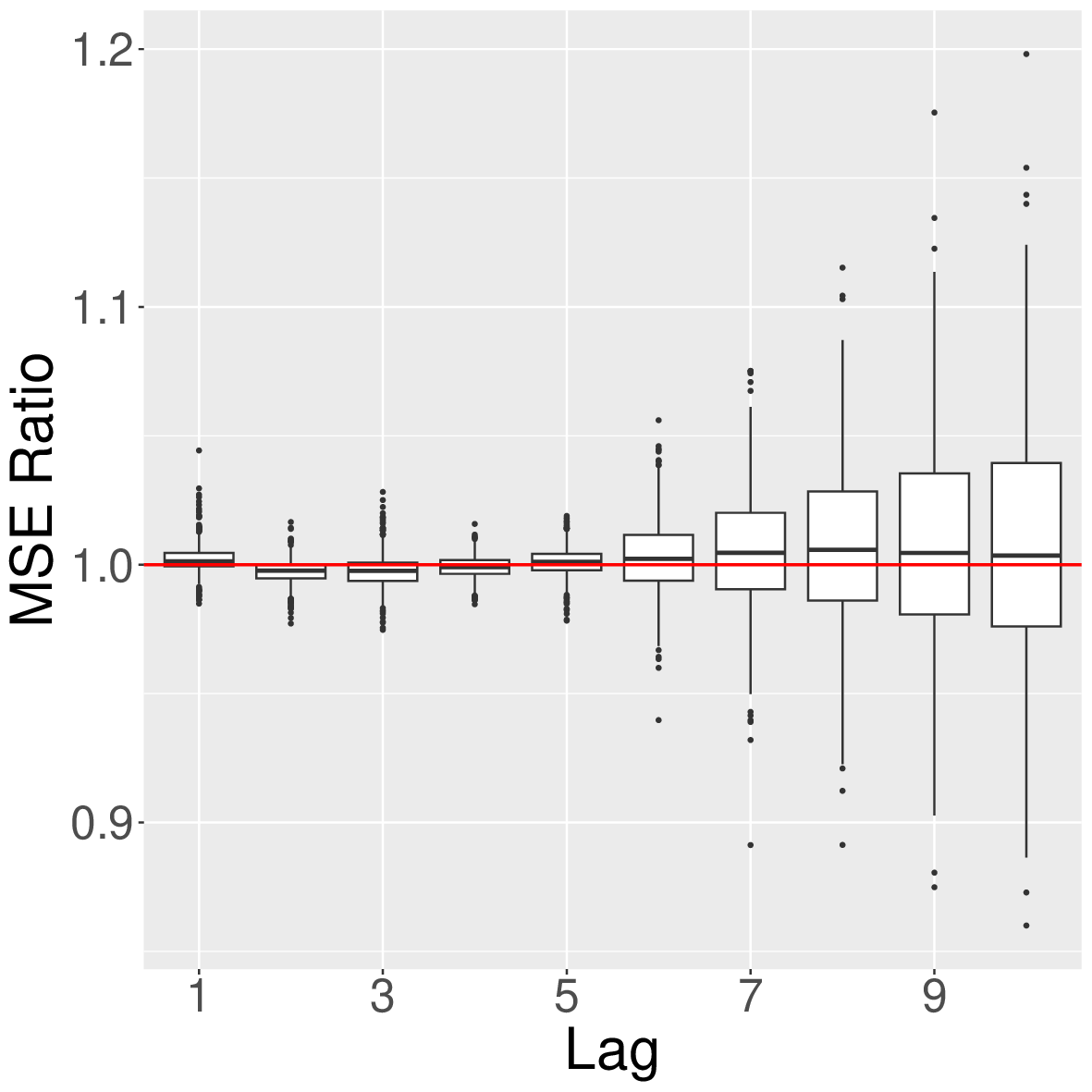}
}
	\caption{\it Here we compare the boxplots of estimated the MSE ratios (nonparametric vs.~ACF-, exponential-, long-memory estimators) for stock A. A ratio of $<1$ indicates smaller MSEs of the nonparametric estimator.}
 \label{fig:ForecastRatios_A}
\end{figure}

\subsection{Application to estimating the busy time distribution of queues}
We will now present a distinctly different application of our new methodology. As mentioned in \cite{BNLSV2014}, there is a close connection between integer-valued trawl processes and stochastic queues. More precisely, let us 
consider a Poisson trawl process in the following. Such a process is closely related to an  M/G/$\infty$ queue, in particular, we can write
$X_t= L(A_t)=X_0+(X_t-X_0)=L(A_t\cap A_0)+\sum_{j=1}^{N_t}\mathbb{I}_{\{\tau_j+a^{-1}(U_j)>t\}}$,
where the $(\tau_j)$s are the arrival times from the Poisson process with rate $v>0$, assuming that $L'\sim \mathrm{Poi}(\nu)$. Here, $X_t$ denotes the number of jobs, people, or particles in a stochastic queue at time $t$.
We denote by 
$F(x):= \mathbb{P}(\tau_j-\tau_{j-1}\leq x)$, for $x\geq 0, j \in \mathbb{N}$
the interarrival time distribution, which is here given by the $\mathrm{Exp}(v)$ distribution. 
Also, $(U_j)$ are independent and identically distributed random variables on $(0, a(0))$.
We can regard the random variables 
 $\zeta_j=a^{-1}(U_j)$ as  service times having cumulative distribution function  
 $H(x)=\mathbb{P}(\zeta_j\leq x)=\mathbb{P}(a(x)\leq U_j)=1-a(x)/a(0)$,
  for $x\geq 0$.
 I.e.~$a(x)=a(0)(1-H(x))$.
 In order to simplify the exposition, we shall from now on assume that $a(0)=1$.
If $X_0=0$, then we would be in the classical setting of an M/G/$\infty$ queue, where the process is typically assumed to start at 0 at time 0.  
In order to adapt the classical queuing framework to the situation of trawl processes, we consider a conditional framework, conditioning on $X_0=0$ in the trawl framework. More precisely, 
we define the conditional probability
$\mathbb{P}_k(t):=\mathbb{P}(X_t=k|X_0=0)$. 
 From  \citet[Lemma 4.1 and Proposition 4.1]{BLSV2021}, we get that, 
		    $\mathbb{P}_0(t)=\mathbb{P}(X_{t}=0|X_0=0)
=\mathbb{P}(L(A_{t}\setminus A_0)=0)$.
For the Poisson trawl, where $L'\sim\mathrm{Poi}(v)$,  we have $X_t \sim \mathrm{Poi}(\Leb(A) v)$,
where 
$\Leb(A)=\int_0^{\infty}a(u)du$ and 
\begin{align}\label{eq:P0}
\mathbb{P}_0(t)
    &=
     \exp(-\Leb(A_t\setminus A_0)v). 
\end{align}

\begin{remark} We note that, if $L'\sim \mathrm{NegBin}(m,\bar p)$, for  $m>0$ and $\bar p \in [0,1]$, then  $L(A_t\setminus A_0)  \sim \mathrm{NegBin}(m \Leb(A_t\setminus A_0),\bar p)$.  
I.e.~we have, for all $t \geq 0$,
		$\mathbb{P}(L(A_t\setminus A_0)  = x) =  \frac{\Gamma(\Leb(A_t\setminus A_0)m+x)}{x!\Gamma(\Leb(A_t\setminus A_0)m)} (1-\bar p)^{\Leb(A_t\setminus A_0)m} \bar p^x$, $x= 0, 1, 2, \ldots$,
	and
   $\mathbb{P}_0(t)=   (1-\bar p)^{\Leb(A_t\setminus A_0)m}$. 
 Recall that the negative binomial distribution can be represented as a compound Poisson distribution with logarithmic distributed jump sizes. In the context of queues, this model (along with other compound Poisson distributions with jump sizes in $\mathbb{N}$) can be interpreted as a system, where jobs arrive in batches and are serviced in those same batches. 
\end{remark}

\subsubsection{Busy times of trawls}

In queueing theory, the concept of the busy time has attracted considerate interest, see e.g.~\cite{T1962}, \cite{Stadje1985},  \cite{LiuShi1996}, \cite{ARTALEJO2007}.
We shall now transfer this concept to the framework of trawl processes via its connection to M/G/$\infty$ queues.
The framework of the M/G/$\infty$ queue aligns with the process of counting particles using a Type II counter, see \citet[Chapter 6]{T1962}.
We say that the system under consideration is in state $E_k$ at time $t$ when $X_t=k$ given that $X_0=0$. Further, we denote by $M_k(t)$ the expected number of transitions from state $E_k$ to state $E_{k+1}$ in the time interval $(0,t]$. 
Using this notation, we regard $E_0$ as the idle time, since $X_t=0$ in that state and no servers are working. 
We denote by $\tau_{\nu_1}, \tau_{\nu_2}, \ldots$ the arrival times of particles when there are no other particles present, these are often referred to as the arrival times of the ``registered particles''. 
We are interested in studying the cumulative distribution function, $B$, of the busy time of the trawl process. 
\begin{theorem} Consider a Poisson trawl process with $a(0)=1$. Define
$C(t):=1-e^{-v \int_0^{t}a(u)du}$ and 
$c(t):=C'(t)=va(t)(1-C(t))$.
Then the 
    cumulative distribution function of the length of a busy period, conditional on $X_0=0$,  is given by
        $B(t)=1-\frac{1}{v}\sum_{n=1}^{\infty}c^{*n}(t)$
    where $c^{*n}$ denotes the $n$-fold convolution of $c$ with itself.
\end{theorem}
\begin{proof}
We note that the proof of the above result follows the same arguments used by 
\cite{Stadje1985} and \cite{T1962}, where we plug in the functional for $\mathbb{P}_0$ obtained in \eqref{eq:P0} rather than the corresponding function obtained for classical queues.
\end{proof}

\begin{figure}[htbp]
\centering
	\captionsetup[subfigure]{aboveskip=-4pt,belowskip=-4pt}
	\subfloat[	\label{fig:busytime}]{	\includegraphics[height=3cm, width=0.45\textwidth]%
 {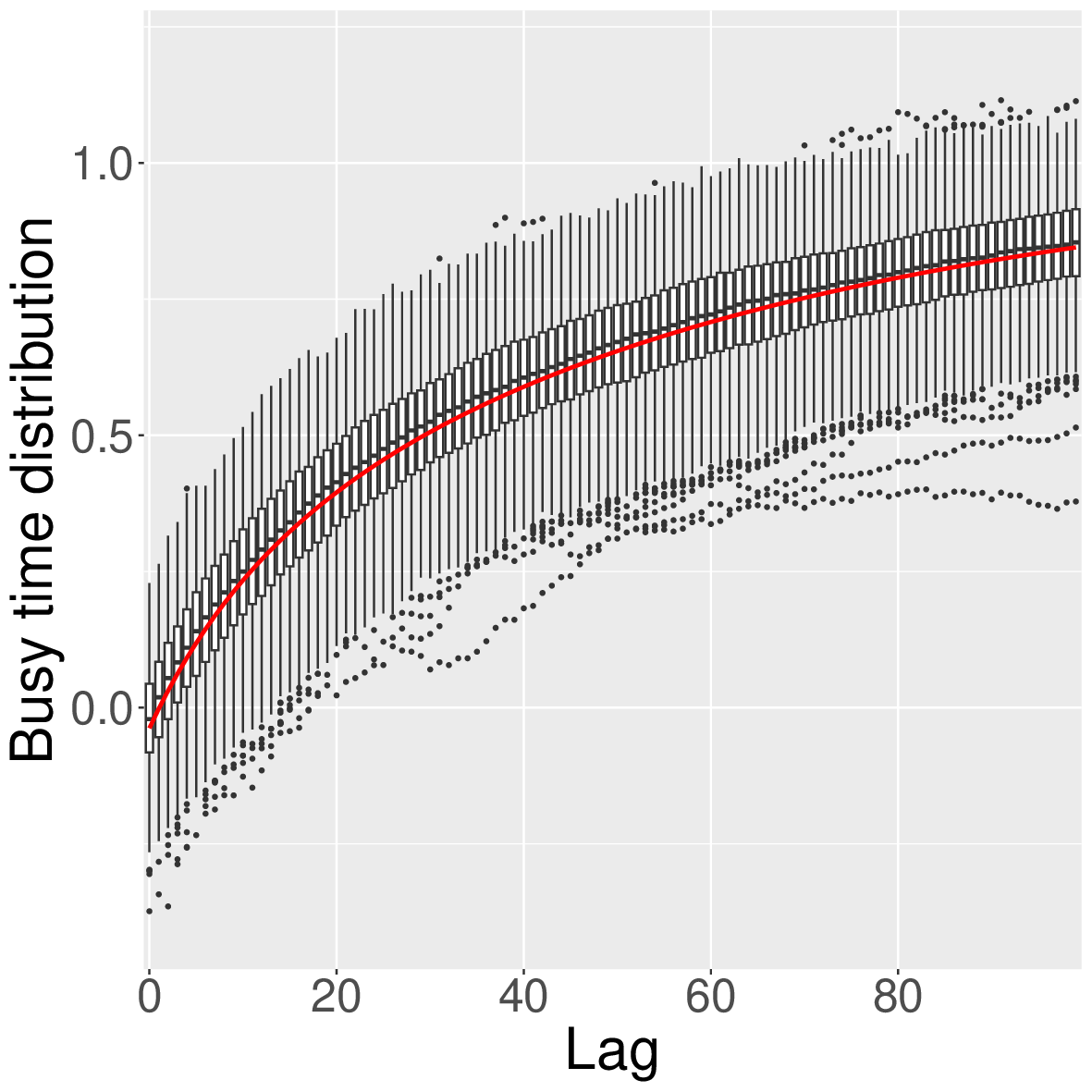} } 
	\captionsetup[subfigure]{aboveskip=-4pt,belowskip=-4pt}
	\captionsetup[subfigure]{aboveskip=-4pt,belowskip=-4pt}
	\subfloat[ \label{fig:busytime-trunc}]{	\includegraphics[height=3cm, width=0.45\textwidth]
 {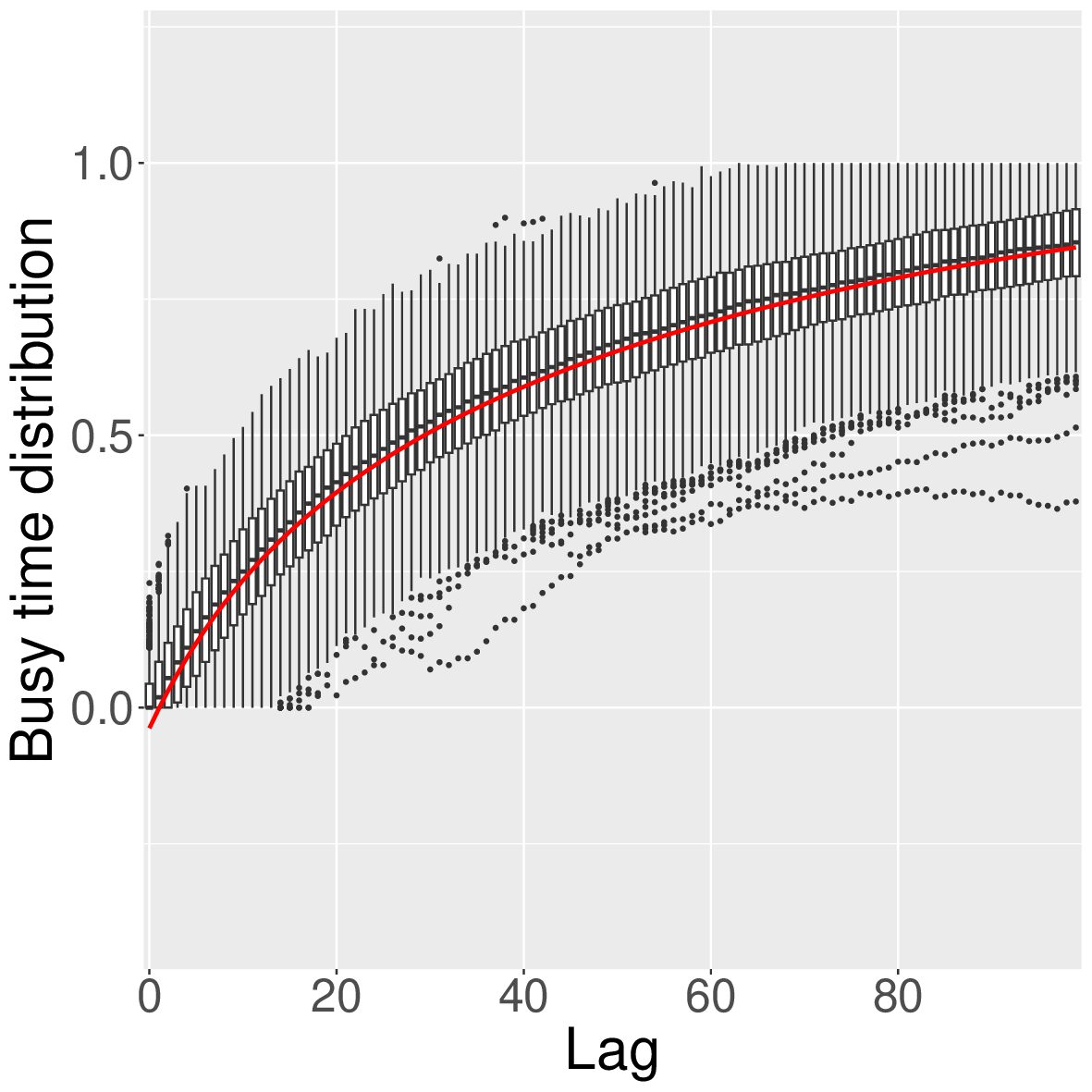} } 
	\caption{\it 
 Boxplots of the estimated busy time distribution for a Poisson-Exponential trawl process for 1000 sample paths.  Figure \ref{fig:busytime} shows the estimated busy time distribution and Figure \ref{fig:busytime-trunc} shows the modified estimator that is truncated between 0 and 1. The solid red line depicts the numerical approximation of the theoretical busy time distribution function.}
 \label{Trawl-BusyTimecdf}
\end{figure}
In the function $C$, the integrated trawl function appears, which can either be estimated using an approximation of the integral as in the context of estimating slices discussed earlier. Alternatively, we note that 
    $\int_0^{t}a(u)du= \Var(X_0)-\Cov(X_t,X_0)$,
which implies that the empirical variance and covariance can be used for estimating the integrated trawl function appearing in the function $C$. 
We note that the busy time distribution features the trawl function $a$ explicitly in the function $c$. Hence, an estimator of $a$ is needed.
 In Figure \ref{Trawl-BusyTimecdf}, we illustrate the estimation of the busy time distribution for a Poisson trawl process with $L'\sim\mathrm{Poi}(1)$ and $a(x)=e^{- x}$, for $x \geq 0$.

\section{Conclusion and discussion}
Trawl processes constitute a large  class of stationary, infinitely divisible stochastic processes which can account for short and long memory and flexible tail behaviour. This paper has established the first nonparametric estimation theory for their trawl function, proved the consistency of the estimator and derived both an infeasible and a feasible version of its asymptotic Gaussianity.
We stress that the selection of any tuning parameters is not needed for the estimator itself, but is required when constructing asymptotic confidence bounds. In our simulation experiments, however, we found a low sensitivity to varying choices of the tuning parameters.
Note also that the  limit theory for the inference on the trawl function is  not influenced by  the memory of the process, which suggest a broad applicability of the new methodology. 
       
We recall that for guaranteeing the identifiability of the trawl function we need to either fix the variance of the L\'{e}vy seed,
which we did by setting $\Var(L')=1$, or the $\Leb(A)$, which could be set to 1. While the former assumption aligns well with equivalent assumptions in Brownian- or general L\'{e}vy-driven models, the latter assumption has also theoretical appeal: If trawl functions were chosen amongst the class of (monotonic) density functions, which guarantees that $\Leb(A)=1$, then the marginal law of the resulting trawl process $L(A_t)$ would be exactly the same as the marginal law of $L'$, whereas in the case considered in this article, the cumulant function of the trawl process and of the corresponding L\'{e}vy seed differ by the proportionality factor $\Leb(A)$.

Our simulation study has revealed an excellent finite sample performance for the estimation of the trawl function, the trawl set and its slices over various choices of the observation scheme. 
Moreover, when using the asymptotic Gaussianity for approximating confidence bounds, we advise using the bias-corrected version of the trawl function estimator and carefully checking whether the observation scheme, i.e.~the choice of $n$ and $\Delta_n$, approximates the  theoretical assumptions well.

Finally, we note that nonparametric methods for estimating and forecasting trawl processes are powerful and simple alternatives to parametric procedures. They do not require the choice of any tuning parameters, as is often the case for other nonparametric methods. They avoid the need for parametric model selection,  are easy to implement and have good forecast performance as illustrated in our empirical work.

\paragraph{Supplementary material:}
The supplementary material   
 contains all the proofs of the theoretical results presented in the main article, the details of the simulation study and some additional figures from the empirical study.

\input{Supplement}

\input{bib}
\end{document}

%% file: Supplement.tex
\newpage
\begin{appendix}
 \section{Supplementary material} 
This supplementary material 
is structured as follows.
Section  \ref{sec:proofs} contains all the proofs of 
the  asymptotic theory presented in the main article.  
Sections \ref{asec:simsetup}, \ref{asec:consistency},
\ref{asec:asymGauss} and 
\ref{asec:slices} contain details on the simulation study:
Section \ref{asec:simsetup} describes the simulation setup, Section \ref{asec:consistency} presents the consistency results of the trawl function estimation, Section \ref{asec:asymGauss}
assesses the finite sample performance of the asymptotic Gaussianity and Section  \ref{asec:slices}  focusses on the trawl function, slice and slice ratio estimation. 
Moreover, Section \ref{sec:suppforSect51}
provides complementary plots for the analysis carried out in Section 5.1 of the main article, now focussing on the case of a Gaussian trawl process. 
Finally, Section \ref{asec:empirics} provides the plots for the stocks DFS, WAT and WM which were considered in the empirical study of the main article. 

\section{Proofs of the main results}
\label{sec:proofs}
We start by stating two key lemmas that will be used constantly during our proofs. The first one, whose proof is straightforward, concerns the moments of homogeneous \Levy bases. 

\begin{lemma}\label{lemmamoments}If $L$ is a centered homogeneous
	\Levy basis, then 
	\begin{align*}
		\mathbb{E}(L(A)^{2}) & =\Leb(A)(\sigma^{2}+\int x^{2}\nu(dx)).\\
		\mathbb{E}(L(A)^{3}) & =\Leb(A)\int x^{3}\nu(dx).\\
		\mathbb{E}(L(A)^{4}) & =\Leb(A)\left\{ c_{4}(L')+3\Leb(A)(\sigma^{2}+\int x^{2}\nu(dx))^{2}\right\} ,
	\end{align*}
    where $c_{4}(L')$ is as in Theorem  \ref{propconsistency}  in the main article. Moreover, if $A\subseteq B\subseteq C$, then 
	\[
	\mathbb{E}\left[L(A)L(B)L(C)\right]=\mathbb{E}\left[L(A)^{3}\right],
	\]
	and 
	\[
	\mathbb{E}\left[L(A)^{2}L(B)L(C)\right]=\mathbb{E}\left[L(A)^{4}\right]+\mathbb{E}\left[L(A)^{2}\right]\mathbb{E}\left[L(B\backslash A)^{2}\right].
	\]
\end{lemma}
The second key result gives conditions for the negligibility of sums of martingale differences. For a proof see Lemma 2.2.11 in \cite{JacProt11}.

\begin{lemma}\label{lemmapproxgammahat-1}Let $(\Omega,\mathscr{F},\mathbb{P})$
	be a complete probability space endowed with a sequence of discrete-time
	filtrations $(\mathscr{H}_{i}^{n})_{i\geq1}$. Suppose that $(D_{i}^{n}:i,n\geq1)$
	is a square-integrable $(\mathscr{H}_{i}^{n})_{i\geq1}$-martingale
	difference and that $N_{n}$ is an $\mathbb{N}$-valued $(\mathscr{H}_{i}^{n})_{i\geq1}$-stopping
	time increasing to $+\infty$. Then 
	\[
	\sum_{i=1}^{N_{n}}D_{i}^{n}\overset{\mathbb{P}}{\rightarrow}0,
	\]
	whenever 
	\[
	\sum_{i=1}^{N_{n}}\mathbb{E}\left[(D_{i}^{n})^{2}\mid\mathscr{H}_{i-1}^{n}\right]\overset{\mathbb{P}}{\rightarrow}0.
	\]
	
\end{lemma}

Next, for $i,j=0,1,\ldots$ and $j\geq i$ define
\begin{equation}
	\mathcal{P}_{A}^{n}(i,j):=\{(r,s):a(t_{j+1}-s)<r\leq a(t_{j}-s),t_{i-1}<s\leq t_{i}\},\label{Pijsetdef}
\end{equation}
where $t_{i}=i\Delta_{n}$ with the convention that $t_{-1}=-\infty$. Note that $\mathcal{P}_{A}^{n}(i,j)$ and $\mathcal{P}_{A}^{n}(i^{\prime},j^{\prime})$ are disjoint if either $i\neq i^{\prime}$ or $j\neq j^{\prime}$. Furthermore, for all $k=0,1,\ldots,$
\begin{equation}
\mathrm{Leb}(A_{k\Delta_n}\backslash\bigcup_{i=0}^{k}\bigcup_{j\geq k}\mathcal{P}_{A}^{n}(i,j))=0. \label{eq:reprA_slices}
\end{equation}
Using these sets, we introduce the filtrations: 
\begin{equation}
	\mathscr{F}_{k}^{n}  :=\sigma\left(\{L(\mathcal{P}_{A}^{n}(m,j))\}_{j\geq m}:m=0,\ldots,k\right),\,\,\,\,k\geq0,\label{filtrationovertime}
\end{equation}
and
\begin{equation}
	\mathscr{G}_{j}^{n}:=\sigma\left(\{L(\mathcal{P}_{A}^{n}(i,m))\}_{0\leq i\leq m}:m=0,\ldots,j\right),\,\,\,j\geq0.\label{filtrationoverspace}
\end{equation}
The following local approximations for $a$ and $\phi$ will
be often used through our computations:
\begin{equation}
	a_{n}(x):=\int_{0}^{1}a(s\Delta_{n}+x)ds,\,\,a_{n}^{\prime}(x):=\int_{0}^{1}\int_{0}^{1}\phi((s+r)\Delta_{n}+x)dsdr\,\,x\geq0.\label{phiapprox}
\end{equation}

\subsection{Proof of the consistency of $\hat{a}(t)$}\label{Secbasidec}
Throughout all our proofs below, the non-random positive constants will
be denoted by the generic symbol $C>0$, and they may change from
line to line. Furthermore, for every $t\geq0$ and $n\in\mathbb{N}$
we will let $l_{n}:=\lfloor t/\Delta_{n}\rfloor$. Let us now introduce some notation. The centred version of
$X$ will be denoted as $\tilde{X}_{t}:=X_{t}-\mathbb{E}(X_{t})$,
$t\in\mathbb{R}$. Similarly, we will let $\tilde{L}(\cdot):=L(\cdot)-\mathbb{E}(L(\cdot))$,
which is clearly a centered \Levy basis satisfying that $\tilde{X}_{t}=\tilde{L}(A_{t})$.
The partial sums of $\tilde{X}$ will be denoted as $S_{n}:=\sum_{k=0}^{n-1}\tilde{X}_{k\Delta_{n}}$.

Now we present the proof of Theorem  \ref{propconsistency}  in the main article. 
First, we derive a key decomposition for $\hat{a}(t)$ that will be used throughout all our proofs:  For all $t\geq0$, 
\begin{equation}
	\hat{a}(t)=\tilde{a}(t)+\mathfrak{b}_{n}(t),\label{eq:decouplinghata}
\end{equation}
where $\mathfrak{b}_{n}(0)=0$ and for $t>0$
\[
\mathfrak{b}_{n}(t):=\frac{1}{n\Delta_{n}}(X_{(n-1)\Delta_{n}}-X_{l_{n}\Delta_{n}})\left(\frac{1}{n}S_{n}\right)+\frac{1}{n\Delta_{n}}(X_{(n-1)\Delta_{n}}-\bar{X}_{n})(X_{(n-1-l_{n})\Delta_{n}}-\bar{X}_{n}),
\]
and 
\[
\tilde{a}(t):=\begin{cases}
	\frac{1}{2n\Delta_{n}}\sum_{k=0}^{n-2}(\delta_{k}\tilde{X})^{2} & \text{if }t=0;\\
	-\frac{1}{n\Delta_{n}}\sum_{k=l_{n}}^{n-2}\tilde{X}_{(k-l_{n})\Delta_{n}}\delta_{k}\tilde{X} & \text{if }t>0,
\end{cases}
\]
in which $\delta_{k}\tilde{X}:=\left(\tilde{X}_{(k+1)\Delta_{n}}-\tilde{X}_{k\Delta_{n}}\right)$.
\begin{proof}[Proof of Theorem  \ref{propconsistency}]
	Clearly $\mathfrak{b}_{n}(t)=\mathrm{o}_\mathbb{P}(1)$. Moreover, an application of \eqref{TrawlACF} lead to 
	\begin{equation}\label{expatilde}
		\mathbb{E}(\tilde{a}(t))=(1-\frac{l_n+1}{n})\frac{1}{\Delta_{n}}\int_{l_{n}\Delta_{n}}^{(l_{n}+1)\Delta_{n}}a(s)ds, \,\,\, \forall \,t\geq0.
	\end{equation}
	As a result, we obtain that $\mathbb{E}(\tilde{a}(t))\rightarrow a(t)$  as $n\rightarrow\infty$, thanks to Assumption \ref{as:basicsamplingscheme} and the continuity of $a$.  It is then left to show that $\tilde{a}(t)-\mathbb{E}(\tilde{a}(t))$ is asymptotically negligible. For $t>0$, we use the following identities (which are easily obtained via \eqref{eq:reprA_slices}), valid for $l_{n}\leq k\leq j\leq n-1$,
	\begin{equation}
		-\delta_{k}\tilde{X}=  \chi_{k}-\varrho_{k}; \,\,\, \tilde{X}_{(k-l_{n})\Delta_{n}}=  \beta_{k,j}^{(1)}+\beta_{k,j}^{(2)}+\sum_{m=j+1}^{n}\beta_{k,m}^{(1)}+\beta_{k}^{(3)},\label{decXbetas}
	\end{equation}
	where (see Figure \ref{figincrementsX})
	\begin{equation}
		\chi_{k}:=\tilde{L}(\cup_{i=0}^{k}\mathcal{P}_{A}^{n}(i,k)),\,\,\,\varrho_{k}:=\tilde{L}(\cup_{j\geq k+1}\mathcal{P}_{A}^{n}(k+1,j)),\label{defchirho}
	\end{equation}
	and 
	\begin{equation}
		\begin{aligned}\beta_{k,j}^{(1)} & =\tilde{L}(\cup_{i=0}^{k-l_{n}}\mathcal{P}_{A}^{n}(i,j));\\
			\beta_{k,j}^{(2)} & =\tilde{L}(\cup_{i=0}^{k-l_{n}}\cup_{m=k-l_{n}}^{j-1}\mathcal{P}_{A}^{n}(i,m));\\
			\beta_{k}^{(3)} & =\tilde{L}(\cup_{i=0}^{k-l_{n}}\cup_{m\geq n+1}\mathcal{P}_{A}^{n}(i,m)).
		\end{aligned}
		\label{defbetasdecompX}
	\end{equation}
    A visual representation of these quantities can be seen in Figure \ref{fig_def_betas}.
	Note that the array $(\tilde{X}_{(k-l_{n})\Delta_{n}}\varrho_{k})_{k\geq l_n}$ is an $	\mathscr{F}_{k+1}^{n}$ square integrable increment martingale with $\mathbb{E}[(\tilde{X}_{(k-l_{n})\Delta_{n}}\varrho_{k})^{2}]\leq C\Delta_{n}$, due to Lemma \ref{lemmamoments}. From Lemma \ref{lemmapproxgammahat-1}
	\[
	\tilde{a}(t)=\frac{1}{n\Delta_{n}}\sum_{k=l_{n}}^{n-1}\tilde{X}_{(k-l_{n})\Delta_{n}}\chi_k+\mathrm{o}_\mathbb{P}(1).
	\
	\]
	The consistency now follows from the relations (\ref{negligibilityerrors}) and (\ref{variance}) below, as well as the boundedness and integrability of $a$. The case $t=0$ can be shown analogously.  \end{proof}
\begin{figure}[h]
    \centering
    \includegraphics[width=0.6\linewidth]{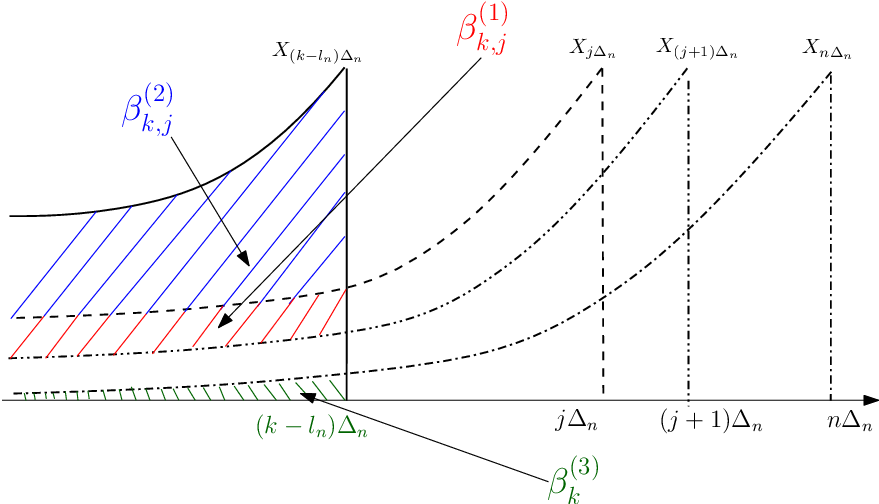}
    \caption{Graphical representation of the random variables introduced in \eqref{defbetasdecompX}.}
    \label{fig_def_betas}
\end{figure}

\subsection{Proof of the  infeasible central limit theorem }
Next, we turn to the proof of Theorem  \ref{thmCLTtrawlapprox}  in the main article. Before presenting our proof, let us state some simple facts.  Identities (\ref{eq:decouplinghata}) and (\ref{expatilde}) together with Assumption \ref{as:trawl}, reveal that for $t\geq0$
\begin{equation}
	\hat{a}(t)-a(l_{n}\Delta_{n})=\tilde{a}(t)-\mathbb{E}(\tilde{a}(t))+\frac{1}{\Delta_{n}}\int_{l_{n}\Delta_{n}}^{(l_{n}+1)\Delta_{n}}\left[a(s)-a(l_{n}\Delta_{n})\right]ds+\mathfrak{b}_{n}(t)+\mathrm{O}(1/(n\Delta_{n})).\label{errordecomp1}
\end{equation}
In view that $\tilde{X}_{t}=\tilde{L}(A_{t})$,
where $\tilde{L}(\cdot):=L(\cdot)-\mathbb{E}(L(\cdot))$ is a centered
\Levy basis, (\ref{errordecomp1}) indicates that without loss of generality
we may and do assume that $L$ is centred which in particular implies
that $\tilde{X}=X$. 
\begin{proof}[Proof of Theorem  \ref{thmCLTtrawlapprox}]
	
	First note that by \citep[Lemma 5.1]{PPSV2021}, Assumption
	\ref{as:samplingscheme} and the stationarity of $X$ we get
	\[
	\mathfrak{b}_{n}(t)=\mathrm{o}_{\mathbb{P}}(1/\sqrt{n\Delta_{n}}).
	\]
	Furthermore, Assumption \ref{as:trawl} easily implies that
	\begin{equation}\label{decompdiscretazationerror}
		\frac{1}{\Delta_{n}}\int_{l_{n}\Delta_{n}}^{(l_{n}+1)\Delta_{n}}\left[a(s)-a(l_{n}\Delta_{n})\right]ds=-\Delta_{n}\int_{0}^{1}\int_{0}^{y}\phi(\Delta_{n}x+l_{n}\Delta_{n})dxdy.
	\end{equation}
	Therefore, thanks to (\ref{errordecomp1}) and the Dominated Convergence
	Theorem, we deduce that
	\begin{equation}
		\hat{a}(t)-a(l_{n}\Delta_{n})=[\tilde{a}(t)-\mathbb{E}(\tilde{a}(t))]-\frac{\Delta_{n}}{2}\phi(t)+\mathrm{o}_{\mathbb{P}}(1/\sqrt{n\Delta_{n}})+\mathrm{o}(\Delta_{n}).\label{decompahat}
	\end{equation}
	The key part of the proof is showing that
	\begin{equation}
		\mathcal{E}_{t}:=\sqrt{n\Delta_{n}}[\tilde{a}(t)-\mathbb{E}(\tilde{a}(t))]\overset{d}{\rightarrow}N(0,\sigma_{a}^{2}(t)).\label{statisticalerror}
	\end{equation}
	If this were true, then by (\ref{decompahat}) and Assumption \ref{as:samplingscheme},
	we would have that for $0\leq\mu<\infty$
	\[
	\sqrt{n\Delta_{n}}(\hat{a}(t)-a(l_{n}\Delta_{n}))=-\frac{1}{2}\sqrt{\mu}\phi(t)+N(0,\sigma_{a}^{2}(t))+\mathrm{o}_{\mathbb{P}}(1),
	\]
	while for $\mu=+\infty$
	\[
	\frac{\hat{a}(t)-a(l_{n}\Delta_{n})}{\Delta_{n}}=-\frac{1}{2}\phi(t)+\mathrm{o}_{\mathbb{P}}(1),
	\]
	which is enough. The verification of (\ref{statisticalerror}) in
	divided in three parts: In the first step we write $\mathcal{E}_{t}$
	as a sum of an increment martingale plus a negligible term. In the
	last two steps we derive the asymptotic variance and check that the
	Lyapunov condition is satisfied under our assumptions. 
	
	\subsubsection*{Step 1: An increment martingale decomposition for $\mathcal{E}_{t}$}
   
	Let us start by assuming that $t>0$ and observe that in this situation
	(recall that $L$ is assumed to be centred) 
	\[
	\mathcal{E}_{t}=-\frac{1}{\sqrt{n\Delta_{n}}}\sum_{k=l_{n}}^{n-1}\left\{ X_{(k-l_{n})\Delta_{n}}\delta_{k}X-\mathbb{E}\left(X_{(k-l_{n})\Delta_{n}}\delta_{k}X\right)\right\} +\mathrm{o}_{\mathbb{P}}(1),
	\]
	Furthermore, using (\ref{decXbetas}) and that
	\begin{equation}
		-\delta_{k}X=  \chi_{k}-\sum_{j=k+1}^{n}\alpha_{k+1,j}-\gamma_{k},\label{incremrel}
	\end{equation}
	with
	\begin{equation}
		\alpha_{k,j}=L(\mathcal{P}_{A}^{n}(k,j)),\,\,\,\gamma_{k}:=L(\cup_{j\geq n+1}\mathcal{P}_{A}^{n}(k+1,j)),\label{defalpha}
	\end{equation}
	we further deduce that 
	\begin{equation}
		\mathcal{E}_{t}=\frac{1}{\sqrt{n\Delta_{n}}}\sum_{j=l_{n}+1}^{n-1}\zeta_{j,n}+\sum_{q=1}^{3}r_{n,q}+\mathrm{o}_{\mathbb{P}}(1),\label{errorrepmtg0}
	\end{equation}
	where $\zeta_{j,n}=\sum_{p=1}^{5}\xi_{j,n}^{(p)}$ in which 
	\begin{equation}
		\begin{aligned}\xi_{j,n}^{(1)} & =\left[\chi_{j}\beta_{j,j}^{(1)}-\mathbb{E}(\chi_{j}\beta_{j,j}^{(1)})\right];\,\,\,\xi_{j,n}^{(2)}=\chi_{j}\beta_{j,j}^{(2)},\,\,\,\xi_{j,n}^{(3)}=\sum_{k=l_{n}}^{j-1}\chi_{k}\beta_{k,j}^{(1)};\\
			\xi_{j,n}^{(4)} & =-\sum_{k=l_{n}}^{j-1}\alpha_{k+1,j}\beta_{k,j}^{(2)};\,\,\,\xi_{j,n}^{(5)}=-\sum_{m=l_{n}+1}^{j}\sum_{k=l_{n}}^{m-1}\alpha_{k+1,m}\beta_{k,j}^{(1)}=-\sum_{k=l_{n}}^{j-1}\sum_{m=k+1}^{j}\alpha_{k+1,m}\beta_{k,j}^{(1)}
		\end{aligned}
		\label{errorstpos}
	\end{equation}
	Furthermore, 
\begin{equation}\label{residuals_def}
\begin{aligned}r_{n,1}= & \frac{1}{\sqrt{n\Delta_{n}}}\sum_{k=l_{n}}^{n-1}\chi_{k}\beta_{k}^{(3)},\,\,r_{n,2}=-\frac{1}{\sqrt{n\Delta_{n}}}\sum_{k=l_{n}}^{n-1}X_{(k-l_{n})\Delta_{n}}\gamma_{k}\\
r_{n,3}= & -\frac{1}{\sqrt{n\Delta_{n}}}\sum_{k=l_{n}}^{n-1}\beta_{k}^{(3)}L(\cup_{j=k+1}^{n}\mathcal{P}_{A}^{n}(k+1,j)).
\end{aligned}
\end{equation}    
	We proceed now to show that $r_{n,q}$ are asymptotically negligible
	for $q=1,2,3$. Observe first that since $(\chi_{k})_{k=l_{n},\ldots,n-1}$
	are i.i.d., centered, and independent of $(\beta_{k}^{(3)})_{k=l_{n},\ldots,n-1}$,
	then the array $(\chi_{k}\beta_{k}^{(3)})_{k=l_{n},\ldots,n-1}$ is
	a martingale difference w.r.t. the filtration 
	\[
	\mathscr{G}_{j}^{n}\lor\sigma(\{\beta_{k}^{(3)}:k=l_{n},\ldots,n-1\}),\,\,j=l_{n},\ldots,n-1.
	\]
	
	It follows from this and the fact that $l_{n}/n=\mathrm{o}(1)$ that
	\begin{equation}\label{negligibilityerrors}
		\begin{aligned}
			\mathbb{E}(r_{n,1}^{2}) 
			& =\frac{1}{n\Delta_{n}}\int_{0}^{\Delta_{n}}a(s)ds\sum_{k=l_{n}}^{n-1}\int_{-\infty}^{t_{k}-l_{n}\Delta_{n}}a((n+1)\Delta_{n}-s)ds\\
			& =\mathrm{O}(1)\times\frac{1}{n\Delta_{n}}\int_{0}^{n\Delta_{n}}\int_{(n+1)\Delta_{n}- \lfloor r/\Delta_{n} \rfloor\Delta_{n}}^{\infty}a(s)dsdr+\mathrm{o}(1)\\
			& =\mathrm{O}(1)\times\int_{0}^{1}\int_{(\frac{n+1-\lfloor nx \rfloor}{n})\Delta_{n}n}^{\infty}a(s)dsdx+\mathrm{o}(1)\rightarrow0,
		\end{aligned}
	\end{equation}
	where in the last step we made the change of variables $x=r/(n\Delta_{n})$
	as well as applied the Dominated Convergence Theorem. Similar arguments
	can be used for $q=2,3$. Hence, we have that
	\begin{equation}
		\mathcal{E}_{t}=\frac{1}{\sqrt{n\Delta_{n}}}\sum_{j=l_{n}+1}^{n-1}\zeta_{j,n}+\mathrm{o}_{\mathbb{P}}(1),\,\,\,t>0,\label{decomperrormtg}
	\end{equation}
	where $\zeta_{j,n}$ is defined as a sum of the r.v.'s introduced
	in (\ref{errorstpos}). On the other hand, by arguing as above, it
	is easy to see that (\ref{decomperrormtg}) is valid also for $t=0$
	but in this situation $\zeta_{j,n}=\frac{1}{2}\sum_{q=1}^{2}\xi_{j,n}^{(q)}$
	in which 
	\begin{equation}
		\begin{aligned}\xi_{j,n}^{(1)} & =\chi_{j}^{2}-\mathbb{E}(\chi_{j}^{2});\,\,\,\xi_{j,n}^{(2)}=\sum_{k=0}^{j-1}[\alpha_{k+1,j}^{2}-\mathbb{E}(\alpha_{k+1,j}^{2})].\end{aligned}
		\label{errorstpos-1}
	\end{equation}
	Note that by construction $\zeta_{j,n}$ is $\mathscr{G}_{j}^{n}$-measurable
	(see (\ref{filtrationoverspace})) and such that $\mathbb{E}(\zeta_{j,n}\mid\mathscr{G}_{j-1}^{n})=0$,
	i.e. the array $(\zeta_{j,n})_{j=l_{n}+1,\ldots,n-1}$ is an increment
	martingale. Hence, according to Theorem 2.2.13 in \cite{JacProt11},
	in order to finish the proof, we must verify that for some $p>2$
	\begin{align}
		1.\,\frac{1}{\Delta_{n}n}\sum_{j=l_{n}+1}^{n-1}\mathbb{E}(\zeta_{j,n}^{2}\mid\mathscr{G}_{j-1}^{n}) & \overset{\mathbb{P}}{\rightarrow}\sigma_{a}^{2}(t),\label{asymptoticvar}\\
		2.\,\frac{1}{(\Delta_{n}n)^{p}}\sum_{j=l_{n}+1}^{n-1}\mathbb{E}(\mid\zeta_{j,n}\mid^{p}\mid\mathscr{G}_{j-1}^{n}) & \overset{\mathbb{P}}{\rightarrow}0.\label{linderbergfellercond}
	\end{align}
	For the rest of the proof, we will verify that this is indeed the
	case.
	
	\subsubsection*{Step 2: The asymptotic variance of $\mathcal{E}_{t}$}

	In Lemma \ref{lemmaavar} below, we show that
	\begin{equation}\label{approxAVAR}
		\frac{1}{\Delta_{n}n}\sum_{j=l_{n}+1}^{n-1}\mathbb{E}(\zeta_{j,n}^{2}\mid\mathscr{G}_{j-1}^{n})=\frac{1}{\Delta_{n}n}\sum_{j=l_{n}+1}^{n-1}\mathbb{E}(\zeta_{j,n}^{2})+\mathrm{o}_{\mathbb{P}}(1).
	\end{equation}

	Once again, we start by considering $t>0$. By the independent scattered
	property of $L$ and \eqref{eq:condcovar} below (see the proof of Lemma \ref{lemmaavar}) we have that for $p=2,3,4$, $\mathbb{E}[\xi_{j,n}^{(1)}\xi_{j,n}^{(p)}]=0$, and $\mathbb{E}[\xi_{j,n}^{(1)}\xi_{j,n}^{(5)}]=\mathrm{O}(\Delta_n)$, uniformly on $j$. Hence,
	\[
	\mathbb{E}(\zeta_{j,n}^{2})=\sum_{p=1}^{5}\mathbb{E}[(\xi_{j,n}^{(p)})^{2}]+2\sum_{p=3}^{5}\sum_{q=2}^{p-1}\mathbb{E}[\xi_{j,n}^{(p)}\xi_{j,n}^{(q)}]+\mathrm{O}(\Delta_n),
	\]
uniformly on $j$. Lemma \ref{lemmamoments} (recall that $\mathrm{Var}(L^{\prime})^{2}=1$)
	and relations \eqref{eq:condvariances} and  \eqref{eq:condcovar} below, imply that 
	\begin{equation}
		\begin{aligned}\mathbb{E}[(\xi_{j,n}^{(1)})^{2}]= & \Delta_{n}c_{4}(L')a_{n}(l_{n}\Delta_{n})+\mathrm{O}(\Delta_{n}^{2});\\
			\mathbb{E}[(\xi_{j,n}^{(2)})^{2}]= & \Delta_{n}a_{n}(0)\int_{0}^{l_{n}\Delta_{n}}a(s)ds;\\
			\mathbb{E}[(\xi_{j,n}^{(3)})^{2}]= & \Delta_{n}a_{n}(0)\int_{l_{n}\Delta_{n}}^{t_{j}}a_{n}(\left\lceil s/\Delta_{n}\right\rceil \Delta_{n})ds;\\
			\mathbb{E}[(\xi_{j,n}^{(4)})^{2}]= &\Delta_{n}a_{n}(0)\int_{0}^{(l_{n}+1)\Delta_{n}}a(s)\mathrm{d}s\\
  & +\Delta_{n}\int_{(l_{n}+1)\Delta_{n}}^{t_{j}}a_{n}(\lfloor s/\Delta_{n}\rfloor\Delta_{n})a(s-l_{n}\Delta_{n})\mathrm{d}s\\
  &-\int_{0}^{t_{j}}a(u)\mathrm{d}u\int_{t_{j}}^{t_{j+1}}a(s-l_{n}\Delta_{n})\mathrm{d}s\\
			\mathbb{E}[(\xi_{j,n}^{(5)})^{2}]= & \Delta_{n}\int_{\Delta_{n}}^{t_{j}-l_{n}\Delta_{n}}a_{n}(\left\lceil s/\Delta_{n}\right\rceil \Delta_{n}+l_{n}\Delta_{n})\left(a_{n}(0)-a(s)\right)ds.
		\end{aligned}
		\label{variance}
	\end{equation}
	as well as 
	\begin{equation}
		\begin{aligned}-\mathbb{E}[\xi_{j,n}^{(3)}\xi_{j,n}^{(4)}]= & \Delta_{n}\sum_{k=2l_{n}+1}^{j-1}\int_{l_{n}\Delta_{n}}^{t_{k}-l_{n}\Delta_{n}}\left[a(t_{j}-s)-a(t_{j+1}-s)\right]\\
			& \times a_{n}(t_{k}-\lfloor s/\Delta_{n}\rfloor\Delta_{n}+l_{n}\Delta_{n})ds\\
			-\mathbb{E}[\xi_{j,n}^{(3)}\xi_{j,n}^{(5)}]= & \Delta_{n}\int_{0}^{t_{j}-l_{n}\Delta_{n}}a_{n}(\left\lceil s/\Delta_{n}\right\rceil \Delta_{n}+l_{n}\Delta_{n})\left[a_{n}(0)-a(s)\right]ds;\\
			\mathbb{E}[\xi_{j,n}^{(4)}\xi_{j,n}^{(5)}]= & 0.
		\end{aligned}
		\label{cov3}
	\end{equation}
	Furthermore, by the definition of $\xi_{j,n}^{(2)}$, if $t=0$, then
	$\mathbb{E}[\xi_{j,n}^{(2)}\xi_{j,n}^{(p)}]=0$ for all $p=1,\ldots,5$,
	while for $t>0$
	\begin{equation}
		\begin{aligned}\mathbb{E}[\xi_{j,n}^{(2)}\xi_{j,n}^{(3)}]= & \Delta_{n}\int_{0}^{l_{n}\Delta_{n}\wedge(t_j-l_{n}\Delta_{n})}a(s+l_{n}\Delta_{n})a_{n}(l_{n}\Delta_{n}-\lfloor s/\Delta_{n}\rfloor  \Delta_{n})ds\\
			-\mathbb{E}[\xi_{j,n}^{(2)}\xi_{j,n}^{(4)}]= & \Delta_{n}\int_{\Delta_{n}}^{t_{j}-l_{n}\Delta_{n}}\left[a(s)-a(s+l_{n}\Delta_{n})\right]\left[a_{n}(0)-a_{n}(\left\lceil s/\Delta_{n}\right\rceil \Delta_{n}-\Delta_{n})\right]ds\\
			& +\Delta_{n}\int_{t_{j}}^{\infty}[a(u-l_{n}\Delta_{n})-a(u)]ds(a_{n}(0)-a_{n}(t_{j}-l_{n}\Delta_{n}))\\
			-\mathbb{E}[\xi_{j,n}^{(2)}\xi_{j,n}^{(5)}]= & \Delta_{n}\int_{l_{n}\Delta_{n}}^{t_{j}-l_{n}\Delta_{n}}\left[a(s-l_{n}\Delta_{n})-a(s)\right]a_{n}(\left\lceil s/\Delta_{n}\right\rceil \Delta_{n}+l_{n}\Delta_{n})ds,
		\end{aligned}
		\label{cov2}
	\end{equation}
	Note that by swapping the order of summation and making the change
	of variables $x=r/(n\Delta_{n})$, we obtain that as $n\rightarrow\infty$
	\begin{equation}
		\begin{aligned}
			&\frac{1}{\Delta_{n}n}\sum_{j=l_{n}+1}^{n-1}\mathbb{E}[\xi_{j,n}^{(3)}\xi_{j,n}^{(4)}]\\
			= & -\frac{1}{n\Delta_{n}}\int_{(2l_{n}+1)\Delta_{n}}^{(n-1)\Delta_{n}}\int_{l_{n}\Delta_{n}}^{( \lfloor r/\Delta_{n} \rfloor-l_{n})\Delta_{n}}a_{n}(\left\lceil s/\Delta_{n}\right\rceil \Delta_{n}+l_{n}\Delta_{n})a(s+\Delta_{n})dsdr+o(1)\\
			\rightarrow & -\int_{t}^{\infty}a(s+t)a(s)ds,
		\end{aligned}
	\end{equation}
	where in the last step we  used the Dominated Convergence Theorem
	(recall that $a$ is continuous, integrable and totally bounded).
	Similar arguments can be applied to (\ref{variance}), (\ref{cov3})
	and (\ref{cov2}), to conclude that
	(\ref{asymptoticvar}) holds.
	
	Suppose now that $t=0$. Then
	\begin{align*}
		\mathbb{E}[(\xi_{j,n}^{(1)})^{2}\mid\mathscr{G}_{j-1}^{n}]=\mathbb{E}[(\xi_{j,n}^{(1)})^{2}]=&\Delta_{n}c_{4}(L')a_{n}(0)+\mathrm{O}(\Delta_{n}^{2});\\
		\mathbb{E}[(\xi_{j,n}^{(2)})^{2}\mid\mathscr{G}_{j-1}^{n}]=\mathbb{E}[(\xi_{j,n}^{(2)})^{2}]=&c_{4}(L')\sum_{k=0}^{j-1}\Leb(\mathcal{P}_{A}^{n}(k+1,j))(1+\mathrm{O}(\Delta_{n}));\\
		\mathbb{E}[\xi_{j,n}^{(1)}\xi_{j,n}^{(2)}\mid\mathscr{G}_{j-1}^{n}]=\mathbb{E}[\xi_{j,n}^{(1)}\xi_{j,n}^{(2)}]=&=c_{4}(L')\sum_{k=0}^{j-1}\Leb(\mathcal{P}_{A}^{n}(k+1,j))(1+\mathrm{O}(\Delta_{n})).
	\end{align*}
	Relation (\ref{asymptoticvar}) follows easily from this.
	
	\subsubsection*{Step 3: The Lyapunov condition}
  	
	Here we only consider the case $t>0$ since the case $t=0$ can be
	analysed in the same way. Thus, in order to achieve (\ref{linderbergfellercond}),
	we will check that for some $r>2$ and every $M=1,2,\ldots,5$ as
	$n\rightarrow\infty$
	\begin{equation}
		\frac{1}{(\Delta_{n}n)^{r/2}}\sum_{j=l_{n}+1}^{n-1}\mathbb{E}(\mid\xi_{j,n}^{(M)}\mid^{r})\rightarrow0.\label{lyapunov}
	\end{equation}
To achieve this we will show that if  $\mathbb{E}(\rvert L^{\prime}\rvert^{2r})<\infty$ for some $r\geq2$, then there
is a constant $C>0$ only depending on  $r,a$ and the characteristic triplet of $L$, such that 
\begin{equation}\label{eq:momentestchi}
\begin{aligned}
      \mathbb{E}(\rvert\xi_{j,n}^{(M)}\rvert^{r})&\leq C\Delta_{n};\,\,\,M=1,2,4;   \\
   \mathbb{E}(\rvert\xi_{j,n}^{(M)}\rvert^{r})&\leq C(\Delta_{n}+\Delta_{n}\int_{0}^{n\Delta_{n}}a(s)s^{r/2-1}\mathrm{d}s),\,\,\,M=3,5.
\end{aligned}
\end{equation} If this were true,  by choosing $2<r<\frac{p_{0}}{2}\land2\alpha$, we conclude from Assumption \ref{as:samplingscheme} that for every $M=1,...,5$, $\mathbb{E}(\mid\xi_{j,n}^{(M)}\mid^{r})\leq C\Delta_n$, from which  \eqref{lyapunov} trivially follows. Thus, we are left to check that \eqref{eq:momentestchi} indeed holds. While doing so, we will constantly use the identities
	presented in Lemma \ref{lemmamoments} along with the estimate
 \begin{equation}
\mathbb{E}(\rvert L(B)\rvert^{r})\leq C(\mathrm{Leb}(B)^{r/2}\lor\mathrm{Leb}(B)),\,\,\forall\,\ensuremath{B\in\mathcal{B}(\mathbb{R}^{2}),}\label{momentineq}
\end{equation}  
($C$ only depends on   $r$ and the characteristic triplet of $L$) whose proof can be found in e.g. \cite{LuschPages08}  as well as \cite{Turner11} (Corollary 1.2.7). Thus, for the sake of readability, we will omit referring to such results and their use should be clear from the context.  The independent scattered property of $L$ results in
	\begin{align*}
		\mathbb{E}(\rvert\xi_{j,n}^{(1)}\rvert^{r}) & \leq 3^{r-1}[\mathbb{E}(\rvert\beta_{j,j}^{(1)}\rvert^{2r})+\mathbb{E}(\rvert\beta_{j,j}^{(1)}\rvert^{r})\mathbb{E}(\rvert L(\cup_{i=j-l_n+1}^{j}\mathcal{P}_{A}^{n}(i,j))\rvert^{r})+\mathbb{E}(\chi_{j}\beta_{j,j}^{(1)})^{r} ]\\
		& \leq C\mathrm{Leb}(\cup_{i=0}^{j-l_n}\mathcal{P}_{A}^{n}(i,j))\leq C a(0)\Delta_{n}.
	\end{align*}
Since for $r\geq2$, $\mathbb{E}(\rvert\beta_{j,j}^{(2)}\rvert^{r})\leq C\mathrm{Leb}(A)$, we also get that
\[\mathbb{E}(\rvert\xi_{j,n}^{(2)}\rvert^{r}) = \mathbb{E}(\rvert\chi_{j}\rvert^{r})\mathbb{E}(\rvert\beta_{j,j}^{(2)}\rvert^{r}) \leq C\mathrm{Leb}(\cup_{i=0}^{j}\mathcal{P}_{A}^{n}(i,j))\leq C a(0)\Delta_{n},
  \]
which shows the claimed estimates for $M=1,2$.   Before proceeding, let us point out that for every $j\geq l_n+1 $
	fix, $\xi_{j,n}^{(3)}$ and $(\xi_{j,n}^{(4)},\xi_{j,n}^{(5)})$
	are sums of martingale difference with respect to $\mathscr{G}_{k}^{n}\lor\sigma(\beta_{k,j}^{(1)})_{i=l_n,\ldots,j-1}$
	and $\mathscr{F}_{k+1}^{n}$, respectively. Thus, by Rosenthal's inequality
	(see e.g. Theorem 2.12 in \citet{HallHeyde80})
\begin{equation}
		\begin{aligned}\mathbb{E}(\mid\xi_{j,n}^{(3)}\mid^{r}) & \leq C\left\{ \mathbb{E}\left[\left(\sum_{k=l_{n}}^{j-1}(\beta_{k,j}^{(1)})^{2}\mathbb{E}(\chi_{k}^{2})\right)^{r/2}\right]+\sum_{k=l_{n}}^{j-1}\mathbb{E}(\mid\chi_{k}\beta_{k,j}^{(1)}\mid^{r})\right\} \\
			& \leq C\mathbb{E}\left[\left(\Delta_{n}\sum_{k=l_{n}}^{j-1}(\beta_{k,j}^{(1)})^{2}\right)^{r/2}\right]+C\mathrm{Leb}(A)\Delta_{n},
		\end{aligned}
		\label{mainineq3}
	\end{equation}
	where in the last inequality we used that 
	\begin{equation}
 \label{eq:sumexpbeta}
		\sum_{i=0}^{j-1-l_n}\mathrm{Leb}(\mathcal{P}_{A}^{n}(i,j))(j-i-l_n)=\sum_{k=l_n}^{j-1}\mathrm{Leb}(\cup_{i=0}^{k-l_n}\mathcal{P}_{A}^{n}(i,j))\leq\mathrm{Leb}(A).
	\end{equation}
	Now, we write
	\[
	\sum_{k=l_n}^{j-1}(\beta_{k,j}^{(1)})^{2}=\sum_{i=0}^{j-1-l_n}(\alpha_{i,j}^{2}+2\eta_{i,j})(j-i-l_n),
	\]
	where $\eta_{i,j}:=\alpha_{i,j}\sum_{i^{\prime}=0}^{i-1}\alpha_{i^{\prime},j}\mathbf{1}_{i\geq1}$.
	Therefore, by Rosenthal's inequality for positive and independent r.v.'s (9.7.b in \citet{LinBai10}) and \eqref{eq:sumexpbeta} we get
	\begin{align*}
		\mathbb{E}\left[\left(\Delta_{n}\sum_{i=0}^{j-1-l_n}\alpha_{i,j}^{2}(j-i-l_n)\right)^{r/2}\right]\leq C & \sum_{i=0}^{j-1-l_n}\mathbb{E}(\mid\alpha_{i,j}\mid^{r})(t_{j}-t_{i}-l_n\Delta_n)^{r/2}\\
		& +C\Delta_{n}^{r/2}\left(\sum_{i=0}^{j-1-l_n}\mathbb{E}(\alpha_{i,j}^{2})(j-i-l_n)\right)^{r/2}\\
		\leq & C\sum_{i=0}^{j-1-l_n}\mathbb{E}(\mid\alpha_{i,j}\mid^{r})(t_{j}-t_{i}-l_n\Delta_n)^{r/2}+C\Delta_{n}^{r/2}.
	\end{align*}
	Furthermore, by the Mean-Value Theorem
	\begin{equation}
		\begin{aligned}\sum_{i=0}^{j-1-l_n}\mathbb{E}(\mid\alpha_{i,j}\mid^{r})(t_{j}-t_{i}-l_n\Delta_n)^{r/2} &\leq C\sum_{i=0}^{j-1-l_n}\mathrm{Leb}(\mathcal{P}_{A}^{n}(i,j))(t_{j}-t_{i}-l_n\Delta_n)^{r/2}\\
			& \leq C(\Delta_{n}\int_{l_n\Delta_n}^{t_{j}}a(s)s^{r/2-1}\mathrm{d}s+\Delta_{n}^{r/2+1})\\
			& +C\mathrm{Leb}(\mathcal{P}_{A}^{n}(0,j))(t_{j}-l_n\Delta_n)^{r/2}\\
			& -C\int_{t_{j}}^{t_{j+1}}a(s)(t_{j-1}-l_n\Delta_n)^{r/2}\mathrm{d}s\\
			& \leq C(\Delta_{n}\int_{0}^{n\Delta_{n}}a(s)s^{r/2-1}\mathrm{d}s+\Delta_{n}^{r/2+1}).
		\end{aligned}
		\label{boundq31}
	\end{equation}
	Now, if $r/2\leq2$, we can apply the inequality $\mathbb{E}(\rvert Y\rvert^{r/2})\leq\mathbb{E}(\rvert Y\rvert^{2})^{r/4}$ in combination with  \eqref{eq:sumexpbeta} to deduce that
	\begin{align*}		(E_{j,n})^{4/r}:=\mathbb{E}\left[\left|\sum_{i=0}^{j-1-l_n}\eta_{i,j}(j-i-l_n)\right|^{r/2}\right]^{4/r} & \leq\sum_{i=1}^{j-1-l_{n}}\mathbb{E}(\alpha_{i,j}^{2})\mathbb{E}[L(\cup_{i^{\prime}=0}^{i-1}\mathcal{P}_{A}^{n}(i,j))^{2}](j-i-l_n)^{2}\\
		& \leq\left(\sum_{k=l_n}^{j-1}\mathrm{Leb}(\cup_{i=0}^{k-l_n}\mathcal{P}_{A}^{n}(i,j))\right)^{2}\leq\mathrm{Leb}(A)^{2}.
	\end{align*}
	Applying this estimate along with (\ref{mainineq3}) to (\ref{boundq31}) 
	shows that \eqref{eq:momentestchi} holds for $2\leq r\leq 4$. If
	$r/2>2$ we can use once again Rosenthal's inequality (now the $\eta_{i,j}$'s
	are $\mathscr{F}_{i}^{n}$-martingale difference) to deduce that  $E_{j,n}$ is bounded up to a constant by 
	\begin{align*}
		E^{(1)}_{j,n}+ E^{(2)}_{j,n}:=& \sum_{i=0}^{j-1-l_n}\mathbb{E}(\rvert\alpha_{i,j}\rvert^{r/2})\mathbb{E}[\rvert L(\cup_{i^{\prime}=0}^{i-1}\mathcal{P}_{A}^{n}(i,j))\rvert^{r/2}](j-i-l_n)^{r/2}\\
		& +\mathbb{E}\left[\left(\sum_{i=0}^{j-1-l_n}L(\cup_{i^{\prime}=0}^{i-1}\mathcal{P}_{A}^{n}(i,j))^{2}\mathrm{Leb}(\mathcal{P}_{A}^{n}(i,j))(j-i-l_n)^{2}\right)^{r/4}\right].
	\end{align*}
	By applying (\ref{eq:sumexpbeta}) and reasoning as in (\ref{boundq31})
	we obtain that 
	\begin{align*}
		\Delta_n^{r/2}E^{(1)}_{j,n} &\leq C \Delta_{n}\sum_{i=0}^{j-1-l_n}\mathrm{Leb}(\mathcal{P}_{A}^{n}(i,j))(j-i-l_n)\sum_{i^{\prime}=0}^{i-1}\mathrm{Leb}(\mathcal{P}_{A}^{n}(i^{\prime},j))(t_{j}-t_{i^{\prime}}-l_n\Delta_n)^{r/2-1}\\
		& \leq C\mathrm{Leb}(A)\Delta_{n}\sum_{i^{\prime}=0}^{j-1-l_n}\mathrm{Leb}(\mathcal{P}_{A}^{n}(i^{\prime},j))(t_{j}-t_{i^{\prime}}-l_n\Delta_n)^{r/2-1}\\
		& \leq C\Delta_{n}\left(\Delta_{n}\int_{0}^{n\Delta_{n}}a(s)s^{r/2-2}\mathrm{d}s+\Delta_{n}^{r/2}\right).
	\end{align*}
	From Jensen's inequality and  (\ref{eq:sumexpbeta}) we further have that $\Delta_n^{r/2}E^{(2)}_{j,n}$ can be  estimated from above by
	\begin{align*}
		 & \mathrm{Leb}(A)^{r/4-1}\Delta_{n}^{r/4}\sum_{i=0}^{j-1-l_n}\mathrm{Leb}(\mathcal{P}_{A}^{n}(i,j))(j-i-l_n)\sum_{i^{\prime}=0}^{i-1}\mathrm{Leb}(\mathcal{P}_{A}^{n}(i^{\prime},j))(t_{j}-t_{i^{\prime}}-l_n\Delta_n)^{r/4}\\
		& \leq C\Delta_{n}^{r/4}\left(\Delta_{n}\int_{0}^{n\Delta_{n}}a(s)s^{r/4-1}\mathrm{d}s+\Delta_{n}^{r/4+1}\right),
	\end{align*}
	where in the last inequality we applied the same arguments used in  (\ref{boundq31}).  Therefore, since $a$ is continuous, we get that once again (\ref{eq:momentestchi})
	holds. By replacing $\chi_{k}$ by $\sum_{m=k+1}^{j}\alpha_{k+1,m}$
	in (\ref{mainineq3}), we get the desired moment estimate for $\xi_{j,n}^{(5)}$.
	Finally, arguing as in (\ref{mainineq3}) and using that  $\mathbb{E}(\rvert\beta_{k,j}^{(2)}\rvert^{r})\leq C\mathrm{Leb}(A)$, we deduce that 
	\[
	\begin{aligned}\mathbb{E}(\mid\xi_{j,n}^{(4)}\mid^{r}) & \leq C\mathbb{E}\left[\left(\sum_{k=l_n}^{j-1}(\beta_{k,j}^{(2)})^{2}\mathbb{E}(\alpha_{k+1,j}^{2})\right)^{r/2}\right]+C\sum_{k=l_n}^{j-1}\mathbb{E}(\mid\alpha_{k+1,j}\mid^{r})\mathbb{E}(\mid\beta_{k,j}^{(2)}\mid^{r})\\
		& \leq C\mathbb{E}\left[\left(\sum_{k=l_n}^{j-1}(\beta_{k,j}^{(2)})^{2}\mathrm{Leb}(\mathcal{P}_{A}^{n}(k+1,j))\right)^{r/2}\right]+C\sum_{k=l_n}^{j-1}\mathrm{Leb}(\mathcal{P}_{A}^{n}(k+1,j))
	\end{aligned}
	\]
	By Jensen's inequality the first quantity in the previous inequality is bounded up to a constant by  
\[\left(\sum_{k=l_n}^{j-1}\mathrm{Leb}(\mathcal{P}_{A}^{n}(k+1,j))\right)^{\frac{r}{2}-1}\sum_{k=l_n}^{j-1}\mathrm{Leb}(\mathcal{P}_{A}^{n}(k+1,j))\mathbb{E}(\mid\beta_{k,j}^{(2)}\mid^{r}),\]
	which is in turn bounded by 
\[C\mathrm{Leb}(A)\left(\sum_{k=l_n}^{j-1}\mathrm{Leb}(\mathcal{P}_{A}^{n}(k+1,j))\right)^{\frac{r}{2}}.\]
 Since $\sum_{k=l_{n}}^{j-1} \mathrm{Leb}(\mathcal{P}_{A}^{n}(k+1,j))\leq a(0)\Delta_{n}$, the inequality in (\ref{eq:momentestchi})  for $\xi_{j,n}^{(4)}$  now follows easily from the previous bounds.
 \end{proof}

\subsection{Proof of the  infeasible central limit theorem for the Gaussian case when $t=0$}\label{sec_proof_t0}
Now we present the  proof of Theorem \ref{thmCLTt0Gaussian}, which follows the same lines
as the proof of Theorem \ref{thmCLTtrawlapprox}. We first observe
that it is enough to show that $Q_{n}\overset{\mathbb{P}}{\rightarrow}c_{4}(L')a(0)$
and that in the Gaussian case, we further have that 
\begin{equation}
	\sqrt{n}\left(\hat{a}(0)-a(0)\right)\overset{d}{\rightarrow}N(0,2a(0)^{2})-\frac{1}{2}\sqrt{\mu_{0}}\phi(0);\,\,\,\text{and }\frac{1}{\Delta_{n}}Q_{n}\overset{\mathbb{P}}{\rightarrow}6a(0)^{2}.\label{eq:whattodothmcltt0}
\end{equation}
\begin{proof}[Proof of Theorem \ref{thmCLTt0Gaussian}] Let us first check that in general, $Q_{n}\overset{\mathbb{P}}{\rightarrow}c_{4}(L')a(0)$.
	Indeed, using (\ref{incremrel}), Lemma \ref{lemmamoments} and
	putting $\varrho_{k}:=L(\cup_{j\geq k+1}\mathcal{P}_{A}^{n}(k+1,j))$,
	we get that 
	\begin{equation}
		\begin{aligned}Q_{n}-\frac{\int_{0}^{\Delta_{n}}a(s)ds}{\Delta_{n}}\left(c_{4}(L')+3\int_{0}^{\Delta_{n}}a(s)ds\right)= & \frac{1}{2n\Delta_{n}}\sum_{k=0}^{n-2}[\chi_{k}^{4}-\mathbb{E}(\chi_{k}^{4})]\\+&\frac{1}{2n\Delta_{n}}\sum_{k=0}^{n-2}[\varrho_{k}^{4}-\mathbb{E}(\varrho_{k}^{4})]\\
			& +\frac{2}{n\Delta_{n}}\sum_{k=0}^{n-2}\chi_{k}^{3}\varrho_{k}+\frac{2}{n\Delta_{n}}\sum_{k=0}^{n-2}\chi_{k}[\varrho_{k}^{3}-\mathbb{E}(\varrho_{k}^{3})]\\
			& +\frac{2\int x^{3}\nu(dx)\int_{0}^{\Delta_{n}}a(s)}{n\Delta_{n}}\sum_{k=0}^{n-2}\chi_{k}\\
			& +\frac{3}{n\Delta_{n}}\sum_{k=0}^{n-2}\chi_{k}^{2}\varrho_{k}^{2}\\
			=: & S_{1}+S_{2}+S_{3}+S_{4}+S_{5}+S_{6},
		\end{aligned}
		\label{eq:quariticydecomp}
	\end{equation}
	plus an error of order $\mathrm{O}(1/n)$. It is clear, due to Lemma
	\ref{lemmamoments}, that $S_{6}\overset{\mathbb{P}}{\rightarrow}0$.
	Furthermore, for every $q=1,\ldots,5$, the summands in $S_{q}$ are
	martingale differences. An application of the von Bahr-Esseen inequality
	(see \citep[9.3.b]{LinBai10}) and (\ref{momentineq}) yield to
	the asymptotic negligibility of $S_{q}$. 
	
	Now assume that $L$ is Gaussian. We start by showing that in this
	case $\frac{1}{\Delta_{n}}Q_{n}\overset{\mathbb{P}}{\rightarrow}6a(0)^{2}$.
	Relation (\ref{eq:quariticydecomp}) still holds with $c_{4}(L')=0$.
	Thus, in view that $(\int_{0}^{\Delta_{n}}a(s)ds)^{-1/2}\chi_{k}$
	are i.i.d. centered Gaussian r.v.s., it follows trivially from the
	law of large numbers that $\frac{1}{\Delta_{n}}S_{1}\overset{\mathbb{P}}{\rightarrow}0.$
	Using the exact argument as well as the fact that the summands in
	$S_{q}$ are martingale differences allow us to conclude that $\frac{1}{\Delta_{n}}S_{q}\overset{\mathbb{P}}{\rightarrow}0$
	for $q=1,\ldots,5,$ also holds in this situation. We finally show
	that $\frac{1}{\Delta_{n}}S_{6}\overset{\mathbb{P}}{\rightarrow}3a(0)^{2}.$
	To do this, we note that by the Gaussianity of $L$ we have that
	\begin{equation}
		\frac{1}{\Delta_{n}}S_{6}=3a(0)^{2}+\mathrm{o}_{\mathbb{P}}(1)+\frac{3}{n\Delta_{n}^{2}}\sum_{k=0}^{n-2}\chi_{k}^{2}[\varrho_{k}^{2}-\mathbb{E}(\varrho_{k}^{2})].\label{eq:S6dec}
	\end{equation}
	In view that the array $\chi_{k}^{2}[\varrho_{k}^{2}-\mathbb{E}(\varrho_{k}^{2})]$
	is martingale difference and $\mathbb{E}\left[(\chi_{k}^{2}[\varrho_{k}^{2}-\mathbb{E}(\varrho_{k}^{2})])^{2}\right]=\mathrm{O}(\Delta_{n}^{4})$,
	uniformly on $k$, the last summand in (\ref{eq:S6dec}) is also $\mathrm{o}_{\mathbb{P}}(1)$,
	which gives the desired convergence. Therefore, it is left to show that 
	\[
	\sqrt{n}\left(\hat{a}(0)-a(0)\right)\overset{d}{\rightarrow}N(0,2a(0)^{2})-\frac{1}{2}\sqrt{\mu_{0}}\phi(0).
	\]
	To verify this, as in the proof of Theorem \ref{thmCLTtrawlapprox}, 	we decompose
	\begin{equation}
		\sqrt{n}\left(\hat{a}(0)-a(0)\right)=\mathcal{\tilde{E}}_{0}-\frac{1}{2}\sqrt{\mu_{0}}\phi(0)+\mathrm{o}(1),\label{statisticalerror-1}
	\end{equation}
	where$\mathcal{\tilde{E}}_{0}=\frac{1}{2\Delta_{n}\sqrt{n}}\sum_{j=1}^{n-1}\tilde{\zeta}_{j,n}$, 
	with $\tilde{\zeta}_{j,n}=\sum_{q=1}^{4}\acute{\xi}_{j,n}^{(q)}$ and
	\[
	\begin{aligned}\acute{\xi}_{j,n}^{(1)} & =\chi_{j}^{2}-\mathbb{E}(\chi_{j}^{2});\,\,\,\acute{\xi}_{j,n}^{(2)}=\sum_{k=0}^{j-1}[\alpha_{k+1,j}^{2}-\mathbb{E}(\alpha_{k+1,j}^{2})];\\
		\acute{\xi}_{j,n}^{(3)} & =2\sum_{k=0}^{j-1}\sum_{j^{\prime}=k+1}^{j-1}\alpha_{k+1,j}\alpha_{k+1,j^{\prime}}\mathbf{1}_{j\geq k+2};\,\,\,\acute{\xi}_{j,n}^{(4)}=-2\sum_{k=0}^{j-1}\chi_{k}\alpha_{k+1,j}.
	\end{aligned}
	\]
	Furthermore, as $n\rightarrow\infty$, $\frac{1}{\Delta_{n}^{8}n^{4}}\sum_{j=l_{n}+1}^{n-1}\mathbb{E}(\mid\acute{\zeta}_{j,n}\mid^{8})\rightarrow0$
	and
	\[
	\frac{1}{4\Delta_{n}^{2}n}\sum_{j=l_{n}+1}^{n-1}\mathbb{E}(\acute{\zeta}_{j,n}^{2}\mid\mathscr{G}_{j-1}^{n})=\frac{1}{4\Delta_{n}^{2}n}\sum_{j=1}^{n-1}\mathbb{E}(\acute{\zeta}_{j,n}^{2})+\mathrm{o}_{\mathbb{P}}(1).
	\]
	Since $L$ is Gaussian and centred, we have by Lemma \ref{lemmamoments}
	\begin{align*}
		\mathbb{E}[(\acute{\xi}_{j,n}^{(1)})^{2}] & =2(\Delta_{n}a_{n}(0))^{2};\,\,\mathbb{E}[(\acute{\xi}_{j,n}^{(2)})^{2}]=2\sum_{k=0}^{j-1}\Leb(\mathcal{P}_{A}^{n}(k+1,j))^{2};\\
		\mathbb{E}[(\acute{\xi}_{j,n}^{(3)})^{2}] & =4\sum_{k=0}^{j-1}\Leb(\mathcal{P}_{A}^{n}(k+1,j))\sum_{j^{\prime}=k+1}^{j-1}\Leb(\mathcal{P}_{A}^{n}(k+1,j^{\prime}))\mathbf{1}_{j\geq k+2};\\
		\mathbb{E}[(\acute{\xi}_{j,n}^{(4)})^{2}] & =4\Delta_{n}a_{n}(0)\sum_{k=0}^{j-1}\Leb(\mathcal{P}_{A}^{n}(k+1,j));\\
		\mathbb{E}[\acute{\xi}_{j,n}^{(2)}\acute{\xi}_{j,n}^{(3)}] & =\mathbb{E}[\acute{\xi}_{j,n}^{(1)}\acute{\xi}_{j,n}^{(3)}]=\mathbb{E}[\acute{\xi}_{j,n}^{(1)}\acute{\xi}_{j,n}^{(4)}]=\mathbb{E}[\acute{\xi}_{j,n}^{(2)}\acute{\xi}_{j,n}^{(4)}]=\mathbb{E}[\acute{\xi}_{j,n}^{(3)}\acute{\xi}_{j,n}^{(4)}]=0,
	\end{align*}
	as well as, uniformly on $j$, $\mathbb{E}[\acute{\xi}_{j,n}^{(1)}\acute{\xi}_{j,n}^{(2)}]=2\sum_{k=0}^{j-1}\Leb(\mathcal{P}_{A}^{n}(k+1,j))^{2}=\mathrm{O}(\Delta_{n}^{3})$,
	due to Assumption \ref{as:trawl}. The first part in (\ref{eq:whattodothmcltt0})
	now follows by summing over $j$ and rearranging terms.\end{proof}

\subsection{Proof of the consistency of the ABI, AVAR and the trawl set and slice estimators}
We now tackle the proofs of Propositions \ref{ABIconv}, \ref{AVARconv} and \ref{prop:consistencyforecast}.
During this part, we keep the notation introduced in the previous
sections. 

\begin{proof}[Proof of Proposition \ref{ABIconv}]
	
	By means of (\ref{incremrel}), we decompose 
	\begin{equation}
		\begin{aligned}\hat{a}^{\prime}(t)-\mathbb{E}(\hat{a}^{\prime}(t))= & -\frac{1}{n\Delta_{n}^{2}}\sum_{k=l_{n}+1}^{n-2}\left[\chi_{k}\lambda_{k}-\mathbb{E}(\chi_{k}\lambda_{k})\right]\\
			& -\frac{1}{n\Delta_{n}^{2}}\sum_{j=l_{n}+2}^{n-1}\sum_{k=l_{n}+1}^{j-1}\alpha_{k-l_{n},j}\chi_{k}-\frac{1}{n\Delta_{n}^{2}}\sum_{k=l_{n}+1}^{n-2}\gamma_{k-l_{n}-1}\chi_{k}\\
			& +\frac{1}{n\Delta_{n}^{2}}\sum_{k=l_{n}+1}^{n-2}\delta_{k-l_{n}-1}X\cdot\varrho_{k}+\frac{1}{n\Delta_{n}^{2}}\sum_{k=l_{n}+1}^{n-2}\chi_{k-l_{n}-1}\chi_{k},
		\end{aligned}
		\label{eq:derdecomp}
	\end{equation}
	where $\lambda_{k}:=L(\cup_{j=k-l_{n}}^{k}\mathcal{P}_{A}^{n}(k-l_{n},j))$.
	From Assumption \ref{as:trawl}, we further have that
	\begin{align*}
		\mathbb{E}(\hat{a}^{\prime}(t)) & =\frac{1}{\Delta_{n}}\int_{0}^{1}[a(s\Delta_{n}+(l_{n}+1)\Delta_{n})-a(s\Delta_{n}+l_{n}\Delta_{n})]ds+\mathrm{O}(1/(n\Delta_{n}))\\
		& =a^{\prime}(t)+\mathrm{o}(1).
	\end{align*}
	Thus, in order to complete the proof, we only need to check that the sums in
	(\ref{eq:derdecomp}) are $\mathrm{O}_{\mathbb{P}}(1/\sqrt{n\Delta_{n}^{2}})$.
	This follows easily from the fact that each summand is a squared integrable martingale
	difference whose second moment (due to Lemma \ref{lemmamoments} and Assumption \ref{as:trawl}) is bounded uniformly up to a constant
	by $\Delta_{n}^{2}$.\end{proof}

\begin{proof}[Proof of Proposition \ref{AVARconv}] We have already verified in the previous subsection
	that \linebreak $Q_{n}\overset{\mathbb{P}}{\rightarrow}c_{4}(L')a(0)$. Furthermore,
	since the arguments used to analyse $\hat{v}_{2}(t)$ can be extended
	to $v_{3}(t)$ and $v_{4}(t)$ without any difficulty, for the rest
	of the proof, we focus on checking that 
	\begin{equation}
		\hat{v}_{2}(t)\overset{\mathbb{P}}{\rightarrow}2\int_{0}^{\infty}a(s)^{2}ds.\label{eq:v2conv}
	\end{equation}
Let us start by observing that for $l=1,\ldots,n-2$
\begin{equation}
	\mathbb{E}[(\tilde{a}(l\Delta_{n})-\mathbb{E}(\tilde{a}(l\Delta_{n})))^{2}]\leq C/(n\Delta_{n}),\label{eq:varianceboundatilde}
\end{equation}
where $C>0$ is a constant independent of $l$ and $n$. Indeed, recall that in view of \eqref{decXbetas} we may write
\[\tilde{a}(l\Delta_{n})-\mathbb{E}(\tilde{a}(l\Delta_{n}))=\frac{1}{n\Delta_{n}}\sum_{j=l}^{n-2}(\xi_{j,n}^{(1)}+\xi_{j,n}^{(2)}+\xi_{j,n}^{(3)})+\frac{r_{n,1}}{\sqrt{n\Delta_{n}}}-\frac{1}{n\Delta_{n}}\sum_{k=l}^{n-2}\tilde{X}_{(k-l)\Delta_{n}}\varrho_{k},
\]
where $r_{n,1}$ is defined as in \eqref{residuals_def}. In the proof of Theorem  \ref{propconsistency} we have already seen that the second moment of the last term in the previous decomposition is bounded by $\mathbb{E}[X_{0}^{2}]/(n\Delta_n)$. Furthermore, the first part of \eqref{negligibilityerrors} shows that $\mathbb{E}(r_{n,1}^2)\leq a(0)\mathrm{Leb}(A)$. Relation \eqref{eq:varianceboundatilde} now follows from the previous estimates as well as \eqref{variance}.

	An application of (\ref{eq:decouplinghata}), (\ref{expatilde}),
	and (\ref{eq:varianceboundatilde}) result in
	\[
	\Delta_{n}\sum_{l=0}^{N_{n}}[\hat{a}^{2}(l\Delta_{n})-a^{2}(l\Delta_{n})]=\Delta_{n}\sum_{l=0}^{N_{n}}[\hat{a}(l\Delta_{n})-a(l\Delta_{n})]^{2}+\mathrm{o}_{\mathbb{P}}(1),
	\]
	where we also used that (thanks to Assumptions \ref{as:trawl}, \ref{as:Nnprop}, and
	the Generalized DCT)
	\[
	\Delta_{n}\sum_{l=1}^{N_{n}}\left((1-l/n)\frac{1}{\Delta_{n}}\int_{l\Delta_{n}}^{(l+1)\Delta_{n}}a(s)ds \right) ^{p}=\int_{0}^{\infty}a(s)^{p}ds+\mathrm{o}(1),\,\,\,p\geq1.
	\]
	Since $n\Delta_{n}^{3}$ is bounded, we may utilize (\ref{errordecomp1})
	and (\ref{decompdiscretazationerror}) once again to conclude that
	\[
	\Delta_{n}\sum_{l=0}^{N_{n}}[\hat{a}(l\Delta_{n})-a(l\Delta_{n})]^{2}\leq C\Delta_{n}\sum_{l=0}^{N_{n}}[\tilde{a}(l\Delta_{n})-\mathbb{E}(\tilde{a}(l\Delta_{n}))]^{2}+\mathrm{o}_{\mathbb{P}}(1)=\mathrm{o}_{\mathbb{P}}(1),
	\]
	where the last equation is obtained by means of (\ref{eq:varianceboundatilde})
	and Assumption \ref{as:Nnprop}. This concludes the proof of (\ref{eq:v2conv}).\end{proof}

\begin{proof}[Proof of Proposition \ref{prop:consistencyforecast}]
The proof of the consistency of the three estimators for the trawl (sub) sets follows similar arguments as in the proof of Proposition 
	\ref{AVARconv}. To see this, note that we work with the approximations
	$\Leb(A)=\int_0^{\infty}a(s)ds\approx\int_0^{n\Delta_n}a(s)ds\approx 
	\sum_{l=0}^{n-1}\hat a(l\Delta_n)\Delta_n=:\widehat{\Leb(A)}$,%
	$\Leb(A\cap A_h)= \int_h^{\infty}a(s)ds\approx\int_h^{n\Delta_n}a(s)ds\approx 
	\sum_{\substack{l\in \{0, \ldots, n-1\}:\\ l\Delta_n \geq h}} \hat a(l\Delta_n)\Delta_n=:\widehat{	\Leb(A\cap A_h)}$,
	$\Leb(A\setminus A_h)= \int_0^{h}a(s)ds\approx 
	\sum_{\substack{l\in \{0, \ldots, n-1\}:\\ l\Delta_n < h}} \hat a(l\Delta_n)\Delta_n=:\widehat{\Leb(A\setminus A_h)}$.
	Moreover, since we only consider linear functionals of the trawl function in the integrands, we do not need to use the additional tuning parameter $N_n$.\end{proof}

\subsection{Approximating the asymptotic variance}

In this part, we verify that equation \eqref{approxAVAR} holds, i.e.
\begin{lemma}\label{lemmaavar}Let Assumptions \ref{as:samplingscheme},
	\ref{as:trawl} and \ref{as:seed} in  the main article hold. 
	Then, we have that 
	\[
	\frac{1}{\Delta_{n}n}\sum_{j=l_{n}+1}^{n-1}\mathbb{E}(\zeta_{j,n}^{2}\mid\mathscr{G}_{j-1}^{n})=\frac{1}{\Delta_{n}n}\sum_{j=l_{n}+1}^{n-1}\mathbb{E}(\zeta_{j,n}^{2})+\mathrm{o}_{\mathbb{P}}(1),
	\]
	where $\zeta_{j,n}=\sum_{p=1}^{5}\xi_{j,n}^{(p)}$ and $\xi_{j,n}^{(p)}$ are defined as in (\ref{errorstpos}). 
	
\end{lemma}

The proof of the previous lemma consists of expressing the approximation error as sums of martingale differences. Thus, Lemma \ref{lemmapproxgammahat-1} will be extensively used.

\begin{proof}[Proof of Lemma \ref{lemmaavar}] Let us start by choosing $n_{0}$ such that $l_{n}=\lfloor t/\Delta_{n}\rfloor>0$, for all $n\geq n_{0}$. As we mentioned above, most of the proof consists of writing the error terms as the sum of martingale differences and then applying Lemma \ref{lemmapproxgammahat-1}.
	Denote by $\delta_{k,j}=L(\cup_{m=k+1}^{j-1}\mathcal{P}_{A}^{n}(k+1,m))$
	and note that in this situation $
\mathbb{E}(\zeta_{j,n}^{2}\mid\mathscr{G}_{j-1}^{n})=\sum_{p,q=1}^{5}\mathbb{E}(\xi_{j,n}^{(p)}\xi_{j,n}^{(q)}\mid\mathscr{G}_{j-1}^{n})$, with a total of $15$ different summands, each of which equals (thanks
	to Lemma \ref{lemmamoments}) to 
	\begin{align}
		\begin{aligned}\mathbb{E}((\xi_{j,n}^{(1)})^{2}\mid\mathscr{G}_{j-1}^{n})= & \mathbb{E}((\xi_{j,n}^{(1)})^{2});\\
			\mathbb{E}((\xi_{j,n}^{(2)})^{2}\mid\mathscr{G}_{j-1}^{n})= & \mathbb{E}((\xi_{j,n}^{(2)})^{2})+\mathrm{O}(\Delta_{n})\left((\beta_{j,j}^{(2)})^{2}-\mathbb{E}[(\beta_{k,j}^{(2)})^{2}]\right);\\
			\mathbb{E}((\xi_{j,n}^{(3)})^{2}\mid\mathscr{G}_{j-1}^{n})= & \sum_{k,k^{\prime}=l_{n}}^{j-1}\chi_{k}\chi_{k^{\prime}}\Leb(\cup_{i=0}^{k\land k^{\prime}-l_{n}}\mathcal{P}_{A}^{n}(i,j));\\
			\mathbb{E}((\xi_{j,n}^{(4)})^{2}\mid\mathscr{G}_{j-1}^{n})= & \mathbb{E}[(\xi_{j,n}^{(4)})^{2}]+\sum_{k=l_{n}}^{j-1}\left((\beta_{k,j}^{(2)})^{2}-\mathbb{E}[(\beta_{k,j}^{(2)})^{2}]\right)\Leb(\mathcal{P}_{A}^{n}(k+1,j));\\
			\mathbb{E}((\xi_{j,n}^{(5)})^{2}\mid\mathscr{G}_{j-1}^{n})= & \sum_{k=l_{n}}^{j-1}\mathbb{E}[(\beta_{k,j}^{(1)})^{2}]\mathbb{E}[\alpha_{k+1,j}^{2}]+\sum_{k,k^{\prime}=l_{n}}^{j-2}\mathbb{E}[(\beta_{k\land k^{\prime},j}^{(1)})^{2}]\delta_{k,j}\delta_{k^{\prime},j},
		\end{aligned}
		\label{eq:condvariances}
	\end{align} 
	and 
	\begin{equation}
		\begin{aligned}\mathbb{E}(\xi_{j,n}^{(1)}\xi_{j,n}^{(2)}\mid\mathscr{G}_{j-1}^{n})=\beta_{j,j}^{(2)}\mathbb{E}[(\beta_{j,j}^{(1)})^{3}]; & \,\,\,\mathbb{E}(\xi_{j,n}^{(1)}\xi_{j,n}^{(3)}\mid\mathscr{G}_{j-1}^{n})=\sum_{k=l_{n}}^{j-1}\chi_{k}\mathbb{E}[(\beta_{k,j}^{(1)})^{3}];\\
			\mathbb{E}(\xi_{j,n}^{(1)}\xi_{j,n}^{(4)}\mid\mathscr{G}_{j-1}^{n})= & -\sum_{k=l_{n}}^{j-1}\beta_{k,j}^{(2)}\mathbb{E}[\alpha_{k+1,j}^{3}];\\
			\mathbb{E}(\xi_{j,n}^{(1)}\xi_{j,n}^{(5)}\mid\mathscr{G}_{j-1}^{n})= &-\sum_{k=l_{n}}^{j-2}\delta_{k,j}\mathbb{E}[(\beta_{k,j}^{(1)})^{3}] +\mathrm{O}(\Delta_n);\\
			\mathbb{E}(\xi_{j,n}^{(2)}\xi_{j,n}^{(3)}\mid\mathscr{G}_{j-1}^{n})= & \beta_{j,j}^{(2)}\sum_{k=l_{n}}^{j-1}\chi_{k}\mathbb{E}[(\beta_{k,j}^{(1)})^{2}];\\
			\mathbb{E}(\xi_{j,n}^{(2)}\xi_{j,n}^{(4)}\mid\mathscr{G}_{j-1}^{n})= & -\beta_{j,j}^{(2)}\sum_{k=l_{n}}^{j-1}\beta_{k,j}^{(2)}\mathbb{E}[\alpha_{k+1,j}^{2}];\\
			\mathbb{E}(\xi_{j,n}^{(2)}\xi_{j,n}^{(5)}\mid\mathscr{G}_{j-1}^{n})= & -\beta_{j,j}^{(2)}\sum_{k=l_{n}}^{j-2}\delta_{k,j}\mathbb{E}[(\beta_{k,j}^{(1)})^{2}];\\
			\mathbb{E}(\xi_{j,n}^{(3)}\xi_{j,n}^{(4)}\mid\mathscr{G}_{j-1}^{n})= & -\sum_{k,k^{\prime}=l_{n}}^{j-1}\chi_{k}\beta_{k^{\prime},j}^{(2)}\mathbb{E}[\alpha_{k^{\prime}+1,j}^{2}]\mathbf{1}_{k^{\prime}\leq k-l_{n}-1};\\
			\mathbb{E}(\xi_{j,n}^{(3)}\xi_{j,n}^{(5)}\mid\mathscr{G}_{j-1}^{n})= & -\sum_{k,k^{\prime}=l_{n}}^{j-2}\chi_{k^{\prime}}\delta_{k,j}\mathbb{E}[(\beta_{k\land k^{\prime},j}^{(1)})^{2}];\\
			\mathbb{E}(\xi_{j,n}^{(4)}\xi_{j,n}^{(5)}\mid\mathscr{G}_{j-1}^{n})= & \sum_{k,k^{\prime}=l_{n}}^{j-2}\delta_{k,j}\beta_{k^{\prime},j}^{(2)}\mathbb{E}[\alpha_{k^{\prime}+1,j}^{2}]\mathbf{1}_{k^{\prime}\leq k-l_{n}-1}.
		\end{aligned}
		\label{eq:condcovar}
	\end{equation}
	uniformly on $j$. We will verify that for all $p=1,\ldots,5$ and for all $q=p,\ldots,5$
	we have that 
	\begin{equation}
		\frac{1}{\Delta_{n}n}\sum_{j=l_{n}+1}^{n-1}\mathbb{E}(\xi_{j,n}^{(p)}\xi_{j,n}^{(q)}\mid\mathscr{G}_{j-1}^{n})=\frac{1}{\Delta_{n}n}\sum_{j=l_{n}+1}^{n-1}\mathbb{E}(\xi_{j,n}^{(p)}\xi_{j,n}^{(q)})+\mathrm{o}_{\mathbb{P}}(1).\label{eq:lemmaapproxxi}
	\end{equation}
	We start analysing the case $p=q$, for $p=2,\ldots,5$ (the case
	$p=1$ follows trivially from (\ref{eq:condvariances})).
	
	\subsubsection*{Variances}
    	\begin{description}
		\item [{$p=2$:}] In view that $\beta_{j,j}^{(2)}$ is $l_{n}$ dependent
		and $\mathbb{E}\left[\left((\beta_{j,j}^{(2)})^{2}-\mathbb{E}[(\beta_{k,j}^{(2)})^{2}]\right)^{2}\right]=\mathrm{O}(1)$
		uniformly on $j$ (due to Lemma \ref{lemmamoments}), it follows from
		the Cauchy--Schwarz inequality that 
		\begin{equation}
			\frac{1}{n^{2}}\mathrm{Var}\left[\sum_{j=l_{n}+1}^{n-1}\left((\beta_{j,j}^{(2)})^{2}-\mathbb{E}[(\beta_{k,j}^{(2)})^{2}]\right)\right]=\mathrm{O}(1/n)+\mathrm{O}(l_{n}/n)\rightarrow0,\label{variancexi2}
		\end{equation}
		which is enough.
		\item [{$p=3$:}] Plainly
		\begin{align*}
			\sum_{k,k^{\prime}=l_{n}}^{j-1}\chi_{k}\chi_{k^{\prime}}\Leb(\cup_{i=0}^{k\land k^{\prime}-l_{n}}\mathcal{P}_{A}^{n}(i,j))= & \mathbb{E}((\xi_{j,n}^{(3)})^{2})+\sum_{k=l_{n}}^{j-1}[\chi_{k}^{2}-\mathbb{E}(\chi_{k}^{2})]\Leb(\cup_{i=0}^{k-l_{n}}\mathcal{P}_{A}^{n}(i,j))\\
			& +2\sum_{k=l_{n}+1}^{j-1}\chi_{k}\sum_{k^{\prime}=l_{n}}^{k-1}\chi_{k^{\prime}}\Leb(\cup_{i=0}^{k-l_{n}}\mathcal{P}_{A}^{n}(i,j)).
		\end{align*}
		Consequently,
		\begin{align*}
			\frac{1}{n\Delta_{n}}\sum_{j=l_{n}+1}^{n-1}\mathbb{E}((\xi_{j,n}^{(3)})^{2}\mid\mathscr{G}_{j-1}^{n})= & \frac{1}{n\Delta_{n}}\sum_{j=l_{n}+1}^{n-1}\mathbb{E}((\xi_{j,n}^{(3)})^{2})\\
			& +\frac{1}{n\Delta_{n}}\sum_{k=l_{n}}^{n-2}[\chi_{k}^{2}-\mathbb{E}(\chi_{k}^{2})]\Leb(\cup_{i=0}^{k-l_{n}}\cup_{j=k+1}^{n-1}\mathcal{P}_{A}^{n}(i,j))\\
			& +2\mathrm{Var}(L^{\prime})\frac{1}{n\Delta_{n}}\sum_{k=l_{n}+1}^{n-2}\chi_{k}\nu_{k}^{(1)},
		\end{align*}
		where $\nu_{k}^{(1)}:=\sum_{k^{\prime}=l_{n}}^{k-1}\chi_{k^{\prime}}\Leb(\cup_{i=0}^{k^{\prime}-l_{n}}\cup_{j=k+1}^{n-1}\mathcal{P}_{A}^{n}(i,j))$.
		Since $(\chi_{k})_{k\geq0}$ are i.i.d., $\mathbb{E}[(\chi_{k}^{2}-\mathbb{E}(\chi_{k}^{2}))^{2}]=\mathrm{O}(\Delta_{n})$
		uniformly on $k$, and $\Leb(\cup_{i=0}^{k-l_{n}}\cup_{j=k+1}^{n+l_{n}}\mathcal{P}_{A}^{n}(i,j))\leq \Leb(A)$,
		one trivially deduce that 
		\[
		\frac{1}{n\Delta_{n}}\sum_{k=l_{n}}^{n-2}[\chi_{k}^{2}-\mathbb{E}(\chi_{k}^{2})]\Leb(\cup_{i=0}^{k-l_{n}}\cup_{j=k+1}^{n-1}\mathcal{P}_{A}^{n}(i,j))=\mathrm{o}_{\mathbb{P}}(1).
		\]
		Note now that $(\chi_{k}\nu_{k}^{(1)})_{k\geq1}$ is an $(\mathscr{F}_{k}^{n})$-martingale
		difference satisfying that 
		\[
		\mathbb{E}[(\chi_{k}\nu_{k}^{(1)})^{2}]\leq C\Delta_{n}\int_{0}^{t_{k+1}}(\int_{s}^{\infty}a(r)dr)^{2}ds.
		\]
		Therefore, 
		\begin{align*}
			\frac{1}{(n\Delta_{n})^{2}}\sum_{k=l_{n}+1}^{n-2}\mathbb{E}\left[(\chi_{k}\nu_{k}^{(1)})^{2}\right] & \leq C\frac{1}{(n\Delta_{n})^{2}}\int_{0}^{n\Delta_{n}}\int_{0}^{x}(\int_{s}^{\infty}a(r)dr)^{2}dsdx\\
			& =C\int_{0}^{1}\int_{0}^{y}(\int_{un\Delta_{n}}^{\infty}a(r)dr)^{2}dudy\rightarrow0,
		\end{align*}
		where we have done the change of variables $y=x/(n\Delta_{n})$ and
		$u=s/n\Delta_{n}$ as well as applied the Dominated Convergence Theorem.
		A simple application of Lemma \ref{lemmapproxgammahat-1} gives that
		(\ref{eq:lemmaapproxxi}) is valid in this case.
		\item [{$p=4$:}] As for $p=3,$ by decomposing 
		\begin{align*}
			&\sum_{k=l_{n}}^{j-1}\left((\beta_{k,j}^{(2)})^{2}-\mathbb{E}[(\beta_{k,j}^{(2)})^{2}]\right)\Leb(\mathcal{P}_{A}^{n}(k+1,j)) \\
			& =\sum_{i=0}^{j-l_{n}-1}\sum_{m=i}^{j-1}[\alpha_{i,m}^{2}-\mathbb{E}(\alpha_{i,m}^{2})]\sum_{k=i+l_{n}}^{(j-1)\land(m+l_{n})}\Leb(\mathcal{P}_{A}^{n}(k+1,j))\\
			& +2\sum_{i=0}^{j-l_{n}-1}\sum_{m=i+1}^{j-1}\sum_{m^{\prime}=i}^{m-1}\alpha_{i,m}\alpha_{i,m^{\prime}}\sum_{k=i+l_{n}}^{(j-1)\land(m^{\prime}+l_{n})}\Leb(\mathcal{P}_{A}^{n}(k+1,j))\\
			& +2\sum_{i=1}^{j-l_{n}-1}\sum_{i^{\prime}=0}^{i-1}\sum_{m,m^{\prime}=i}^{j-1}\alpha_{i,m}\alpha_{i^{\prime},m^{\prime}}\sum_{k=i+l_{n}}^{(j-1)\land(m^{\prime}+l_{n})\land(m+l_{n})}\Leb(\mathcal{P}_{A}^{n}(k+1,j))
		\end{align*}
		we obtain that
		\begin{equation}
			\begin{aligned}\frac{1}{n\Delta_{n}}\sum_{j=l_{n}+1}^{n-1}\mathbb{E}((\xi_{j,n}^{(4)})^{2}\mid\mathscr{G}_{j-1}^{n})= & \frac{1}{\Delta_{n}n}\sum_{j=l_{n}+1}^{n-1}\mathbb{E}((\xi_{j,n}^{(4)})^{2})\\
				& +\frac{1}{n\Delta_{n}}\sum_{i=0}^{n-2-l_{n}}\sum_{m=i}^{n-2}[\alpha_{i,m}^{2}-\mathbb{E}(\alpha_{i,m}^{2})]\theta_{i,m,m^{\prime}}^{(1)}\\
				& +\frac{2}{n\Delta_{n}}\sum_{i=0}^{n-2-l_{n}}\sum_{m=i+1}^{n-2}\sum_{m^{\prime}=i}^{m-1}\alpha_{i,m}\alpha_{i,m^{\prime}}\theta_{i,m,m^{\prime}}^{(2)}\\
				& +\frac{2}{n\Delta_{n}}\sum_{i=0}^{n-2-l_{n}}\sum_{i^{\prime}=0}^{i-1}\sum_{m,m^{\prime}=i}^{n-2}\alpha_{i,m}\alpha_{i^{\prime},m^{\prime}}\theta_{i,m,m^{\prime}}^{(3)},
			\end{aligned}
			\label{eq:decompq4}
		\end{equation}
		where 
		\begin{equation}
			\begin{aligned}\theta_{i,m,m^{\prime}}^{(1)} & =\sum_{j=(i+l_{n}+1)\lor(m+1)}^{n-1}\sum_{k=i+l_{n}+1}^{j\land(m+l_{n}+1)}\Leb(\mathcal{P}_{A}^{n}(k,j));\\
				\theta_{i,m,m^{\prime}}^{(2)} & =\sum_{j=(i+l_{n}+1)\lor(m+1)}^{n-1}\sum_{k=i+l_{n}+1}^{j\land(m^{\prime}+l_{n}+1)}\Leb(\mathcal{P}_{A}^{n}(k,j));\\
				\theta_{i,m,m^{\prime}}^{(3)} & =\sum_{j=(i+l_{n}+1)\lor(m+1)\lor(m^{\prime}+1)}^{n-1}\sum_{k=i+l_{n}+1}^{j\land(m^{\prime}+l_{n}+1)\land(m+l_{n}+1)}\Leb(\mathcal{P}_{A}^{n}(k,j)).
			\end{aligned}
			\label{eq:that123def}
		\end{equation}
		Observe that all the error terms in (\ref{eq:decompq4}) are sums
		of $(\mathscr{F}_{i}^{n})$-martingale differences and for $\ell=1,2,3$
		\begin{align*}
			\left|\theta_{i,m,m^{\prime}}^{(\ell)}\right| & \leq\sum_{j=(i+l_{n}+1)\lor(m+1)}^{m+l_{n}}\sum_{k=i+l_{n}+1}^{j}\Leb(\mathcal{P}_{A}^{n}(k,j))+\sum_{j=m+l_{n}+1}^{n}\sum_{k=i+l_{n}+1}^{m+l_{n}+1}\Leb(\mathcal{P}_{A}^{n}(k,j))\\
			& \leq C\Delta_{n}l_{n}+\Leb(A)=\mathrm{O}(1).
		\end{align*}
		An application of the previous bound, Lemma \ref{lemmamoments}, 
		and the independent scattered property of $L$ gives us that 
		\begin{align*}
			\sum_{i=0}^{n-2-l_{n}}\mathbb{E}\left[\left(\sum_{m=i}^{n-2}[\alpha_{i,m}^{2}-\mathbb{E}(\alpha_{i,m}^{2})]\theta_{i,m,m^{\prime}}^{(1)}\right)^{2}\right] & \leq C\sum_{i=0}^{n}\sum_{m=i}^{n}\Leb(\mathcal{P}_{A}^{n}(i,m))\leq C\Delta_{n}n,\\
			\sum_{i=0}^{n-2-l_{n}}\mathbb{E}\left[\left(\sum_{m=i+1}^{n-2}\alpha_{i,m}\sum_{m^{\prime}=i}^{m-1}\alpha_{i,m^{\prime}}\theta_{i,m,m^{\prime}}^{(2)}\right)^{2}\right] & \leq C\sum_{i=0}^{n}\left(\sum_{m=i}^{n}\Leb(\mathcal{P}_{A}^{n}(i,m))\right)^{2}\leq C\Delta_{n}n,
		\end{align*}
		and that 
		\begin{equation}
			\begin{aligned}\sum_{i=0}^{n-2-l_{n}}\mathbb{E}\left(\sum_{i^{\prime}=0}^{i-1}\sum_{m,m^{\prime}=i}^{n-2}\alpha_{i,m}\alpha_{i^{\prime},m^{\prime}}\right)^{2}= & \sum_{i=0}^{n-2-l_{n}}\sum_{i^{\prime}=0}^{i-1}\sum_{m,m^{\prime}=i}^{n-2}\mathbb{E}[\alpha_{i,m}^{2}]\mathbb{E}[\alpha_{i^{\prime},m^{\prime}}](\theta_{i,m,m^{\prime}}^{(3)})^{2}\\
				& \leq C\sum_{i=1}^{n}\sum_{m=i}^{n}\sum_{i^{\prime}=0}^{i}\sum_{m^{\prime}=i}^{n}\Leb(\mathcal{P}_{A}^{n}(i,m))\\
				&\times\Leb(\mathcal{P}_{A}^{n}(i^{\prime},m^{\prime}))\\
				& \leq C\sum_{i=0}^{n}\sum_{m=i}^{n}\Leb(\mathcal{P}_{A}^{n}(i,m))\\
				& \leq C\Delta_{n}n.
			\end{aligned}
			\label{eq:bound2varp4}
		\end{equation}
		(\ref{eq:lemmaapproxxi}) is obtained as a simple application of the
		previous previous estimates along with Lemma  \ref{lemmamoments}.
		\item [{$p=5$:}] Since in this situation
		\begin{align*}
			\frac{1}{n\Delta_{n}}\sum_{j=l_{n}+1}^{n-1}\mathbb{E}((\xi_{j,n}^{(5)})^{2}\mid\mathscr{G}_{j-1}^{n})= & \frac{1}{n\Delta_{n}}\sum_{j=l_{n}+1}^{n-1}\mathbb{E}((\xi_{j,n}^{(5)})^{2})\\
			& +\frac{1}{n\Delta_{n}}\sum_{k=l_{n}+1}^{n-2}\sum_{m=k}^{n-2}\left(\alpha_{k,m}^{2}-\mathbb{E}[\alpha_{k,m}^{2}]\right)\theta_{k,m}^{(4)}\\
			& +\frac{2}{n\Delta_{n}}\sum_{k=l_{n}+1}^{n-3}\sum_{m=k+1}^{n-2}\sum_{m^{\prime}=k}^{m-1}\alpha_{k,m}\alpha_{k,m^{\prime}}\theta_{k,m}^{(4)}\\
			& +\frac{2}{n\Delta_{n}}\sum_{k=l_{n}+2}^{n-2}\sum_{m=k}^{n-2}\sum_{k^{\prime}=l_{n}+1}^{k-1}\sum_{m^{\prime}=k}^{n-2}\alpha_{k,m}\alpha_{k^{\prime},m^{\prime}}\theta_{k,m,m^{\prime}}^{(5)},
		\end{align*}
		with the error term being once again sums of $(\mathscr{F}_{k}^{n})$-martingale
		differences, and where 
		\begin{equation}
			\theta_{k,m}^{(4)}=\Leb(\cup_{i=0}^{k-1-l_{n}}\cup_{j=m+1}^{n-1}\mathcal{P}_{A}^{n}(i,j));\,\,\theta_{k,m,m^{\prime}}^{(5)}=\Leb(\cup_{i=0}^{k-1-l_{n}}\cup_{j=m\lor m^{\prime}+1}^{n-1}\mathcal{P}_{A}^{n}(i,j)),\label{eq:theta45def}
		\end{equation}
		are uniformly bounded by $\Leb(A)$, we can argue as in the case $p=4$ to obtain
		the desired result.
	\end{description}
	
	\subsubsection*{Covariances}
		In this part we verify that (\ref{eq:lemmaapproxxi}) holds when $p\neq q$,
	for $p=1,\ldots,4$. Below we will use the notation $v_{3}=\int x^{3}\nu(dx)$.
	\begin{description}
		\item [{$p=1$:}] The case $q=2$ can be analysed similarly to the case
		$p=q=2$. Furthermore, since
		\begin{align*}
			\frac{1}{\Delta_{n}n}\sum_{j=l_{n}+1}^{n-1}\mathbb{E}(\xi_{j,n}^{(1)}\xi_{j,n}^{(3)}\mid\mathscr{G}_{j-1}^{n})= & v_{3}\frac{1}{\Delta_{n}n}\sum_{k=l_{n}}^{n-2}\chi_{k}\theta_{k-1,k}^{(4)},\\
			\frac{1}{\Delta_{n}n}\sum_{j=l_{n}+1}^{n-1}\mathbb{E}(\xi_{j,n}^{(1)}\xi_{j,n}^{(4)}\mid\mathscr{G}_{j-1}^{n})= & -\frac{v_{3}}{n\Delta_{n}}\sum_{i=0}^{n-1-l_{n}}\sum_{m=i}^{n-2}\alpha_{i,m}\theta_{i,m,m^{\prime}}^{(1)},\\
			\frac{1}{\Delta_{n}n}\sum_{j=l_{n}+1}^{n-1}\mathbb{E}(\xi_{j,n}^{(1)}\xi_{j,n}^{(5)}\mid\mathscr{G}_{j-1}^{n})= & -\frac{v_{3}}{n\Delta_{n}}\sum_{k=l_{n}+1}^{n-2}\sum_{m=k}^{n-2}\alpha_{k,m}\theta_{k,m}^{(4)}+\mathrm{O}(\Delta_n),
		\end{align*}
		in which $\theta_{i,m,m^{\prime}}^{(1)}$ and $\theta_{k,m}^{(4)}$
		are is as in  (\ref{eq:that123def}) and (\ref{eq:theta45def}), respectively,
		we can reproduce the argument used above for $p=q$, $p=3,4,5$, to
		conclude that (\ref{eq:lemmaapproxxi}) is also satisfied in this
		situation. 
		\item [{$p=2$:}] By the Cauchy-Schwartz inequality and Assumption 
		\eqref{as:trawl}
		in the main article, 
		for all $l_{n}+1\leq j<2l_{n}$
		\[
		\mathbb{E}\left[\left|\mathbb{E}(\xi_{j,n}^{(2)}\xi_{j,n}^{(3)}\mid\mathscr{G}_{j-1}^{n})\right|\right]\leq C\mathbb{E}\left[\left(\sum_{k=l_{n}}^{j-1}\chi_{k}\Leb(\cup_{i=0}^{k-l_{n}}\mathcal{P}_{A}^{n}(i,j))\right)^{2}\right]^{1/2}\leq C\Delta_{n},
		\]
		which easily implies that
		\[
		\frac{1}{\Delta_{n}n}\sum_{j=l_{n}+1}^{2l_{n}}\mathbb{E}(\xi_{j,n}^{(2)}\xi_{j,n}^{(3)}\mid\mathscr{G}_{j-1}^{n})\overset{\mathbb{P}}{\rightarrow}0.
		\]
		For $j\geq2l_{n}+1$ we have that 
		\begin{align*}
			\mathbb{E}(\xi_{j,n}^{(2)}\xi_{j,n}^{(3)}\mid\mathscr{G}_{j-1}^{n})= & \beta_{j,j}^{(2)}\vartheta_{j}^{(1)}+\beta_{j,j}^{(2)}\vartheta_{j}^{(2)},
		\end{align*}
		where 
		\[
		\vartheta_{j}^{(1)}:=\sum_{k=l_{n}}^{j-l_{n}-1}\chi_{k}\mathbb{E}[(\beta_{k,j}^{(1)})^{2}];\,\,\,\vartheta_{j}^{(2)}:=\sum_{k=j-l_{n}}^{j-1}\chi_{k}\mathbb{E}[(\beta_{k,j}^{(1)})^{2}]
		\]
		Note that $\beta_{j,j}^{(2)}\vartheta_{j}^{(1)}$ and $\beta_{j,j}^{(2)}\vartheta_{j}^{(2)}-\mathbb{E}(\beta_{j,j}^{(2)}\vartheta_{j}^{(2)})$
		are $l_{n}$-uncorrelated. Furthermore, by Rosenthal's inequality
		and relation \eqref{momentineq}, we get that 
		\[
		\mathbb{E}[\mid\vartheta_{j}^{(v)}\mid{}^{4}]\leq C\Delta_{n}^{4},\,\,\,v=1,2,
		\]
		which in turn implies that 
		\begin{align*}
			\mathbb{E}[(\beta_{j,j}^{(2)}\vartheta_{j}^{(v)})^{2}] & \leq C\Delta_{n}^{2},\,\,\,v=1,2.
		\end{align*}
		Using this bound, we deduce as in (\ref{variancexi2}) that 
		\[
		\frac{1}{n\Delta_{n}}\sum_{j=2l_{n}+1}^{n-1}\beta_{j,j}^{(2)}\vartheta_{j}^{(1)}+\frac{1}{n\Delta_{n}}\sum_{j=2l_{n}+1}^{n-1}\left[\beta_{j,j}^{(2)}\vartheta_{j}^{(2)}-\mathbb{E}(\beta_{j,j}^{(2)}\vartheta_{j}^{(2)})\right]=\mathrm{o}_{\mathbb{P}}(1),
		\]
		which is clearly enough for (\ref{eq:lemmaapproxxi}) for the case
		$q=3$. Furthermore, since for $j\geq2l_{n}+1$
		\begin{align*}
			\mathbb{E}(\xi_{j,n}^{(2)}\xi_{j,n}^{(4)}\mid\mathscr{G}_{j-1}^{n})= & -\beta_{j,j}^{(2)}(\vartheta_{j}^{(3)}+\vartheta_{j}^{(4)}),\\
			\mathbb{E}(\xi_{j,n}^{(2)}\xi_{j,n}^{(5)}\mid\mathscr{G}_{j-1}^{n})= & -\beta_{j,j}^{(2)}(\vartheta_{j}^{(5)}+\vartheta_{j}^{(6)}+\vartheta_{j}^{(7)});
		\end{align*}
		where 
		\begin{align*}
			\vartheta_{j}^{(3)}:=\sum_{i=0}^{j-l_{n}-1}\sum_{m=i}^{j-l_{n}-1}\alpha_{i,m}\theta_{i,m,j}^{(6)}; & \,\,\vartheta_{j}^{(4)}:=\sum_{i=0}^{j-l_{n}-1}\sum_{m=j-l_{n}}^{j-1}\alpha_{i,m}\theta_{i,m,j}^{(6)};\\
			\vartheta_{j}^{(5)}:=\sum_{k=l_{n}+1}^{j-l_{n}}\sum_{m=k}^{j-l_{n}-1}\alpha_{k,m}\theta_{k,j}^{(7)}; & \,\,\vartheta_{j}^{(6)}:=\sum_{k=l_{n}+1}^{j-l_{n}}\sum_{m=j-l_{n}}^{j-1}\alpha_{k,m}\theta_{k,j}^{(7)},
		\end{align*}
		and $\vartheta_{j}^{(7)}=\sum_{k=j-l_{n}+1}^{j-1}\sum_{m=k}^{j-1}\alpha_{k,m}\theta_{k,j}^{(7)}$,
		in which
		\begin{align*}
			\theta_{i,m,j}^{(6)}\lor\theta_{i,m,j}^{(7)} & \leq C\Delta_{n}a(t_{j}-t_{m}-(l_{n}+1)\Delta_{n}),
		\end{align*}
		and $\beta_{j,j}^{(2)}\vartheta_{j}^{(l)}-\mathbb{E}(\beta_{j,j}^{(2)}\vartheta_{j}^{(v)})$
		are $l_{n}$-uncorrelated for $v=3,4,5,6$, we can repeat the preceding
		arguments to conclude that (\ref{eq:lemmaapproxxi}) is also valid
		for these cases.
		\item [{$p=3$:}] By decomposing 
		\begin{equation}
			\beta_{k^{\prime},j}^{(2)}=\beta_{k^{\prime},k}^{(1)}+\beta_{k^{\prime},k}^{(2)}+\sum_{m=k+1}^{j-1}\beta_{k^{\prime},m}^{(1)},\label{eq:decompp34}
		\end{equation}
		for $k^{\prime}\leq k\leq j-2$, we obtain, due to (\ref{eq:condcovar}),
		that 
		\begin{equation}
			\begin{aligned}\sum_{j=l_{n}+1}^{n-1}\mathbb{E}(\xi_{j,n}^{(3)}\xi_{j,n}^{(4)}\mid\mathscr{G}_{j-1}^{n})= & \sum_{j=l_{n}+1}^{n-1}\mathbb{E}(\xi_{j,n}^{(3)}\xi_{j,n}^{(4)})\\
				& +\sum_{m=2l_{n}+2}^{n-1}\sum_{k=2l_{n}+1}^{m-1}\sum_{k^{\prime}=l_{n}}^{k-l_{n}-1}\chi_{k}\beta_{k^{\prime},m}^{(1)}\theta_{m,k^{\prime}}^{(8)}\\
				& +\sum_{k=2l_{n}+1}^{n-2}\sum_{k^{\prime}=l_{n}}^{k-l_{n}-1}\chi_{k}\beta_{k^{\prime},k}^{(2)}\theta_{k,k^{\prime}}^{(8)}\\
				& +\sum_{k=2l_{n}+1}^{n-2}\sum_{k^{\prime}=l_{n}}^{k-l_{n}-1}[\chi_{k}\beta_{k^{\prime},k}^{(1)}-\mathbb{E}(\chi_{k}\beta_{k^{\prime},k}^{(1)})]\theta_{k,k^{\prime}}^{(8)},
			\end{aligned}
			\label{eq:decompp3q4}
		\end{equation}
		where $\theta_{k,k^{\prime}}^{(8)}:=\Leb(\cup_{j=k+1}^{n}\mathcal{P}_{A}^{n}(k^{\prime}+1,j)).$
		Note that once again all the error terms are sums of $(\mathscr{G}_{k})$-increment
		martingales. Due to Lemma \ref{lemmapproxgammahat-1}, we only need to verify
		that the second moment of each of these summands is of order $\mathrm{O}(\Delta_{n})$.
		For the last two sums, this follows easily by using that $\mathbb{E}[(\beta_{k^{\prime},k}^{(2)})^{2}]\leq C,$
		and
		\[
		\mathbb{E}(\chi_{k}^{2}\beta_{k^{\prime},k}^{(1)}\beta_{m,k}^{(1)})\leq C\Delta_{n}a(t_{k}-t_{k^{\prime}}\land t_{m}),\,\,\,\left|\theta_{k,k^{\prime}}^{(8)}\right|\leq C\Delta_{n}a(t_{k}-t_{k^{\prime}}).
		\]
		The required bound for the second sum in (\ref{eq:decompp3q4}) is
		a bit more involved. Let us start by defining  $\vartheta_{m}^{(8)}=\sum_{k=2l_{n}+1}^{m-1}\sum_{k^{\prime}=l_{n}}^{k-l_{n}-1}\chi_{k}\beta_{k^{\prime},m}^{(1)}\theta_{m,k^{\prime}}^{(8)}$.
		In view that $\mathbb{E}[(\beta_{q,m}^{(1)})^{2}]\leq C\Delta_{n}a(t_{m}-t_{q})$,
		the sought-after bound is obtained by the following estimates
		\begin{align*}
			\mathbb{E}[(\vartheta_{m}^{(8)})^{2}]\leq & C\Delta_{n}\sum_{k=2l_{n}+1}^{m-1}\sum_{k^{\prime}=l_{n}}^{k-l_{n}-1}\sum_{q=l_{n}}^{k^{\prime}}\mathbb{E}[(\beta_{q,m}^{(1)})^{2}]\theta_{m,k^{\prime}}^{(8)}\theta_{m,q}^{(8)}\\
			\leq & C\Delta_{n}^{2}\sum_{k=2l_{n}+1}^{m-1}\int_{l_{n}\Delta_{n}}^{t_{k+1}-l_{n}\Delta_{n}}a(t_{m}-r)\left(\int_{l_{n}\Delta_{n}}^{r}a(t_{m}-s)^{2}ds\right)dr\\
			= & C\Delta_{n}^{2}\sum_{k=2l_{n}+1}^{m-1}\int_{t_{m}-t_{k+1}+l_{n}\Delta_{n}}^{t_{m}-l_{n}\Delta_{n}}a(x)\left(\int_{x}^{t_{m}-l_{n}\Delta_{n}}a(y)^{2}dy\right)dx\\
			\leq & C\Delta_{n}^{2}\sum_{k=2l_{n}+1}^{m-1}\int_{t_{m}-t_{k+1}+l_{n}\Delta_{n}}^{t_{m}-l_{n}\Delta_{n}}a(y)^{2}dy\\
			\leq & C\Delta_{n}\int_{(2l_{n}+1)\Delta_{n}}^{t_{m}}\int_{t_{m}-u+(l_{n}-1)\Delta_{n}}^{t_{m}-l_{n}\Delta_{n}}a(y)^{2}dydu\\
			\leq & C\Delta_{n}\int_{(2l_{n}+1)\Delta_{n}}^{t_{m}-l_{n}\Delta_{n}}\int_{w}^{t_{m}-l_{n}\Delta_{n}}a(y)^{2}dydw\\
			\leq & C\Delta_{n}\int_{0}^{t_{m}}a(y)^{2}ydy\leq C\Delta_{n},
		\end{align*}
		where in the last step we used the fact that $a(s)=\mathrm{O}(s^{-\alpha})$,
		as $s\rightarrow+\infty$ for some $\alpha>1$. For $q=5$, we use
		the decomposition 
		\[
		\sum_{q=k+1}^{j-2}\chi_{q}=L(\cup_{i=0}^{k}\cup_{q=k+1}^{j-2}\mathcal{P}_{A}^{n}(i,q))+\sum_{i=k}^{j-3}\delta_{i,j-1},
		\]
		to deduce that 
		\begin{align*}
			\sum_{j=l_{n}+1}^{n-1}\mathbb{E}(\xi_{j,n}^{(3)}\xi_{j,n}^{(5)}\mid\mathscr{G}_{j-1}^{n})= & \sum_{j=l_{n}+1}^{n-1}\mathbb{E}(\xi_{j,n}^{(3)}\xi_{j,n}^{(5)})+\sum_{j=l_{n}+1}^{n-1}\vartheta_{j}^{(9)}\\
			& +\sum_{k=l_{n}}^{n-3}\sum_{m=k+1}^{n-2}\alpha_{k+1,m}\vartheta_{k,m}^{(10)},
		\end{align*}
		where $\vartheta_{k,m}^{(10)}  =\sum_{q=l_{n}}^{k}\chi_{q}\theta_{q+1,m,m}^{(5)}$, and
		\begin{align*}
			\vartheta_{j}^{(9)}= & \sum_{k=l_{n}}^{j-3}\sum_{i=k}^{j-3}\mathbb{E}[(\beta_{k,j}^{(1)})^{2}][\delta_{k,j}\delta_{k,j-1}-\mathbb{E}(\delta_{k,j}\delta_{k,j-1})]\\
			& +\sum_{k=l_{n}}^{j-3}\delta_{k,j}L(\cup_{i=0}^{k}\cup_{q=k+1}^{j-2}\mathcal{P}_{A}^{n}(i,q)).	
		\end{align*}
		Note that the sum associated with $\vartheta_{j}^{(9)}$ can be analysed
		in the same way as in the case $p=q=5$. Moreover, the last error
		term are sums of $(\mathscr{F}_{k})$-increment martingale. After
		some simple changes of variables and order of integration we further
		have that
		\begin{equation*}
			\sum_{k=l_{n}}^{n-3}\sum_{m=k+1}^{n-2}\mathbb{E}[(\alpha_{k+1,m}\vartheta_{k,m}^{(10)})^{2}]  \leq C\int_{0}^{n\Delta_{n}}\int_{0}^{r}\varGamma(s)^{2}dsdr=\mathrm{o}[(n\Delta_{n})^{2}],
		\end{equation*}
		which according to Lemma \ref{lemmaavar}, let us conclude that (\ref{eq:lemmaapproxxi})
		is also valid in this case.
		\item [{$p=4$}] Lastly, from (\ref{eq:decompp34}), we obtain easily that
		\[
		\sum_{j=l_{n}+1}^{n-1}\mathbb{E}(\xi_{j,n}^{(4)}\xi_{j,n}^{(5)}\mid\mathscr{G}_{j-1}^{n})=\sum_{k=2l_{n}+1}^{n-3}\sum_{m=k+1}^{n-2}\alpha_{k+1,m}\vartheta_{k,m}^{(11)},
		\]
		with 
		\[
		\vartheta_{k,m}^{(11)}=\sum_{q=l_{n}}^{k-l_{n}-1}\sum_{p=q-l_{n}}^{n-2}\beta_{q,p}^{(1)}\Leb(\cup_{j=m\lor p+1}^{n-2}\mathcal{P}_{A}^{n}(q+1,j)).
		\]
		Since $\alpha_{k+1,m}\vartheta_{k,m}^{(11)}$ is an $(\mathscr{F}_{k})$-increment
		martingale and for $m\geq k+1$, $\mathbb{E}[(\vartheta_{k,m}^{(11)})^{2}]$ is bounded up to a constant by
		\begin{align*} \sum_{q=l_{n}}^{k-l_{n}-1}\sum_{q^{\prime}=l_{n}}^{q}&\Leb(\cup_{i=0}^{q^{\prime}-l_{n}}\cup_{p=q-l_{n}}^{n-2}\mathcal{P}_{A}^{n}(q+1,p))\Leb(\cup_{j=m+1}^{n-2}\mathcal{P}_{A}^{n}(q+1,j))\\
			&\times\Leb(\cup_{j=m+1}^{n-2}\mathcal{P}_{A}^{n}(q^{\prime}+1,j))	\end{align*}
		which is in turn uniformly bounded, we deduce as before that (\ref{eq:lemmaapproxxi}) holds in this situation.
	\end{description}
	
\end{proof}

\newpage
\section{Simulation study: Set-up}\label{asec:simsetup}
\subsection{Choice of the marginal distributions and trawl functions}
In the simulation study, we consider three different marginal distributions for the L\'{e}vy seed $L'$:
\begin{itemize}
	\item Negative binomial distribution: $L'\sim\mathrm{NegBin}(m, \theta)$, where the corresponding probability mass function is given by
	$\mathrm{P}(L'=x)=\frac{1}{x!}\frac{\Gamma \left( m
		+x\right) }{\Gamma \left( m \right) }\left( 1-\theta \right)^{m }\theta^{x}$ for $x \in \{0, 1, \ldots\}$. We choose $m=3.2$ and $\theta =0.2$. Then 
	\begin{align*}
		\mathbb{E}(L')&=m \theta/(1-\theta)=0.8,\\
		\Var(L')&=1, \\
		c_4(L')&=m \theta(\theta^2+4\theta+1)/(\theta-1)^4=2.875.
	\end{align*}
	
	\item Gamma distribution: $L'\sim \Gamma(\alpha_g, \sigma_g)$, where
	$\alpha_g>0$ is the shape parameter and $\sigma_g>0$ the scale parameter.
	Then the
	corresponding density is given by
	$$
	f(x)=\frac{1}{\sigma_g^{\alpha_g}\Gamma(\alpha_{g})}x^{\alpha_g-1}e^{-x/\sigma_g},
	$$
	for $x>0$. 
	We choose $\alpha_g =0.64$ and $\sigma_g =1.25$. 
	We note that 
	\begin{align*}
		\mathbb{E}(L')&=\alpha_g \sigma_g=0.8,\\
		\mathrm{Var}(L')&=\alpha_g \sigma_g^2=1,\\
		c_4(L')&=6\alpha \sigma^4=9.375.
	\end{align*}
	\item Gaussian distribution: 
	$L'\sim \mathrm{N}(\mu, \sigma^2)$. We choose $\mu=0.8$ and $\sigma=1$. 
	We note that 	\begin{align*}
		\mathbb{E}(L')&=\mu=0.8,\\
		\mathrm{Var}(L')&=\sigma^2=1,\\
		c_4(L')&=0.
	\end{align*}
\end{itemize}
For the trawl function, we consider two parametric settings:
\begin{itemize}
	\item Exponential trawl function: $a(x)=\exp(-\lambda x)$, for $\lambda >0$ and $x\geq 0$. We choose $\lambda =1$.
	\item SupGamma trawl function: $a(x) = (1+x/\overline{\alpha})^{-H}$, for $\overline{\alpha}>0, H>1$ and $x\geq 0$. Note that for $H\in (1,2]$, the resulting trawl process exhibits long memory and for $H> 2$ short memory. We choose $\overline{\alpha}=2$, 
    $H=1.5$ to study the long-memory setting. 
\end{itemize}
Using the {\tt R} package {\tt ambit}, we simulate trawl processes on a grid of width $\Delta_n \in \{0.5, 0.1, 0.01\}$ for $n=10,000$ and consider relevant subsets of the trawl processes when $n \in \{2000, 5000, 10000\}$.
We will consider the observations on the grid $i\Delta_n$ for $i\in \{0, 1, \ldots, n-1\}$ for the trawl function estimation in the following.
All results are reported based on 1000 Monte Carlo runs. 

The exponential and a supGamma trawl used in the simulation study are depicted in Figure \ref{fig:TrawlSimExpLM}. 
\begin{figure}[htbp]
    \centering
 \captionsetup[subfigure]{aboveskip=-4pt, belowskip=-4pt}
 \subfloat[Exponential and supGamma (LM) trawl\label{fig:TrawlSimExpLM}]{%
   \includegraphics[scale=0.25]{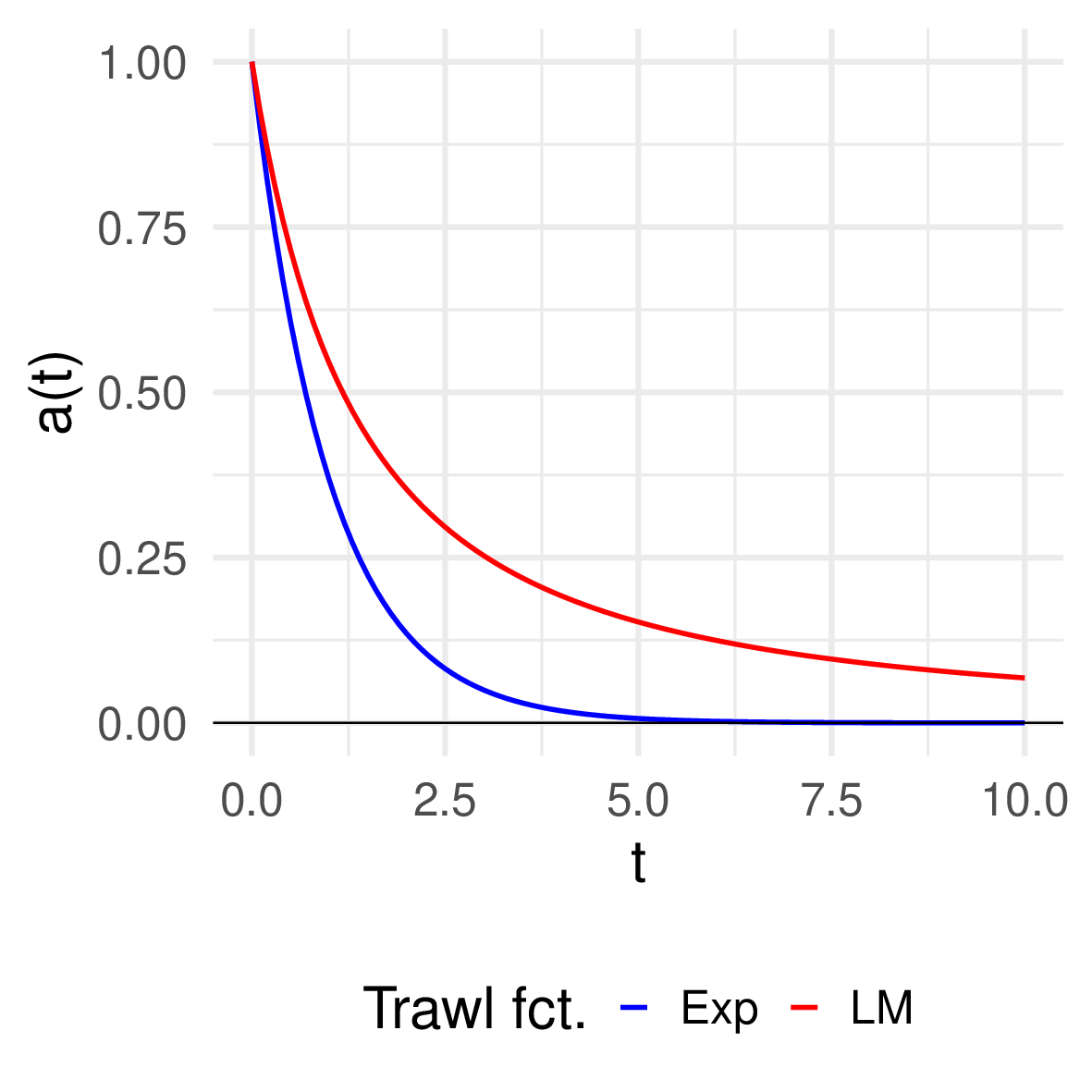}
 }
 \subfloat[Exponential trawl\label{fig:TrawlSimExp_points}]{%
   \includegraphics[scale=0.25]{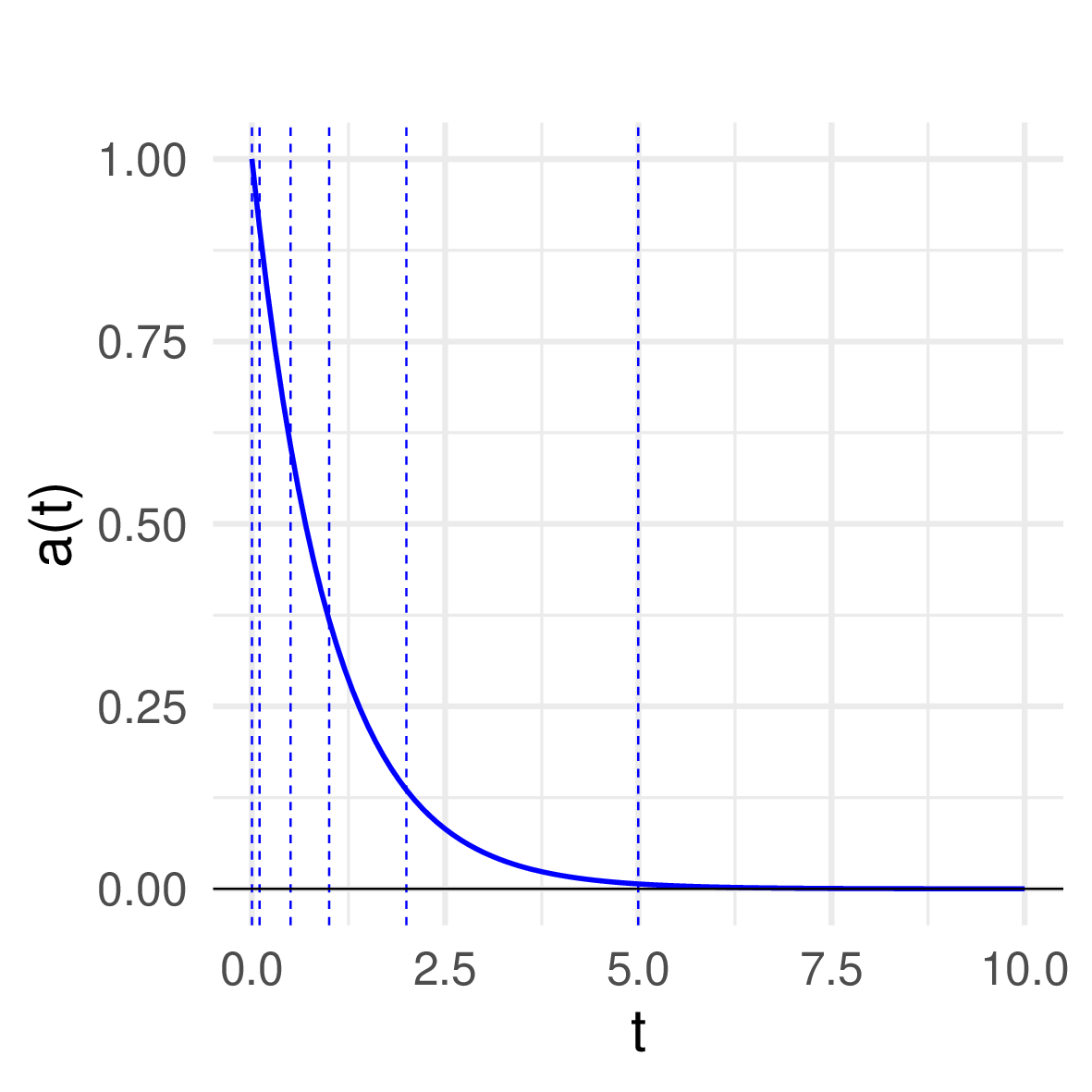}
 }
 \subfloat[SupGamma (LM) trawl \label{fig:TrawlSimLM_points}]{%
   \includegraphics[scale=0.25]{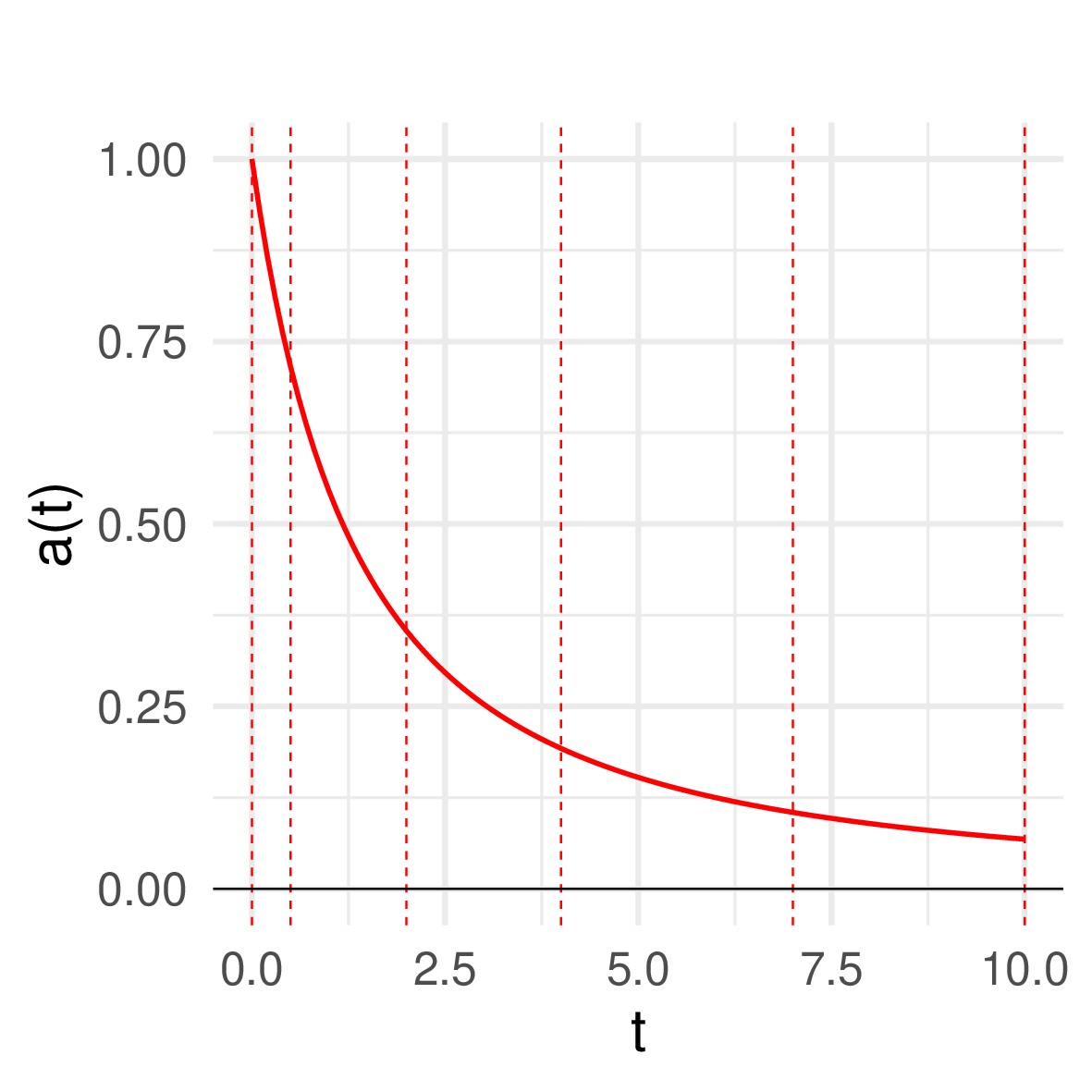}
 }
    \caption{\it Trawl functions used in simulation study: Figure \ref{fig:TrawlSimExpLM} depicts an exponential trawl (in blue) and a supGamma (LM) trawl (in red). Figures \ref{fig:TrawlSimExp_points} and \ref{fig:TrawlSimLM_points} indicate the six evaluation points for the exponential trawl and the supGamma (LM) trawl, respectively, as vertical lines. }
    \label{fig:TrawlSimExpLM}
\end{figure}
We  evaluate each trawl function at six representative points
to assess the performance of our estimator at different levels of the trawl function. 
For this, we fix the time points $t$ and consider three different values for the width of the observation grid, i.e.~$\Delta_n\in \{0.5, 0.1, 0.01\}$. We then choose $i$ such that the fixed-$t$ values are obtained, please see Table \ref{tab:Grid-fixedt-Exp} for details for the exponential trawl and Table \ref{tab:Grid-fixedt_LM} for the supGamma (LM) trawl.  
\begin{table}[ht]
	\centering
	\begin{tabular}{l|lllllll}
		\toprule
		$t$ &0&0.1&0.5&1&2&5\\ \midrule
        If $\Delta_n=0.5$, then $i$ & 0&2&10&20&40&100\\
		If $\Delta_n=0.1$, then $i$ & 0&1&5&10&20&50 \\
		If $\Delta_n=0.01$, then $i$ &0&10&50&100&200&500
		\\
        \midrule
		$a(t)=\exp(-\lambda t)$ & 
        1& 0.90& 0.61& 0.37 &0.14& 0.01\\
		for  $\lambda =1$ &&&&&&
		\\
		\bottomrule
	\end{tabular}
	\caption{\label{tab:Grid-fixedt-Exp}
		Recall that $t=i\Delta_n$. The last row contains the true values for which we estimate the exponential trawl function in the following simulation study.
	}
\end{table}

\begin{table}[htbp]
	\centering
	\begin{tabular}{l|lllllll}
		\toprule
		$t$&0& 0.5& 2& 4& 7& 10\\ \midrule
        If $\Delta_n=0.5$, then $i$ & 0&1&4&8&14&20 \\
		If $\Delta_n=0.1$, then $i$ & 0&5&20&40&70&100 \\
		If $\Delta_n=0.01$, then $i$ &0&50&200&400&700&1000
		\\ \midrule
		$a(t)=(1+t/\overline{\alpha})^{-H}$, & 1 &0.72 & 0.35 & 0.19 & 0.1 & 0.07  \\
		for  $\overline{\alpha} =2, H=1.5$ &&&&&&
		\\
		\bottomrule
	\end{tabular}
	\caption{\label{tab:Grid-fixedt_LM}
		Recall that $t=i\Delta_n$.  True values for which we estimate the supGamma trawl function in the following simulation study.
	}
\end{table}
Recall that our asymptotic theory relies on a double-asymptotic scheme where we need simultaneously that $\Delta_n \to 0$,  $n \to \infty$ and $n \Delta_n \to \infty$.
Moreover, recall that we set $n \Delta_n^3 \to \mu$ and $n \Delta_n^2 \to \mu_0$. 
Throughout our simulation experiments, we shall focus on the case when $\mu=0$ (and $\mu_0=0$) which suggests a relation of $\Delta_n=O(n^{-\beta})$ for $\beta \in (1/3,1)$ (or $\beta \in (1/2, 1)$). 
In finite samples, we will not be able to determine the particular rate $\beta$ of the double-asymptotic regime, hence we focus on three choices of $\Delta_n\in \{0.5, 0.1, 0.01\}$ and three choices of $n$ ($n\in \{2000, 5000, 10000\}$) to cover different in-fill and long-span asymptotic settings.

\section{Simulation study: Consistency results}\label{asec:consistency}

First of all, we study the consistency and bias of the proposed estimators in finite samples. 
We consider the trawl function estimator $\widehat a(t)$ and the bias-corrected version $\hat a(t)-\frac{1}{2}\Delta_n \hat a'(t)$.

For both estimators, we report the boxplots of the  trawl estimates for the six evaluation points over 1000 Monte Carlo runs
in Figures \ref{fig:negbin_exp_results} -- \ref{fig:gauss_lm_results}.
We observe that for both the exponential and the supGamma (LM) trawl function, the estimation bias is generally small for both estimators, with the one for the bias-corrected version, as expected, being even smaller. 
We note that the choice of $n$ has almost no impact on the size of the bias (for fixed $\Delta_n$), but an increase in $n$ decreases the standard deviation of the estimators.
When comparing the boxplots row-wise, we observe that the length of the observation window $\Delta_n n$ noticeably affects the performance of the finite sample. I.e.~shrinking $\Delta_n$ without at the same time increasing $n$ does not necessarily lead to improved finite-sample performance.  

Overall, we conclude that the proposed estimators work well in finite samples across different marginal distributions, short- and long-memory settings and different in-fill and long-span asymptotic settings.
In the long-memory case, we advise the use of the bias-corrected estimator, in particular for small times $t$.

Before presenting the boxplots of our simulation results, we briefly visualise the consistency result in   Figure \ref{fig:ConsistencyTrawl}. We recall that our asymptotic theory comprises a double asymptotic scheme, where we require that $\Delta_n \to 0$,  $n \to \infty$ and $n \Delta_n \to \infty$. Solely for the visualisation, we present a scenario where we set $\Delta_n=0.1$ and simulate one path of a Poisson trawl process with 10,000 observations. We depict the sample path in Figure \ref{fig:path} and its autocorrelation function in Figure \ref{fig:acf} (with the theoretical autocorrelation function superimposed as a solid red line).  We then use the first $n\in \{20, 50, 100, 500, 1000, 2000, 5000, 7000, 10000\}$ observations of that path to estimate the trawl function, resulting in choices $t=\Delta_n n \in \{2, 5, 10,  50, 100, 200, 500, 700, 1000\}$, see Figure \ref{fig:trawlfct}. We note that, for $n$ and $t$ small, the  trawl function estimate exhibits strong variability as expected, but  the estimated trawl function gets closer  to its theoretical counterpart depicted with a solid red line as $n$ and $t$ grow.  
\begin{figure}[htbp]
\centering
	\captionsetup[subfigure]{aboveskip=-4pt,belowskip=-4pt}
	\subfloat[	\label{fig:path}]{	\includegraphics[scale =0.3]{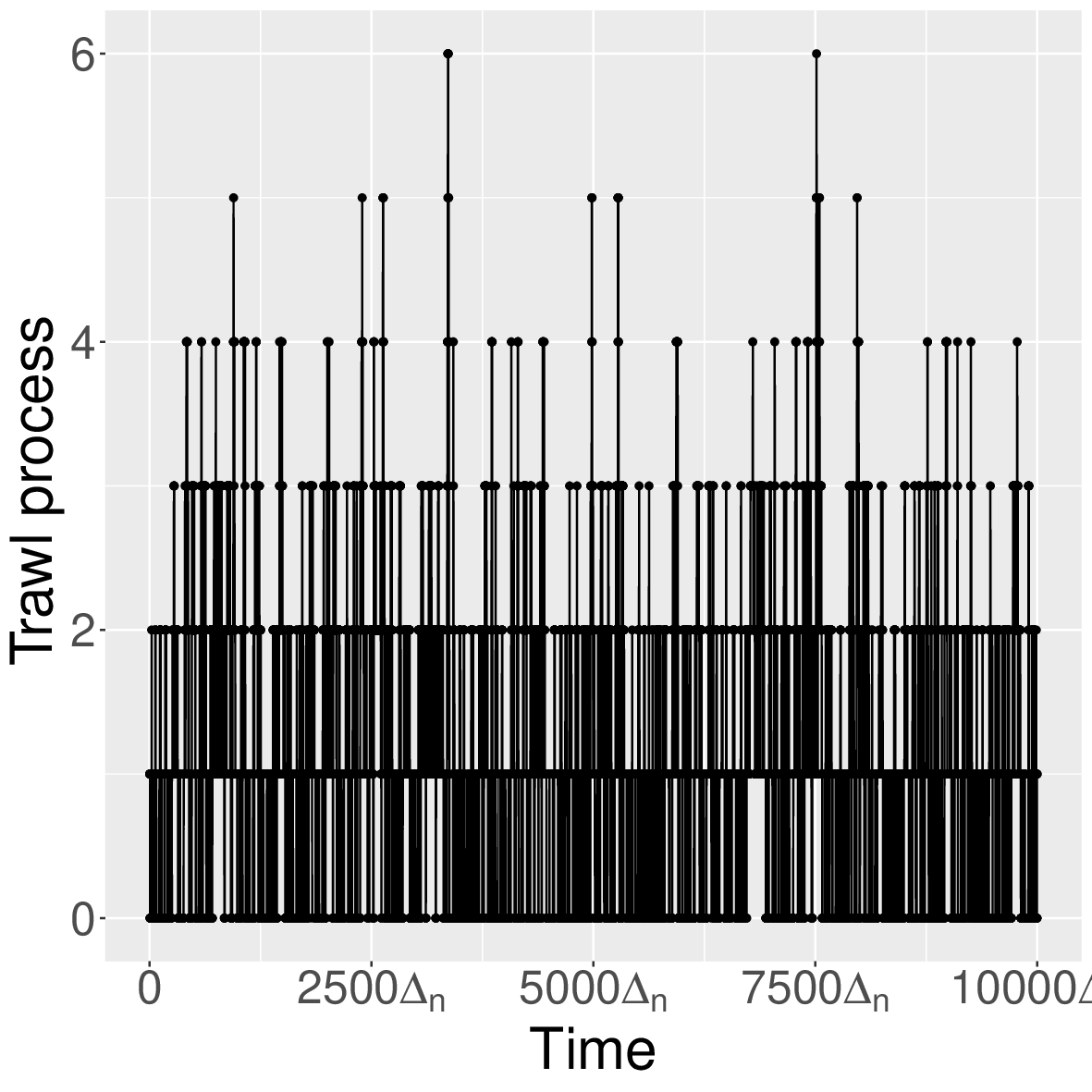} } 
  \subfloat[ \label{fig:acf}]{	\includegraphics[scale=0.3]{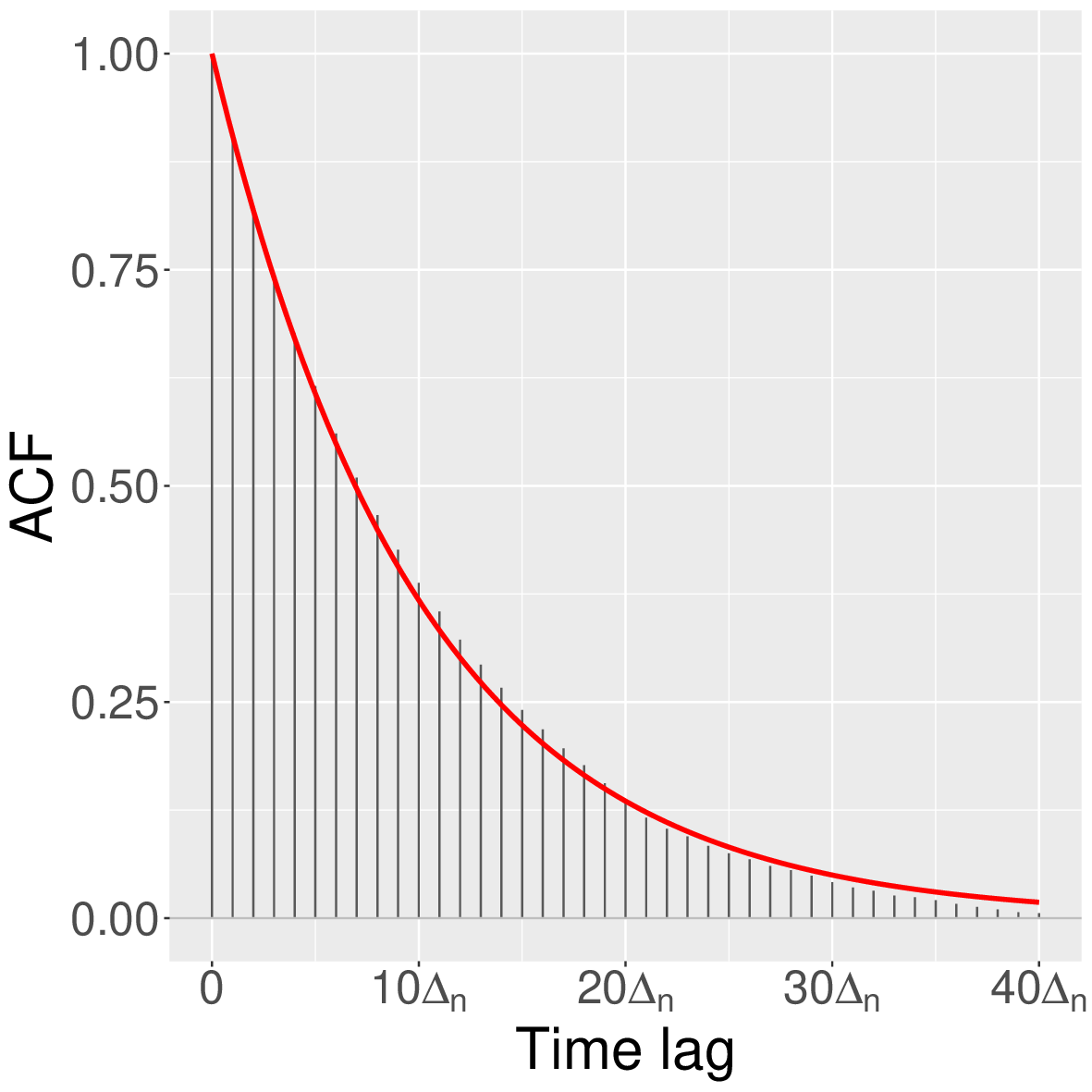} } 
 \\
	\subfloat[ \label{fig:trawlfct}]{	\includegraphics[scale=0.5]{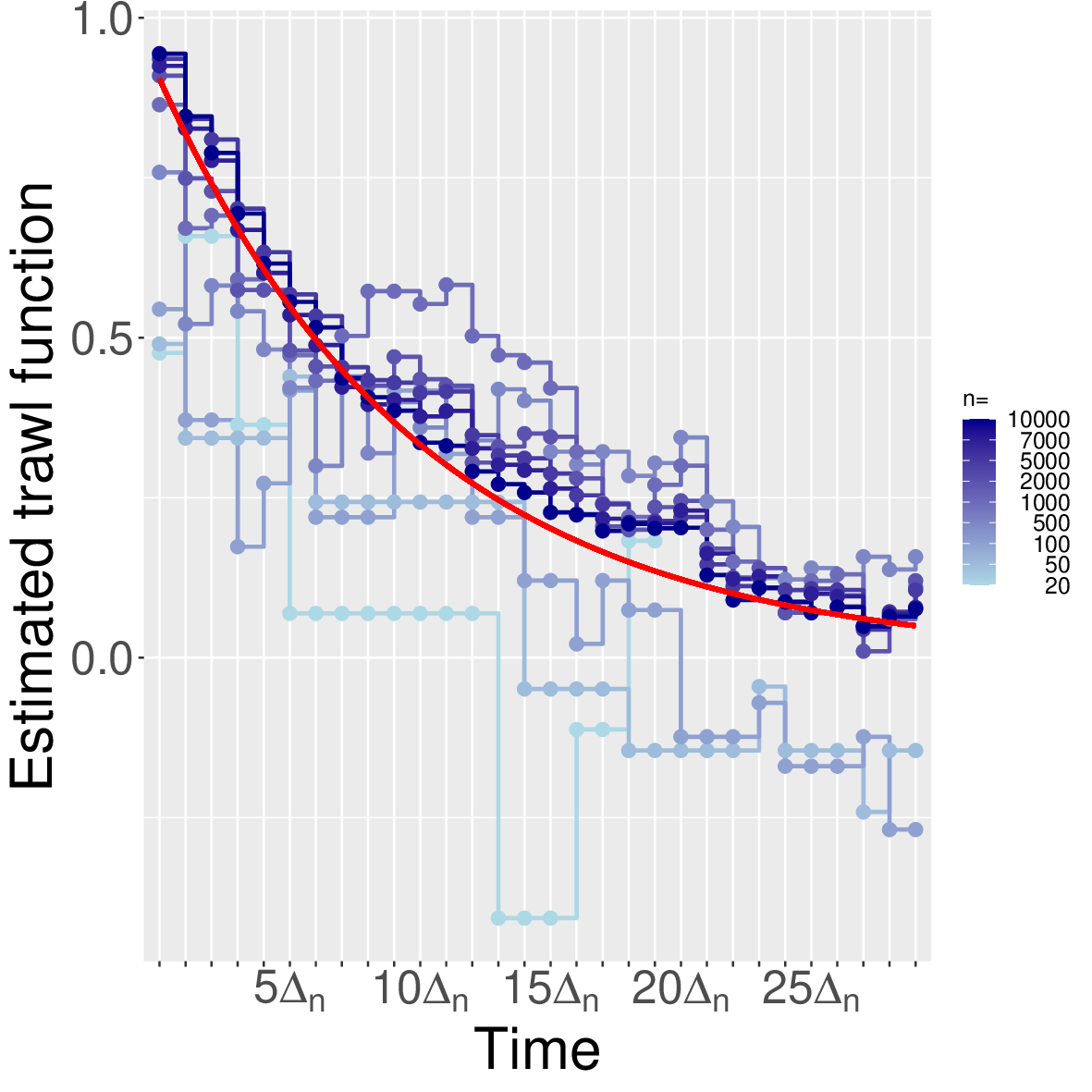} } 
	\caption{\it  We illustrate the consistency result  in a simulation experiment: 
 We choose $\Delta_n=0.1$ and vary $n\in \{20, 50, 100, 500, 1000, 2000, 5000, 7000\}$, resulting in $t=\Delta_n n \in \{2, 5, 10,  50, 100, 200, 500, 700\}$. The Levy seed is chosen as $L'\sim\mathrm{Poi}(1)$ and the trawl function is given by $a(x)=\exp(-x), x>0$. 
 }
 \label{fig:ConsistencyTrawl}
\end{figure}

\begin{figure}[htbp]
\centering
\captionsetup[subfigure]{aboveskip=-4pt, belowskip=-4pt}
\subfloat[NegBin Exp $\Delta_n = 0.5$ - Estimates\label{fig:NegBinExp0p5}]{%
  \includegraphics[scale=0.3]{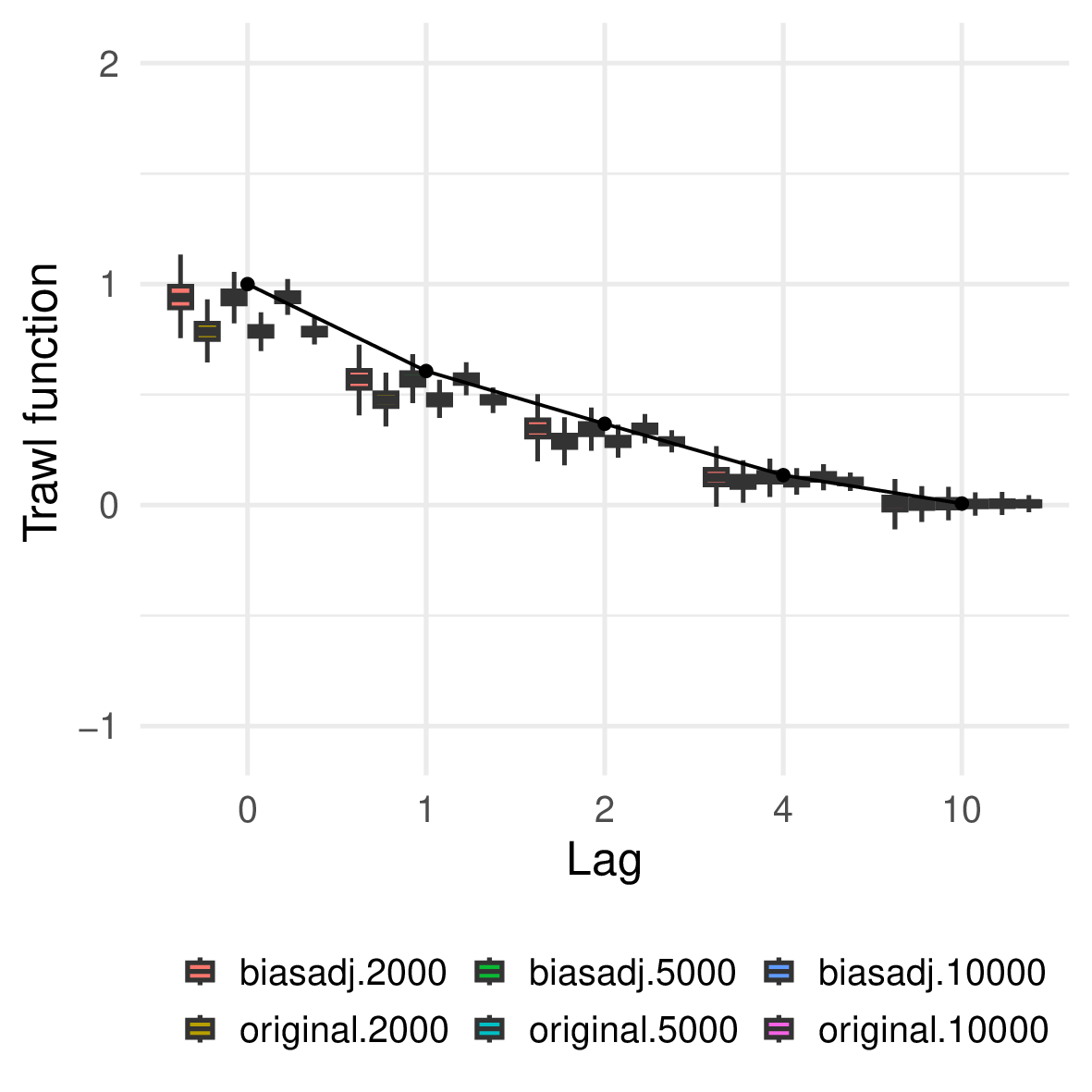}
}
\hfill
\subfloat[NegBin Exp $\Delta_n = 0.5$ - Bias\label{fig:NegBinExp0p5-Bias}]{%
  \includegraphics[scale=0.3]{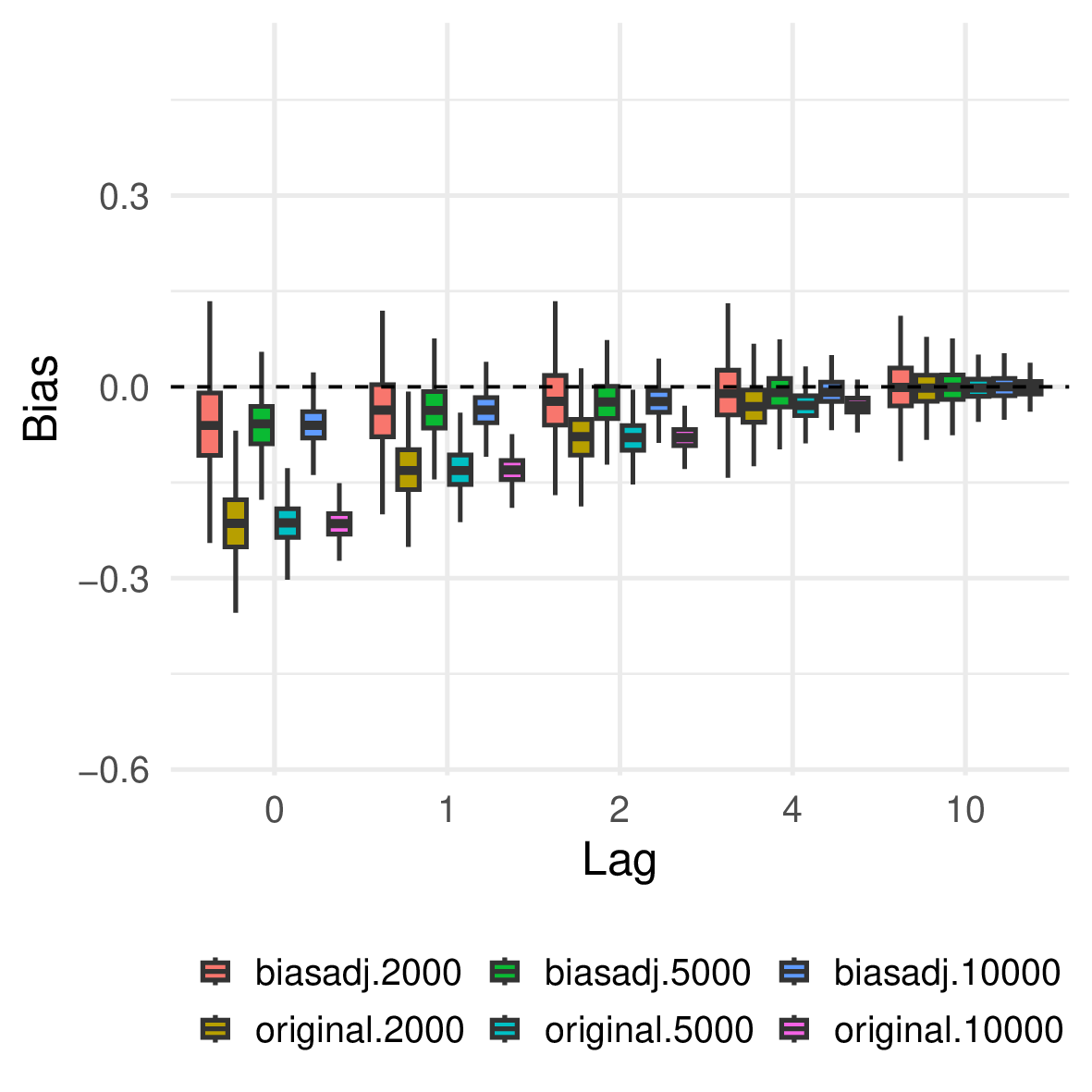}
}
\\
\subfloat[NegBin Exp $\Delta_n = 0.1$ - Estimates\label{fig:NegBinExp0p1}]{%
  \includegraphics[scale=0.3]{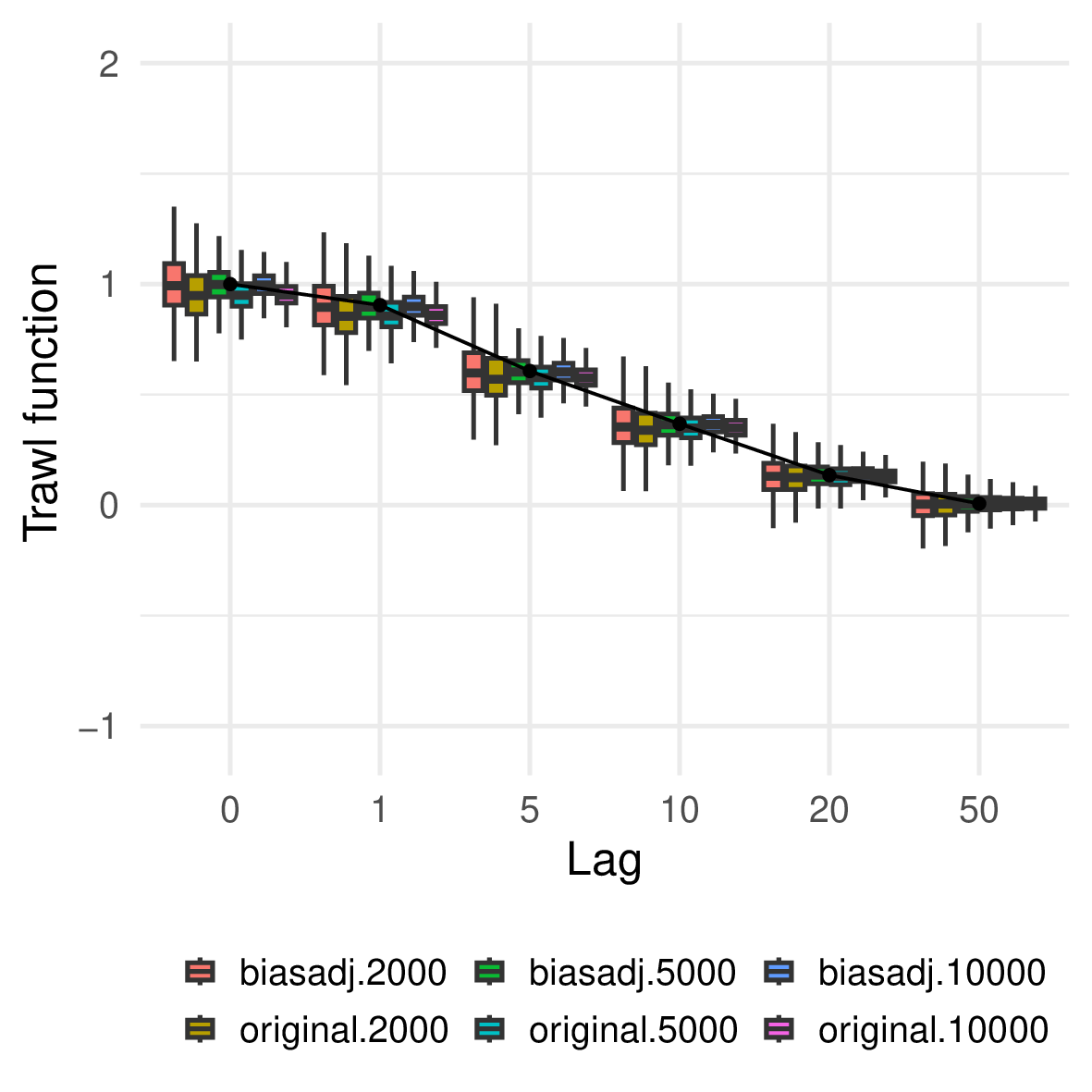}
}
\hfill
\subfloat[NegBin Exp $\Delta_n = 0.1$ - Bias\label{fig:NegBinExp0p1-Bias}]{%
  \includegraphics[scale=0.3]{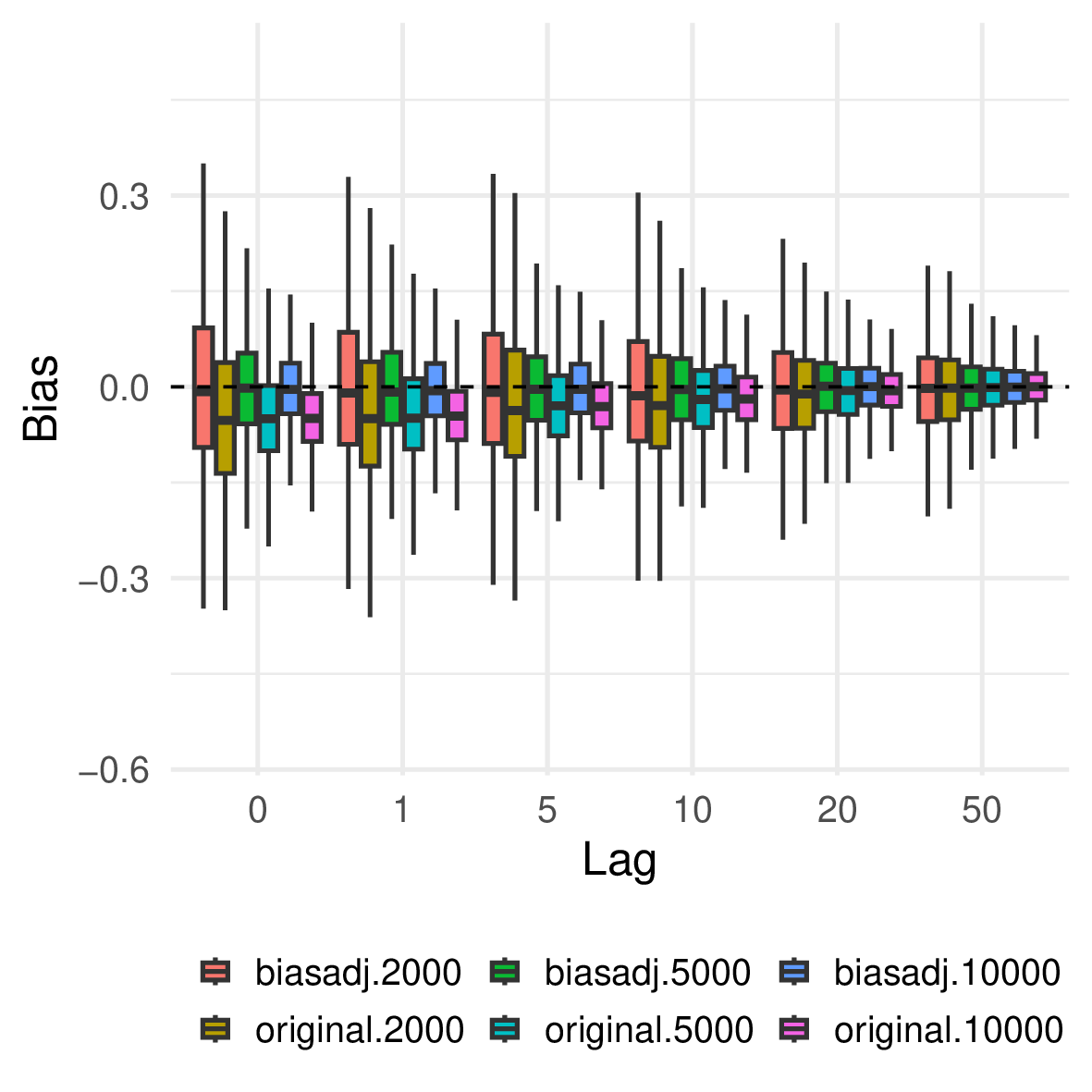}
}
\\
\subfloat[NegBin Exp $\Delta_n = 0.01$ - Estimates\label{fig:NegBinExp0p01}]{%
  \includegraphics[scale=0.3]{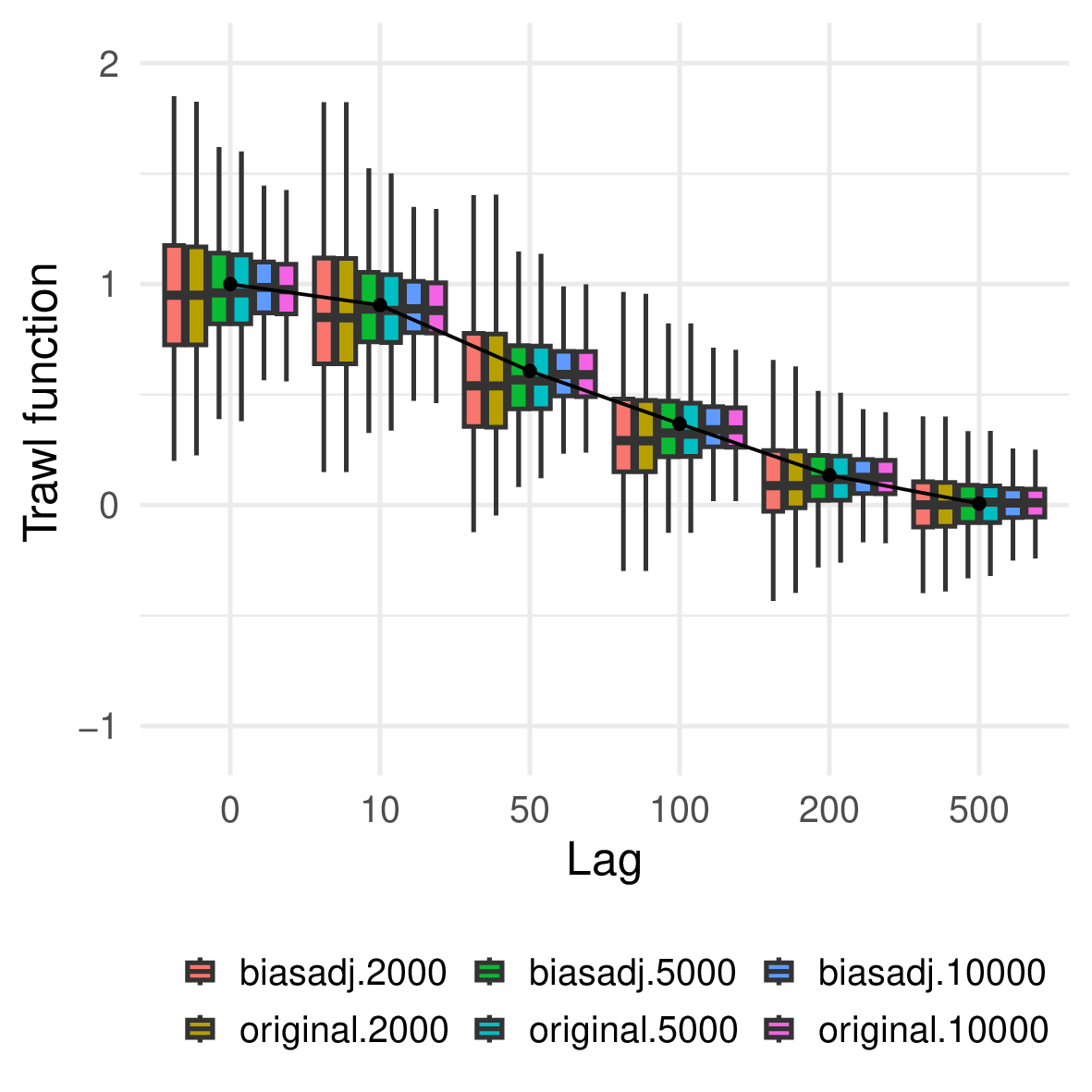}
}
\hfill
\subfloat[NegBin Exp $\Delta_n = 0.01$ - Bias\label{fig:NegBinExp0p01-Bias}]{%
  \includegraphics[scale=0.3]{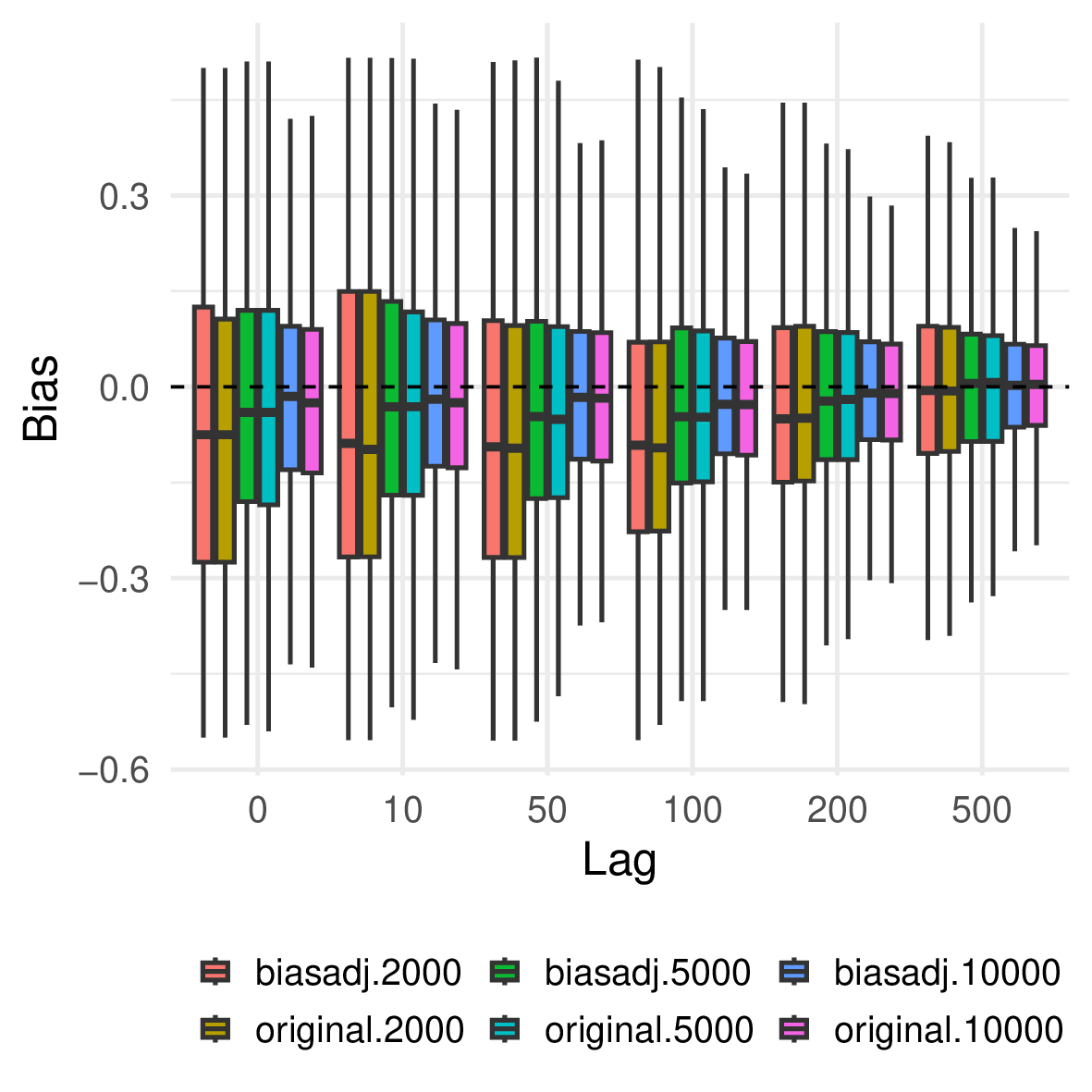}
}
\caption{Negative Binomial distribution with exponential trawl function results for different $\Delta_n$ values.}
\label{fig:negbin_exp_results}
\end{figure}

\begin{figure}[htbp]
\centering
\captionsetup[subfigure]{aboveskip=-4pt, belowskip=-4pt}
\subfloat[Gamma Exp $\Delta_n = 0.5$ - Estimates\label{fig:GammaExp0p5}]{%
  \includegraphics[scale=0.3]{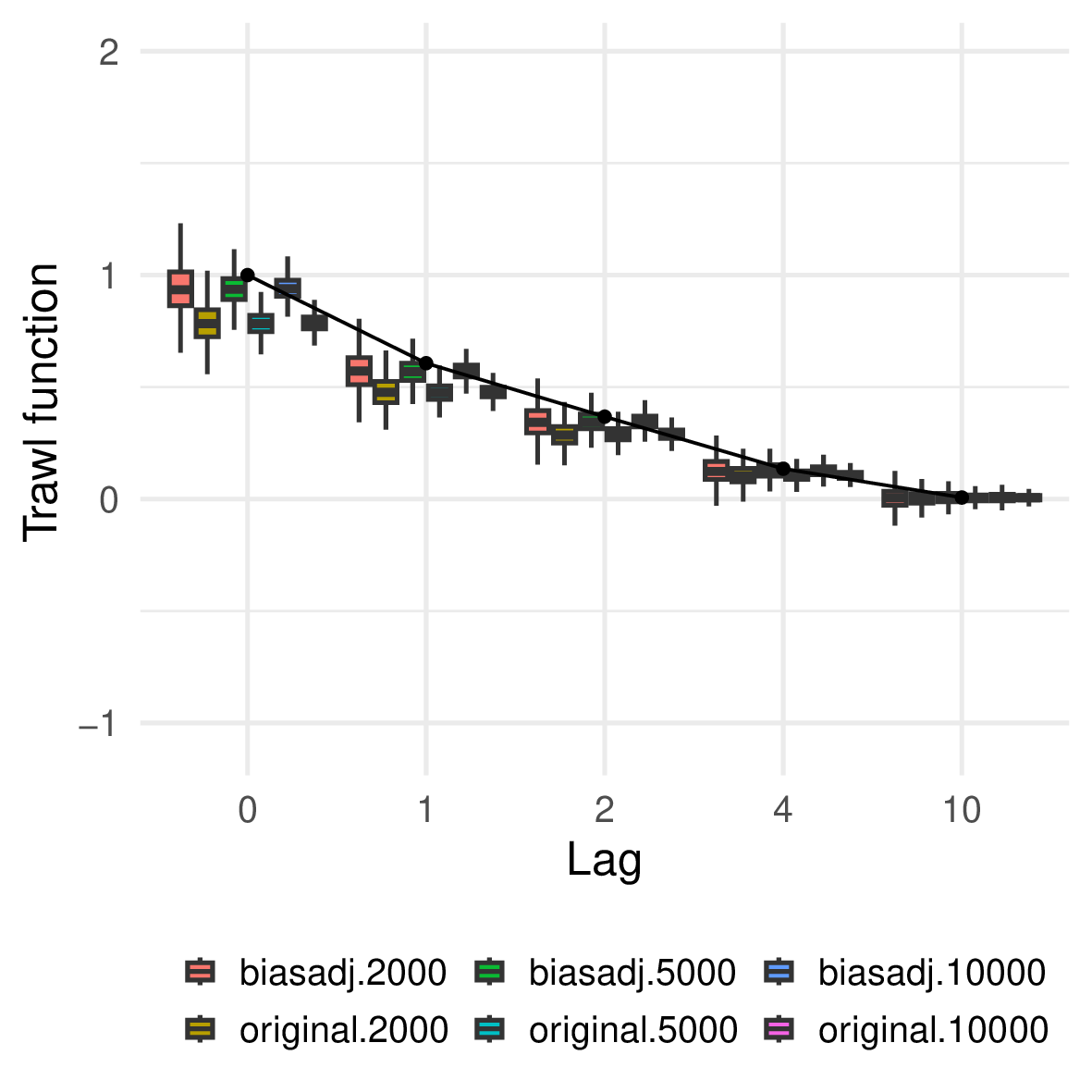}
}
\hfill
\subfloat[Gamma Exp $\Delta_n = 0.5$ - Bias\label{fig:GammaExp0p5-Bias}]{%
  \includegraphics[scale=0.3]{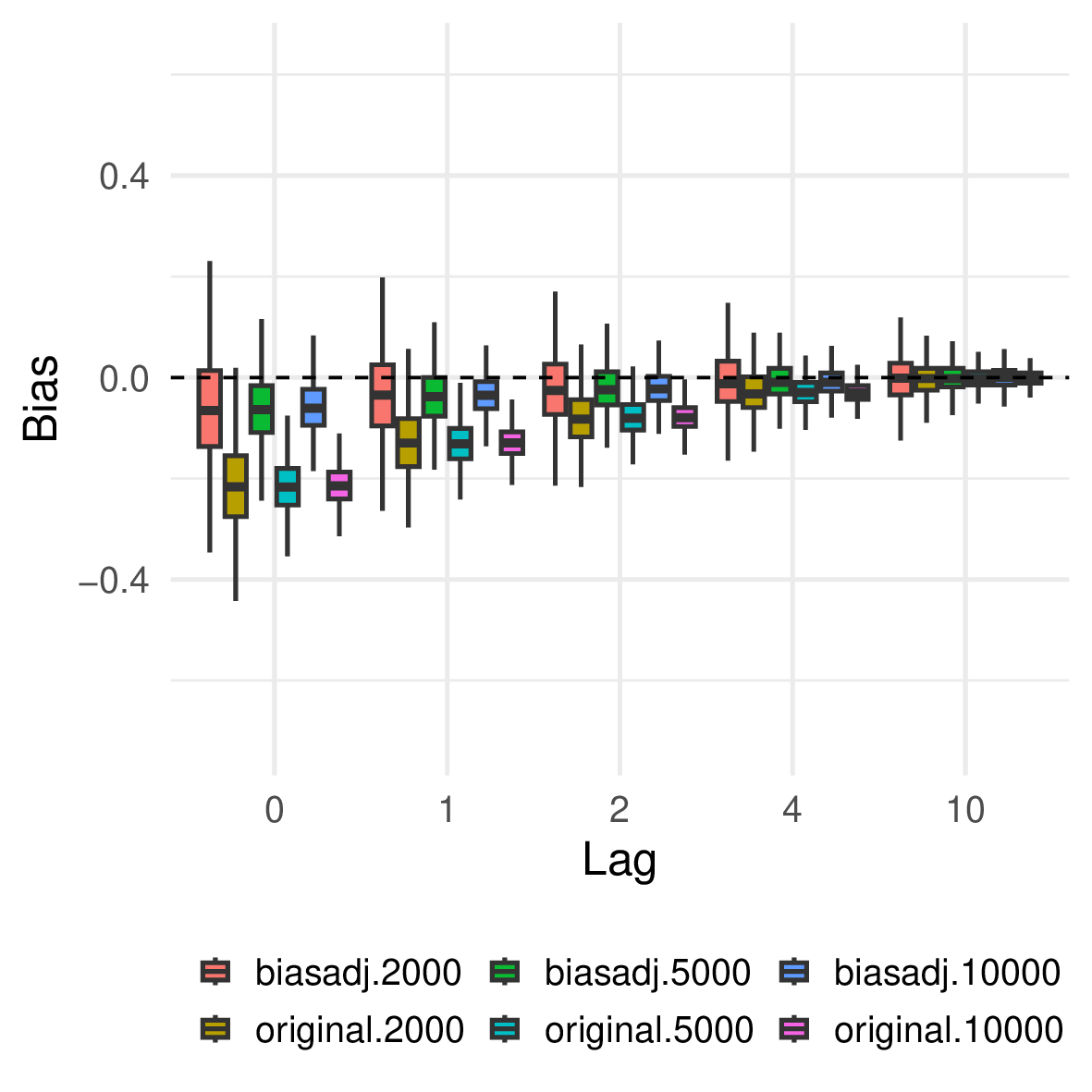}
}
\\
\subfloat[Gamma Exp $\Delta_n = 0.1$ - Estimates\label{fig:GammaExp0p1}]{%
  \includegraphics[scale=0.3]{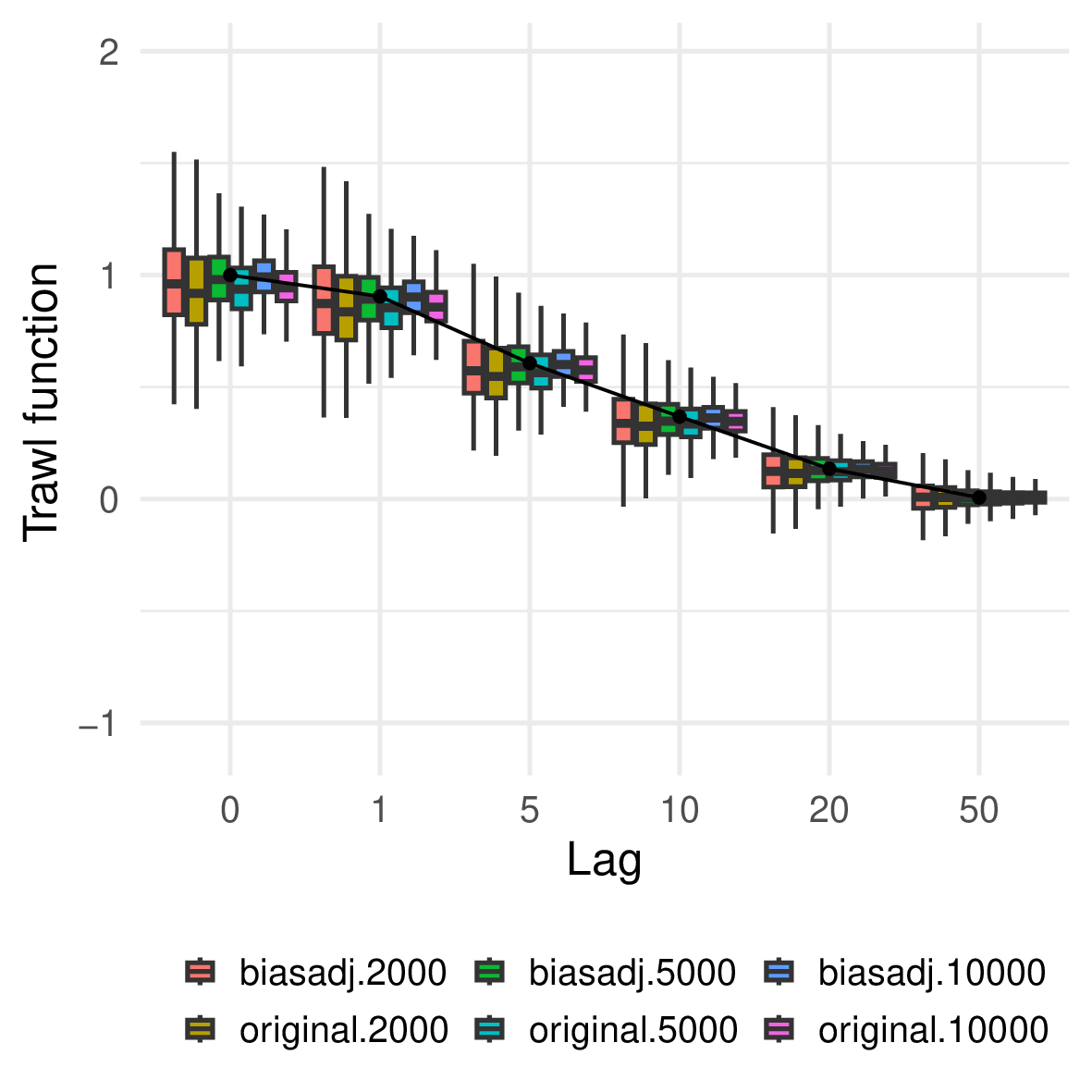}
}
\hfill
\subfloat[Gamma Exp $\Delta_n = 0.1$ - Bias\label{fig:GammaExp0p1-Bias}]{%
  \includegraphics[scale=0.3]{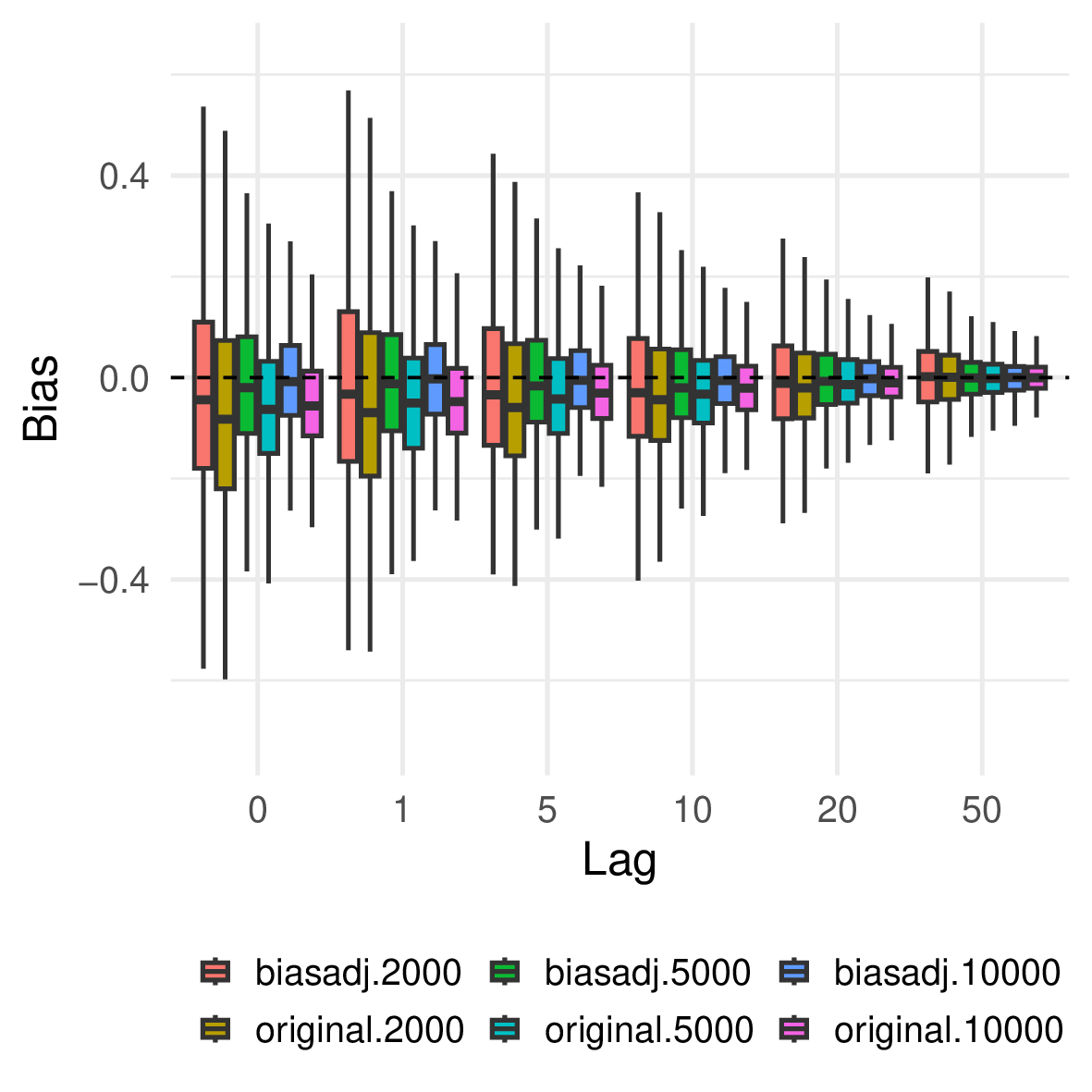}
}
\\
\subfloat[Gamma Exp $\Delta_n = 0.01$ - Estimates\label{fig:GammaExp0p01}]{%
  \includegraphics[scale=0.3]{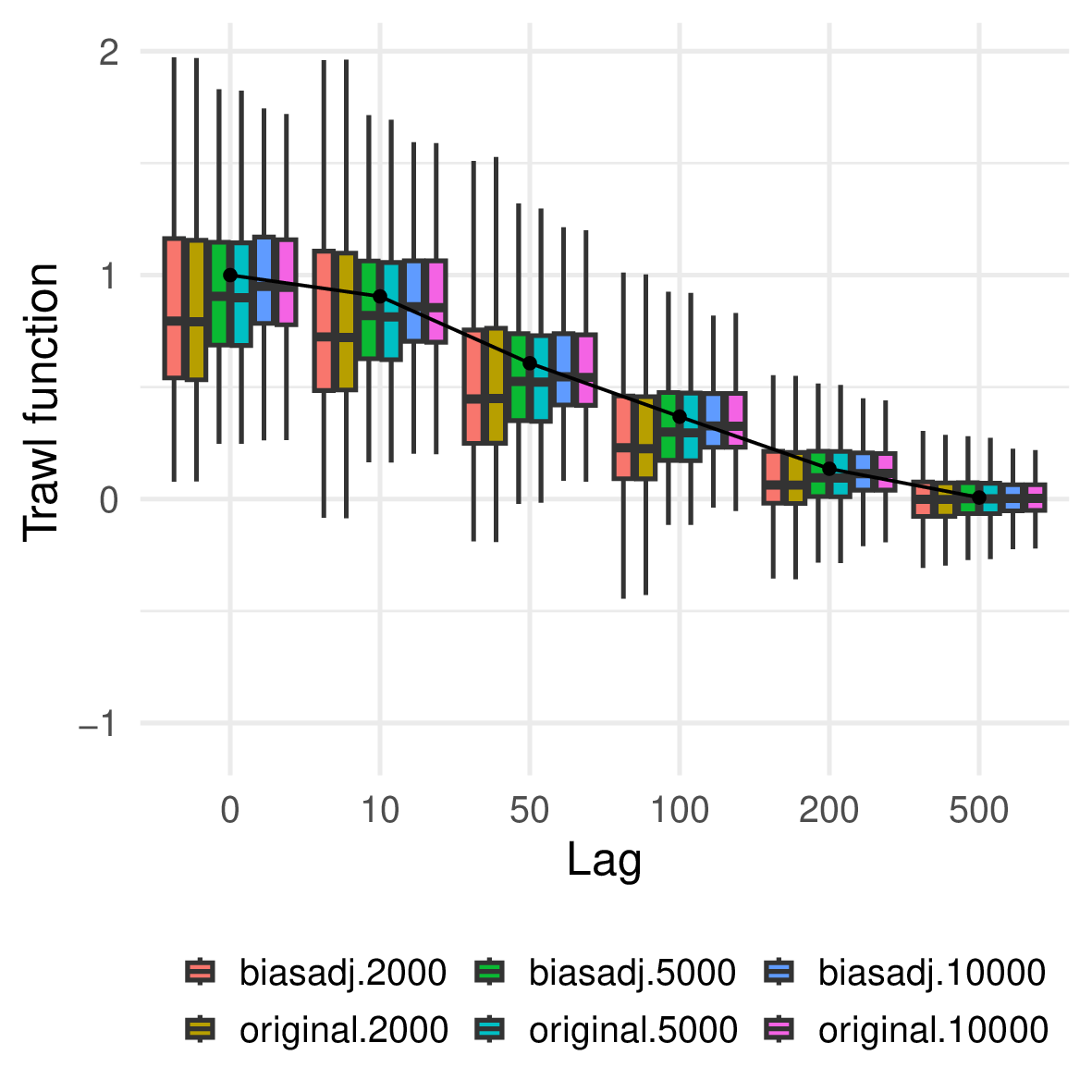}
}
\hfill
\subfloat[Gamma Exp $\Delta_n = 0.01$ - Bias\label{fig:GammaExp0p01-Bias}]{%
  \includegraphics[scale=0.3]{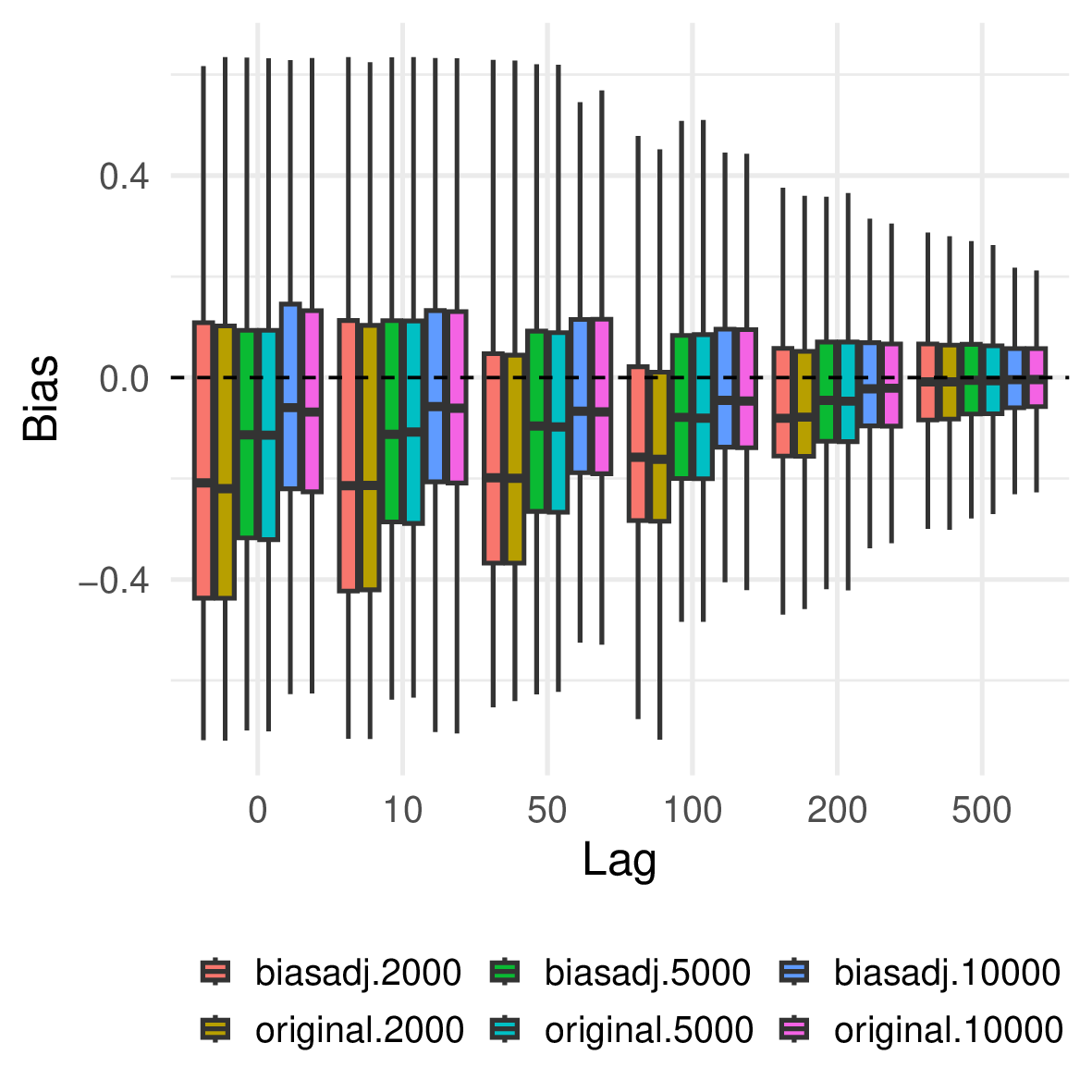}
}
\caption{Gamma distribution with exponential trawl function results for different $\Delta_n$ values.}
\label{fig:gamma_exp_results}
\end{figure}

\begin{figure}[htbp]
\centering
\captionsetup[subfigure]{aboveskip=-4pt, belowskip=-4pt}
\subfloat[Gaussian Exp $\Delta_n = 0.5$ - Estimates\label{fig:GaussExp0p5}]{%
  \includegraphics[scale=0.3]{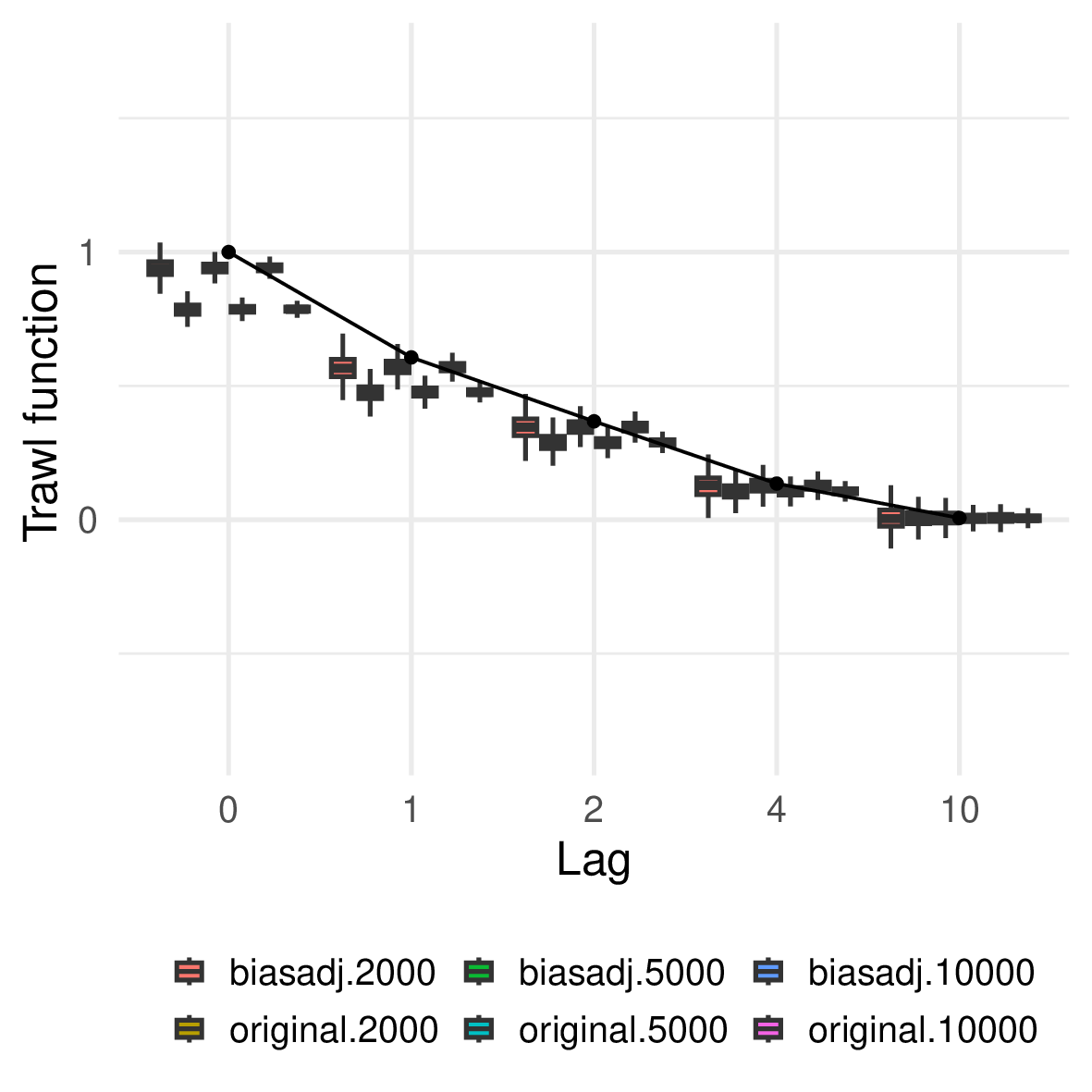}
}
\hfill
\subfloat[Gaussian Exp $\Delta_n = 0.5$ - Bias\label{fig:GaussExp0p5-Bias}]{%
  \includegraphics[scale=0.3]{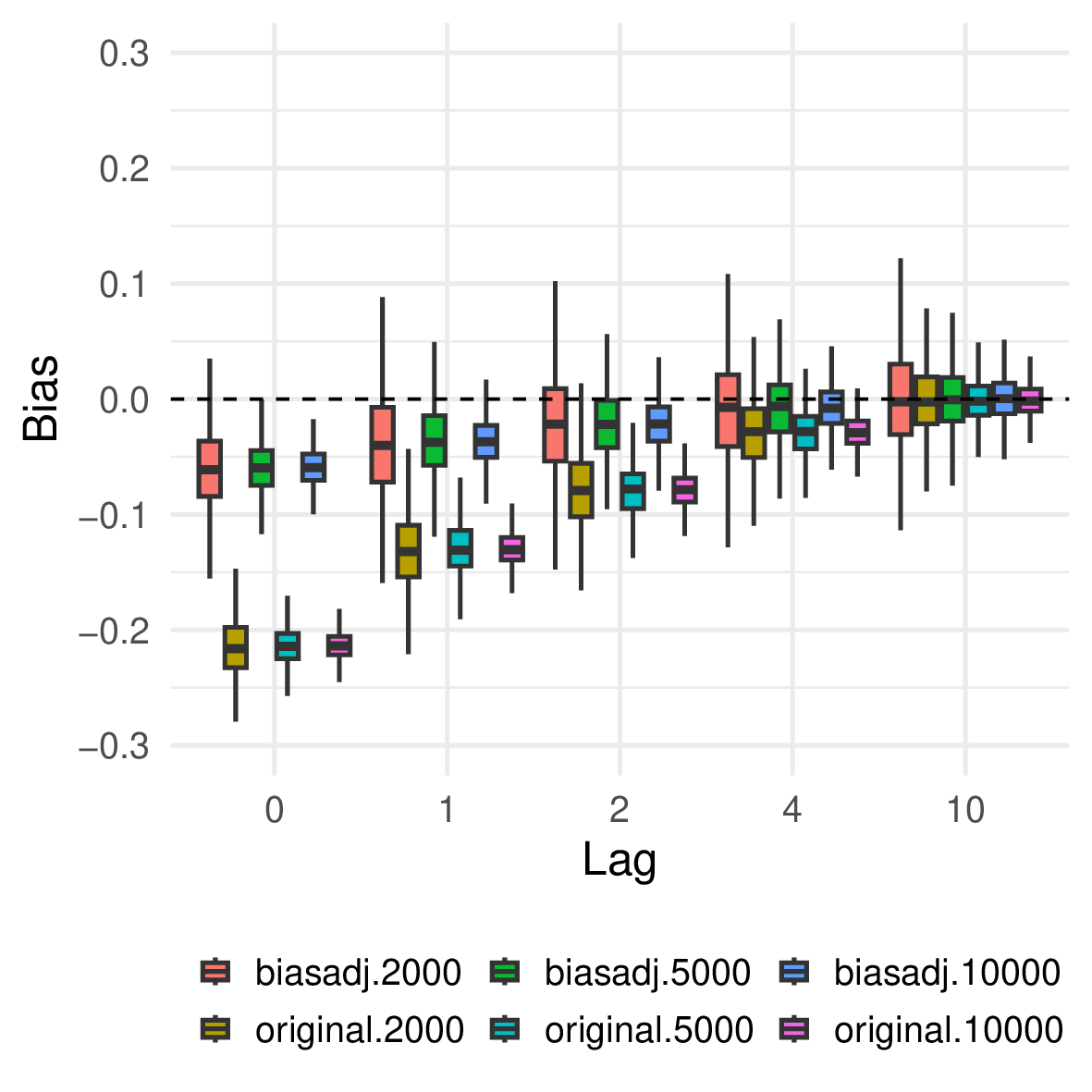}
}
\\
\subfloat[Gaussian Exp $\Delta_n = 0.1$ - Estimates\label{fig:GaussExp0p1}]{%
  \includegraphics[scale=0.3]{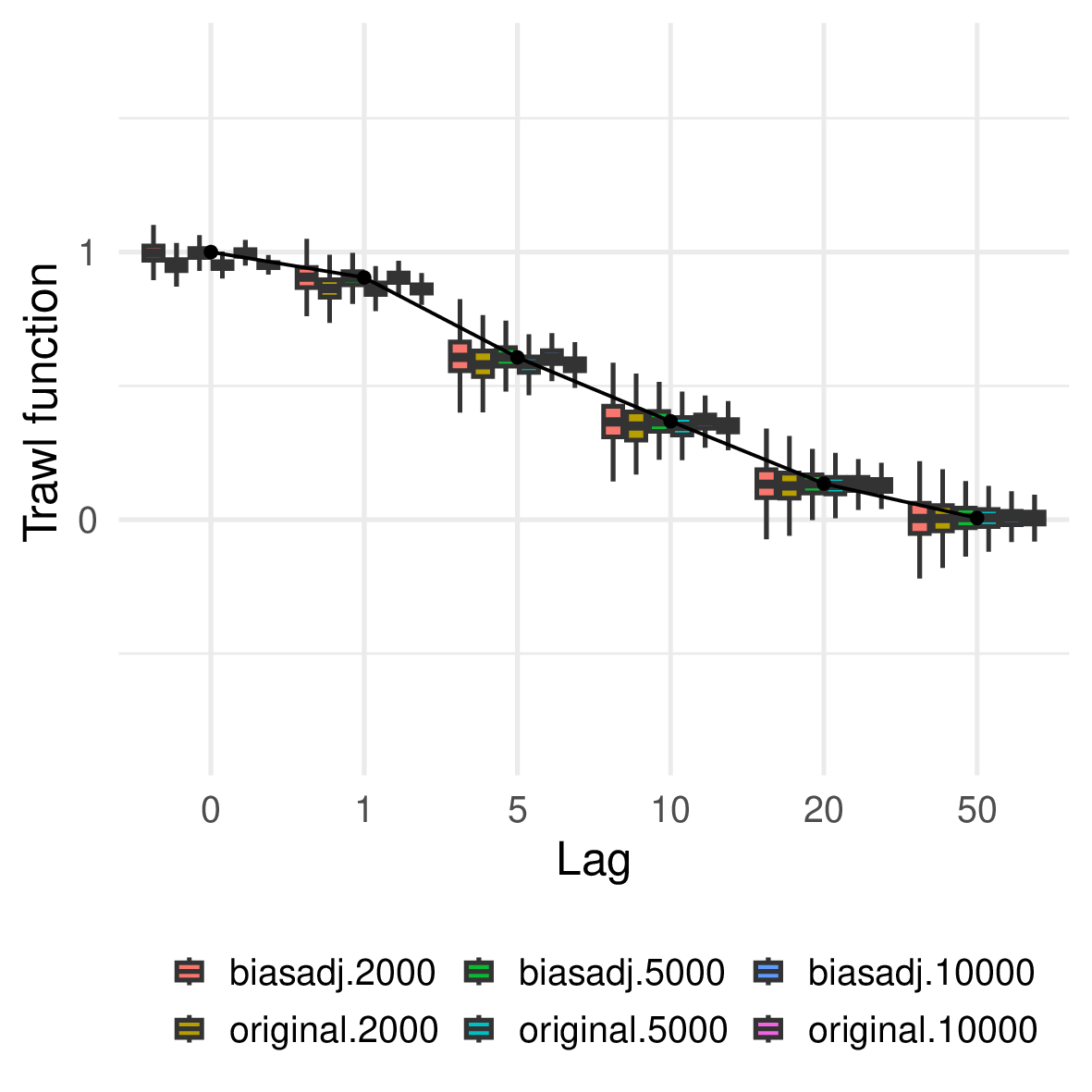}
}
\hfill
\subfloat[Gaussian Exp $\Delta_n = 0.1$ - Bias\label{fig:GaussExp0p1-Bias}]{%
  \includegraphics[scale=0.3]{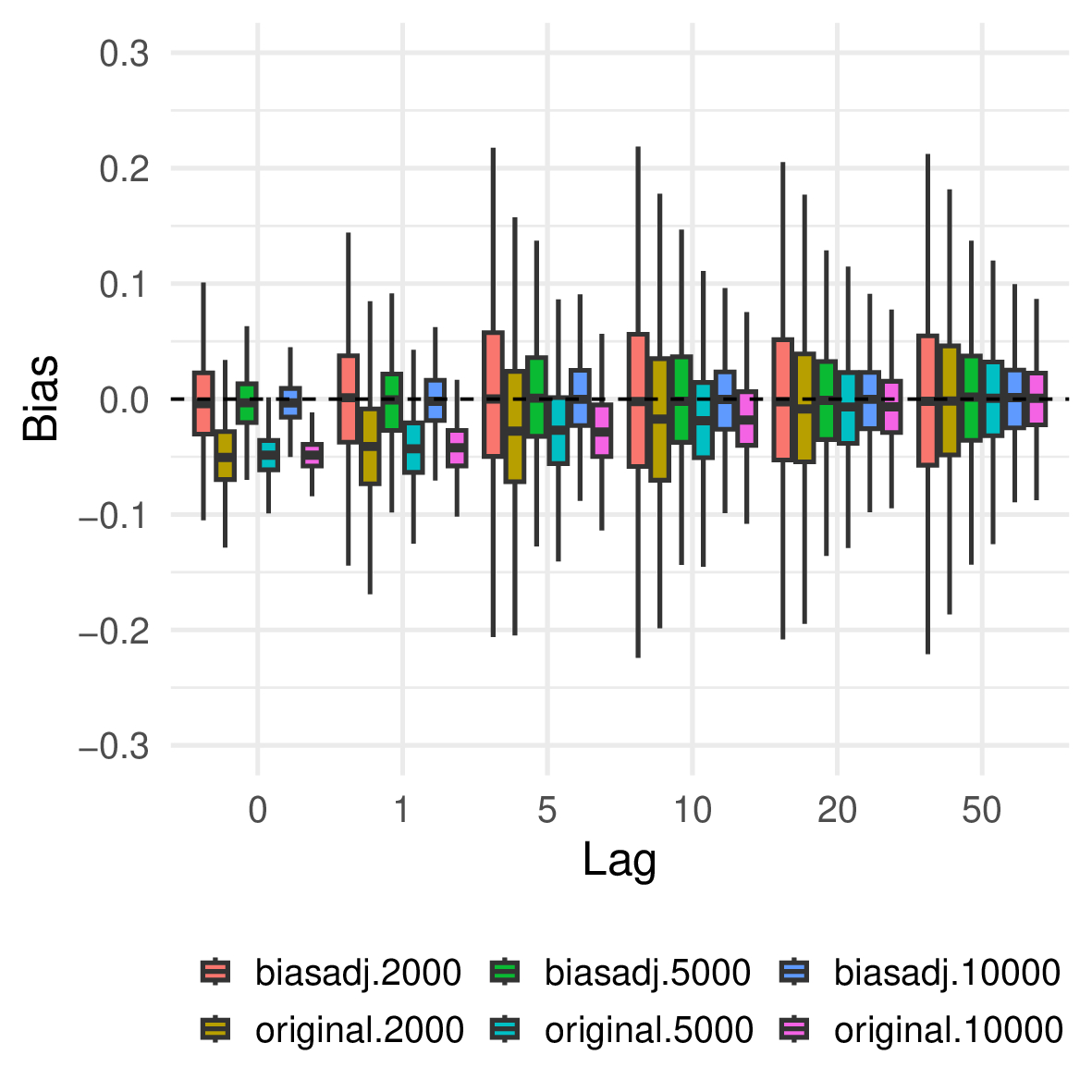}
}
\\
\subfloat[Gaussian Exp $\Delta_n = 0.01$ - Estimates\label{fig:GaussExp0p01}]{%
  \includegraphics[scale=0.3]{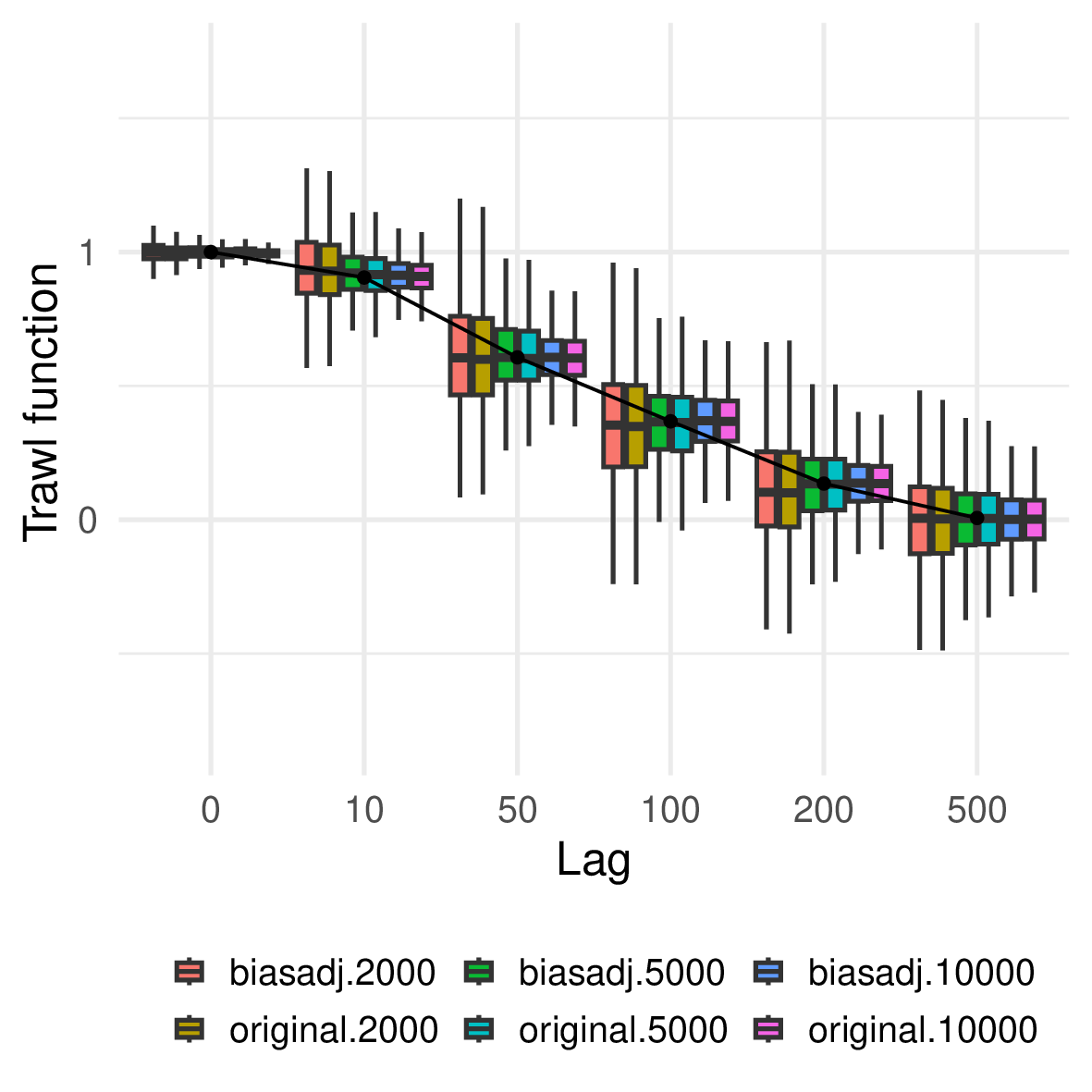}
}
\hfill
\subfloat[Gaussian Exp $\Delta_n = 0.01$ - Bias\label{fig:GaussExp0p01-Bias}]{%
  \includegraphics[scale=0.3]{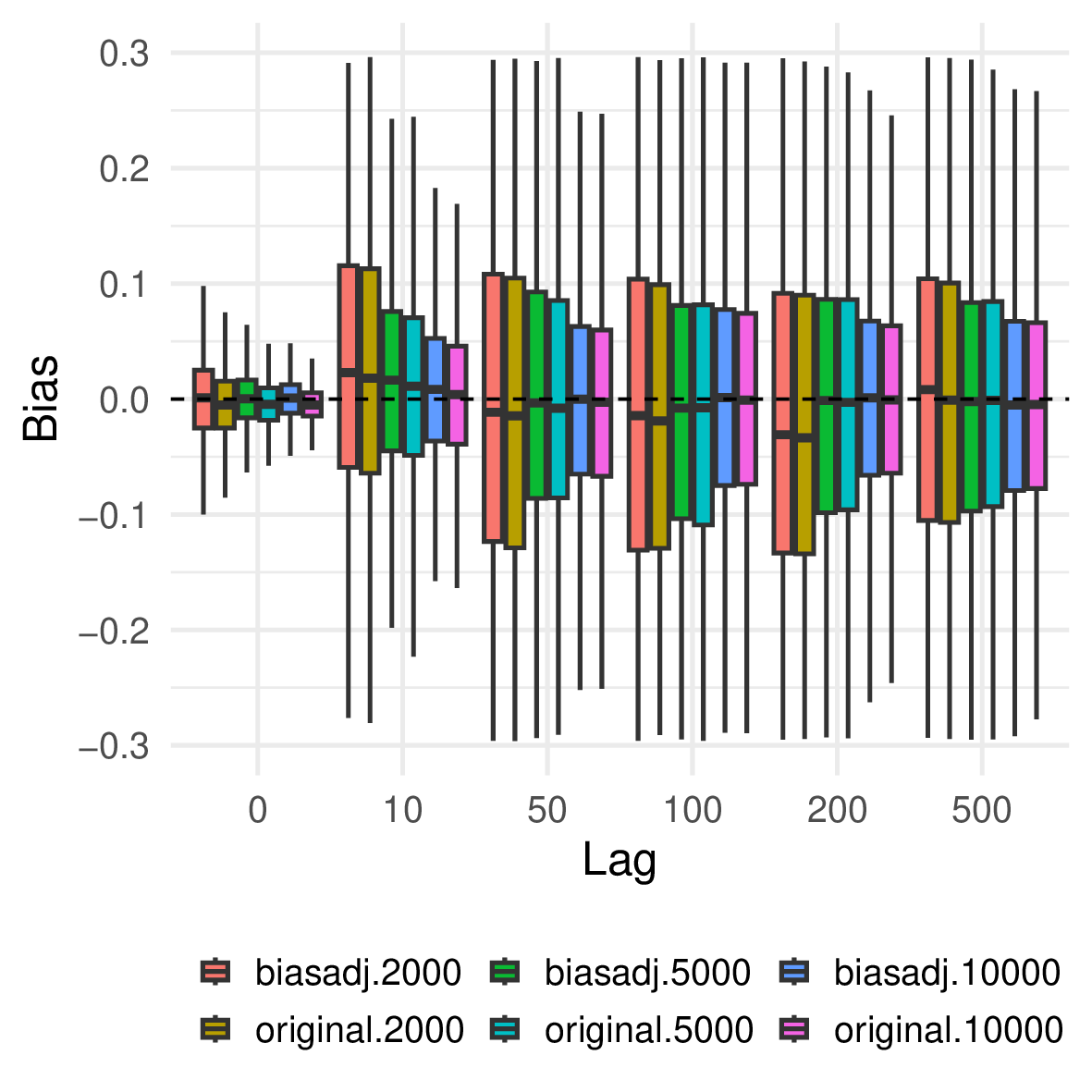}
}
\caption{Gaussian distribution with exponential trawl function results for different $\Delta_n$ values.}
\label{fig:gauss_exp_results}
\end{figure}


\begin{figure}[htbp]
\centering
\captionsetup[subfigure]{aboveskip=-4pt, belowskip=-4pt}
\subfloat[NegBin LM $\Delta_n = 0.5$ - Estimates\label{fig:NegBinLM0p5}]{%
  \includegraphics[scale=0.3]{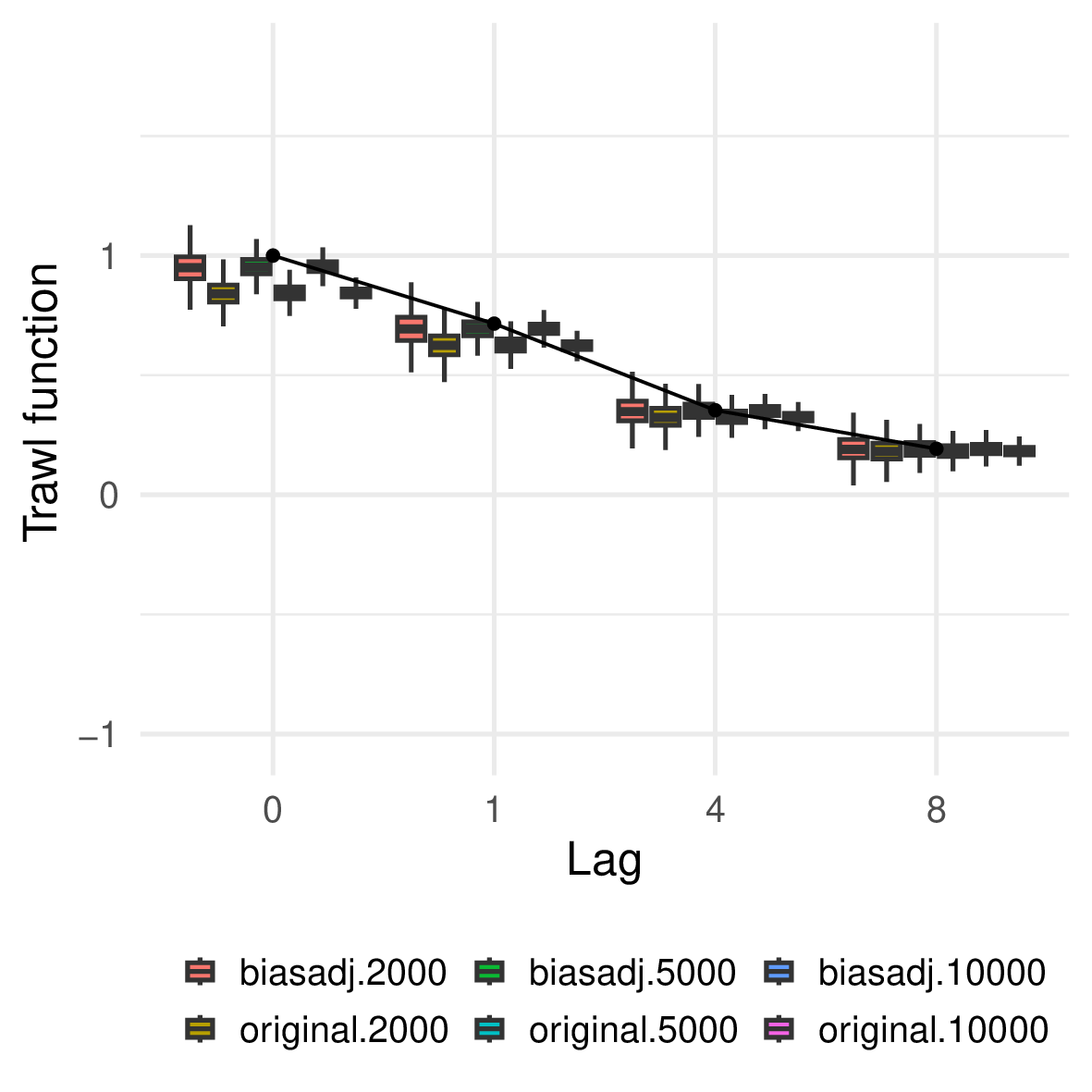}
}
\hfill
\subfloat[NegBin LM $\Delta_n = 0.5$ - Bias\label{fig:NegBinLM0p5-Bias}]{%
  \includegraphics[scale=0.3]{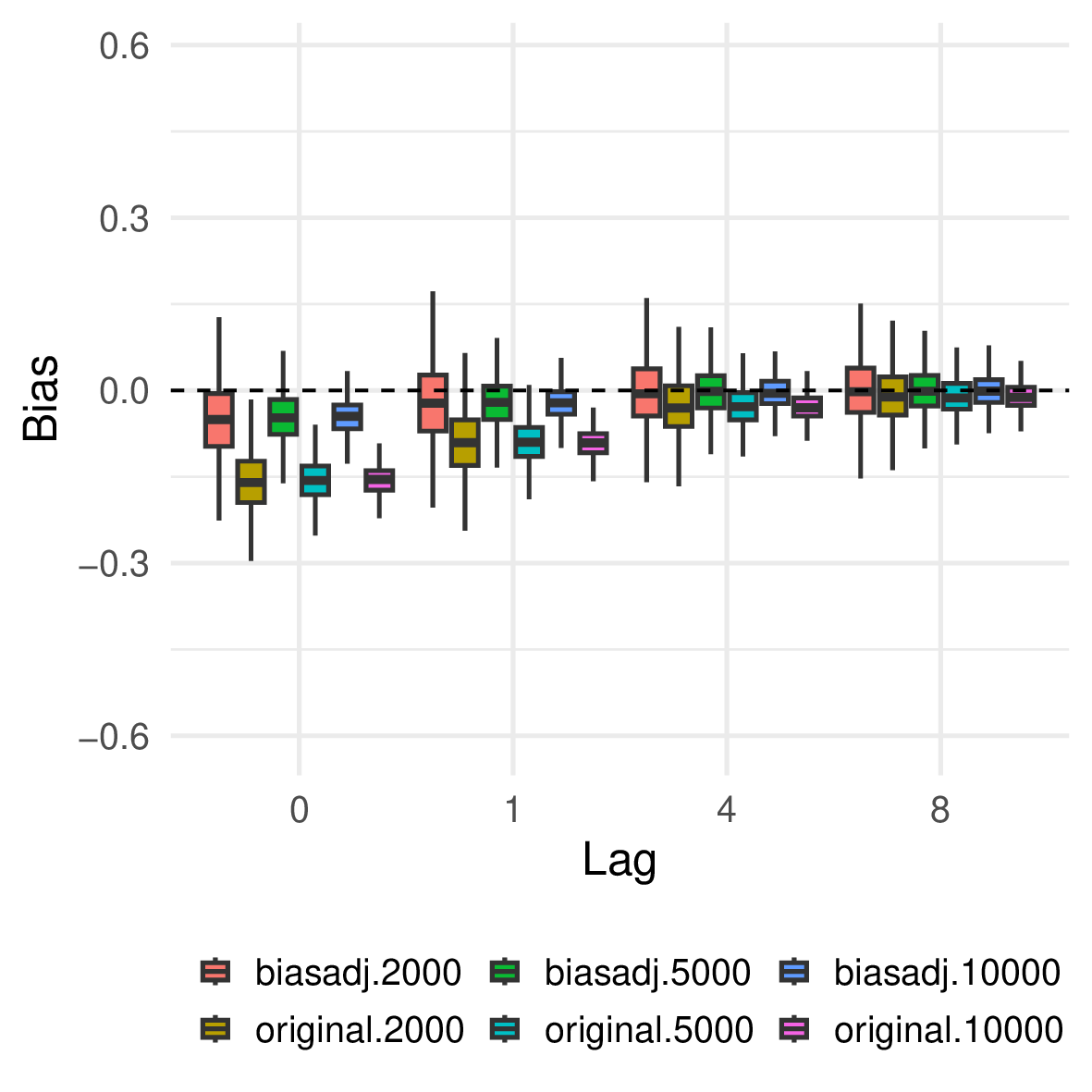}
}
\\
\subfloat[NegBin LM $\Delta_n = 0.1$ - Estimates\label{fig:NegBinLM0p1}]{%
  \includegraphics[scale=0.3]{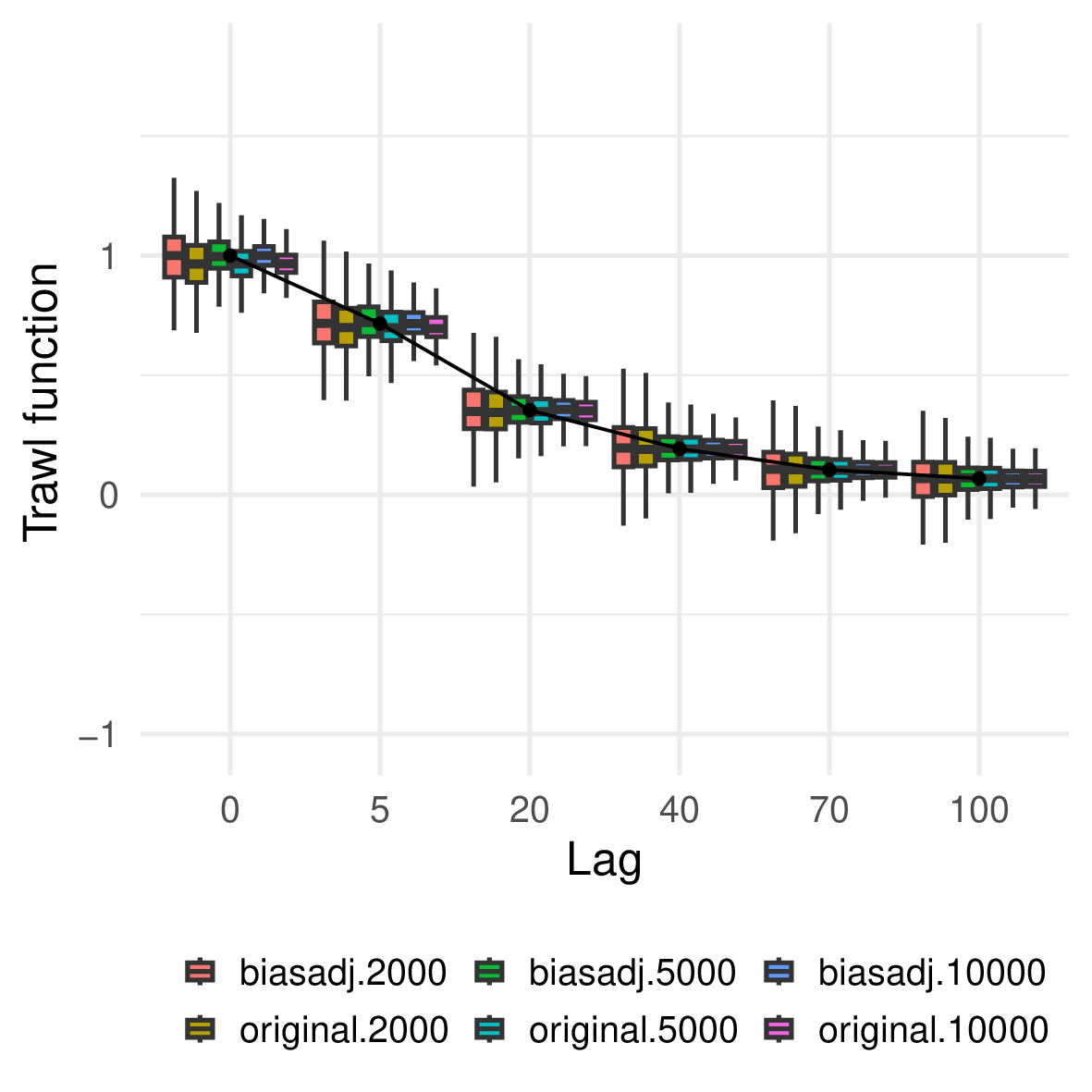}
}
\hfill
\subfloat[NegBin LM $\Delta_n = 0.1$ - Bias\label{fig:NegBinLM0p1-Bias}]{%
  \includegraphics[scale=0.3]{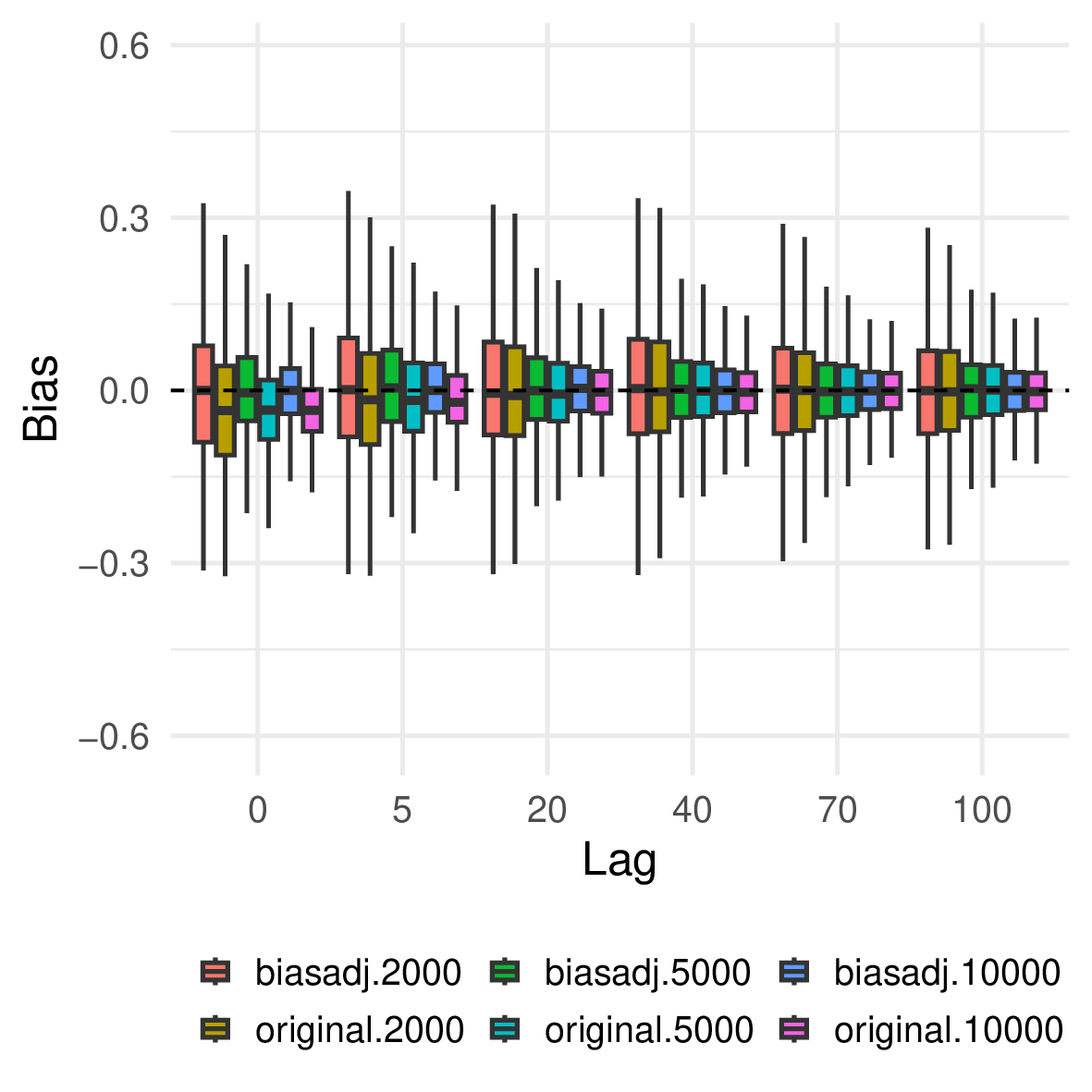}
}
\\
\subfloat[NegBin LM $\Delta_n = 0.01$ - Estimates\label{fig:NegBinLM0p01}]{%
  \includegraphics[scale=0.3]{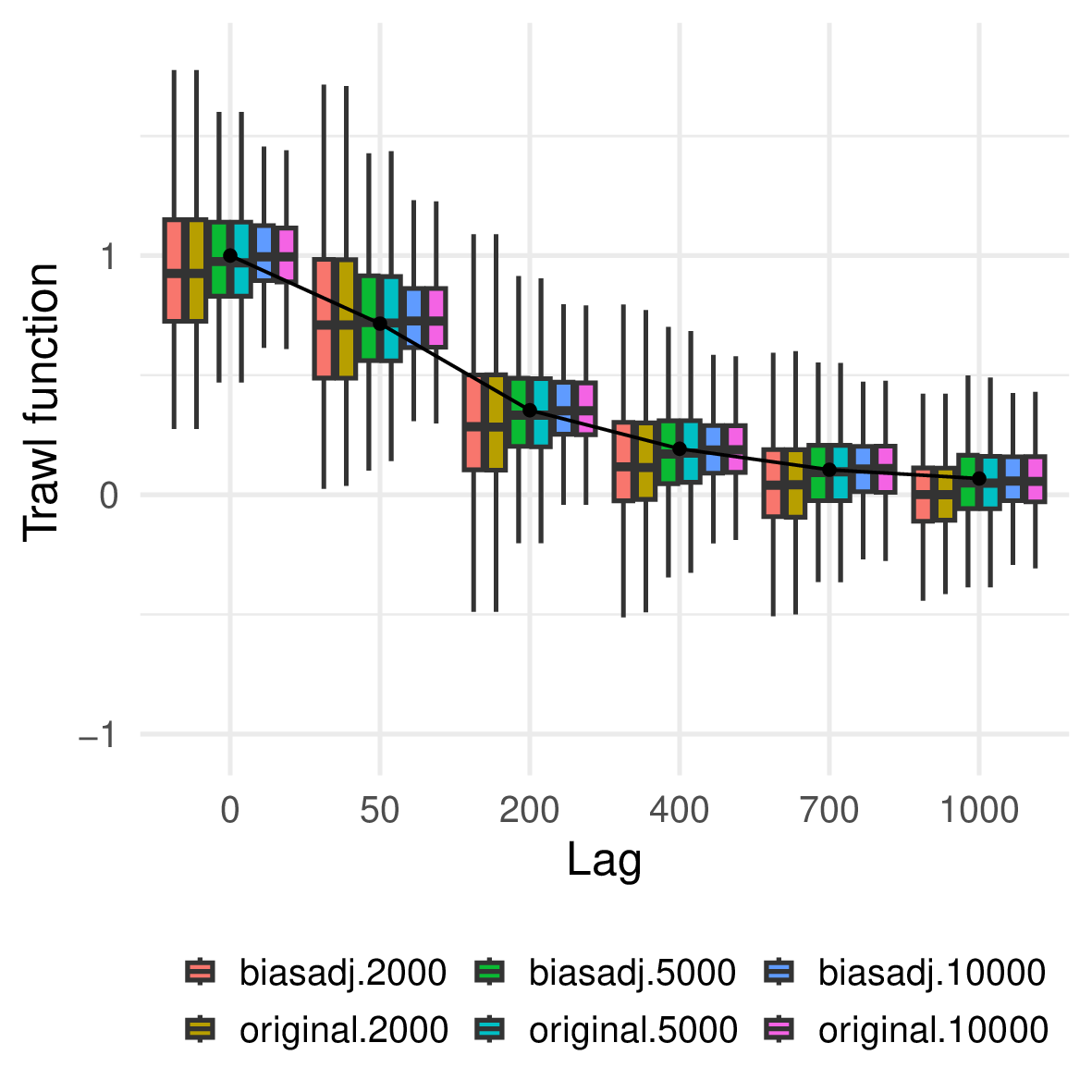}
}
\hfill
\subfloat[NegBin LM $\Delta_n = 0.01$ - Bias\label{fig:NegBinLM0p01-Bias}]{%
  \includegraphics[scale=0.3]{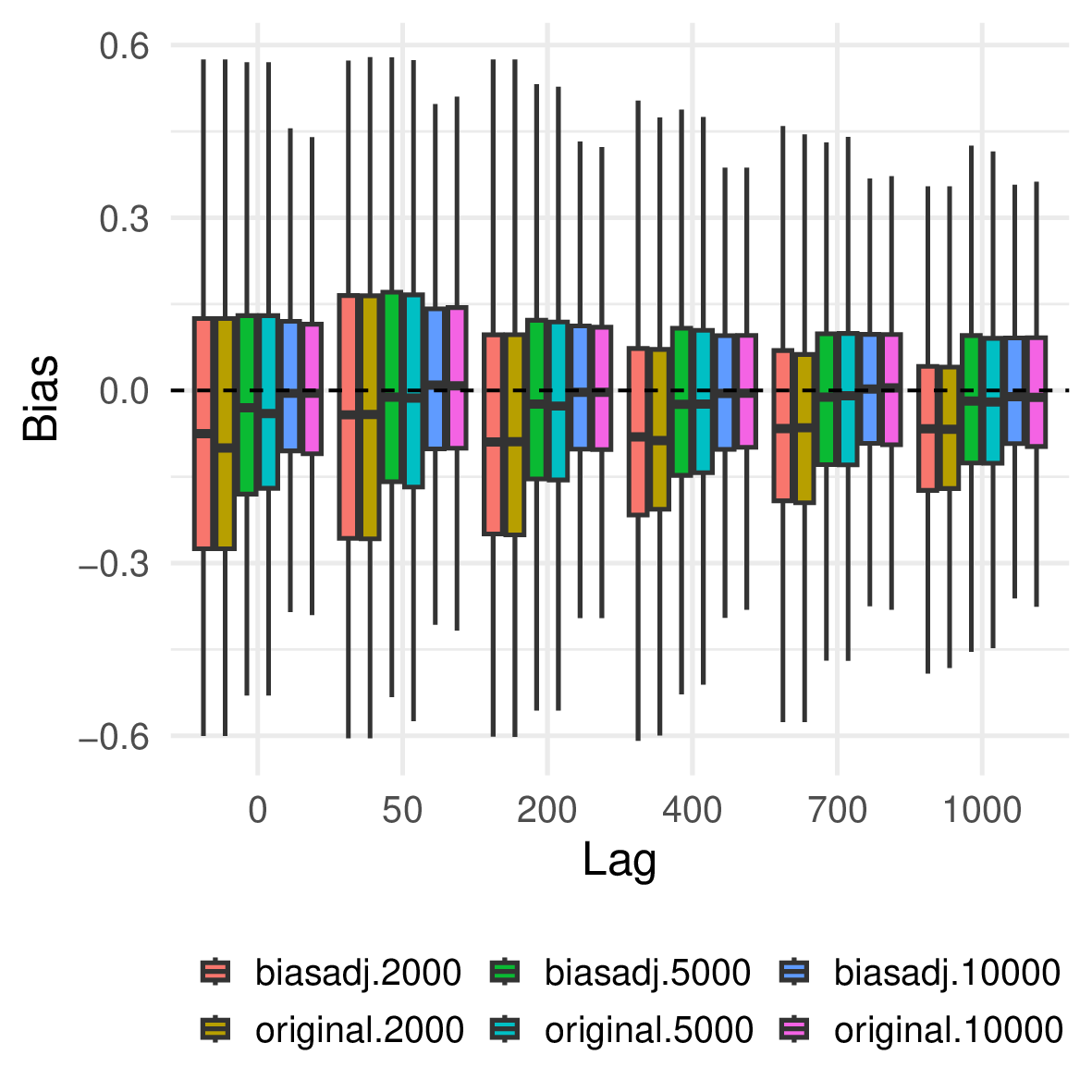}
}
\caption{Negative Binomial distribution with LM trawl function results for different $\Delta_n$ values.}
\label{fig:negbin_lm_results}
\end{figure}

\begin{figure}[htbp]
\centering
\captionsetup[subfigure]{aboveskip=-4pt, belowskip=-4pt}
\subfloat[Gamma LM $\Delta_n = 0.5$ - Estimates\label{fig:GammaLM0p5}]{%
  \includegraphics[scale=0.3]{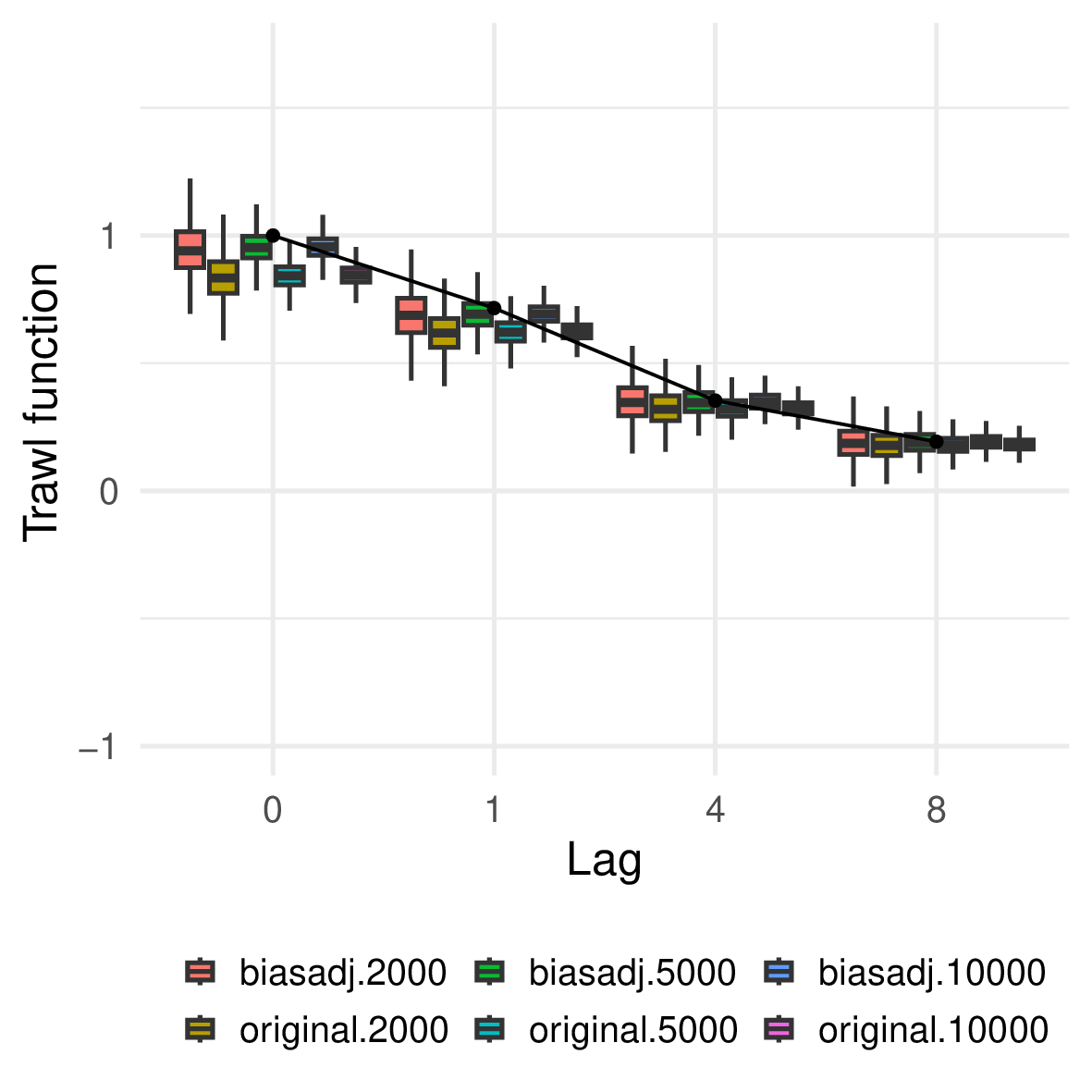}
}
\hfill
\subfloat[Gamma LM $\Delta_n = 0.5$ - Bias\label{fig:GammaLM0p5-Bias}]{%
  \includegraphics[scale=0.3]{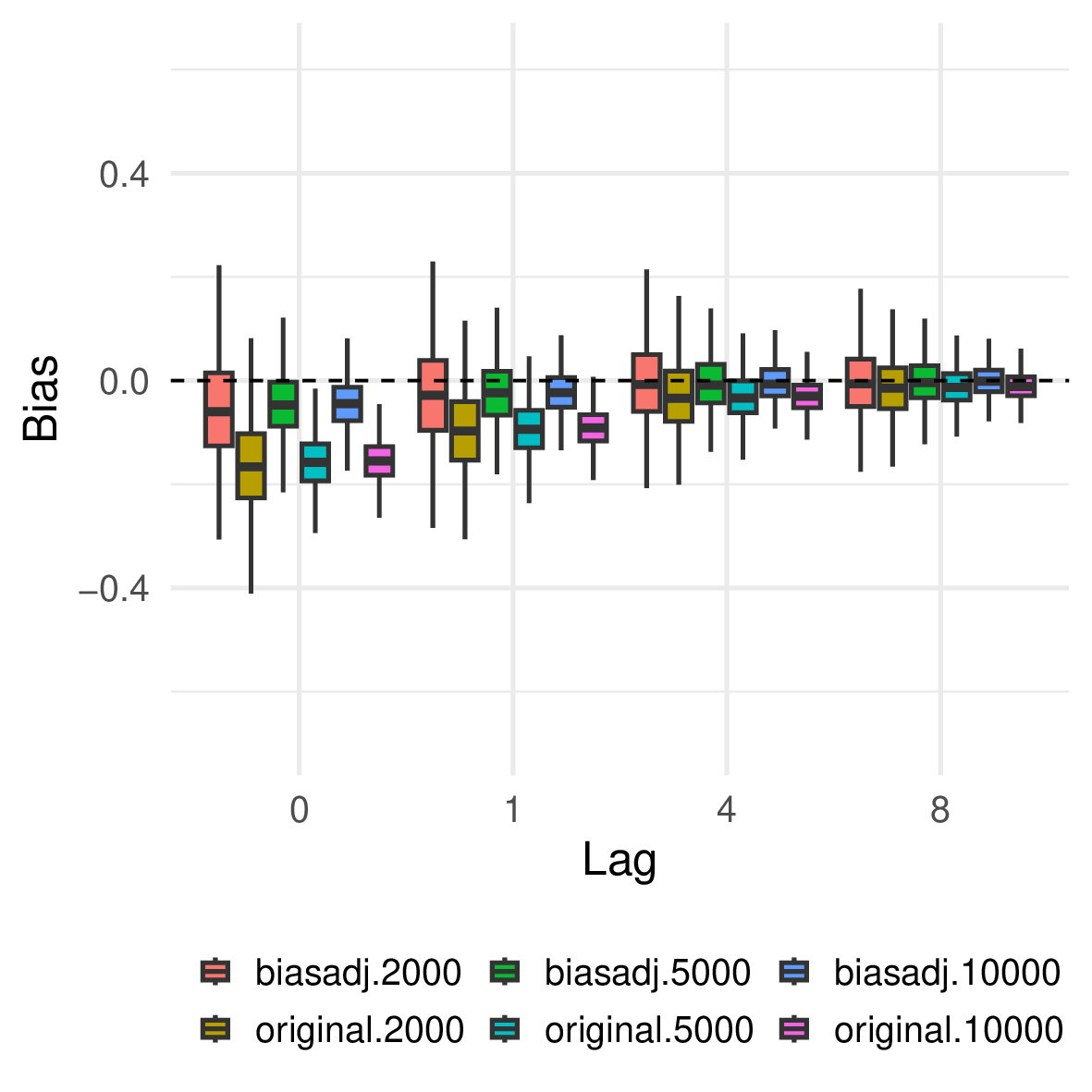}
}
\\
\subfloat[Gamma LM $\Delta_n = 0.1$ - Estimates\label{fig:GammaLM0p1}]{%
  \includegraphics[scale=0.3]{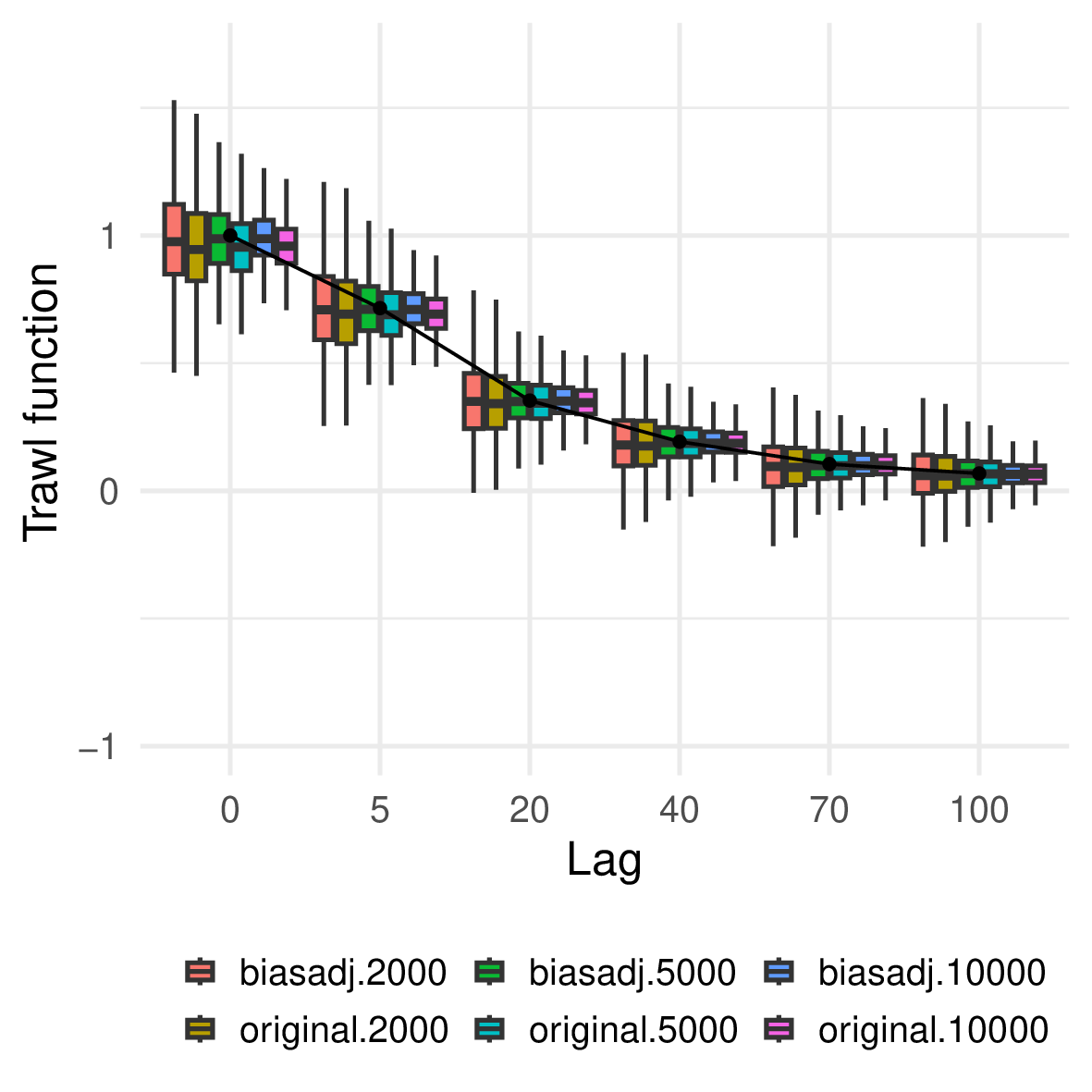}
}
\hfill
\subfloat[Gamma LM $\Delta_n = 0.1$ - Bias\label{fig:GammaLM0p1-Bias}]{%
  \includegraphics[scale=0.3]{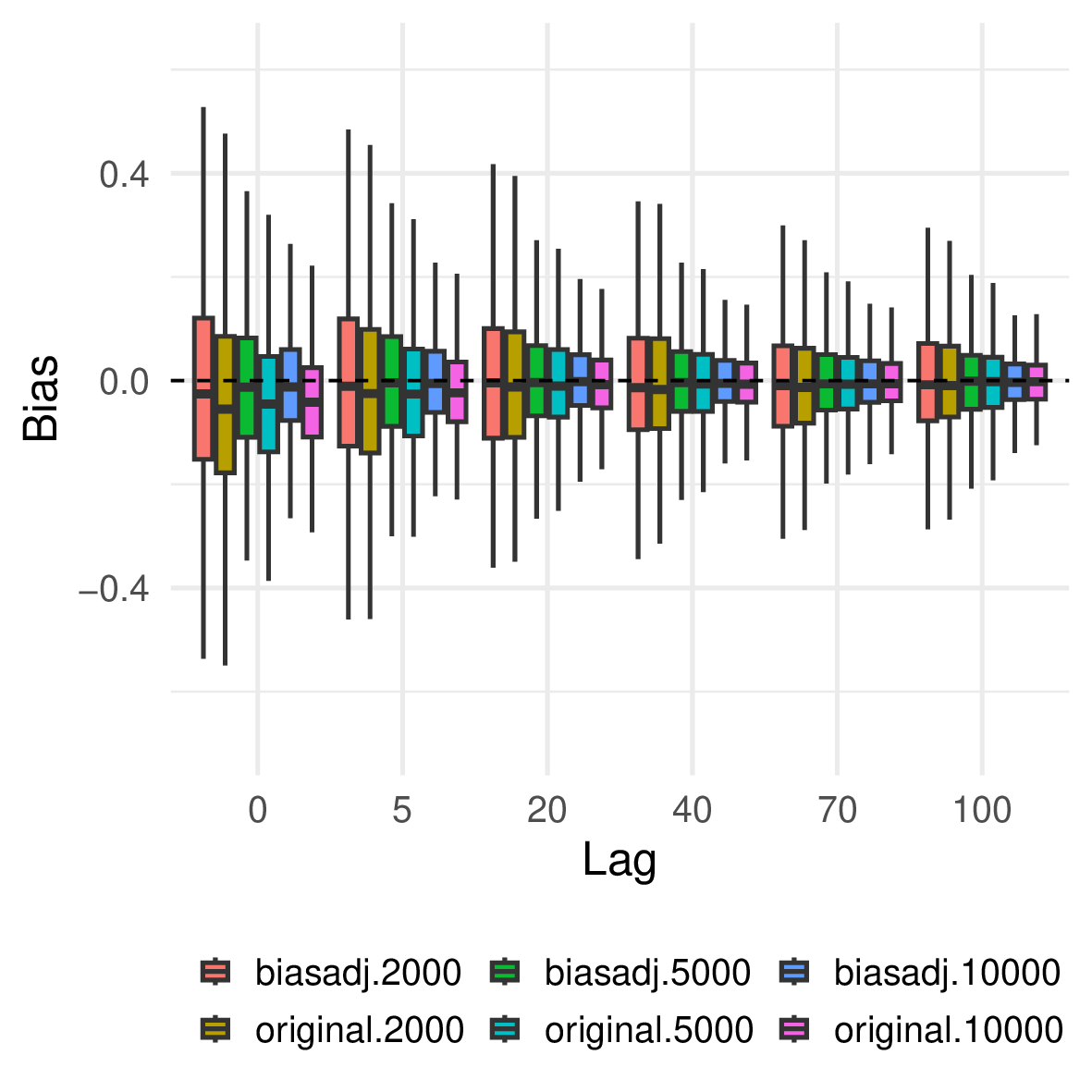}
}
\\
\subfloat[Gamma LM $\Delta_n = 0.01$ - Estimates\label{fig:GammaLM0p01}]{%
  \includegraphics[scale=0.3]{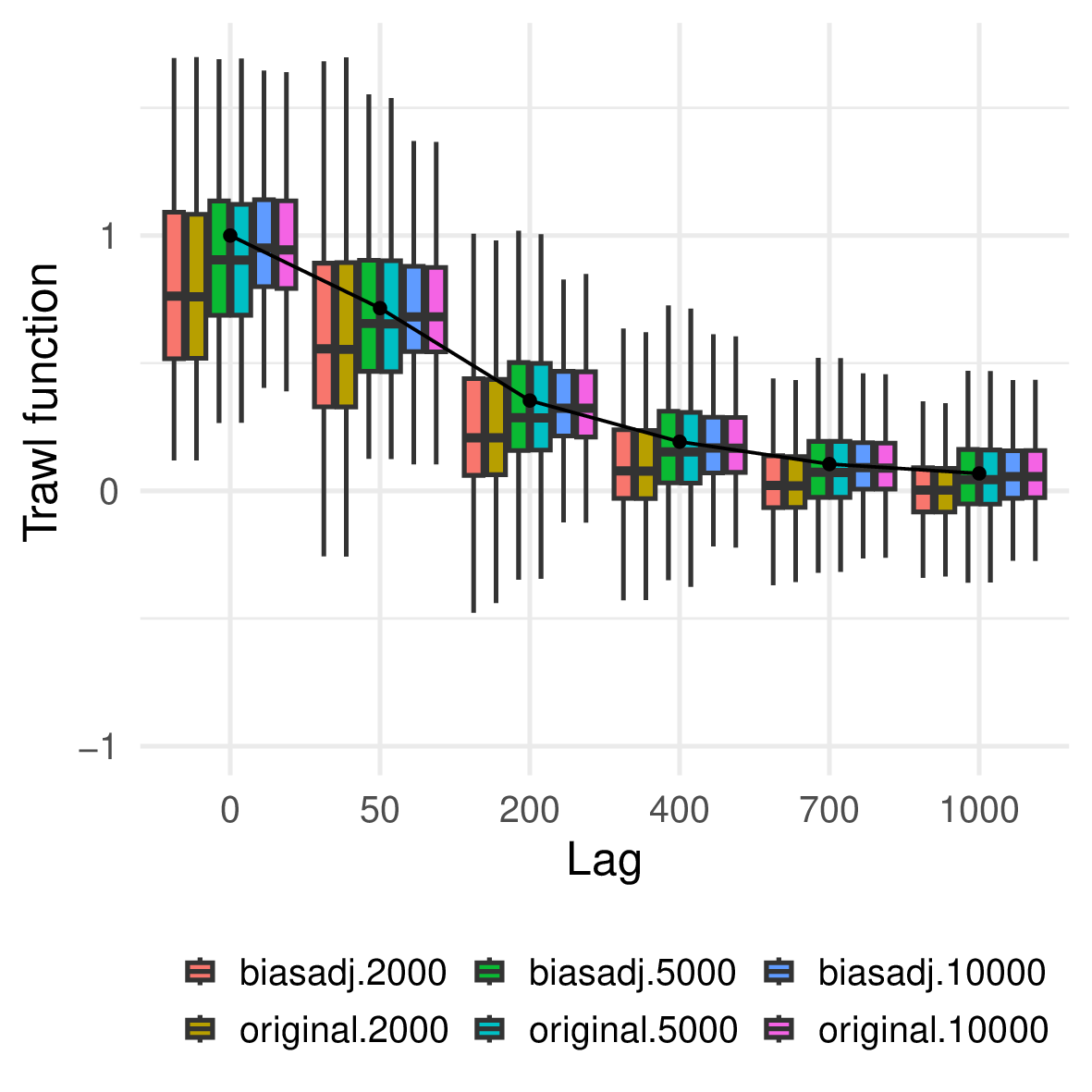}
}
\hfill
\subfloat[Gamma LM $\Delta_n = 0.01$ - Bias\label{fig:GammaLM0p01-Bias}]{%
  \includegraphics[scale=0.3]{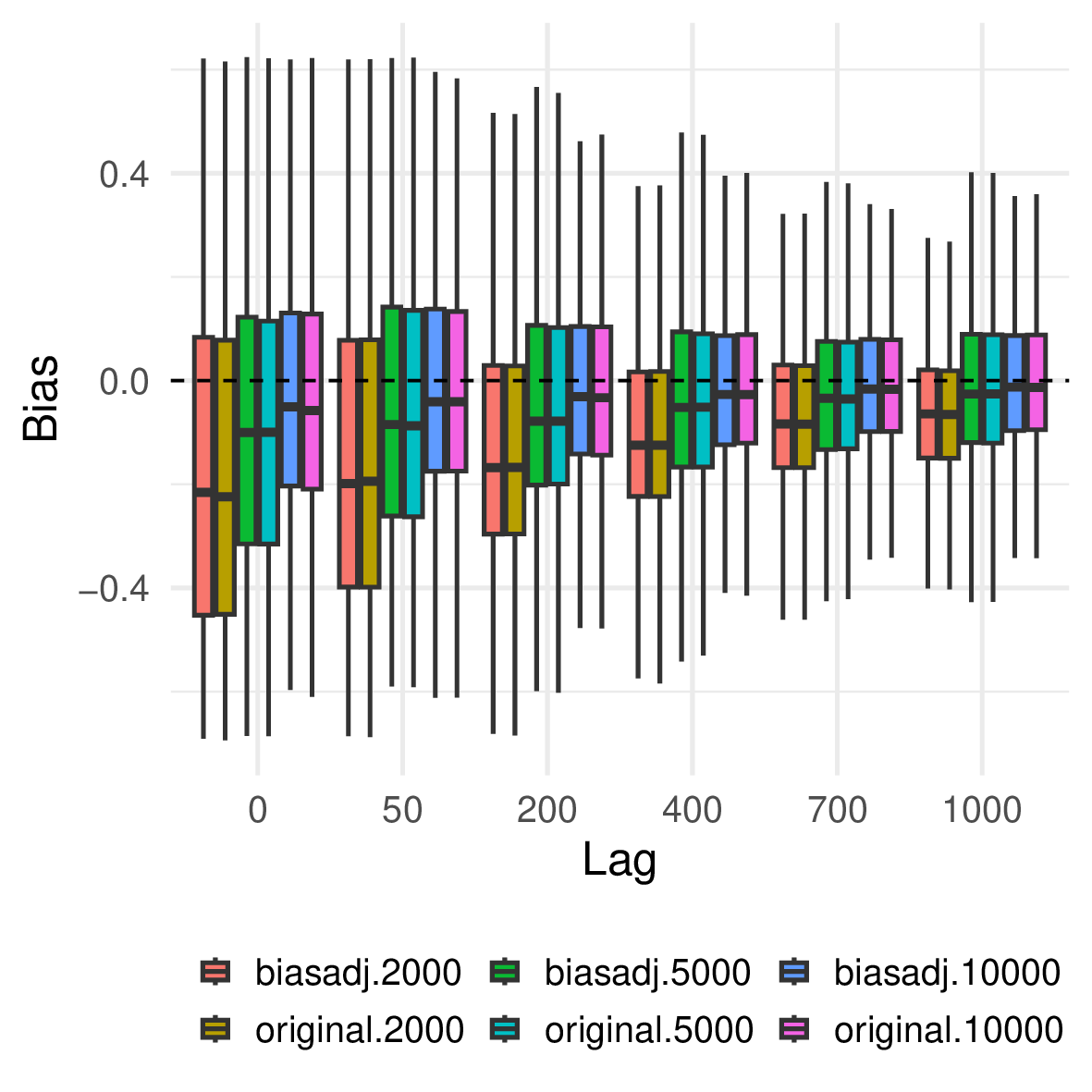}
}
\caption{Gamma distribution with LM trawl function results for different $\Delta_n$ values.}
\label{fig:gamma_lm_results}
\end{figure}

\begin{figure}[htbp]
\centering
\captionsetup[subfigure]{aboveskip=-4pt, belowskip=-4pt}
\subfloat[Gaussian LM $\Delta_n = 0.5$ - Estimates\label{fig:GaussLM0p5}]{%
  \includegraphics[scale=0.3]{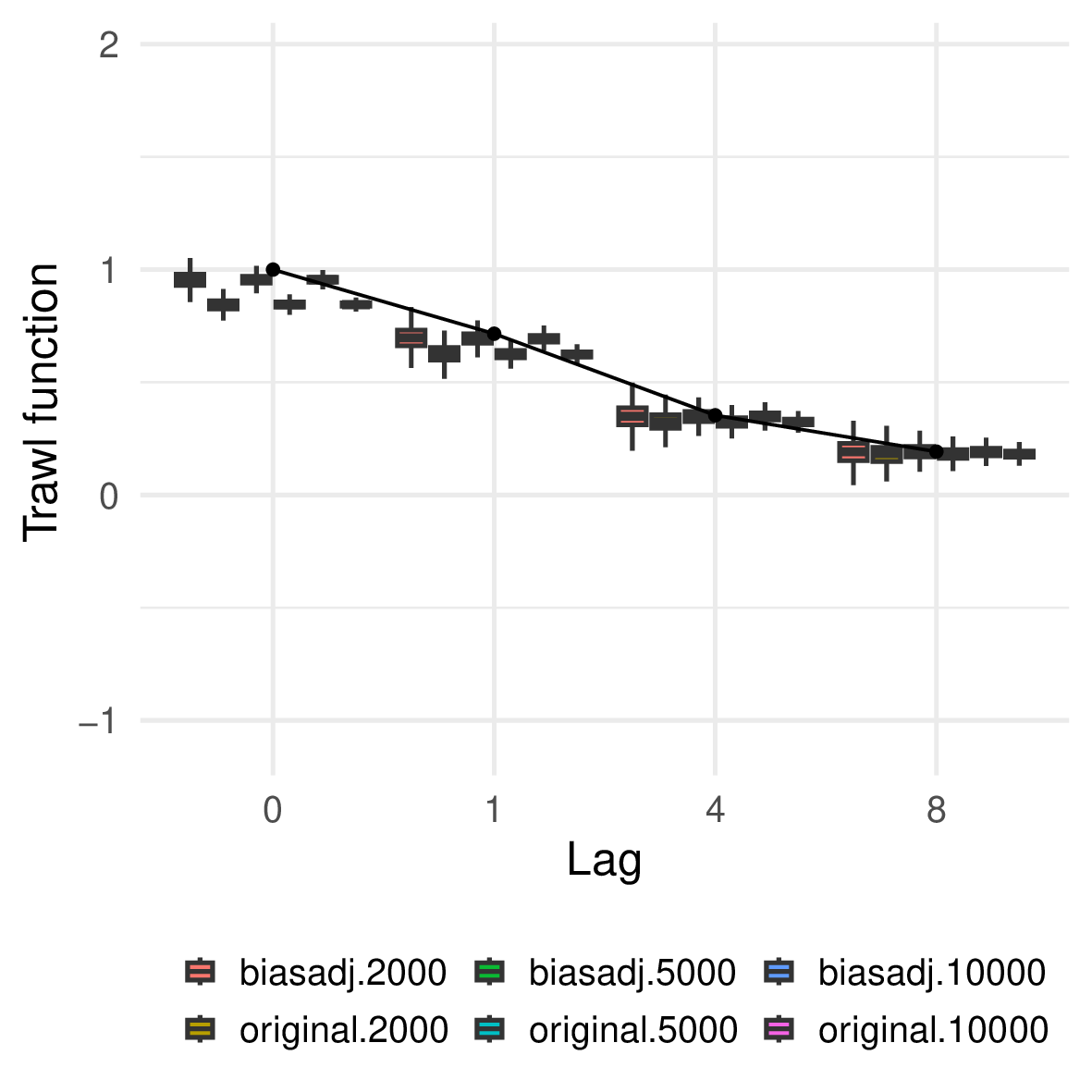}
}
\hfill
\subfloat[Gaussian LM $\Delta_n = 0.5$ - Bias\label{fig:GaussLM0p5-Bias}]{%
  \includegraphics[scale=0.3]{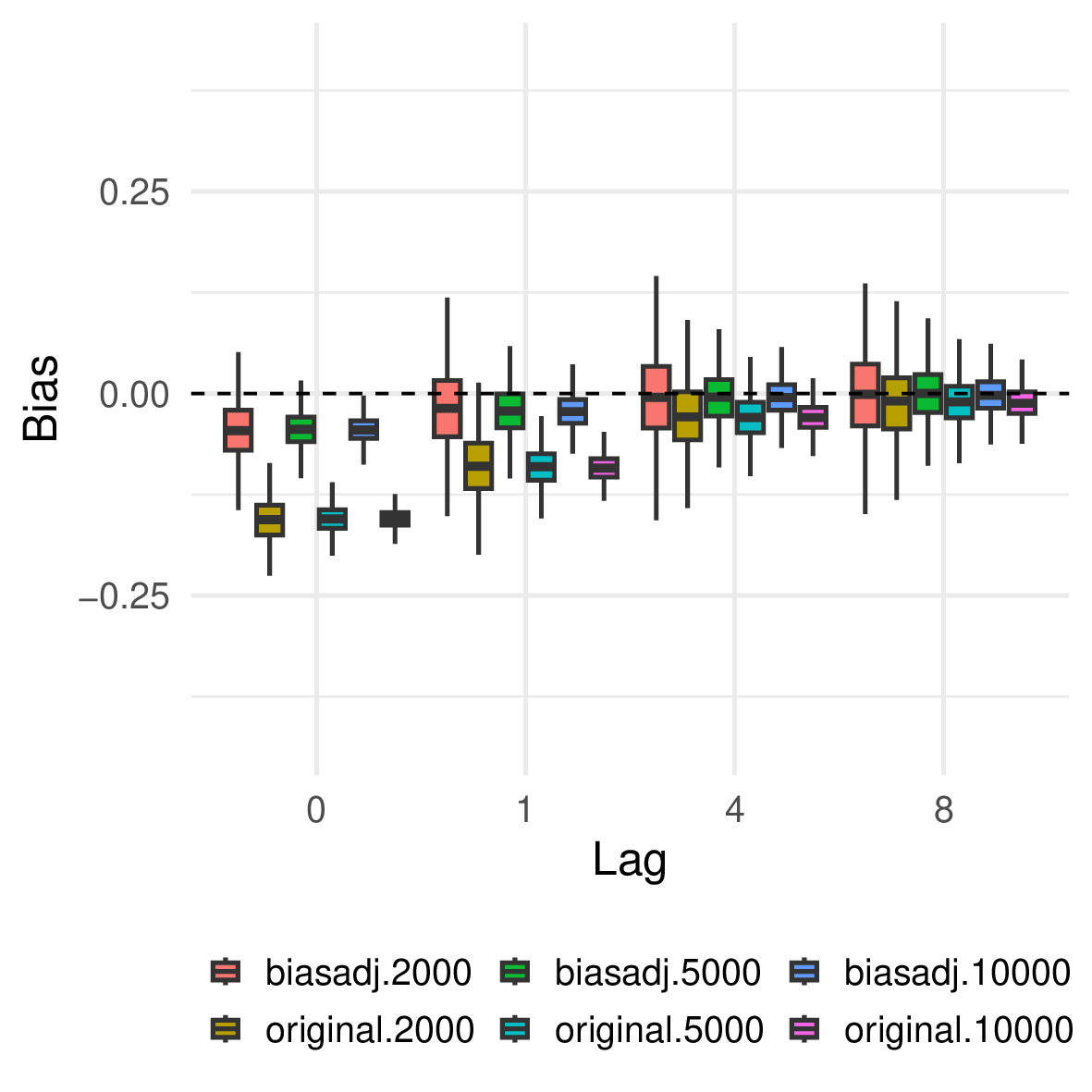}
}
\\
\subfloat[Gaussian LM $\Delta_n = 0.1$ - Estimates\label{fig:GaussLM0p1}]{%
  \includegraphics[scale=0.3]{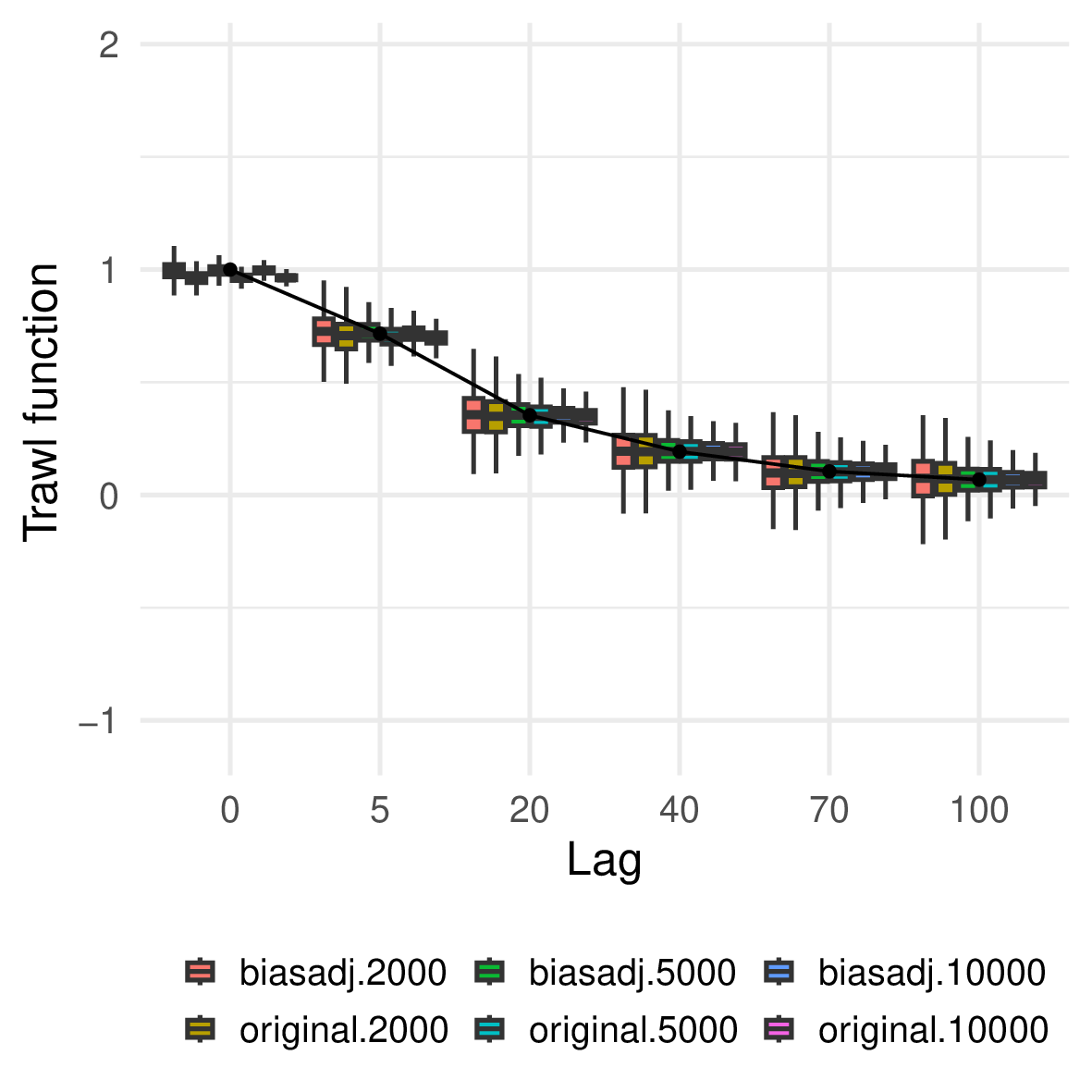}
}
\hfill
\subfloat[Gaussian LM $\Delta_n = 0.1$ - Bias\label{fig:GaussLM0p1-Bias}]{%
  \includegraphics[scale=0.3]{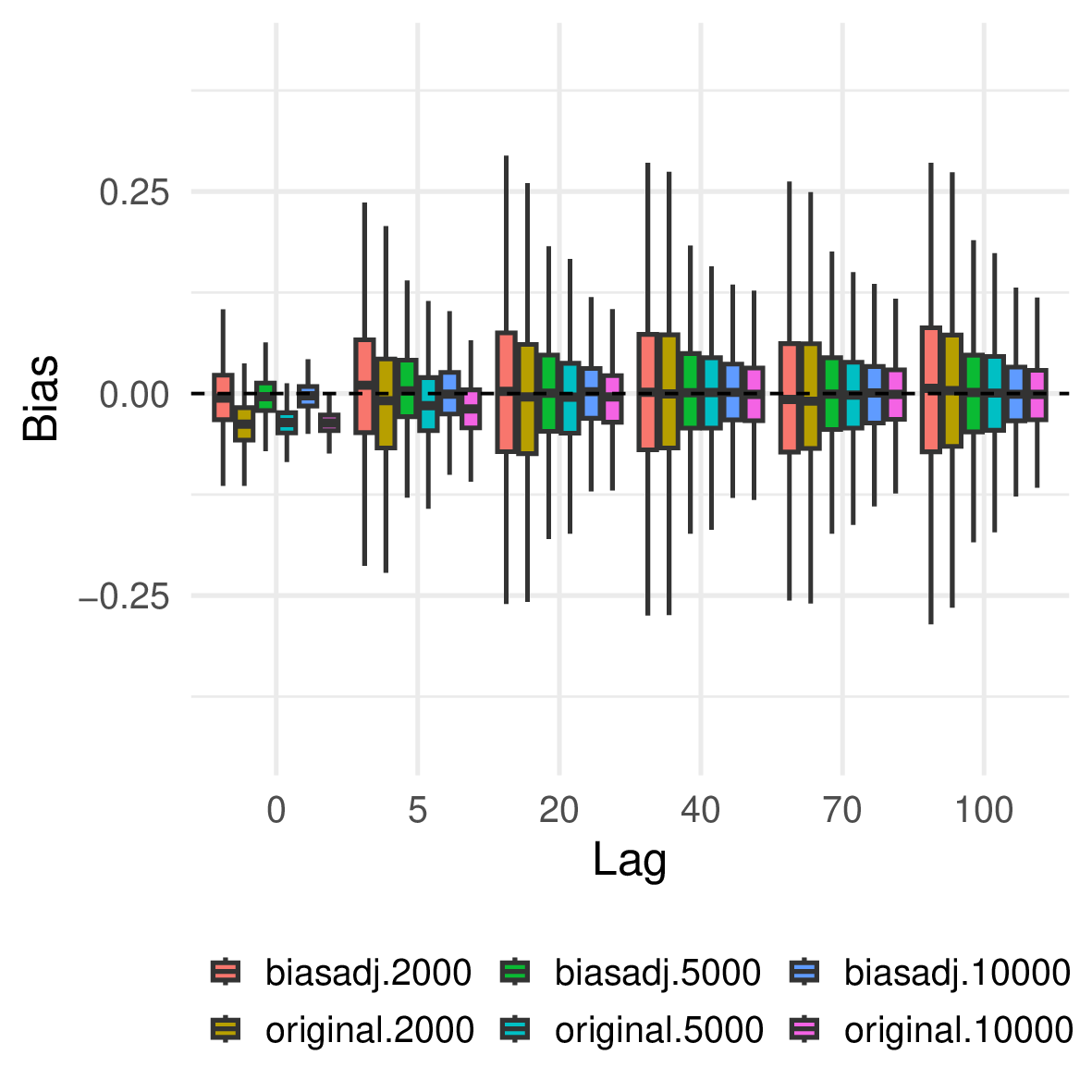}
}
\\
\subfloat[Gaussian LM $\Delta_n = 0.01$ - Estimates\label{fig:GaussLM0p01}]{%
  \includegraphics[scale=0.3]{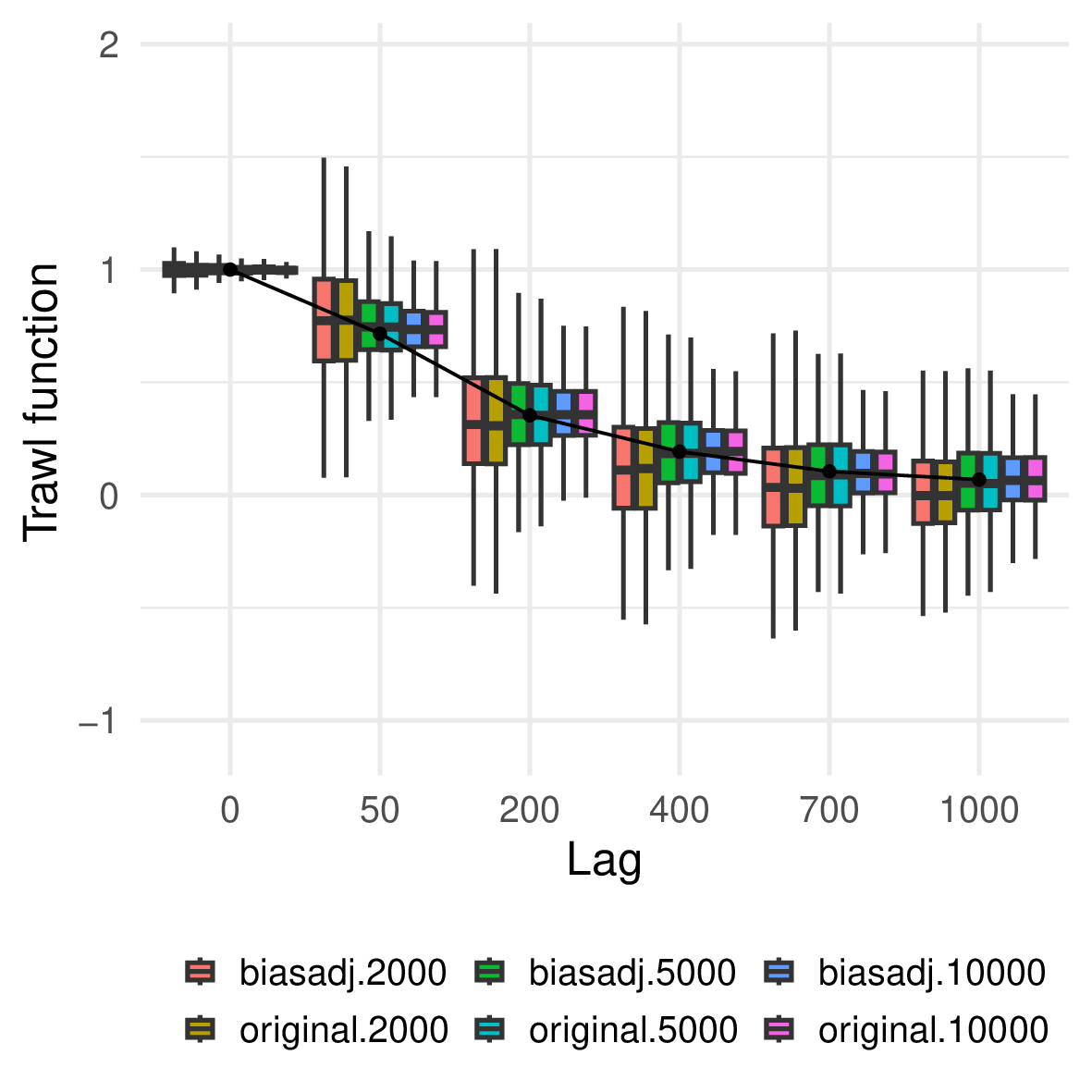}
}
\hfill
\subfloat[Gaussian LM $\Delta_n = 0.01$ - Bias\label{fig:GaussLM0p01-Bias}]{%
  \includegraphics[scale=0.3]{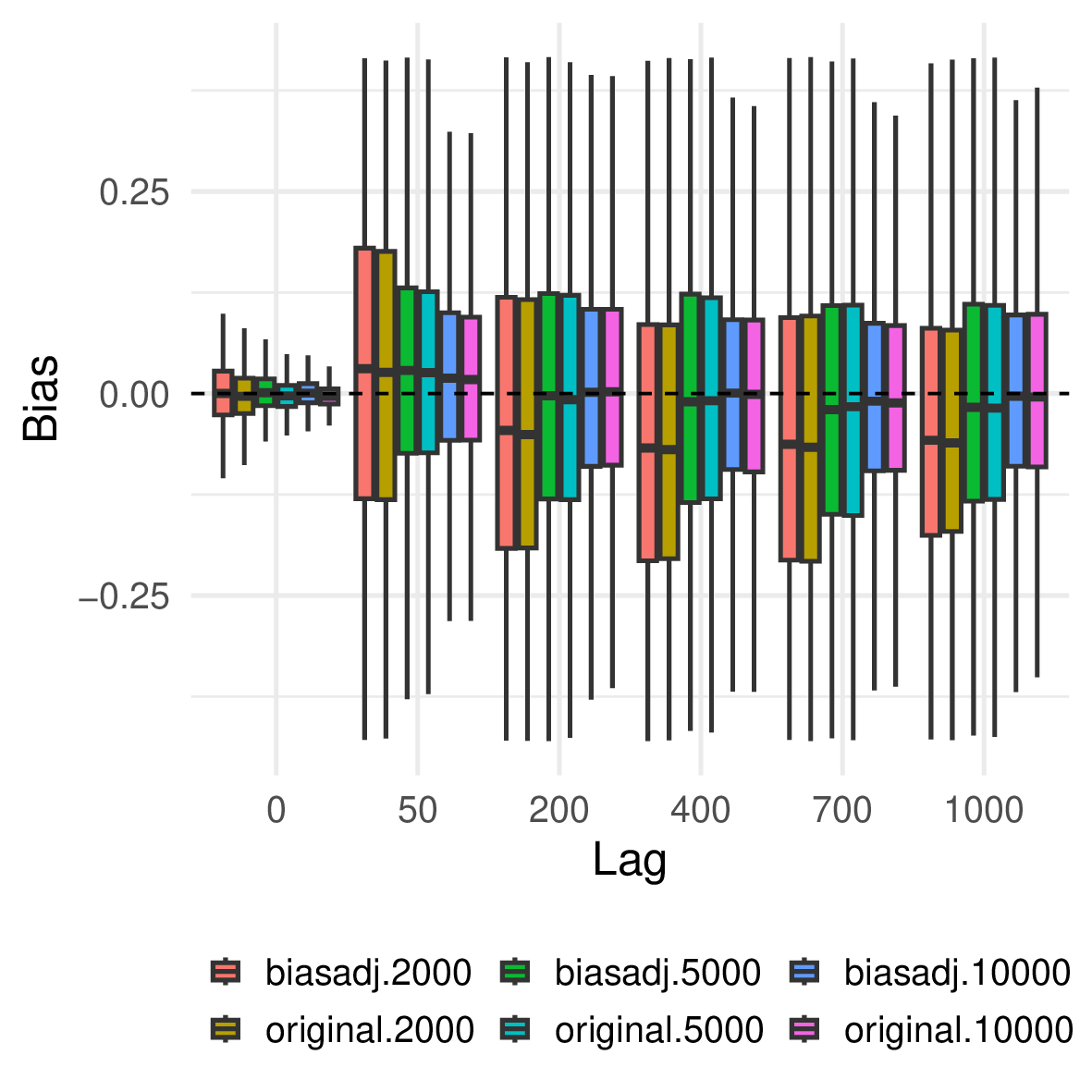}
}
\caption{Gaussian distribution with LM trawl function results for different $\Delta_n$ values.}
\label{fig:gauss_lm_results}
\end{figure}
\clearpage

\newlength{\origtopmargin}
\setlength{\origtopmargin}{\topmargin}

\section{Simulation study: Asymptotic normality results}\label{asec:asymGauss}

Next, we investigate the finite sample behaviour of the central limit theorems.
We consider three statistics and report their mean, standard deviation and coverage probabilities at levels 90\%, 95\% and 99\% based on 1000 Monte Carlo samples:
\begin{description}
	\item[Infeasible statistic:] The infeasible statistic is given by 
	$$T^{IF}(t):=\sqrt{n\Delta_{n}}\frac{\left(\hat{a}(t)-a(t)\right)}{\sqrt{\sigma_a^2(t)}},$$
	where the true (theoretical) variance $\sigma_a^2(t)$ is used. 
	\item[Feasible statistic without bias correction:]
	The feasible statistic without bias correction is given by 
	$$T^{F}:=\sqrt{n\Delta_{n}}\frac{\left(\hat{a}(t)-a(t)\right)}{\sqrt{\hat \sigma_a^2(t)}},$$
	where $\hat \sigma_a^2(t)$ is the estimator proposed in the main article. 
	\item[Feasible statistic with bias correction:]
	The feasible statistic with bias correction is given by
	$$T^{F}_{\mathrm{bias \, corrected}}:=\sqrt{n\Delta_{n}}\frac{\left(\hat{a}(t)-a(t)-\frac{1}{2}\Delta_n \hat a'(t)\right)}{\sqrt{\hat \sigma_a^2(t)}},$$
	where the estimator $\hat a$ proposed in the main article is used. 
\end{description}
Based on our asymptotic theory, we expect the three statistics to follow a standard normal distribution for large enough $n$ and small enough $\Delta_n$.

There is one case, however, which needs to be treated separately: In the purely Gaussian case with $t=0$, we multiply the test statistics by $\sqrt{3}$ to reflect the fact that they would otherwise follow an $N(0, 1/3)$ distribution. Note also that in that particular case, there is no infeasible statistic corresponding to the central limit theorems we have presented in the main article and hence the cell in the table is left empty. 

\begin{remark}
	In the purely Gaussian case, when $t=0$, we could consider the following infeasible statistic
	\begin{align*}
		\sqrt{n} \frac{ (\hat a(0)-a(0))}{\sqrt{2 (a(0))^2}},
	\end{align*}
	which asymptotically follows a standard normal distribution under the assumptions of Theorem  \ref{thmCLTt0Gaussian} (Part 1.) when $\mu=0$ and $\mu_0=0$.
\end{remark}

\subsection{Exponential trawl function}
We start with the negative binomial trawl process with an exponential trawl function; see
Table \ref{tab:NB-Exp-fixedt}.
For the larger value of $\Delta_n=0.1$, we observe a noticeable finite sample bias for all three statistics, in particular, when $t$ is small, but at the same time, the coverage probabilities are reasonable. The best results are obtained for the bias-corrected statistic. We note, however, that the estimate of the asymptotic variance gets less reliable when $t$ increases.
When moving to the case of the smaller $\Delta_n=0.01$, the finite sample bias is reduced. The problem that the estimate of the asymptotic variance is slightly less reliable for increasing $t$ persists. 
We have to keep in mind that, when decreasing $\Delta_n$, while keeping $n$ fixed as in the set-up here, $\Delta_n n$ decreases, so we might not necessarily get an improvement in the finite sample performance since the performance will depend on the balance between the infill and long-span asymptotic regime. 

When moving to the case of a Gamma marginal distribution with Exponential trawl, see Table 
\ref{tab:Gamma-Exp-fixedt},
the results are similar, but generally slightly worse than in the negative binomial case. 
Here we note that the bias correction seems to be more useful for the case when $\Delta_n=0.1$ than for $\Delta_n=0.01$.
Also, when $\Delta_n n$ is relatively small (e.g.~for $\Delta_n=0.1, n=2000$ or $\Delta_n=0.01, n \in \{2000, 5000\}$) the finite sample performance of the test statistic for large $t\in \{5, 10\}$ is significantly worse than for small $t$. 

Moving on to the Gaussian case, see Table 
\ref{tab:Gaussian-Exp-fixedt}, we find again that the bias-corrected statistic performs better, in particular for small values of $t$. We also note that the case $t=0$ appears generally tricky, where the best performance out of the cases considered are the cases $\Delta_n=0.1, n=2000$ and $\Delta_n=0.01, n=2000$. It is important to remember that this finding is not surprising given the additional (stronger) assumptions we have in the central limit theorem for this particular case.

\input{Tables/All_Results_NB_Exp_fixedt_v4.tex}

\input{Tables/All_Results_Gamma_Exp_fixedt_v4.tex}

\input{Tables/All_Results_Gaussian_Exp_fixedt_v4.tex}

\subsection{SupGamma trawl function}
Let us next study the finite sample performance in the long-memory regime, see Tables  
\ref{tab:NB-LM-fixedt},
\ref{tab:Gamma-LM-fixedt} and 
\ref{tab:Gaussian-LM-fixedt}. 
We recall that we found that, in this case, there was a noticeable estimation bias for all three marginal distributions when $t=0$. These results are reflected in our findings for the three statistics.
When $t=0$ and $\Delta_n=0.1$, the finite sample results are very poor. However, when $\Delta_n=0.01$, the results are generally much better and improve significantly as $n$ increases.
We also recall the observation from the exponential case that, when $t$ is increasing, the estimates of the mean and standard deviation of the test statistic deviate greatly from 0 and 1, respectively, when $\Delta_n n$ is comparably small. We observe the same pattern in the supGamma setting, where the findings are even more pronounced. We conclude here that for larger values of $t$, the most reliable performance of the test statistics is in the setting when $\Delta_n=0.1$ and $n \in \{5000, 10000\}$. 

\input{Tables/All_Results_NB_LM_fixedt_v4.tex}
\input{Tables/All_Results_Gamma_LM_fixedt_v4.tex}
\input{Tables/All_Results_Gaussian_LM_fixedt_v4.tex}

\FloatBarrier \clearpage
\section{Simulation study: Consistency of the slice (ratio) estimation}\label{asec:slices}
As we have seen in the main article, for forecast purposes it is of interest to estimate slices of the trawl set. Hence, in our simulation study, we investigate the finite sample performance of our proposed slice estimators. 
We concentrate on the case that is relevant for our empirical study and choose $\Delta_n=0.1$ and $n=5000$. 
Recall that, in the main article, we proposed the following estimators: 
\begin{align*}
	\widehat{\Leb(A)}&=	\sum_{l=0}^{n-1}\hat a(l\Delta_n)\Delta_n, 
	\\
	\widehat{	\Leb(A\cap A_h)} &=\sum_{\substack{l\in \{0, \ldots, n-1\}:\\ l\Delta_n \geq h}} \hat a(l\Delta_n)\Delta_n,\\
	\\
	\widehat{\Leb(A\setminus A_h)}&=\sum_{\substack{l\in \{0, \ldots, n-1\}:\\ l\Delta_n < h}} \hat a(l\Delta_n)\Delta_n.
	\end{align*}
We note that the estimates could become negative, so in our simulation experiment, we made the following adjustment to the estimators:

\begin{align*}
	T_h&:=
	\sum_{\substack{l\in \{0, \ldots, n-1\}:\\ l\Delta_n \geq h}} \hat a(l\Delta_n)\Delta_n\\
	\widehat{	\Leb(A\cap A_h)} &=
	\left\{ \begin{array}{ll}
		T_h,& \mathrm{if} \; T_h< 	\widehat{\Leb(A)}\\
		\widehat{\Leb(A)}, & \mathrm{otherwise}.
	\end{array}\right. \\
	\\
	\widehat{\Leb(A\setminus A_h)}&=\widehat{\Leb(A)}-T_h.
\end{align*}
Based on these estimators, we  estimate 
$\mathrm{\Leb}(A)$, $\mathrm{\Leb}(A\cap A_h)$, 
$\mathrm{\Leb}(A \setminus A_h)$, $\mathrm{\Leb}(A\cap A_h)/\mathrm{\Leb}(A)$ and  $\mathrm{\Leb}(A \setminus A_h)/\mathrm{\Leb}(A)$. 
For each of the five estimators, we report their empirical mean, bias and standard deviation based on 1000 Monte Carlo samples. 
We observe that the slices and corresponding ratios can be estimated 
with high precision for all three marginal distributions and for both the short-memory and the long-memory settings, see Tables \ref{tab:slices-exp} and \ref{tab:slices-lm}, respectively. 
When estimating the slices, we also implemented the estimator using a bias correction, but find that the bias-corrected estimator performs worse when estimating the slices and ratios.

As an alternative estimator, we use
\begin{align*}
	\widehat{\Leb(A)}&=\widehat{\Var(X)}
	\\
	S_h & := \widehat{\mathrm{Cov}(X_0, X_h)},\\
	\widehat{	\Leb(A\cap A_h)} &=\left\{ \begin{array}{ll}
		S_h,& \mathrm{if} \; S_h< 	\widehat{\Leb(A)}\\
		\widehat{\Leb(A)}, & \mathrm{otherwise}.
	\end{array}\right. \\
	\\
	\widehat{\Leb(A\setminus A_h)}&=\widehat{\Leb(A)}-S_h,
	\end{align*}
where $\widehat{\Var(X)}$ and $\widehat{\mathrm{Cov}(X_0, X_h)}$ denote the empirical variance and covariance, respectively.

Based on these alternative estimators, we  estimate 
$\mathrm{\Leb}(A)$, $\mathrm{\Leb}(A\cap A_h)$, 
$\mathrm{\Leb}(A \setminus A_h)$, $\mathrm{\Leb}(A\cap A_h)/\mathrm{\Leb}(A)$ and  $\mathrm{\Leb}(A \setminus A_h)/\mathrm{\Leb}(A)$. 
For each of the five estimators, we report their empirical mean, bias and standard deviation based on 1000 Monte Carlo samples. 
We observe that the slices and corresponding ratios can be estimated 
with high precision for all three marginal distributions and for both the short-memory and the long-memory setting, see Tables the last section of Tables
\ref{tab:slices-exp} and \ref{tab:slices-lm}, 
respectively.  
In particular, it appears that these empirical acf-based estimators outperform the estimator based on summing up the individual trawl function estimates.

\input{Tables/Slices-Exp-Combined.tex}

\input{Tables/Slices-LM-Combined.tex}

\clearpage
\section{Additional figures for the model misspecification testing in  Section 5.1}\label{sec:suppforSect51}
In this section, we repeat the analysis carried out in Section 5.1 for for the case of a  Gaussian trawl process. 
Recall the trawl functions
\begin{align*}
    a_{\text{pe}}(x) = a(x)= \begin{cases}
5 e^{-0.1x} & \text{if } x < 2 \\
5 e^{-0.2} \cdot e^{-0.5(x-2)} & \text{if } x \geq 2
\end{cases},
&& a_{\text{e}}(x)=a(x)=5 e^{-0.3x}, \quad x\geq 0.
\end{align*}

\begin{figure}[htbp]
\centering
\subfloat[Piecewise exp.~$a(x)$]{%
    \includegraphics[width=0.24\textwidth]{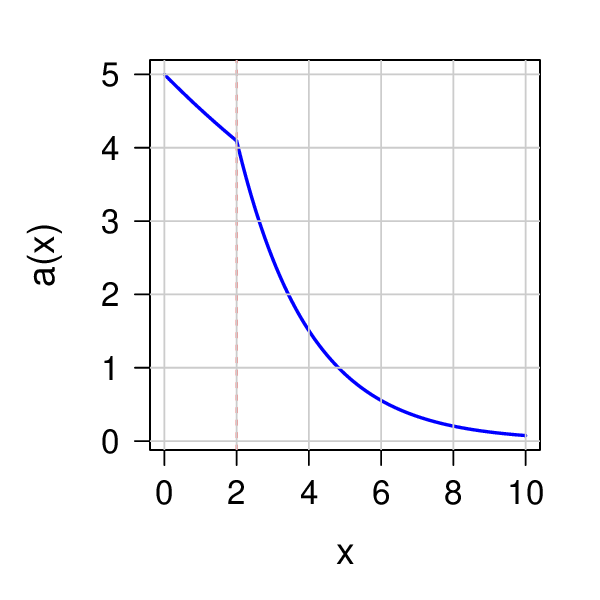}%
}\hfill
\subfloat[ACF $\rho(x)$ (piecewise exp.)]{%
    \includegraphics[width=0.24\textwidth]{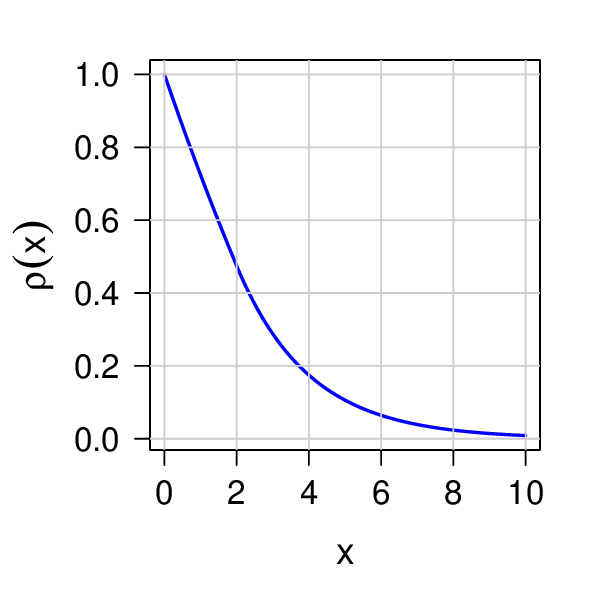}%
}\hfill
\subfloat[Simple exponential $a(x)$]{%
    \includegraphics[width=0.24\textwidth]{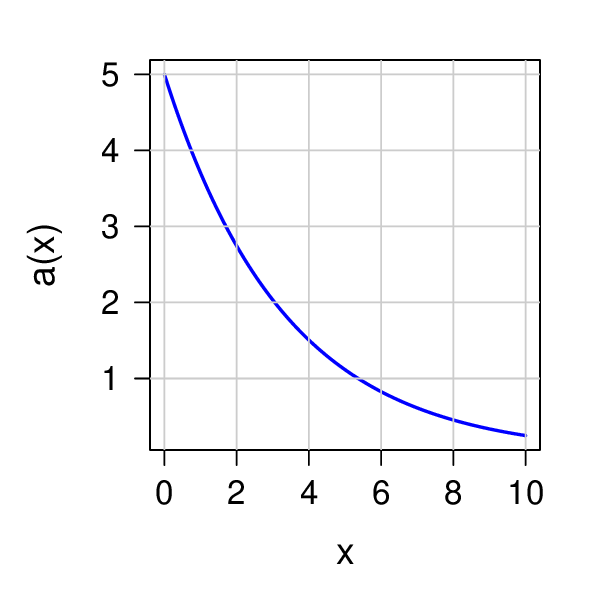}%
}\hfill
\subfloat[ACF $\rho(x)$ (exponential)]{%
    \includegraphics[width=0.24\textwidth]{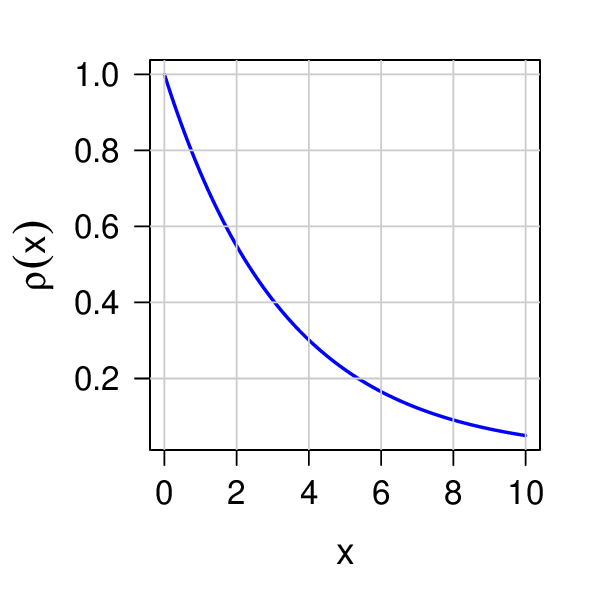}%
}
\caption{Comparison of two trawl functions $a(x)$, specified above, and their associated ACFs $\rho(x)=\Gamma(x)/\Gamma(0)$, where $\Gamma(x)=\int_x^{\infty}a(s)ds$. While the piecewise exponential specification exhibits a clear breakpoint, the corresponding ACFs are visually nearly indistinguishable. This illustrates why parametric ACF-based methods can fail to detect model misspecification.\label{fig:trawl_fcts_ill-supp}}
\end{figure}

Figure \ref{fig:gaussian_comparison} 
shows the true trawl function (red), the nonparametrically estimated trawl function (black), and the corresponding (pointwise) 95\% confidence bounds (blue). In addition, the misspecified exponential trawl function is estimated and displayed in orange, and the correctly specified trawl function is estimated and added as a purple line. 
 The fitted exponential curve does not consistently lie within the confidence bounds obtained from the nonparametric estimator, with departures occurring at specific lag regions.  We also add the corresponding values of the test statistics for the exponential and the piecewise exponential fit, indicating that an exponential trawl is misspecified for very short lags (1-4) and for lags 15-19, close to the breakpoint at lag 20.

\begin{figure}[htbp]
    \centering
       \subfloat[Comparison of nonparametric estimate (black dots) with 95\% confidence 
    bounds (blue), true trawl function (red), and fitted exponential (orange) and 
    piecewise exponential (purple) models.
        \label{fig:gaussian_comparison_top}]{%
        \includegraphics[width=0.5\textwidth, height=6cm]{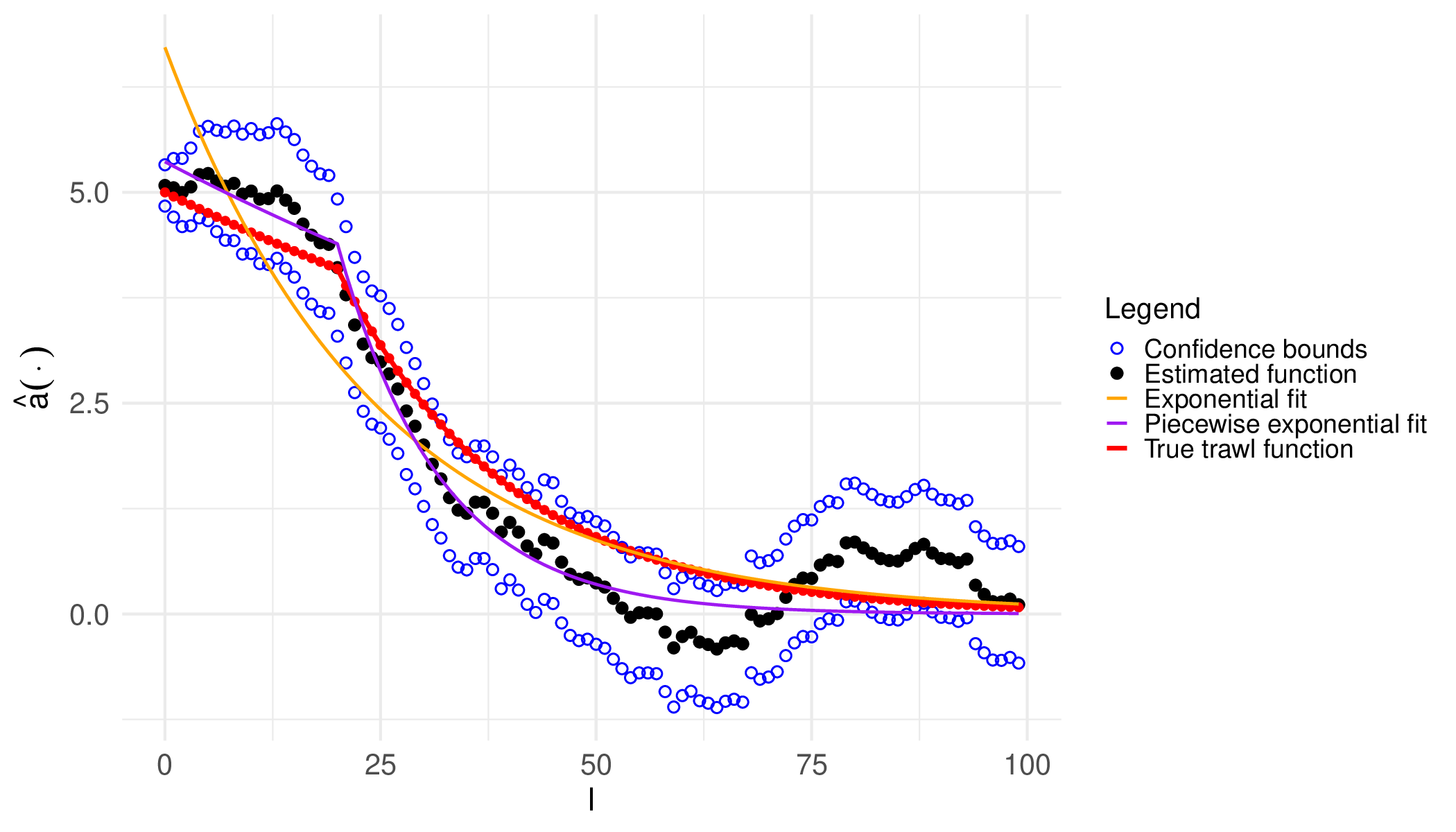}%
    }
    \hfill
    \subfloat[Exponential specification shows systematic rejections concentrated at very short lags and near 
    the structural breakpoint at lag 20 (time = 2.0).\label{fig:gaussian_exponential}]{%
        \includegraphics[width=0.24\textwidth, height=6cm]{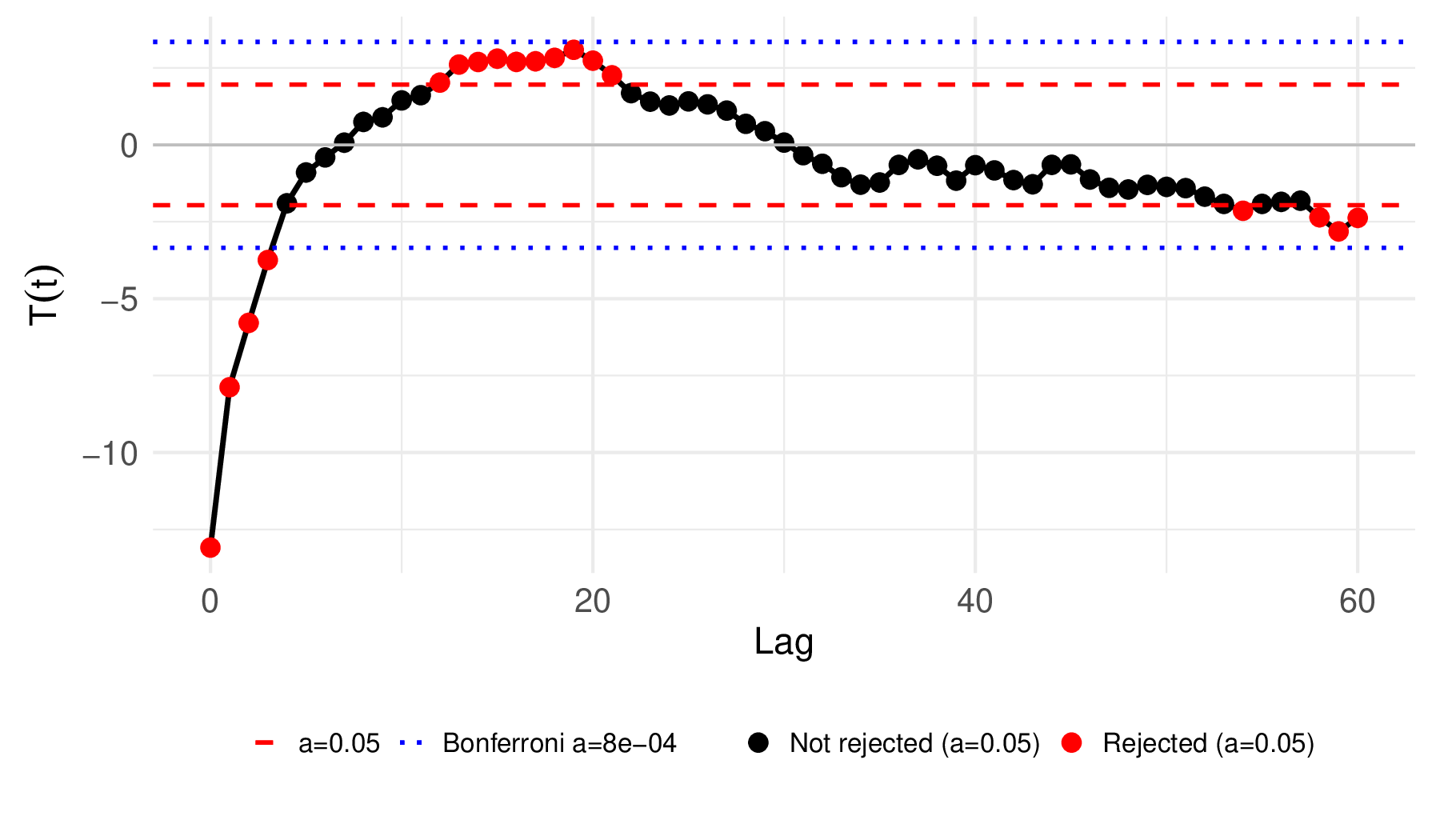}%
    }
    \hfill
    \subfloat[Piecewise exponential (true model) shows test statistics initially within expected 
    range ($\pm 2$) with some rejections for high lags.\label{fig:gaussian_piecewise}]{%
        \includegraphics[width=0.24\textwidth, height=6cm]{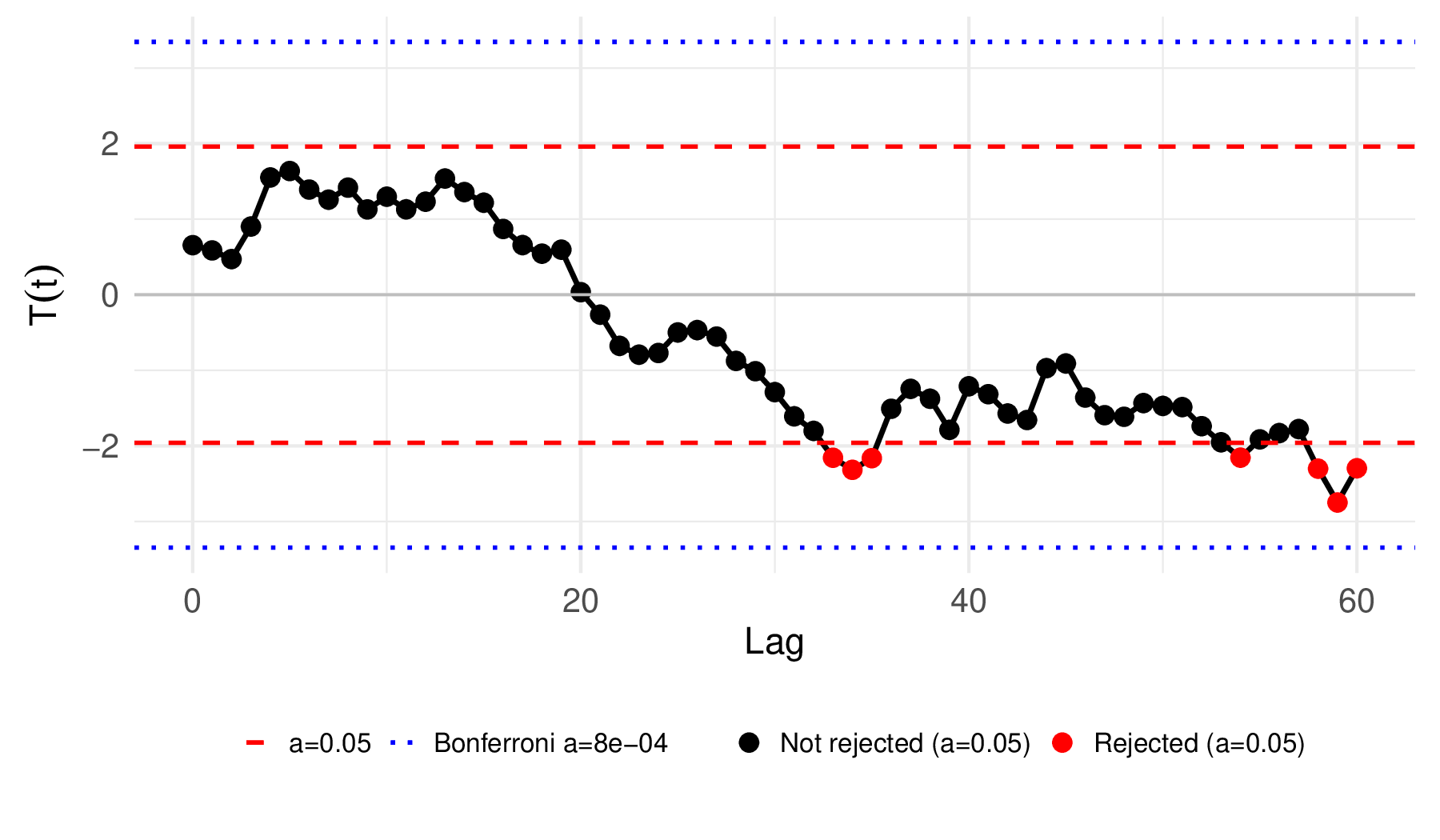}%
    }    
    \caption{Trawl function estimation and specification testing for Gaussian trawl process with piecewise-exponential trawl function. 
    Test statistics for specification tests across lags $0:60$. Dashed lines indicate 
    $\alpha=0.05$ critical values; dotted lines show the Bonferroni-corrected thresholds 
    ($\alpha/61\approx0.0008$).}
    \label{fig:gaussian_comparison}
\end{figure}

\begin{figure}
 \subfloat[Gaussian trawl process: Identification of breakpoint for hybrid estimation.\label{fig:Gaussianbreakpoint}]{%
        \includegraphics[width=0.48\textwidth, height=6cm]{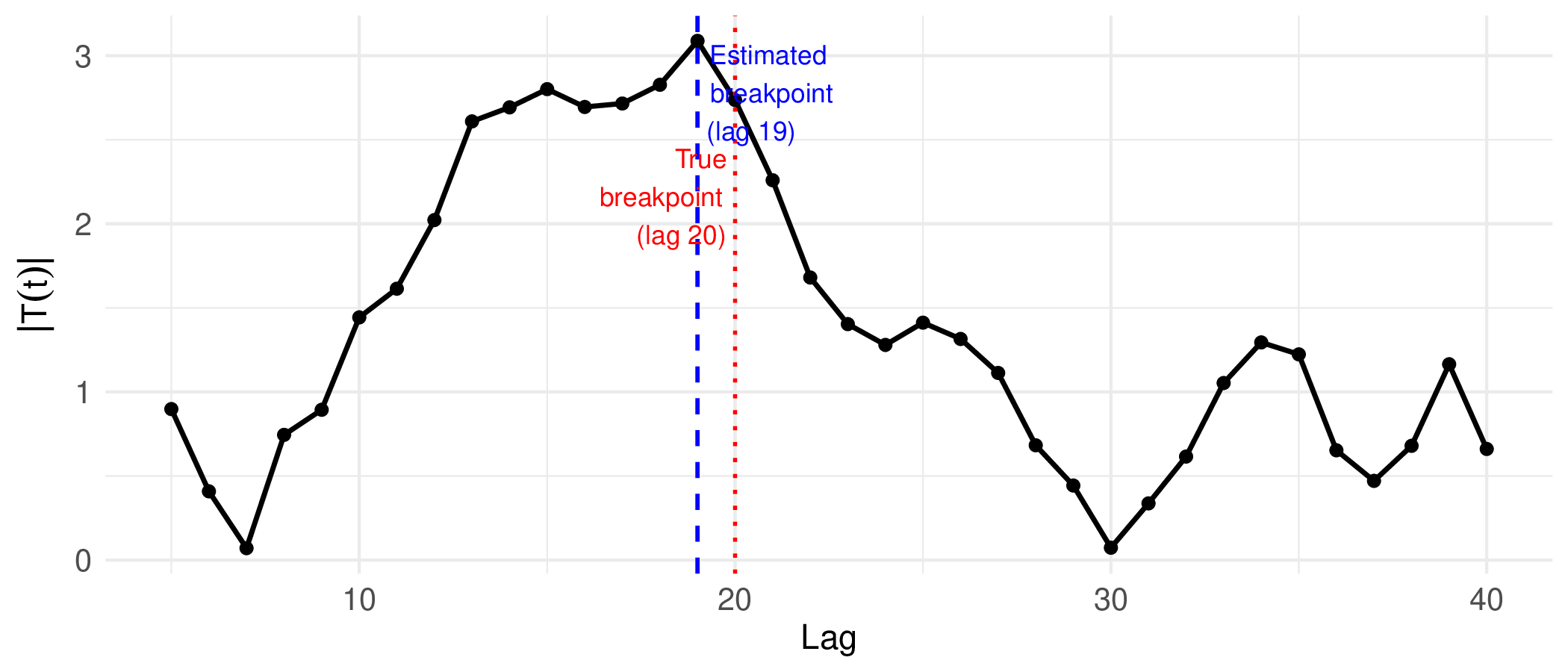}%
    }
    \subfloat[Gaussian trawl process: Identification of breakpoint for hybrid estimation.\label{fig:Gaussianbreakpoint}]{%
        \includegraphics[width=0.48\textwidth, height=6cm]{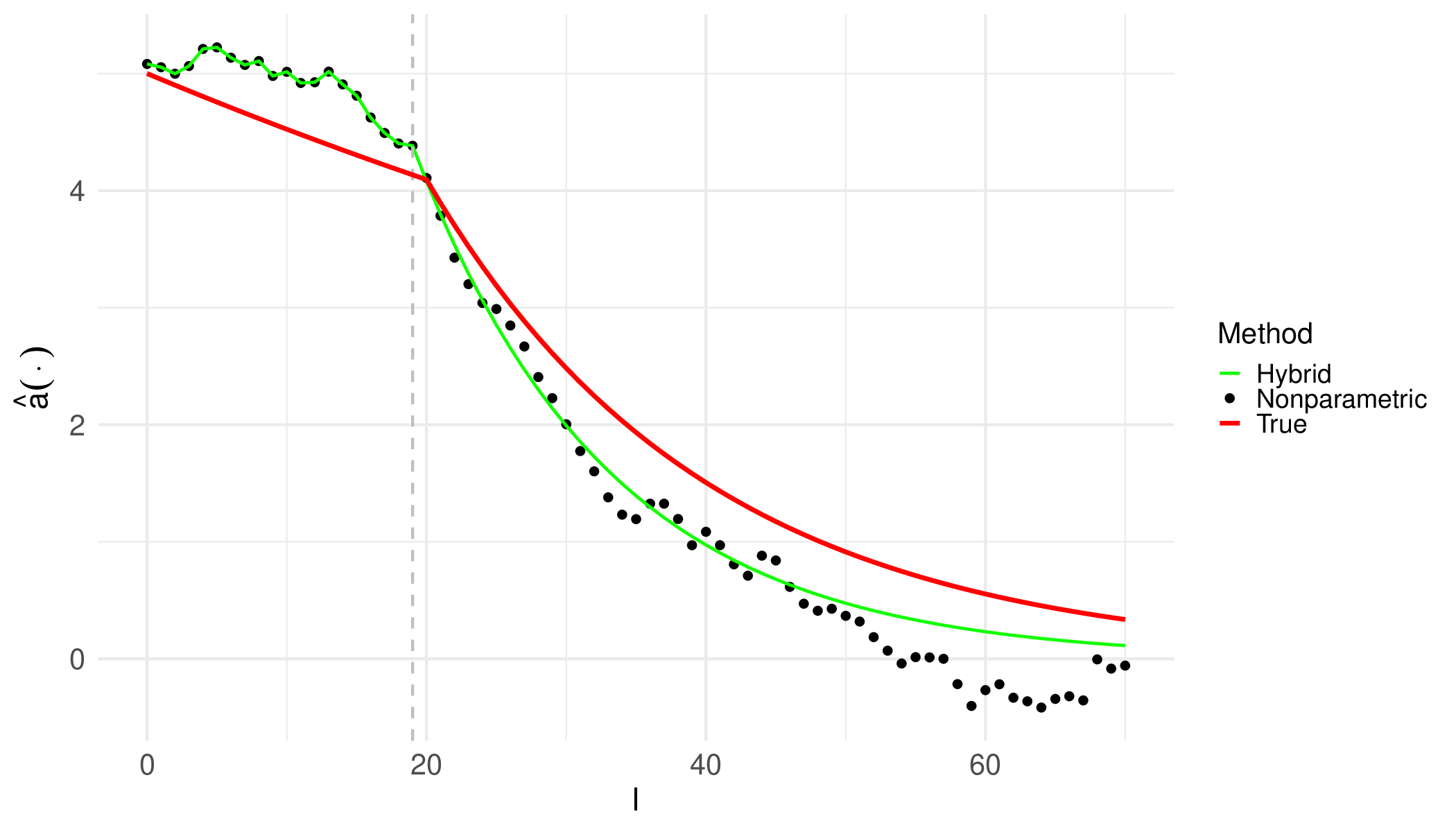}%
    }
    \caption{Using the maximum of the test statistic (comparing with an exponential trawl function) to identify the breakpoint in a hybrid trawl estimation. The hybrid estimator uses nonparametric estimates $\hat{a}(t)$ up to the detected breakpoint $\hat{c}_1$ (identified as the lag maximizing $|T(t)|$), then switches to an exponential tail $\hat{a}(\hat{c}_1) \exp(-\hat{\lambda}(t - \hat{c}_1))$ where $\hat{\lambda}$ is fitted via log-linear regression on the next 30 lags.}
\end{figure}

\begin{figure}[htbp]
    \centering
    \subfloat[Gaussian: RMSE\label{fig:Gaussianforecast}]{%
       \includegraphics[width=0.48\textwidth, height=6cm]{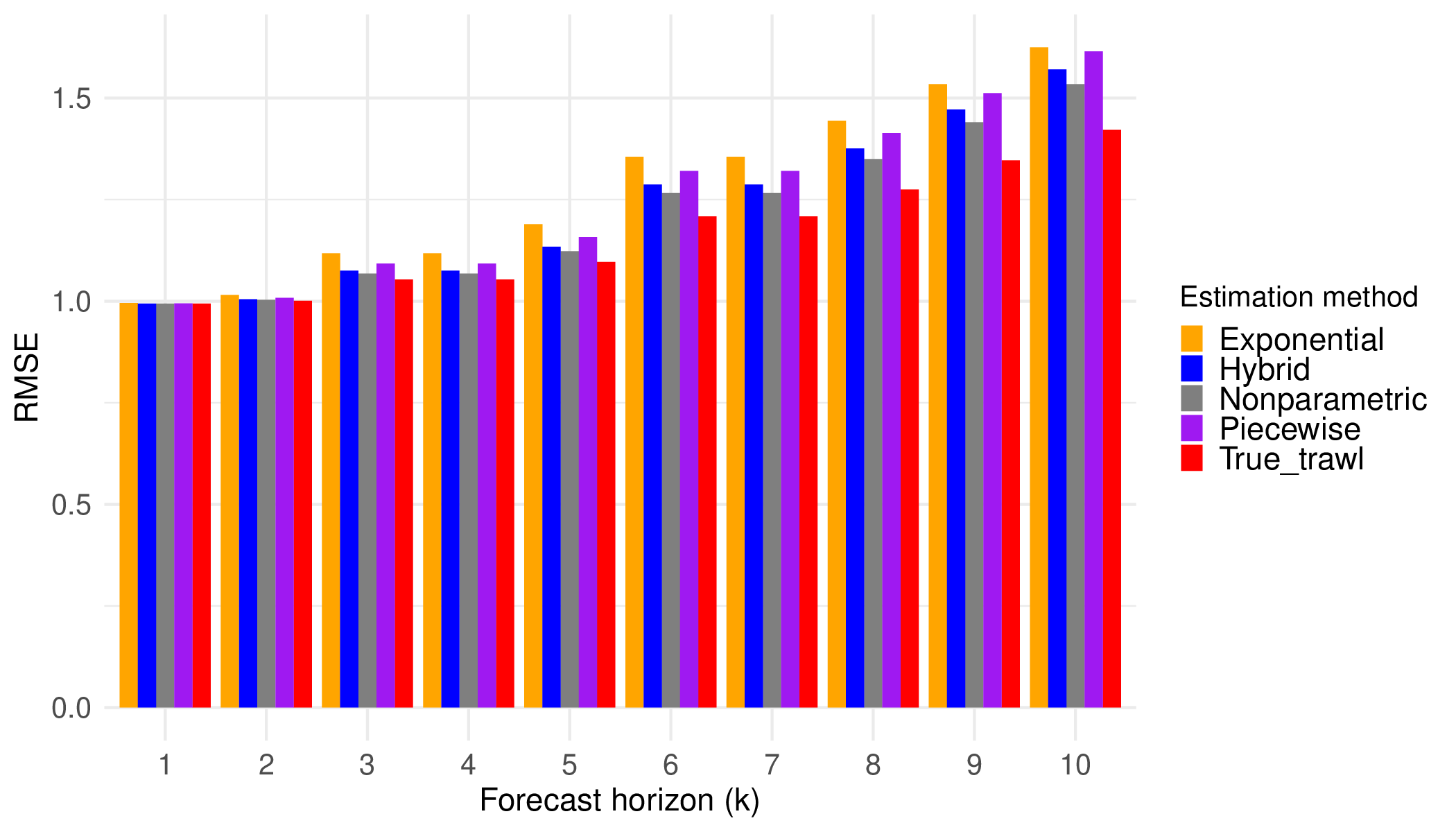}%
    }
    \hfill
     \subfloat[Gaussian: \% penalty\label{fig:Gaussianforecast-pen}]{%
        \includegraphics[width=0.48\textwidth, height=6cm]{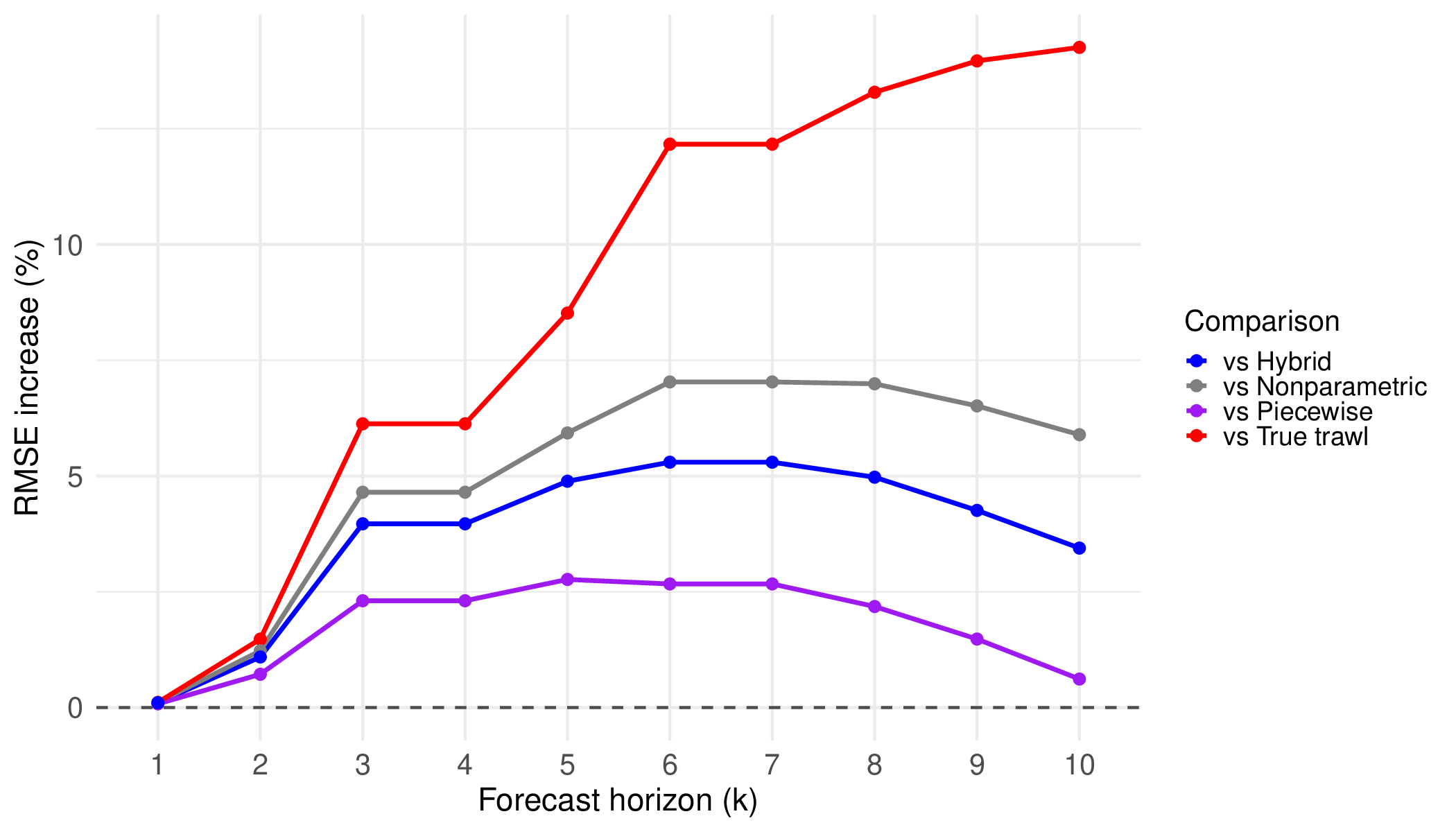}%
    }
    \caption{Forecast accuracy for the Gaussian  trawl process based on 2000 out-of-sample forecasts. RMSE across methods (exponential, hybrid, nonparametric, piecewise-exponential, true) and percentage penalty of using a misspecified exponential specification.}
    \label{fig:forecast_comparison-supp}
\end{figure}

\clearpage
\section{Additional figures from the empirical study}\label{asec:empirics}
In the following, we provide additional figures from the empirical study for the stocks DFS, WAT and WM.

\begin{figure}[htbp]
\centering
\captionsetup[subfigure]{aboveskip=-4pt, belowskip=-4pt}

\subfloat[Boxplots of estimated trawl function]{%
  \includegraphics[scale=0.30]{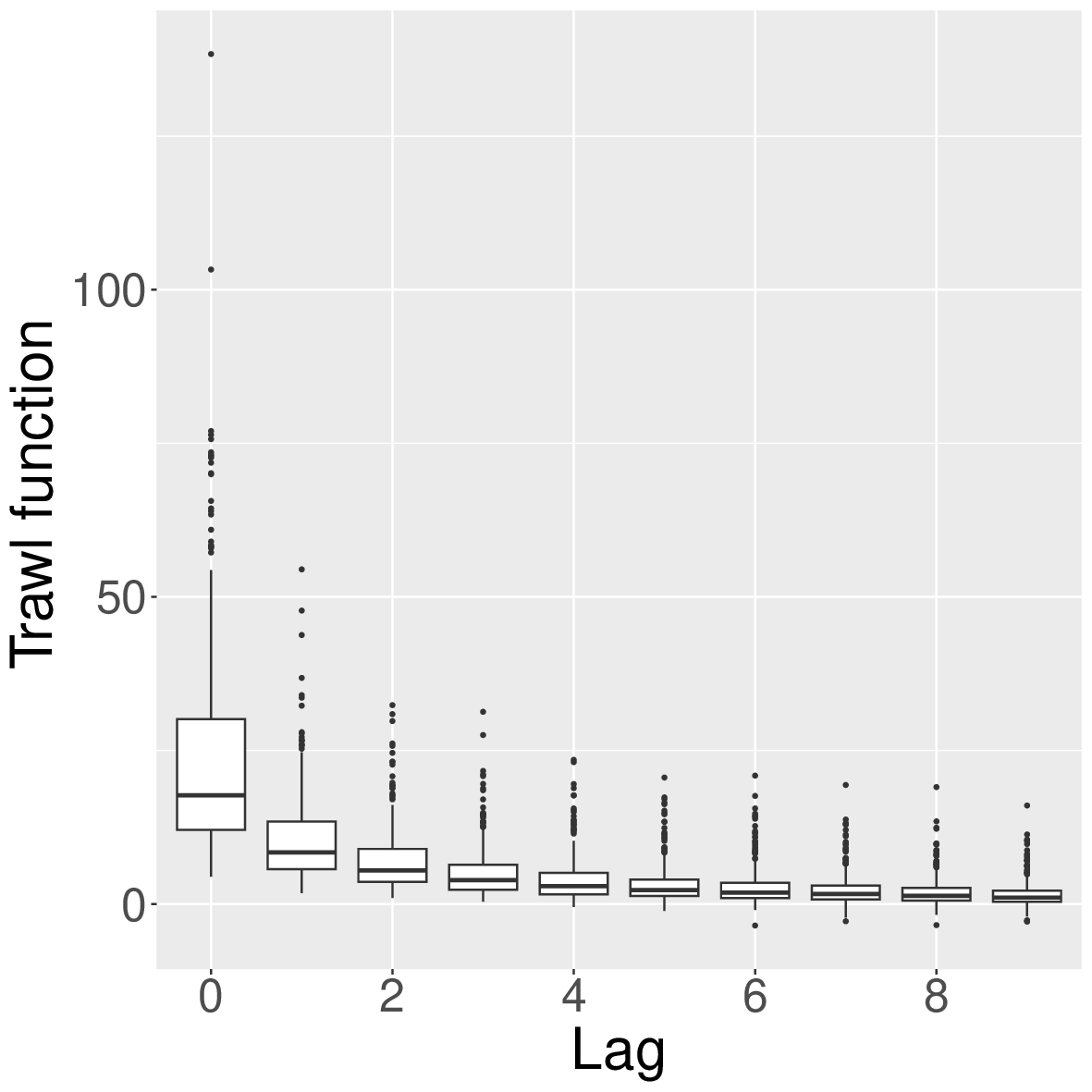}
}
\hfill
\subfloat[Estimated forecast weights]{%
  \includegraphics[scale=0.30]{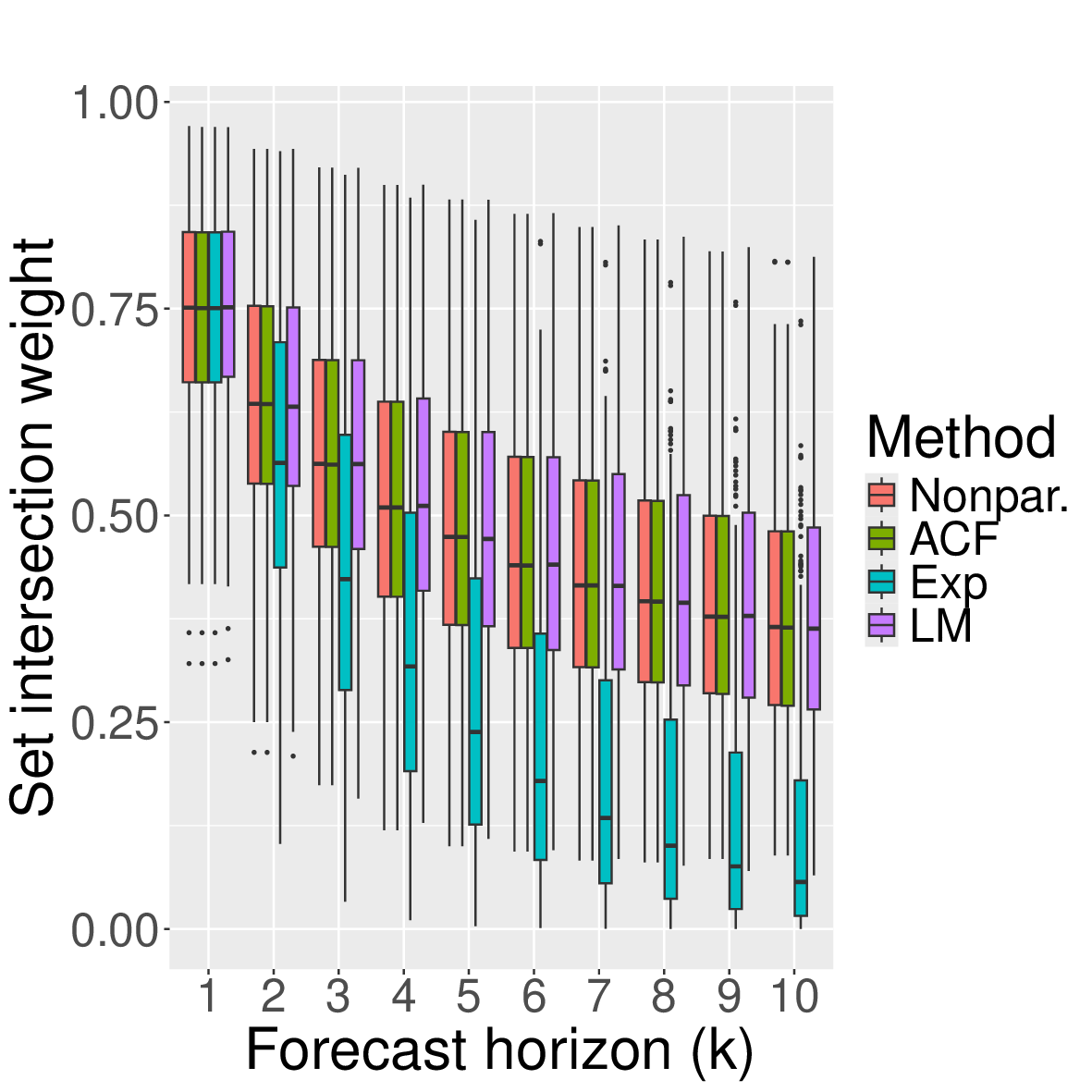}
}

	\caption{\it 
 The first picture shows the box plots of the estimated trawl functions for stock DFS. The second figure shows the estimated forecast weight $\Leb(A \cap A_h)/\Leb(A)$ using a nonparametric trawl (Nonparametric), and acf-based estimate (ACF), an estimated exponential trawl (Exp) and an estimated long memory/supGamma trawl (LM).}
 \label{fig:Trawl_DFS}
\end{figure}

\begin{figure}[htbp]
\centering
\captionsetup[subfigure]{aboveskip=-4pt, belowskip=-4pt}

\subfloat[Nonparametric versus ACF]{%
  \includegraphics[scale=0.20]{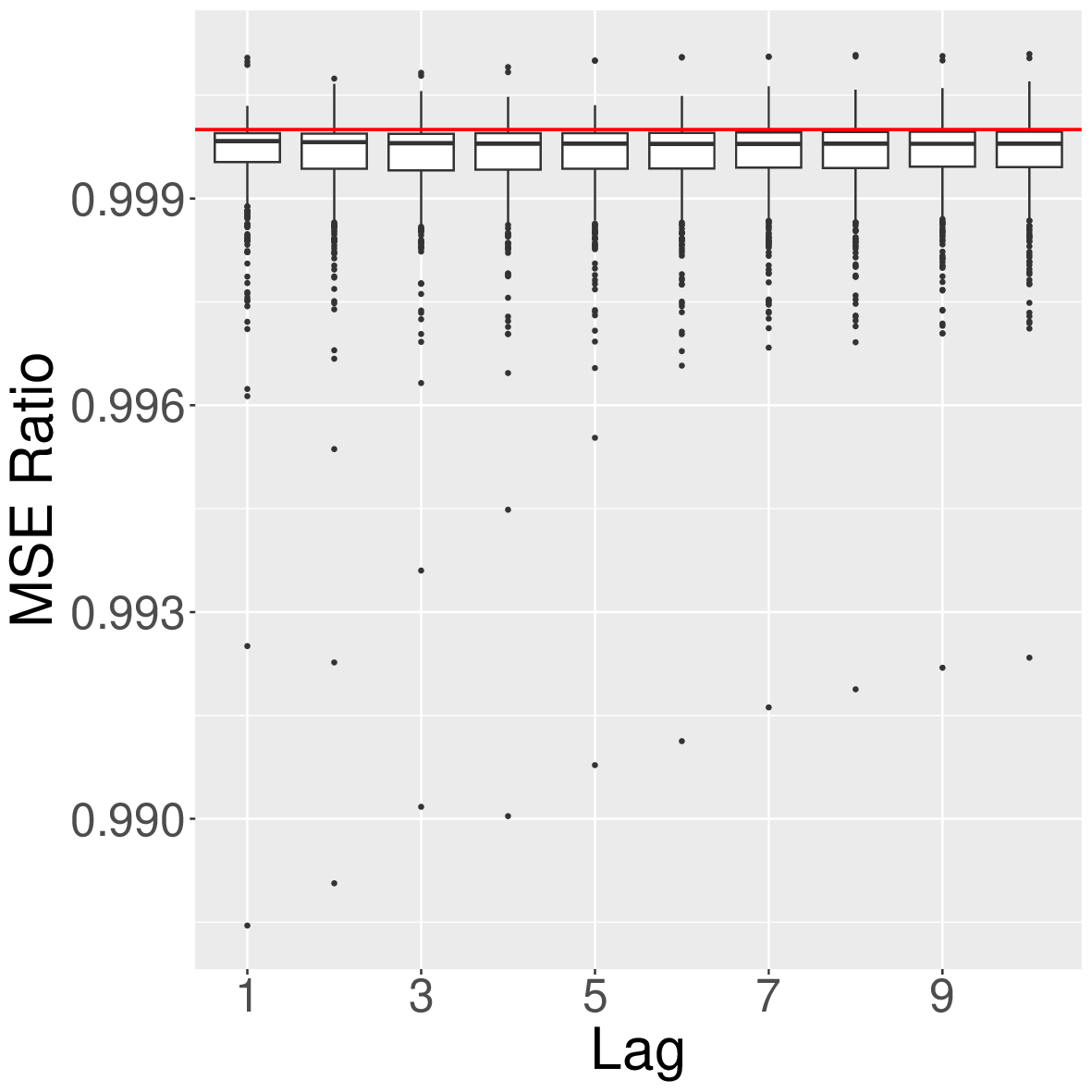}
}
\hfill
\subfloat[Nonparametric versus Exp]{%
  \includegraphics[scale=0.20]{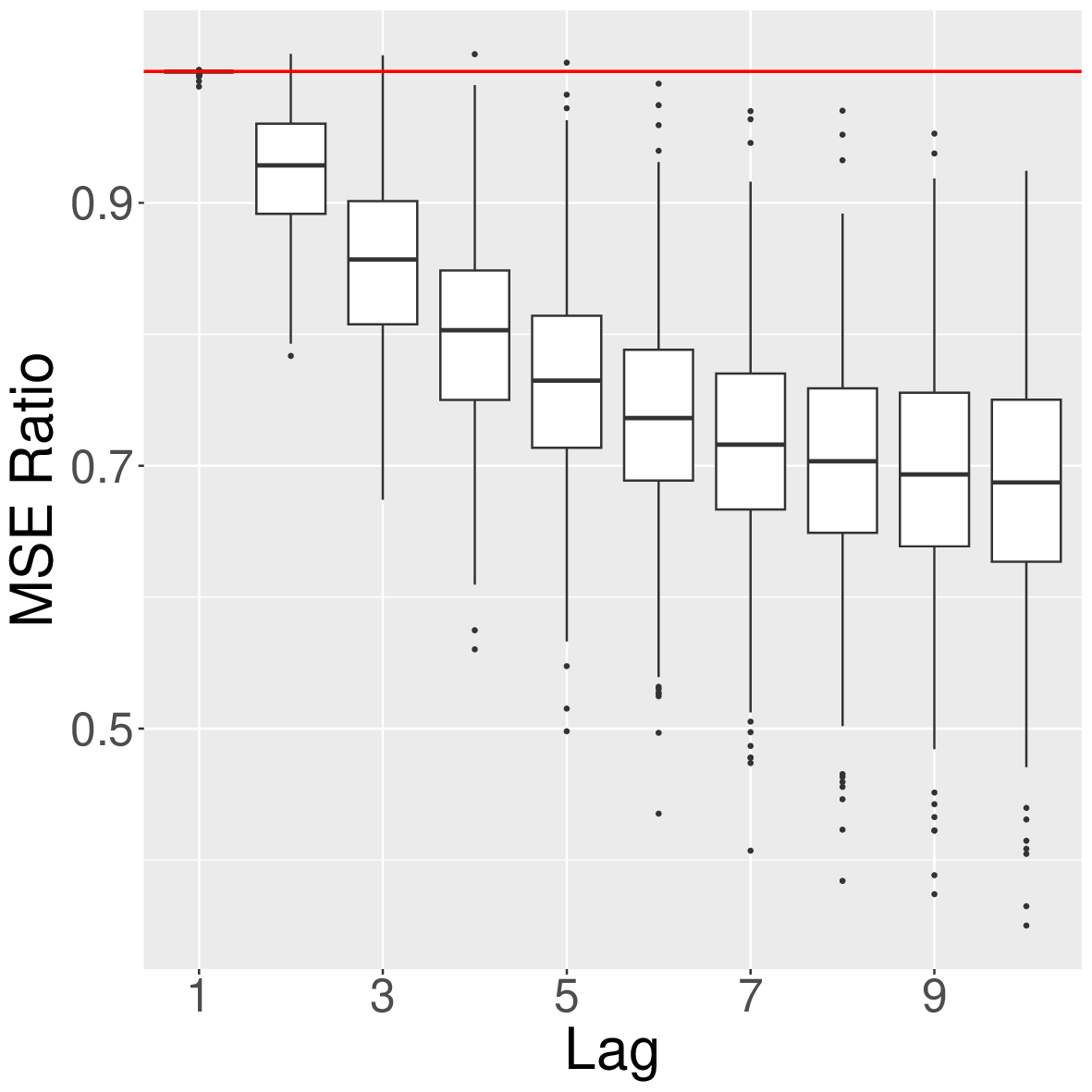}
}
\hfill
\subfloat[Nonparametric versus LM]{%
  \includegraphics[scale=0.20]{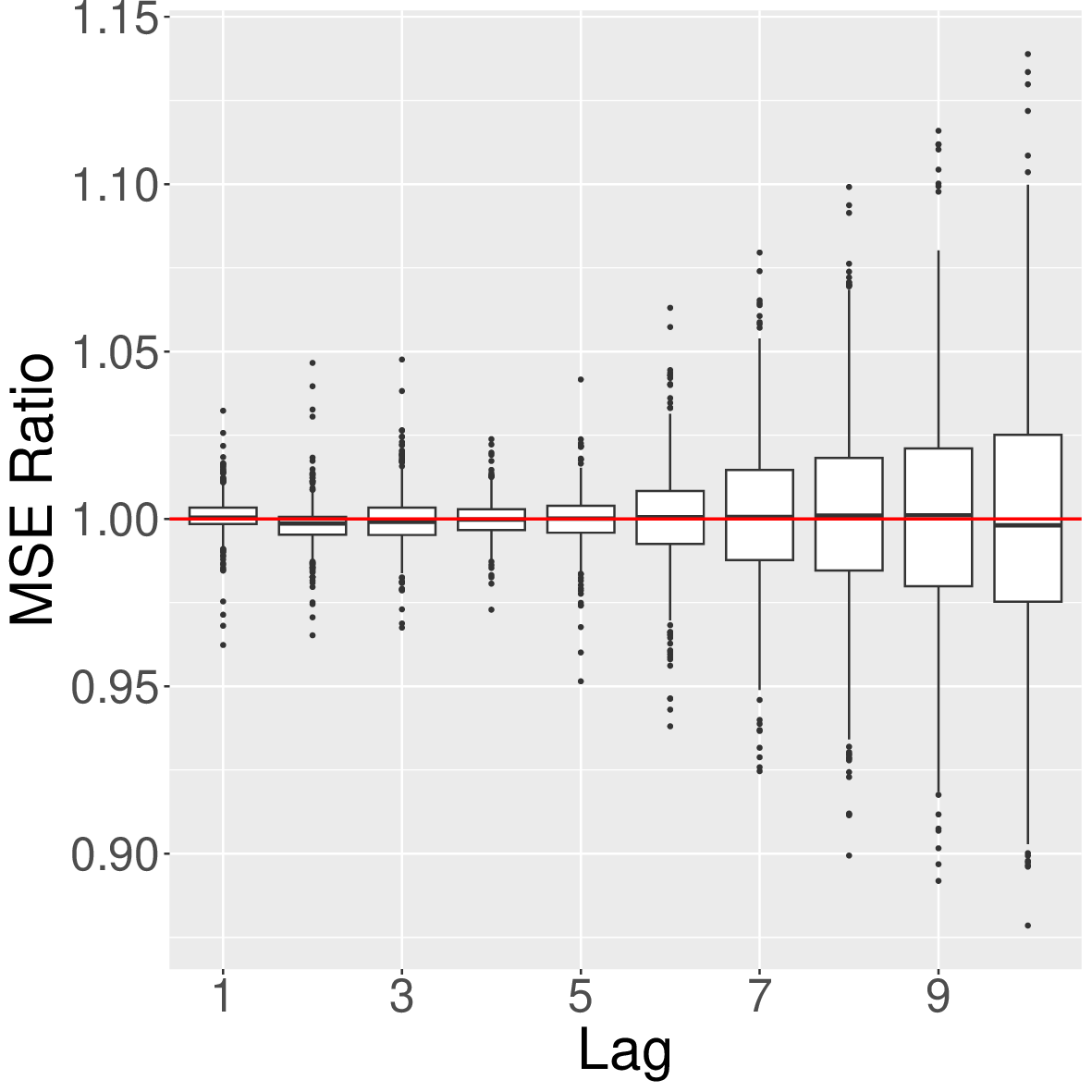}
}
	\caption{\it Here we compare the boxplots of estimated the MSE ratios (nonparametric vs.~ACF-, exponential-, long-memory estimators) for stock DFS. A ratio of $<1$ indicates smaller MSEs of the nonparametric estimator.}
 \label{fig:ForecastRatios_DFS}
\end{figure}


\begin{figure}[htbp]
\centering
\captionsetup[subfigure]{aboveskip=-4pt, belowskip=-4pt}

\subfloat[Boxplots of estimated trawl function]{%
  \includegraphics[scale=0.30]{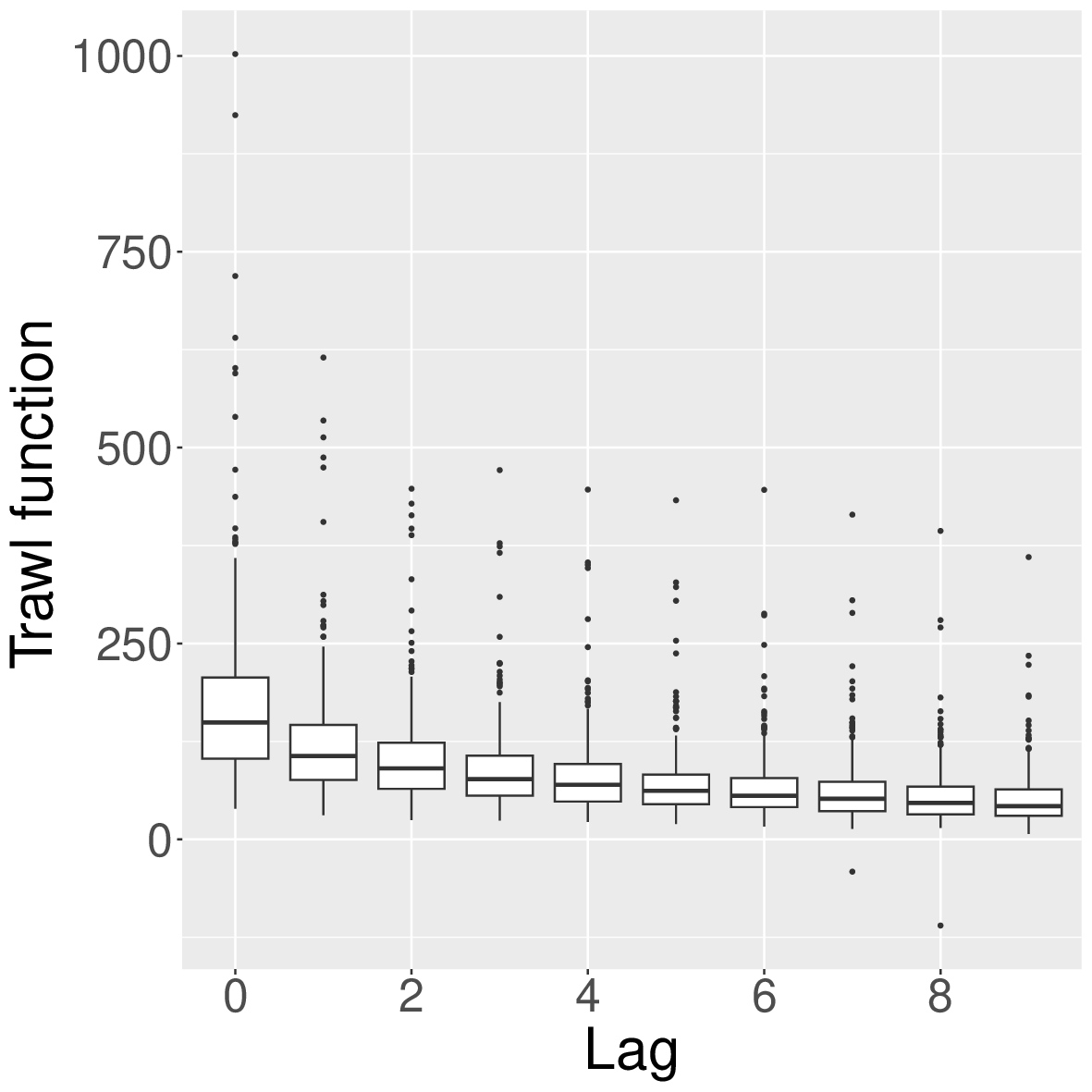}
}
\hfill
\subfloat[Estimated forecast weights]{%
  \includegraphics[scale=0.30]{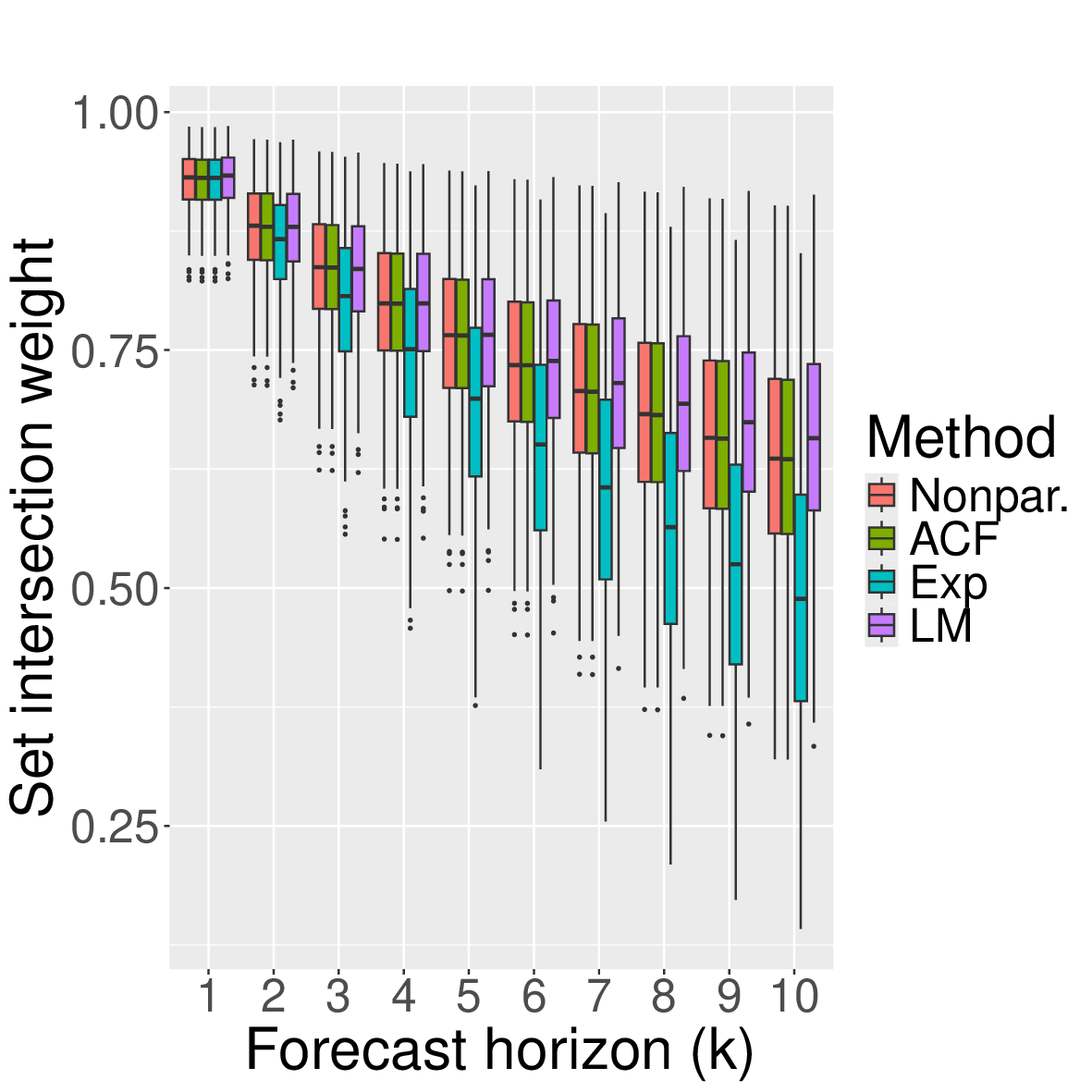}
}

	\caption{\it 
 The first picture shows the box plots of the estimated trawl functions for stock WAT. The second figure shows the estimated forecast weight $\Leb(A \cap A_h)/\Leb(A)$ using a nonparametric trawl (Nonparametric), and acf-based estimate (ACF), an estimated exponential trawl (Exp) and an estimated long memory/supGamma trawl (LM).}
 \label{fig:Trawl_WAT}
\end{figure}

\begin{figure}[htbp]
\centering
\captionsetup[subfigure]{aboveskip=-4pt, belowskip=-4pt}

\subfloat[Nonparametric versus ACF]{%
  \includegraphics[scale=0.20]{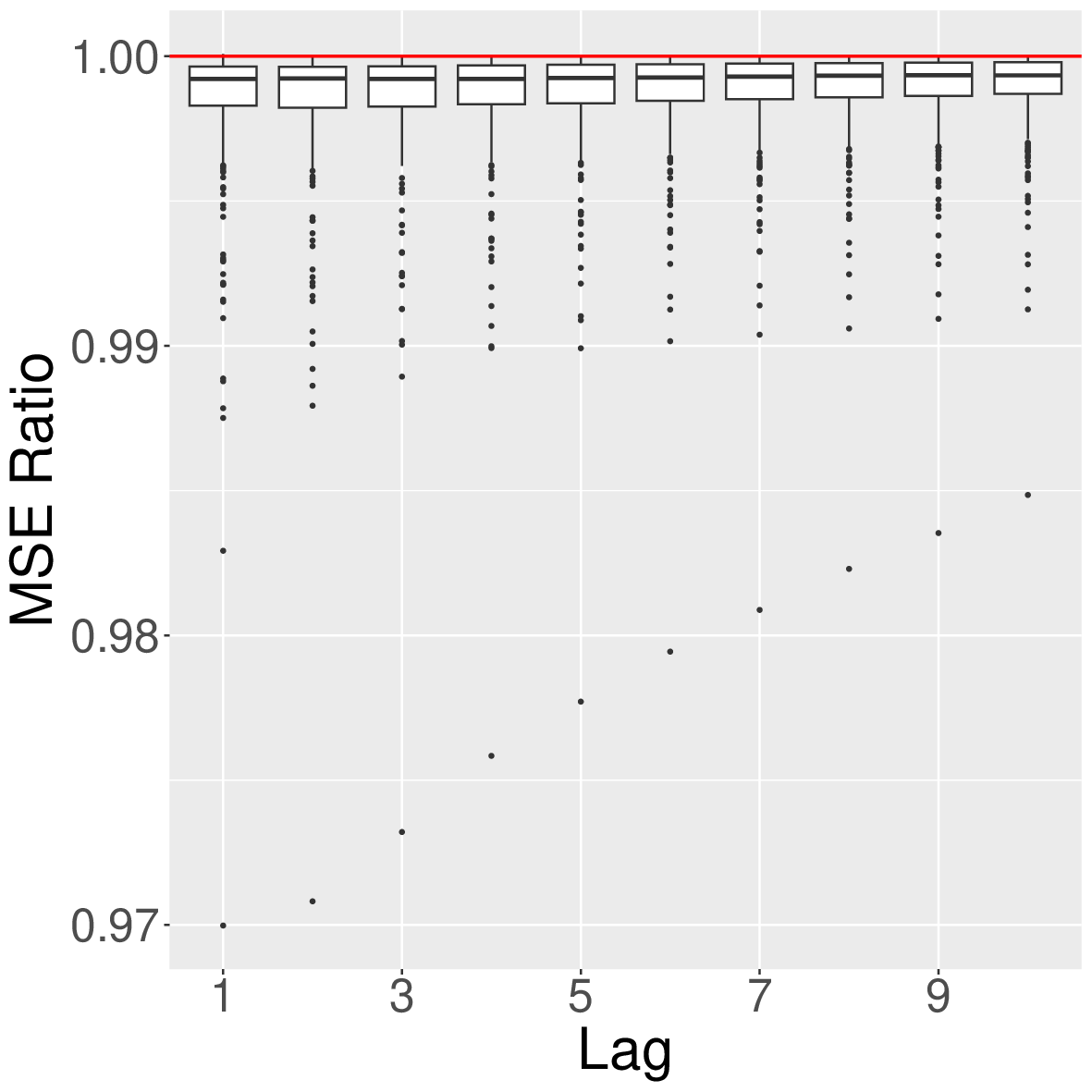}
}
\hfill
\subfloat[Nonparametric versus Exp]{%
  \includegraphics[scale=0.20]{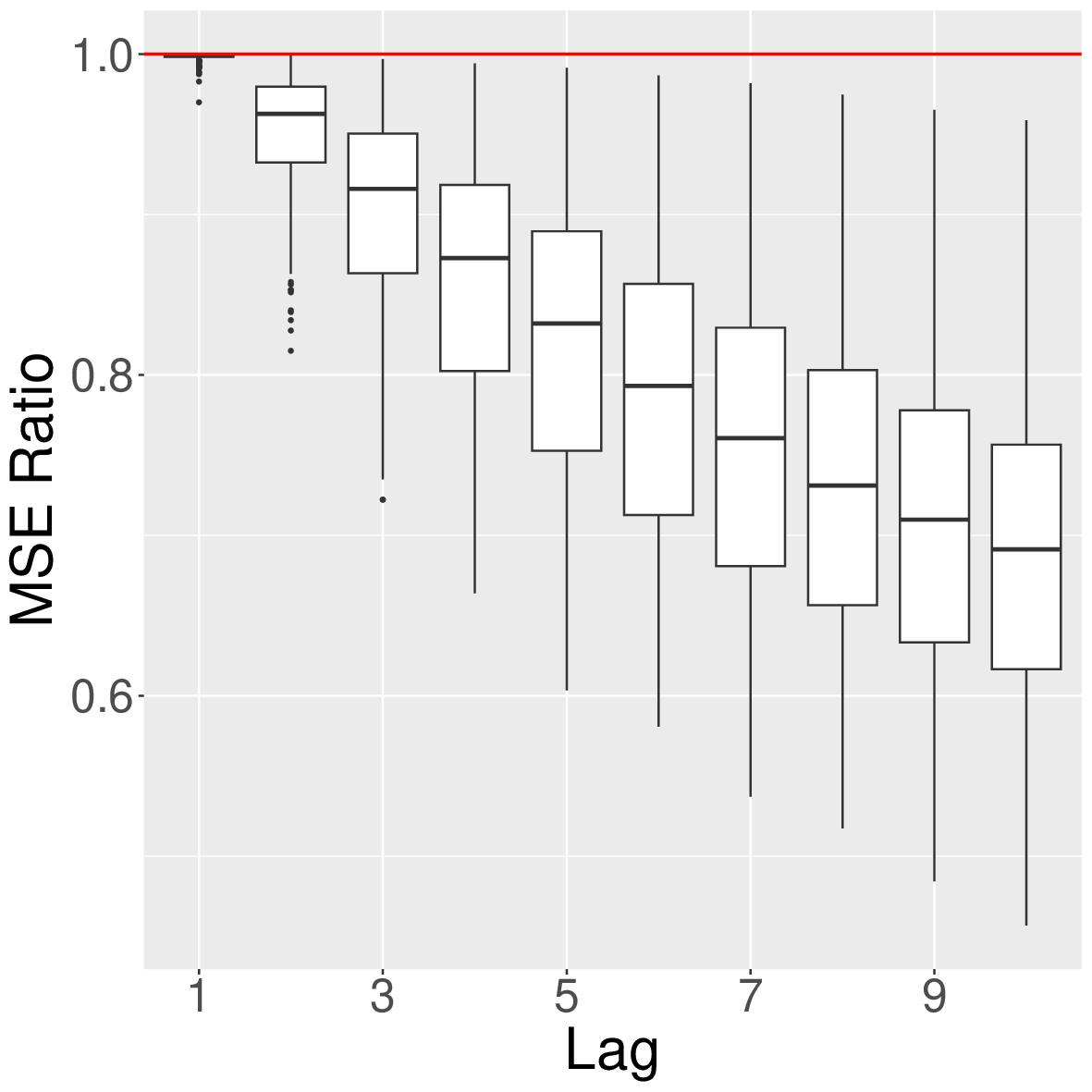}
}
\hfill
\subfloat[Nonparametric versus LM]{%
  \includegraphics[scale=0.20]{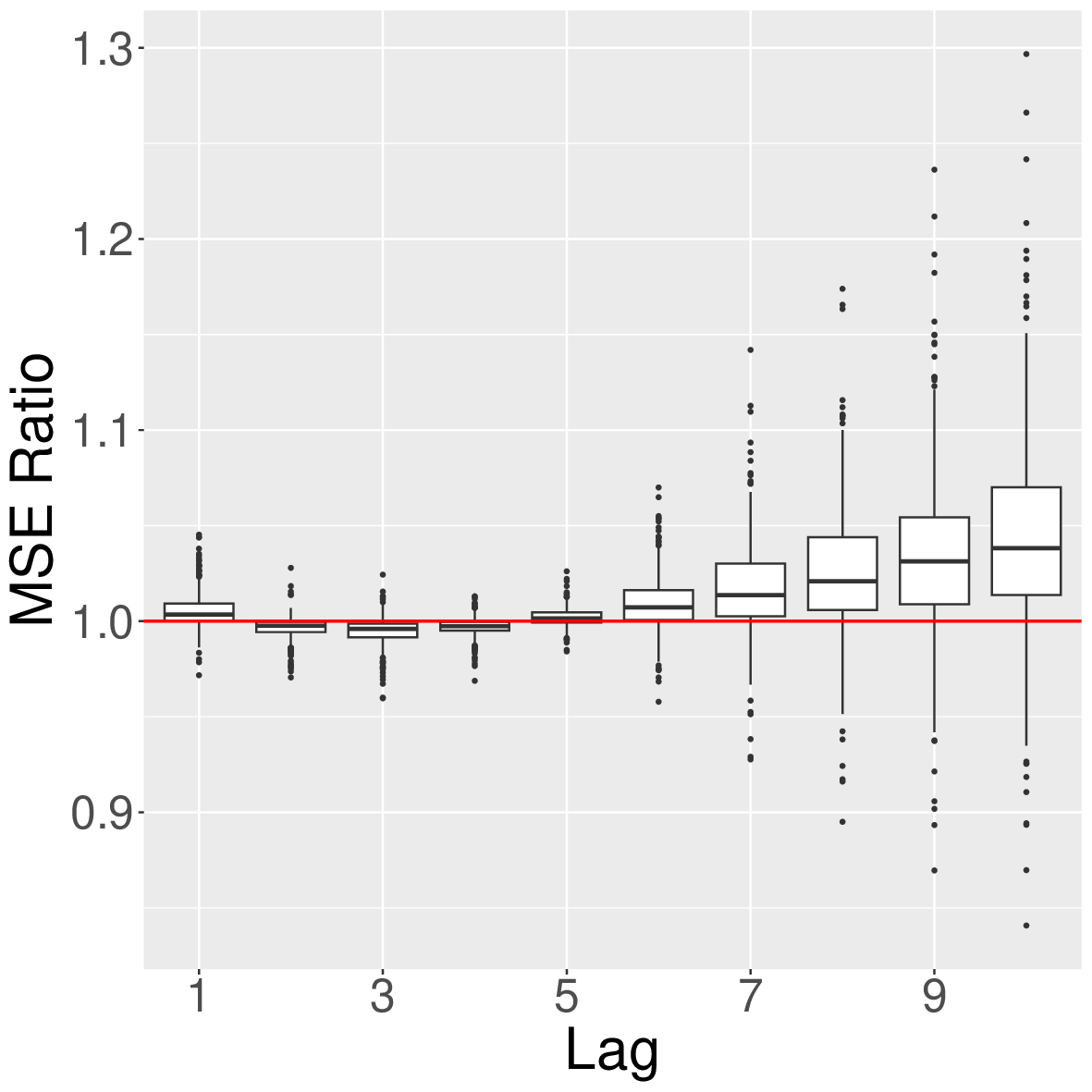}
}
	\caption{\it Here we compare the boxplots of estimated the MSE ratios (nonparametric vs.~ACF-, exponential-, long-memory estimators) for stock WAT. A ratio of $<1$ indicates smaller MSEs of the nonparametric estimator.}
 \label{fig:ForecastRatios_WAT}
\end{figure}

\begin{figure}[htbp]
\centering
\captionsetup[subfigure]{aboveskip=-4pt, belowskip=-4pt}

\subfloat[Boxplots of estimated trawl function]{%
  \includegraphics[scale=0.30]{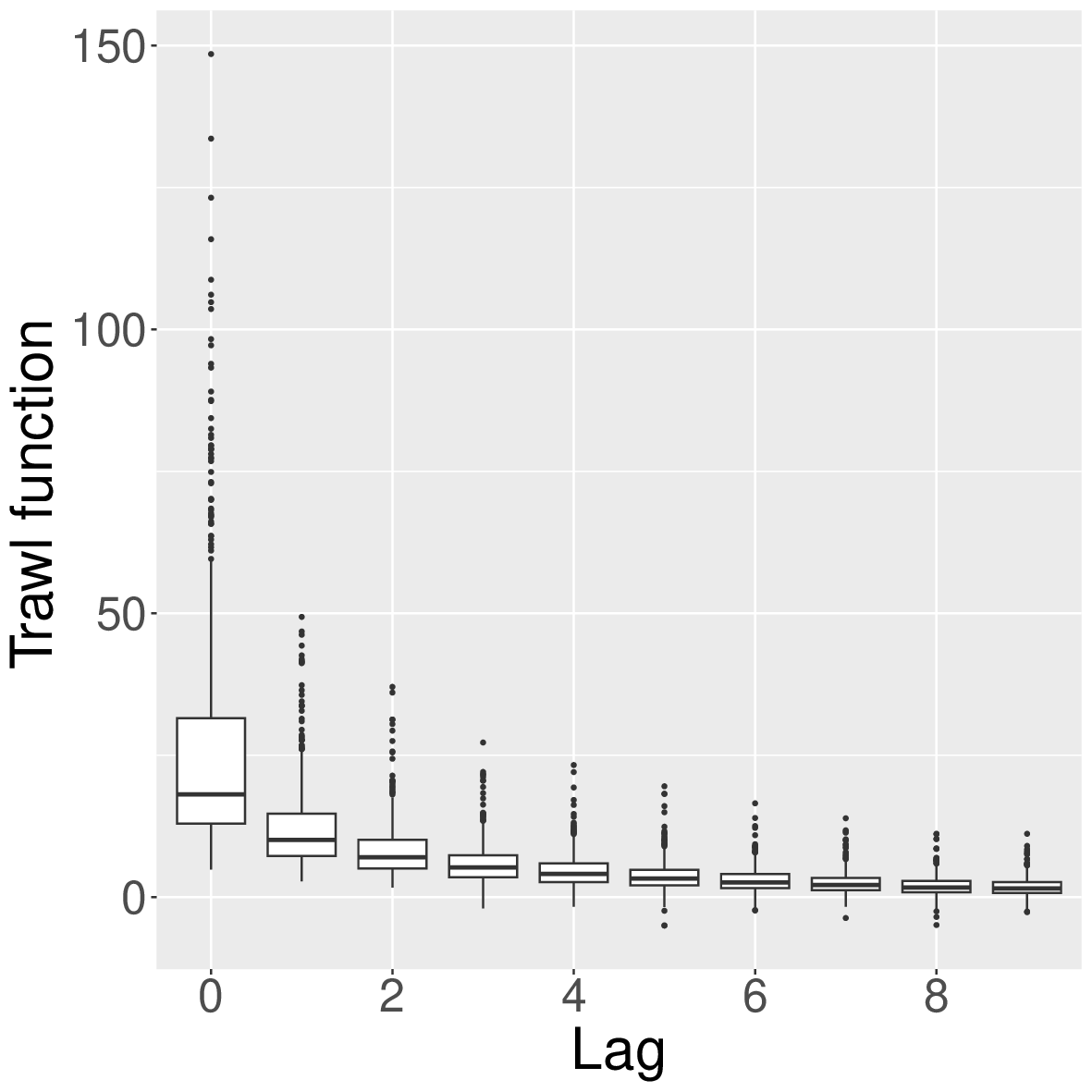}
}
\hfill
\subfloat[Estimated forecast weights]{%
  \includegraphics[scale=0.30]{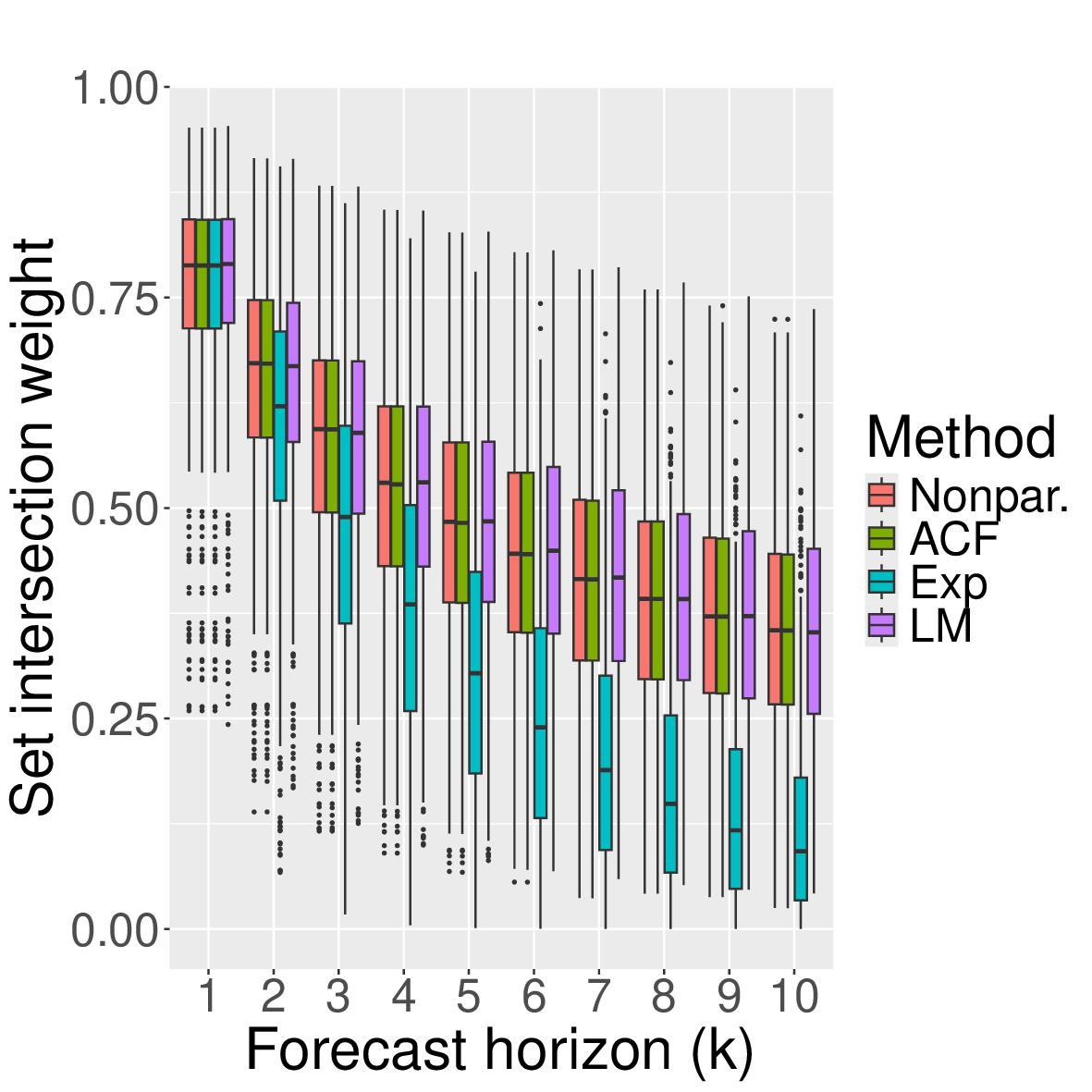}
}

	\caption{\it 
 The first picture shows the box plots of the estimated trawl functions for stock WM. The second figure shows the estimated forecast weight $\Leb(A \cap A_h)/\Leb(A)$ using a nonparametric trawl (Nonparametric), and acf-based estimate (ACF), an estimated exponential trawl (Exp) and an estimated long memory/supGamma trawl (LM).}
 \label{fig:Trawl_WM}
\end{figure}

\begin{figure}[htbp]
\centering
\captionsetup[subfigure]{aboveskip=-4pt, belowskip=-4pt}

\subfloat[Nonparametric versus ACF]{%
  \includegraphics[scale=0.20]{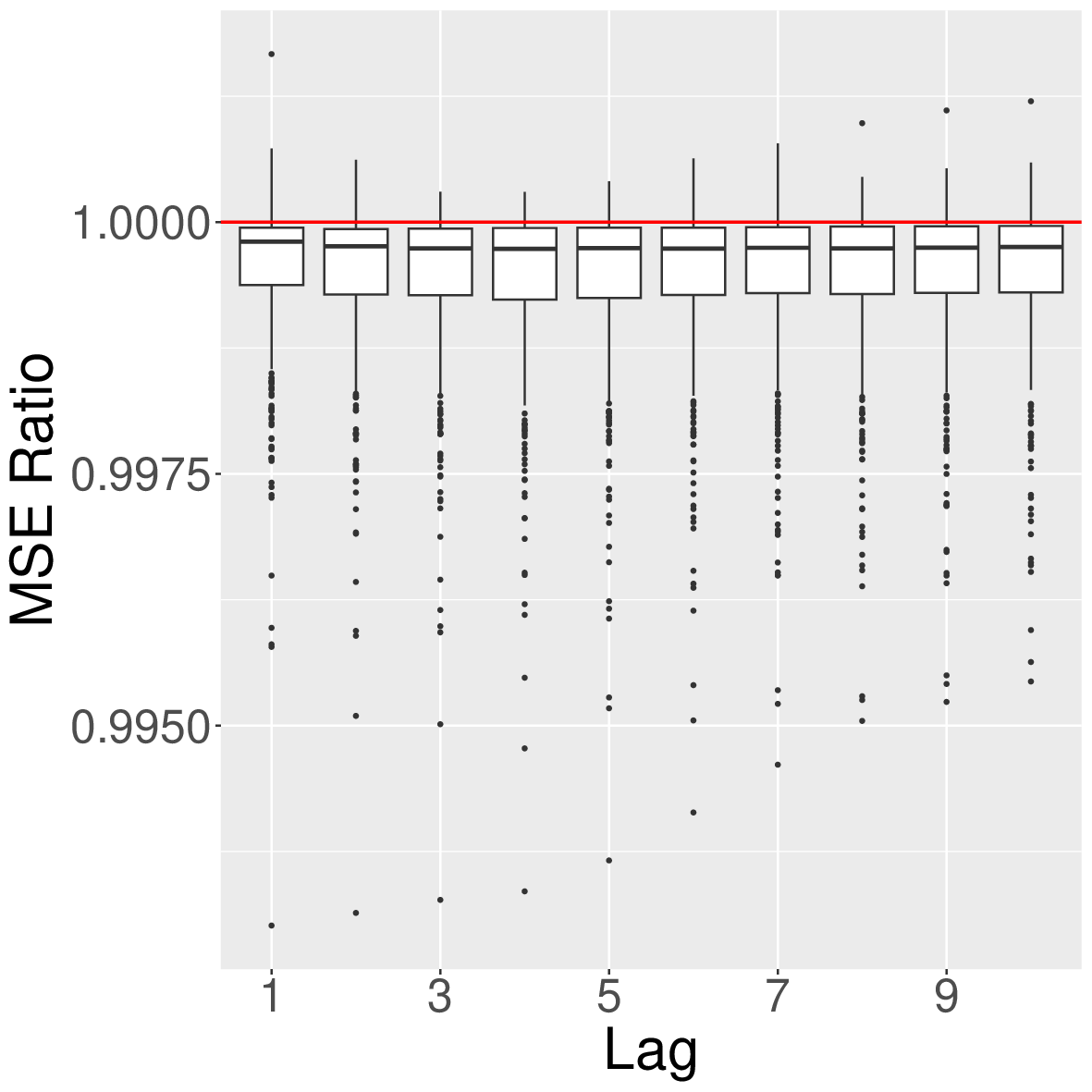}
}
\hfill
\subfloat[Nonparametric versus Exp]{%
  \includegraphics[scale=0.20]{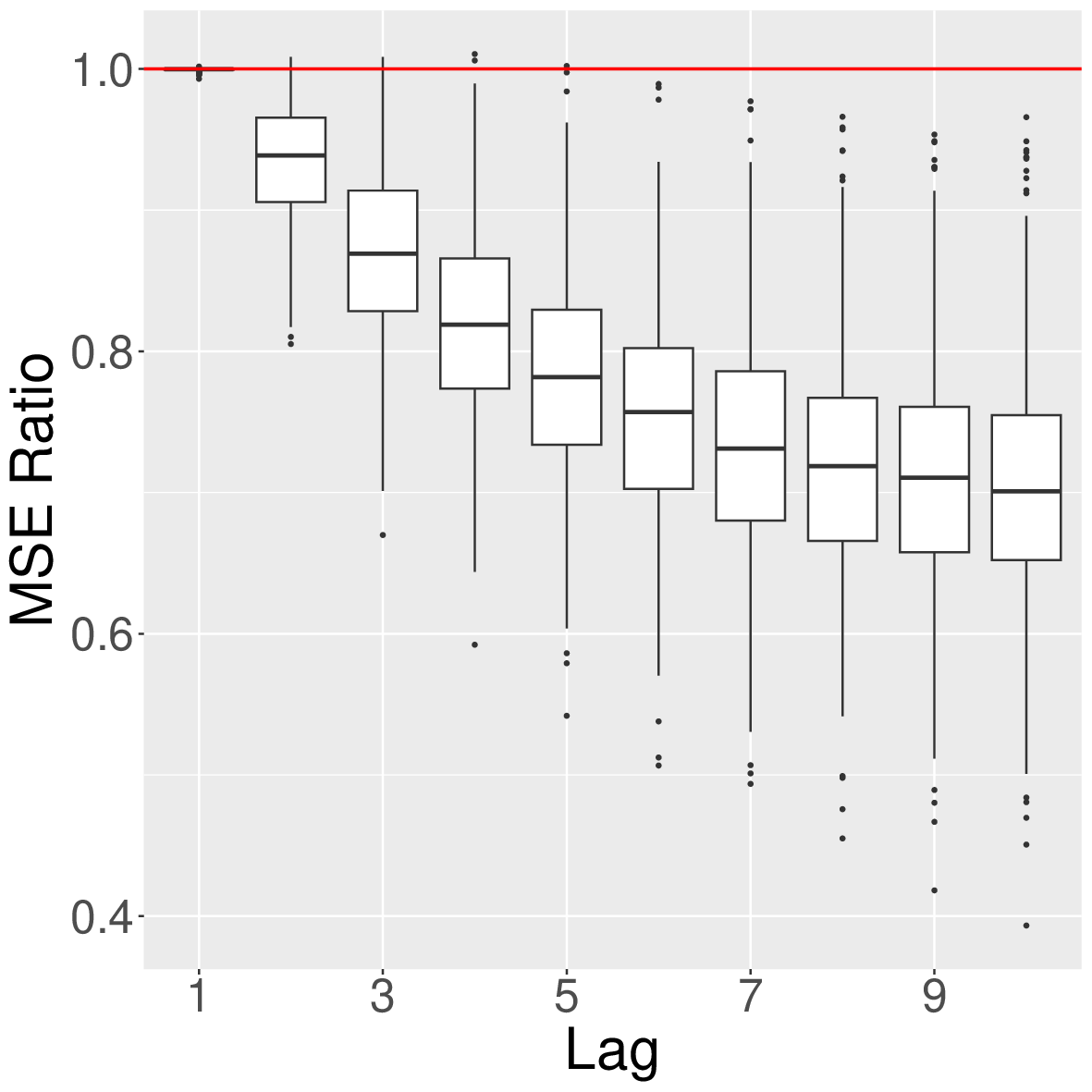}
}
\hfill
\subfloat[Nonparametric versus LM]{%
  \includegraphics[scale=0.20]{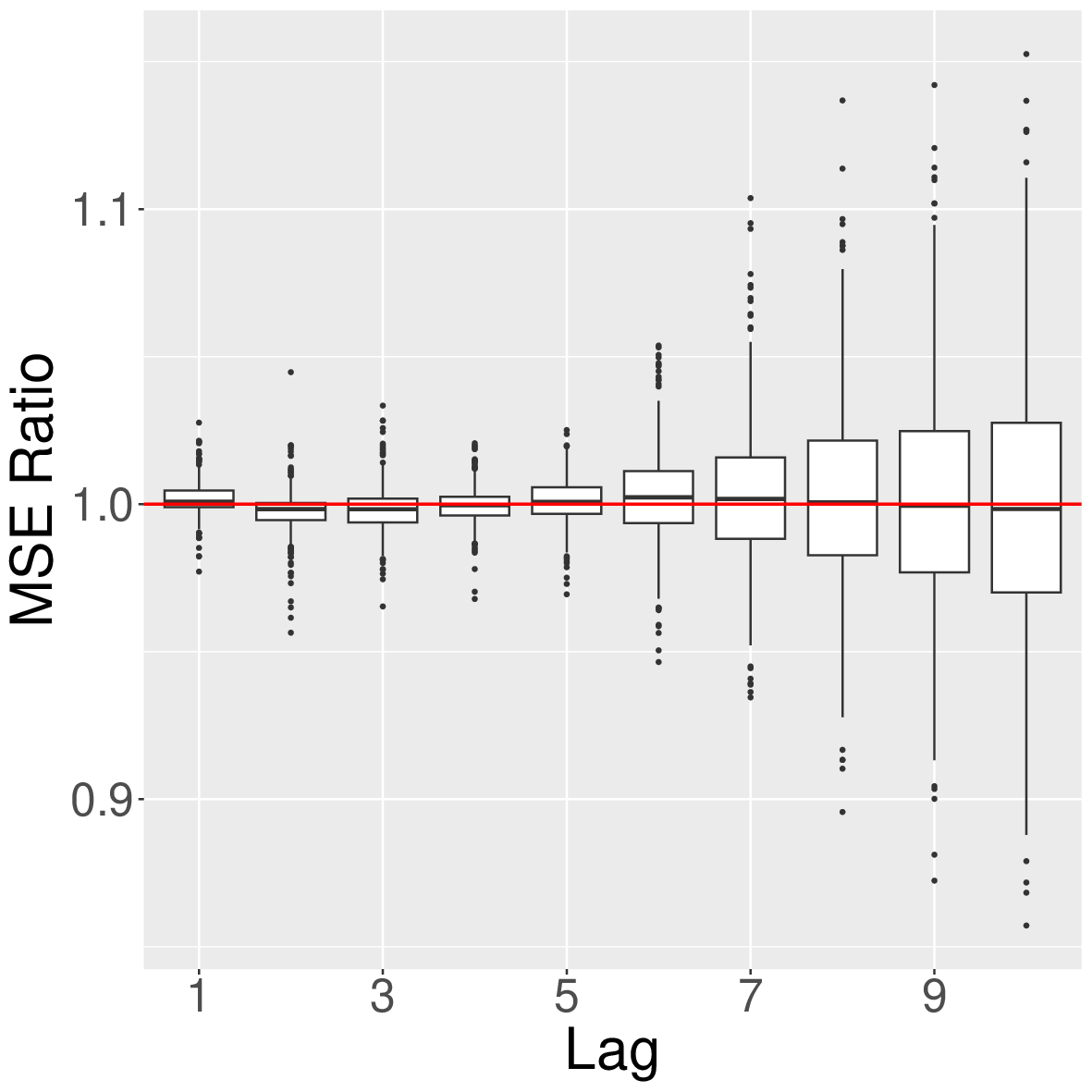}
}
	\caption{\it Here we compare the boxplots of estimated the MSE ratios (nonparametric vs.~ACF-, exponential-, long-memory estimators) for stock WM. A ratio of $<1$ indicates smaller MSEs of the nonparametric estimator.}
 \label{fig:ForecastRatios_WM}
\end{figure}

\end{appendix}

%% file: Tables/All_Results_NB_Exp_fixedt_v4.tex
\afterpage{
\begin{table}[ht]  \scriptsize
\centering
\setlength{\tabcolsep}{5pt} 
\renewcommand{\arraystretch}{1.1} 
\begin{tabular}{r|rrrrr|rrrrr|rrrrr}
	\toprule
		$t=$	&\multicolumn{5}{c}{Infeasible statistic}
	&\multicolumn{5}{c}{Feasible statistic without bias correction}
	&	\multicolumn{5}{c}{Feasible statistic with bias correction}
	\\
 $i\Delta_n$& \multicolumn{15}{c}{}\\
	\midrule
    $t$	& Mean & SD &  90\% & 95\% & 99\%
     & Mean & SD & 90\% & 95\% & 99\%
     & Mean & SD & 90\% & 95\% & 99\%
 		 \\ 
	\midrule
& \multicolumn{15}{c}{$\Delta_n=0.1, n= 2000$} \\
0 & -0.38 & 1.03 & 0.87 & 0.93 & 0.99 & -0.54 & 1.09 & 0.84 & 0.90 & 0.96 & -0.17 & 1.11 & 0.88 & 0.94 & 0.97 \\ 
  0.1 & -0.31 & 1.03 & 0.88 & 0.94 & 0.99 & -0.43 & 0.97 & 0.90 & 0.93 & 0.97 & -0.12 & 1.00 & 0.91 & 0.95 & 0.98 \\ 
  0.5 & -0.21 & 1.02 & 0.90 & 0.94 & 0.98 & -0.31 & 0.88 & 0.92 & 0.96 & 0.99 & -0.12 & 0.92 & 0.93 & 0.96 & 0.99 \\ 
  1 & -0.19 & 1.03 & 0.89 & 0.94 & 0.98 & -0.28 & 0.86 & 0.93 & 0.96 & 0.99 & -0.16 & 0.92 & 0.92 & 0.96 & 0.99 \\ 
  2 & -0.12 & 0.97 & 0.90 & 0.95 & 0.99 & -0.18 & 0.76 & 0.96 & 0.98 & 0.99 & -0.13 & 0.84 & 0.95 & 0.98 & 0.99 \\ 
  5 & -0.04 & 1.01 & 0.90 & 0.95 & 0.98 & -0.07 & 0.72 & 0.98 & 0.99 & 1.00 & -0.07 & 0.82 & 0.95 & 0.98 & 1.00 \\ 

\midrule
& \multicolumn{15}{c}{$\Delta_n=0.1, n= 5000$} \\
0 & -0.62 & 1.03 & 0.82 & 0.90 & 0.98 & -0.70 & 1.07 & 0.81 & 0.88 & 0.95 & -0.12 & 1.07 & 0.89 & 0.93 & 0.97 \\ 
  0.1 & -0.55 & 1.05 & 0.84 & 0.91 & 0.98 & -0.58 & 0.99 & 0.85 & 0.91 & 0.97 & -0.12 & 1.01 & 0.89 & 0.95 & 0.99 \\ 
  0.5 & -0.36 & 1.00 & 0.87 & 0.93 & 0.99 & -0.38 & 0.87 & 0.90 & 0.95 & 0.99 & -0.08 & 0.90 & 0.93 & 0.96 & 0.99 \\ 
  1 & -0.26 & 1.01 & 0.88 & 0.94 & 0.99 & -0.29 & 0.83 & 0.92 & 0.96 & 1.00 & -0.09 & 0.88 & 0.93 & 0.97 & 1.00 \\ 
  2 & -0.10 & 1.02 & 0.90 & 0.96 & 0.99 & -0.14 & 0.78 & 0.96 & 0.99 & 1.00 & -0.06 & 0.86 & 0.94 & 0.97 & 0.99 \\ 
  5 & -0.02 & 1.00 & 0.90 & 0.94 & 0.99 & -0.04 & 0.71 & 0.98 & 0.99 & 1.00 & -0.04 & 0.81 & 0.95 & 0.98 & 1.00 \\

\midrule
& \multicolumn{15}{c}{$\Delta_n=0.1, n= 10000$} \\
0 & -0.90 & 1.05 & 0.76 & 0.84 & 0.95 & -0.93 & 1.07 & 0.76 & 0.84 & 0.93 & -0.12 & 1.07 & 0.88 & 0.93 & 0.97 \\ 
  0.1 & -0.81 & 1.05 & 0.77 & 0.86 & 0.95 & -0.79 & 1.00 & 0.80 & 0.88 & 0.95 & -0.13 & 1.02 & 0.90 & 0.95 & 0.99 \\ 
  0.5 & -0.55 & 1.02 & 0.83 & 0.90 & 0.98 & -0.52 & 0.88 & 0.87 & 0.94 & 0.99 & -0.09 & 0.92 & 0.92 & 0.96 & 1.00 \\ 
  1 & -0.36 & 1.03 & 0.87 & 0.94 & 0.99 & -0.34 & 0.84 & 0.94 & 0.97 & 1.00 & -0.06 & 0.88 & 0.94 & 0.97 & 1.00 \\ 
  2 & -0.14 & 0.99 & 0.90 & 0.95 & 0.99 & -0.15 & 0.75 & 0.97 & 0.99 & 1.00 & -0.03 & 0.84 & 0.95 & 0.98 & 1.00 \\ 
  5 & 0.01 & 1.01 & 0.90 & 0.95 & 0.99 & -0.01 & 0.71 & 0.97 & 0.99 & 1.00 & -0.01 & 0.81 & 0.95 & 0.98 & 1.00 \\ 

\midrule
& \multicolumn{15}{c}{$\Delta_n=0.01, n= 2000$} \\
0 & -0.03 & 0.98 & 0.92 & 0.95 & 0.99 & -0.48 & 1.34 & 0.83 & 0.87 & 0.93 & -0.46 & 1.35 & 0.83 & 0.87 & 0.93 \\ 
  0.1 & 0.07 & 1.02 & 0.92 & 0.95 & 0.97 & -0.33 & 1.18 & 0.87 & 0.92 & 0.96 & -0.32 & 1.18 & 0.87 & 0.92 & 0.96 \\ 
  0.5 & 0.01 & 1.01 & 0.94 & 0.95 & 0.98 & -0.39 & 1.16 & 0.90 & 0.93 & 0.96 & -0.38 & 1.17 & 0.90 & 0.93 & 0.96 \\ 
  1 & -0.06 & 0.91 & 0.94 & 0.97 & 0.99 & -0.43 & 1.38 & 0.90 & 0.94 & 0.97 & -0.43 & 1.45 & 0.91 & 0.94 & 0.97 \\ 
  2 & -0.06 & 0.96 & 0.92 & 0.95 & 0.98 & 2.73 & 54.68 & 0.95 & 0.97 & 0.99 & 2.73 & 54.68 & 0.95 & 0.97 & 0.98 \\ 
  5 & -0.02 & 0.86 & 0.94 & 0.96 & 0.99 & 1.92 & 44.66 & 0.98 & 0.99 & 1.00 & 1.92 & 44.66 & 0.98 & 0.99 & 1.00 \\ 

\midrule
& \multicolumn{15}{c}{$\Delta_n=0.01, n= 5000$} \\
0 & -0.05 & 1.01 & 0.91 & 0.95 & 0.99 & -0.38 & 1.26 & 0.84 & 0.90 & 0.94 & -0.35 & 1.27 & 0.84 & 0.90 & 0.94 \\ 
  0.1 & 0.00 & 1.00 & 0.92 & 0.95 & 0.99 & -0.27 & 1.11 & 0.88 & 0.92 & 0.96 & -0.26 & 1.11 & 0.89 & 0.92 & 0.96 \\ 
  0.5 & -0.02 & 1.01 & 0.92 & 0.96 & 0.98 & -0.28 & 0.97 & 0.91 & 0.94 & 0.98 & -0.27 & 0.97 & 0.91 & 0.94 & 0.98 \\ 
  1 & -0.05 & 0.97 & 0.93 & 0.97 & 0.98 & -0.28 & 0.89 & 0.92 & 0.95 & 0.98 & -0.27 & 0.90 & 0.92 & 0.95 & 0.98 \\ 
  2 & -0.02 & 1.00 & 0.92 & 0.95 & 0.98 & 0.82 & 31.61 & 0.95 & 0.98 & 0.99 & 0.82 & 31.61 & 0.95 & 0.97 & 0.99 \\ 
  5 & 0.02 & 0.93 & 0.92 & 0.95 & 0.99 & -0.04 & 0.68 & 0.99 & 0.99 & 1.00 & -0.04 & 0.69 & 0.98 & 0.99 & 1.00 \\ 

\midrule
& \multicolumn{15}{c}{$\Delta_n=0.01, n= 10000$} \\
0 & -0.06 & 1.01 & 0.90 & 0.95 & 0.99 & -0.29 & 1.15 & 0.86 & 0.90 & 0.96 & -0.26 & 1.16 & 0.85 & 0.90 & 0.96 \\ 
  0.1 & -0.04 & 1.00 & 0.90 & 0.95 & 0.99 & -0.24 & 1.05 & 0.88 & 0.93 & 0.97 & -0.21 & 1.05 & 0.88 & 0.93 & 0.97 \\ 
  0.5 & -0.03 & 0.98 & 0.92 & 0.96 & 0.99 & -0.20 & 0.89 & 0.92 & 0.96 & 0.99 & -0.18 & 0.90 & 0.92 & 0.96 & 0.99 \\ 
  1 & -0.06 & 0.97 & 0.92 & 0.96 & 0.99 & -0.21 & 0.82 & 0.94 & 0.97 & 0.99 & -0.19 & 0.83 & 0.94 & 0.97 & 0.99 \\ 
  2 & -0.02 & 0.98 & 0.92 & 0.96 & 0.98 & -0.13 & 0.77 & 0.96 & 0.98 & 1.00 & -0.13 & 0.77 & 0.96 & 0.98 & 0.99 \\ 
  5 & 0.04 & 0.96 & 0.92 & 0.96 & 0.98 & -0.02 & 0.70 & 0.99 & 0.99 & 1.00 & -0.02 & 0.71 & 0.99 & 0.99 & 1.00 \\ 
   \bottomrule
\end{tabular}
\caption{ \label{tab:NB-Exp-fixedt} Negative binomial marginal distribution with exponential trawl function (fixed $t$): We consider three statistics and report their mean, standard deviation and coverage probabilities at levels 90\%, 95\% and 99\% based on 1000 Monte Carlo samples:
Infeasible statistic:
	$T^{IF}(t):=\sqrt{n\Delta_{n}}\left(\hat{a}(t)-a(t)\right)/(\sigma_a^2(t))^{1/2}$; feasible statistic without bias correction:
	$T^{F}:=\sqrt{n\Delta_{n}}\left(\hat{a}(t)-a(t)\right)/(\hat \sigma_a^2(t))^{1/2}$; Feasible statistic with bias correction:
	$T^{F}_{\mathrm{bias \, corrected}}:=\sqrt{n\Delta_{n}}\left(\hat{a}(t)-a(t)-\frac{1}{2}\Delta_n \hat a'(t)\right)/(\hat \sigma_a^2(t))^{1/2}$.}
\end{table}
	
}

%% file: Tables/All_Results_Gamma_Exp_fixedt_v4.tex
\afterpage{
\begin{table}[ht]  \scriptsize
\centering
\setlength{\tabcolsep}{5pt} 
\renewcommand{\arraystretch}{1.1} 
\begin{tabular}{r|rrrrr|rrrrr|rrrrr}
	\toprule
		$t=$	&\multicolumn{5}{c}{Infeasible statistic}
	&\multicolumn{5}{c}{Feasible statistic without bias correction}
	&	\multicolumn{5}{c}{Feasible statistic with bias correction}
	\\
 $i\Delta_n$& \multicolumn{15}{c}{}\\
	\midrule
		%
	%
 $t$	& Mean & SD &  90\% & 95\% & 99\%
     & Mean & SD & 90\% & 95\% & 99\%
     & Mean & SD & 90\% & 95\% & 99\%
 		 \\ 
	\midrule
	&\multicolumn{15}{c}{$\Delta_n=0.1, n= 2000$}
		\\
0 & -0.28 & 0.96 & 0.91 & 0.95 & 0.99 & -0.68 & 1.35 & 0.78 & 0.84 & 0.91 & -0.45 & 1.33 & 0.81 & 0.87 & 0.93 \\ 
  0.1 & -0.23 & 0.97 & 0.91 & 0.95 & 0.99 & -0.58 & 1.28 & 0.82 & 0.86 & 0.93 & -0.38 & 1.30 & 0.83 & 0.89 & 0.94 \\ 
  0.5 & -0.18 & 0.94 & 0.93 & 0.97 & 0.99 & -0.45 & 1.07 & 0.87 & 0.91 & 0.97 & -0.31 & 1.08 & 0.89 & 0.92 & 0.97 \\ 
  1 & -0.16 & 0.96 & 0.92 & 0.96 & 0.99 & -0.42 & 1.02 & 0.88 & 0.92 & 0.97 & -0.32 & 1.06 & 0.88 & 0.93 & 0.97 \\ 
  2 & -0.09 & 0.93 & 0.93 & 0.96 & 0.99 & 0.72 & 31.61 & 0.93 & 0.96 & 0.98 & 0.76 & 31.61 & 0.92 & 0.96 & 0.98 \\ 
  5 & 0.02 & 0.97 & 0.91 & 0.96 & 0.99 & 4.89 & 70.51 & 0.95 & 0.97 & 0.98 & 4.88 & 70.51 & 0.93 & 0.96 & 0.98 \\ 
		%
	%
	\midrule
&\multicolumn{15}{c}{$\Delta_n=0.1, n= 5000$}
\\
0 & -0.39 & 0.99 & 0.88 & 0.94 & 0.99 & -0.65 & 1.21 & 0.81 & 0.86 & 0.93 & -0.30 & 1.20 & 0.84 & 0.90 & 0.95 \\ 
  0.1 & -0.33 & 0.99 & 0.90 & 0.95 & 0.99 & -0.56 & 1.14 & 0.84 & 0.89 & 0.94 & -0.24 & 1.15 & 0.88 & 0.92 & 0.96 \\ 
  0.5 & -0.26 & 0.99 & 0.90 & 0.95 & 0.99 & -0.44 & 1.02 & 0.87 & 0.92 & 0.98 & -0.21 & 1.03 & 0.89 & 0.93 & 0.98 \\ 
  1 & -0.23 & 1.00 & 0.91 & 0.95 & 0.99 & -0.39 & 0.98 & 0.89 & 0.95 & 0.98 & -0.24 & 1.01 & 0.90 & 0.95 & 0.98 \\ 
  2 & -0.11 & 0.94 & 0.92 & 0.96 & 0.99 & -0.22 & 0.82 & 0.95 & 0.97 & 0.99 & -0.15 & 0.90 & 0.94 & 0.97 & 0.99 \\ 
  5 & -0.02 & 0.99 & 0.90 & 0.95 & 0.98 & 2.92 & 54.67 & 0.97 & 0.99 & 0.99 & 2.91 & 54.67 & 0.95 & 0.97 & 0.99 \\ 
		%
	%
		\midrule
	&\multicolumn{15}{c}{$\Delta_n=0.1, n= 10000$}
	\\
0 & -0.51 & 0.96 & 0.87 & 0.93 & 0.99 & -0.68 & 1.12 & 0.82 & 0.88 & 0.94 & -0.21 & 1.11 & 0.87 & 0.92 & 0.97 \\ 
  0.1 & -0.46 & 0.97 & 0.87 & 0.94 & 0.99 & -0.60 & 1.07 & 0.82 & 0.89 & 0.96 & -0.17 & 1.07 & 0.89 & 0.94 & 0.98 \\ 
  0.5 & -0.31 & 0.97 & 0.90 & 0.95 & 0.99 & -0.42 & 0.97 & 0.89 & 0.93 & 0.98 & -0.10 & 0.99 & 0.91 & 0.95 & 0.99 \\ 
  1 & -0.27 & 0.98 & 0.90 & 0.95 & 0.99 & -0.37 & 0.94 & 0.90 & 0.94 & 0.98 & -0.15 & 0.97 & 0.91 & 0.95 & 0.99 \\ 
  2 & -0.14 & 0.97 & 0.91 & 0.96 & 0.99 & -0.21 & 0.83 & 0.94 & 0.97 & 0.99 & -0.11 & 0.89 & 0.93 & 0.97 & 1.00 \\ 
  5 & 0.00 & 0.99 & 0.92 & 0.96 & 0.99 & -0.06 & 0.74 & 0.98 & 0.99 & 1.00 & -0.07 & 0.83 & 0.96 & 0.98 & 0.99 \\ 
	\midrule
	%
	%
	\midrule
	&\multicolumn{15}{c}{$\Delta_n=0.01, n= 2000$}
	\\
0 & 0.01 & 1.00 & 0.94 & 0.96 & 0.97 & -1.26 & 2.99 & 0.73 & 0.76 & 0.82 & -1.25 & 2.99 & 0.73 & 0.76 & 0.82 \\ 
  0.1 & 0.05 & 1.03 & 0.94 & 0.95 & 0.97 & -0.01 & 31.72 & 0.76 & 0.80 & 0.85 & 0.00 & 31.72 & 0.76 & 0.80 & 0.85 \\ 
  0.5 & 0.01 & 0.99 & 0.94 & 0.96 & 0.97 & 1.11 & 44.74 & 0.78 & 0.83 & 0.88 & 1.12 & 44.74 & 0.78 & 0.83 & 0.88 \\ 
  1 & -0.01 & 1.03 & 0.94 & 0.95 & 0.97 & 12.22 & 113.32 & 0.80 & 0.85 & 0.90 & 12.22 & 113.32 & 0.81 & 0.85 & 0.90 \\ 
  2 & -0.01 & 1.00 & 0.94 & 0.96 & 0.97 & 23.46 & 153.06 & 0.87 & 0.90 & 0.94 & 23.45 & 153.06 & 0.87 & 0.90 & 0.93 \\ 
  5 & -0.05 & 0.92 & 0.94 & 0.96 & 0.98 & 30.81 & 173.26 & 0.94 & 0.95 & 0.96 & 30.81 & 173.26 & 0.93 & 0.94 & 0.96 \\
   %
   %
   \midrule
   &\multicolumn{15}{c}{$\Delta_n=0.01, n= 5000$}
   \\
 0 & 0.01 & 1.03 & 0.93 & 0.95 & 0.97 & -0.67 & 1.62 & 0.79 & 0.83 & 0.88 & -0.65 & 1.62 & 0.79 & 0.83 & 0.88 \\ 
  0.1 & 0.04 & 1.02 & 0.92 & 0.94 & 0.97 & -0.56 & 1.45 & 0.82 & 0.85 & 0.91 & -0.54 & 1.45 & 0.82 & 0.85 & 0.91 \\ 
  0.5 & -0.02 & 0.98 & 0.94 & 0.96 & 0.97 & -0.53 & 1.24 & 0.82 & 0.88 & 0.94 & -0.52 & 1.24 & 0.82 & 0.88 & 0.94 \\ 
  1 & 0.03 & 1.04 & 0.93 & 0.95 & 0.97 & -0.46 & 1.14 & 0.86 & 0.91 & 0.96 & -0.45 & 1.15 & 0.86 & 0.91 & 0.96 \\ 
  2 & 0.00 & 1.03 & 0.93 & 0.95 & 0.97 & 9.64 & 99.49 & 0.90 & 0.93 & 0.96 & 9.64 & 99.49 & 0.89 & 0.92 & 0.96 \\ 
  5 & -0.02 & 0.95 & 0.93 & 0.96 & 0.98 & 29.86 & 170.52 & 0.93 & 0.95 & 0.96 & 29.86 & 170.52 & 0.93 & 0.94 & 0.96 \\ 
   %
   %
   \midrule
   &\multicolumn{15}{c}{$\Delta_n=0.01, n= 10000$}
   \\
 0 & 0.01 & 1.03 & 0.92 & 0.95 & 0.97 & -0.49 & 1.48 & 0.82 & 0.87 & 0.92 & -0.47 & 1.48 & 0.82 & 0.87 & 0.92 \\ 
  0.1 & 0.02 & 1.02 & 0.92 & 0.95 & 0.98 & -0.42 & 1.37 & 0.83 & 0.89 & 0.94 & -0.41 & 1.37 & 0.83 & 0.89 & 0.94 \\ 
  0.5 & -0.02 & 1.04 & 0.92 & 0.96 & 0.98 & -0.41 & 1.17 & 0.87 & 0.91 & 0.95 & -0.40 & 1.18 & 0.87 & 0.91 & 0.95 \\ 
  1 & 0.04 & 1.03 & 0.93 & 0.95 & 0.97 & -0.29 & 0.97 & 0.91 & 0.94 & 0.98 & -0.29 & 0.97 & 0.91 & 0.94 & 0.98 \\ 
  2 & 0.01 & 1.03 & 0.91 & 0.94 & 0.97 & 3.73 & 63.11 & 0.93 & 0.95 & 0.98 & 3.74 & 63.11 & 0.92 & 0.95 & 0.98 \\ 
  5 & -0.01 & 0.97 & 0.92 & 0.95 & 0.98 & 14.82 & 121.52 & 0.95 & 0.96 & 0.97 & 14.82 & 121.52 & 0.95 & 0.96 & 0.97 \\ 
   \bottomrule
\end{tabular}
\caption{ \label{tab:Gamma-Exp-fixedt} Gamma marginal distribution with exponential trawl function (fixed $t$): We consider three statistics and report their mean, standard deviation and coverage probabilities at levels 90\%, 95\% and 99\% based on 1000 Monte Carlo samples:
Infeasible statistic:
	$T^{IF}(t):=\sqrt{n\Delta_{n}}\left(\hat{a}(t)-a(t)\right)/(\sigma_a^2(t))^{1/2}$; feasible statistic without bias correction:
	$T^{F}:=\sqrt{n\Delta_{n}}\left(\hat{a}(t)-a(t)\right)/(\hat \sigma_a^2(t))^{1/2}$; Feasible statistic with bias correction:
	$T^{F}_{\mathrm{bias \, corrected}}:=\sqrt{n\Delta_{n}}\left(\hat{a}(t)-a(t)-\frac{1}{2}\Delta_n \hat a'(t)\right)/(\hat \sigma_a^2(t))^{1/2}$.}
\end{table}
	
}

%% file: Tables/All_Results_Gaussian_Exp_fixedt_v4.tex
\afterpage{
\begin{table}[ht]  \scriptsize
\centering
\setlength{\tabcolsep}{5pt} 
\renewcommand{\arraystretch}{1.1} 
\begin{tabular}{r|rrrrr|rrrrr|rrrrr}
	\toprule
		$t=$	&\multicolumn{5}{c}{Infeasible statistic}
	&\multicolumn{5}{c}{Feasible statistic without bias correction}
	&	\multicolumn{5}{c}{Feasible statistic with bias correction}
	\\
 $i\Delta_n$& \multicolumn{15}{c}{}\\
	\midrule
\midrule
		%
	%
 $t$	& Mean & SD &90\% & 95\% & 99\%
 	& Mean & SD &90\% & 95\% & 99\%
 		& Mean & SD &90\% & 95\% & 99\%
 		 \\ 
	\midrule
	&\multicolumn{15}{c}{$\Delta_n=0.1, n= 2000$}
		\\
0 &  &  &  &  &  & -0.96 & 0.61 & 0.87 & 0.95 & 1.00 & -0.09 & 0.75 & 0.97 & 0.99 & 1.00 \\ 
  0.1 & -0.99 & 1.18 & 0.69 & 0.79 & 0.90 & -0.53 & 0.63 & 0.95 & 0.99 & 1.00 & -0.03 & 0.69 & 0.98 & 0.99 & 1.00 \\ 
  0.5 & -0.34 & 1.02 & 0.88 & 0.94 & 0.98 & -0.27 & 0.68 & 0.97 & 0.99 & 1.00 & -0.02 & 0.74 & 0.97 & 0.99 & 1.00 \\ 
  1 & -0.22 & 1.02 & 0.88 & 0.94 & 0.99 & -0.20 & 0.71 & 0.97 & 0.99 & 1.00 & -0.06 & 0.77 & 0.97 & 0.99 & 1.00 \\ 
  2 & -0.09 & 0.96 & 0.91 & 0.96 & 0.99 & -0.09 & 0.66 & 0.99 & 1.00 & 1.00 & -0.03 & 0.74 & 0.98 & 0.99 & 1.00 \\ 
  5 & -0.04 & 0.99 & 0.91 & 0.95 & 0.99 & -0.03 & 0.69 & 0.99 & 1.00 & 1.00 & -0.03 & 0.80 & 0.97 & 0.99 & 1.00 \\  
		%
	%
	\midrule
&\multicolumn{15}{c}{$\Delta_n=0.1, n= 5000$}
\\
0 &  &  &  & &  & -1.48 & 0.62 & 0.61 & 0.78 & 0.96 & -0.11 & 0.76 & 0.97 & 0.99 & 1.00 \\ 
  0.1 & -1.60 & 1.21 & 0.50 & 0.61 & 0.79 & -0.84 & 0.64 & 0.90 & 0.96 & 1.00 & -0.05 & 0.70 & 0.98 & 1.00 & 1.00 \\ 
  0.5 & -0.58 & 1.03 & 0.83 & 0.91 & 0.97 & -0.41 & 0.69 & 0.96 & 0.98 & 1.00 & -0.00 & 0.74 & 0.97 & 0.99 & 1.00 \\ 
  1 & -0.37 & 1.00 & 0.88 & 0.93 & 0.98 & -0.28 & 0.70 & 0.96 & 0.98 & 1.00 & -0.05 & 0.76 & 0.97 & 0.99 & 1.00 \\ 
  2 & -0.17 & 0.98 & 0.90 & 0.95 & 0.99 & -0.13 & 0.68 & 0.98 & 1.00 & 1.00 & -0.05 & 0.77 & 0.97 & 0.99 & 1.00 \\ 
  5 & 0.01 & 1.00 & 0.91 & 0.95 & 0.99 & 0.00 & 0.70 & 0.99 & 1.00 & 1.00 & 0.00 & 0.81 & 0.96 & 0.99 & 1.00 \\  
		%
	%
		\midrule
	&\multicolumn{15}{c}{$\Delta_n=0.1, n= 10000$}
	\\
0 &  &  &  &  &  & -2.08 & 0.62 & 0.24 & 0.42 & 0.78 & -0.14 & 0.77 & 0.97 & 0.99 & 1.00 \\ 
  0.1 & -2.29 & 1.20 & 0.30 & 0.40 & 0.59 & -1.19 & 0.64 & 0.77 & 0.88 & 0.98 & -0.05 & 0.71 & 0.98 & 1.00 & 1.00 \\ 
  0.5 & -0.88 & 1.04 & 0.77 & 0.86 & 0.95 & -0.59 & 0.69 & 0.93 & 0.97 & 0.99 & -0.02 & 0.74 & 0.98 & 0.99 & 1.00 \\ 
  1 & -0.50 & 1.02 & 0.85 & 0.92 & 0.97 & -0.37 & 0.71 & 0.96 & 0.98 & 1.00 & -0.04 & 0.78 & 0.95 & 0.98 & 1.00 \\ 
  2 & -0.23 & 1.00 & 0.89 & 0.95 & 0.99 & -0.17 & 0.70 & 0.98 & 0.99 & 1.00 & -0.05 & 0.78 & 0.97 & 0.99 & 1.00 \\ 
  5 & -0.01 & 1.02 & 0.90 & 0.95 & 0.99 & -0.01 & 0.72 & 0.98 & 0.99 & 1.00 & -0.01 & 0.81 & 0.96 & 0.99 & 1.00 \\  
	\midrule
	%
	%
	\midrule
	&\multicolumn{15}{c}{$\Delta_n=0.01, n= 2000$}
	\\
0 &  &  &  &  &  & -0.11 & 0.56 & 0.99 & 1.00 & 1.00 & -0.02 & 0.71 & 0.97 & 0.99 & 1.00 \\ 
  0.1 & 0.23 & 1.07 & 0.86 & 0.92 & 0.98 & 0.08 & 0.70 & 0.98 & 1.00 & 1.00 & 0.11 & 0.71 & 0.98 & 0.99 & 1.00 \\ 
  0.5 & 0.05 & 0.98 & 0.92 & 0.96 & 0.99 & -0.09 & 0.71 & 0.97 & 0.99 & 1.00 & -0.08 & 0.72 & 0.97 & 0.99 & 1.00 \\ 
  1 & -0.03 & 1.01 & 0.92 & 0.95 & 0.98 & -0.14 & 0.71 & 0.97 & 0.99 & 1.00 & -0.14 & 0.72 & 0.97 & 0.99 & 1.00 \\ 
  2 & -0.08 & 0.96 & 0.91 & 0.95 & 0.99 & -0.12 & 0.66 & 0.99 & 1.00 & 1.00 & -0.12 & 0.67 & 0.99 & 1.00 & 1.00 \\ 
  5 & -0.05 & 0.88 & 0.94 & 0.97 & 0.99 & -0.03 & 0.58 & 1.00 & 1.00 & 1.00 & -0.02 & 0.60 & 1.00 & 1.00 & 1.00 \\ 
   %
   %
   \midrule
   &\multicolumn{15}{c}{$\Delta_n=0.01, n= 5000$}
   \\
 0 &  &  &  &  &  & -0.15 & 0.59 & 0.99 & 1.00 & 1.00 & -0.01 & 0.72 & 0.98 & 0.99 & 1.00 \\ 
  0.1 & 0.15 & 1.06 & 0.90 & 0.94 & 0.98 & 0.06 & 0.70 & 0.98 & 1.00 & 1.00 & 0.10 & 0.71 & 0.98 & 0.99 & 1.00 \\ 
  0.5 & 0.05 & 0.97 & 0.91 & 0.95 & 0.99 & -0.05 & 0.68 & 0.98 & 0.99 & 1.00 & -0.03 & 0.69 & 0.98 & 0.99 & 1.00 \\ 
  1 & -0.01 & 1.03 & 0.89 & 0.94 & 0.98 & -0.10 & 0.73 & 0.97 & 0.99 & 1.00 & -0.09 & 0.74 & 0.97 & 0.99 & 1.00 \\ 
  2 & -0.01 & 0.97 & 0.90 & 0.96 & 0.99 & -0.04 & 0.68 & 0.99 & 0.99 & 1.00 & -0.04 & 0.68 & 0.98 & 0.99 & 1.00 \\ 
  5 & -0.03 & 0.98 & 0.91 & 0.95 & 0.99 & -0.02 & 0.67 & 0.99 & 1.00 & 1.00 & -0.01 & 0.68 & 0.99 & 1.00 & 1.00 \\ 
   %
   %
   \midrule
   &\multicolumn{15}{c}{$\Delta_n=0.01, n= 10000$}
   \\
  0 &  &  &  &  &  & -0.19 & 0.61 & 0.99 & 1.00 & 1.00 & 0.02 & 0.73 & 0.97 & 0.99 & 1.00 \\ 
  0.1 & 0.07 & 1.05 & 0.89 & 0.95 & 0.98 & 0.02 & 0.70 & 0.98 & 0.99 & 1.00 & 0.07 & 0.71 & 0.98 & 0.99 & 1.00 \\ 
  0.5 & -0.01 & 0.98 & 0.90 & 0.95 & 0.99 & -0.06 & 0.69 & 0.97 & 0.99 & 1.00 & -0.05 & 0.70 & 0.97 & 0.99 & 1.00 \\ 
  1 & 0.03 & 1.06 & 0.89 & 0.93 & 0.98 & -0.05 & 0.74 & 0.96 & 0.99 & 1.00 & -0.04 & 0.75 & 0.97 & 0.99 & 1.00 \\ 
  2 & -0.00 & 0.96 & 0.91 & 0.96 & 0.99 & -0.03 & 0.67 & 0.98 & 1.00 & 1.00 & -0.02 & 0.68 & 0.98 & 1.00 & 1.00 \\ 
  5 & -0.05 & 1.01 & 0.91 & 0.95 & 0.99 & -0.02 & 0.71 & 0.98 & 1.00 & 1.00 & -0.02 & 0.72 & 0.98 & 1.00 & 1.00 \\ 
   \bottomrule
\end{tabular}
\caption{ \label{tab:Gaussian-Exp-fixedt} Gaussian marginal distribution with exponential trawl function (fixed $t$): We consider three statistics and report their mean, standard deviation and coverage probabilities at levels 90\%, 95\% and 99\% based on 1000 Monte Carlo samples:
Infeasible statistic:
	$T^{IF}(t):=\sqrt{n\Delta_{n}}\left(\hat{a}(t)-a(t)\right)/(\sigma_a^2(t))^{1/2}$; feasible statistic without bias correction:
	$T^{F}:=\sqrt{n\Delta_{n}}\left(\hat{a}(t)-a(t)\right)/(\hat \sigma_a^2(t))^{1/2}$; Feasible statistic with bias correction:
	$T^{F}_{\mathrm{bias \, corrected}}:=\sqrt{n\Delta_{n}}\left(\hat{a}(t)-a(t)-\frac{1}{2}\Delta_n \hat a'(t)\right)/(\hat \sigma_a^2(t))^{1/2}$.}
\end{table}
	
}

%% file: Tables/All_Results_NB_LM_fixedt_v4.tex
\afterpage{%
	\begin{table}[ht]  \scriptsize
\centering
\setlength{\tabcolsep}{5pt} 
\renewcommand{\arraystretch}{1.1}
\begin{tabular}{r|rrrrr|rrrrr|rrrrr}
	\toprule
	$t=$	&\multicolumn{5}{c}{Infeasible statistic}
&\multicolumn{5}{c}{Feasible statistic without bias correction}
&	\multicolumn{5}{c}{Feasible statistic with bias correction}
\\
$i\Delta_n$& \multicolumn{15}{c}{}\\
	\midrule
		%
	%
 $t$	& Mean & SD &90\% & 95\% & 99\%
 	& Mean & SD &90\% & 95\% & 99\%
 		& Mean & SD &90\% & 95\% & 99\%
 		 \\ 
	\midrule
	&\multicolumn{15}{c}{$\Delta_n=0.1, n= 2000$}
		\\
0 & -0.26 & 1.01 & 0.87 & 0.94 & 0.99 & -0.41 & 1.05 & 0.87 & 0.92 & 0.97 & -0.13 & 1.07 & 0.88 & 0.94 & 0.97 \\ 
  0.5 & -0.06 & 1.01 & 0.89 & 0.94 & 0.98 & -0.17 & 0.84 & 0.94 & 0.97 & 0.99 & -0.05 & 0.87 & 0.94 & 0.97 & 0.99 \\ 
  2 & 0.03 & 1.00 & 0.90 & 0.94 & 0.99 & -0.09 & 0.75 & 0.97 & 0.99 & 1.00 & -0.06 & 0.79 & 0.96 & 0.99 & 0.99 \\ 
  4 & 0.04 & 0.97 & 0.93 & 0.97 & 0.99 & -0.06 & 0.72 & 0.98 & 0.99 & 1.00 & -0.05 & 0.76 & 0.97 & 0.98 & 1.00 \\ 
  7 & 0.02 & 0.96 & 0.92 & 0.95 & 0.99 & -0.05 & 0.69 & 0.98 & 0.99 & 1.00 & -0.05 & 0.74 & 0.97 & 0.99 & 1.00 \\ 
  10 & 0.00 & 0.97 & 0.92 & 0.95 & 0.98 & -0.05 & 0.68 & 0.99 & 1.00 & 1.00 & -0.05 & 0.72 & 0.98 & 0.99 & 1.00 \\ 
		%
	%
	\midrule
&\multicolumn{15}{c}{$\Delta_n=0.1, n= 5000$}
\\
0 & -0.40 & 1.02 & 0.87 & 0.94 & 0.98 & -0.48 & 1.00 & 0.87 & 0.92 & 0.98 & -0.06 & 1.02 & 0.90 & 0.95 & 0.99 \\ 
  0.5 & -0.15 & 1.05 & 0.89 & 0.94 & 0.99 & -0.20 & 0.88 & 0.94 & 0.97 & 0.99 & -0.01 & 0.91 & 0.94 & 0.97 & 0.99 \\ 
  2 & -0.02 & 1.04 & 0.88 & 0.94 & 0.99 & -0.09 & 0.79 & 0.96 & 0.98 & 0.99 & -0.03 & 0.83 & 0.95 & 0.98 & 0.99 \\ 
  4 & -0.00 & 1.01 & 0.90 & 0.95 & 0.98 & -0.06 & 0.75 & 0.96 & 0.98 & 1.00 & -0.05 & 0.79 & 0.96 & 0.98 & 0.99 \\ 
  7 & 0.00 & 0.95 & 0.92 & 0.96 & 0.99 & -0.04 & 0.68 & 0.98 & 0.99 & 1.00 & -0.04 & 0.72 & 0.98 & 0.99 & 1.00 \\ 
  10 & 0.00 & 0.96 & 0.92 & 0.96 & 0.99 & -0.03 & 0.68 & 0.98 & 0.99 & 1.00 & -0.04 & 0.73 & 0.98 & 0.99 & 1.00 \\ 
		%
	%
		\midrule
	&\multicolumn{15}{c}{$\Delta_n=0.1, n= 10000$}
	\\
0 & -0.61 & 1.02 & 0.83 & 0.91 & 0.98 & -0.65 & 0.99 & 0.84 & 0.91 & 0.97 & -0.06 & 0.99 & 0.90 & 0.95 & 0.99 \\ 
  0.5 & -0.27 & 1.03 & 0.88 & 0.94 & 0.99 & -0.28 & 0.86 & 0.93 & 0.97 & 0.99 & 0.00 & 0.89 & 0.94 & 0.97 & 1.00 \\ 
  2 & -0.05 & 1.02 & 0.90 & 0.95 & 0.98 & -0.09 & 0.78 & 0.96 & 0.98 & 1.00 & -0.00 & 0.82 & 0.95 & 0.98 & 1.00 \\ 
  4 & -0.04 & 0.98 & 0.91 & 0.95 & 0.99 & -0.07 & 0.73 & 0.98 & 0.99 & 1.00 & -0.05 & 0.77 & 0.97 & 0.99 & 1.00 \\ 
  7 & -0.01 & 0.94 & 0.93 & 0.96 & 0.99 & -0.04 & 0.68 & 0.98 & 0.99 & 1.00 & -0.03 & 0.72 & 0.98 & 0.99 & 1.00 \\ 
  10 & -0.02 & 0.97 & 0.91 & 0.96 & 0.99 & -0.04 & 0.69 & 0.99 & 0.99 & 1.00 & -0.05 & 0.73 & 0.98 & 0.99 & 1.00 \\ 
	\midrule
	%
	%
	\midrule
	&\multicolumn{15}{c}{$\Delta_n=0.01, n= 2000$}
	\\
0 & -0.03 & 0.98 & 0.92 & 0.95 & 0.99 & -0.48 & 1.34 & 0.83 & 0.87 & 0.93 & -0.46 & 1.35 & 0.83 & 0.87 & 0.93 \\ 
  0.1 & 0.07 & 1.02 & 0.92 & 0.95 & 0.97 & -0.33 & 1.18 & 0.87 & 0.92 & 0.96 & -0.32 & 1.18 & 0.87 & 0.92 & 0.96 \\ 
  0.5 & 0.01 & 1.01 & 0.94 & 0.95 & 0.98 & -0.39 & 1.16 & 0.90 & 0.93 & 0.96 & -0.38 & 1.17 & 0.90 & 0.93 & 0.96 \\ 
  1 & -0.06 & 0.91 & 0.94 & 0.97 & 0.99 & -0.43 & 1.38 & 0.90 & 0.94 & 0.97 & -0.43 & 1.45 & 0.91 & 0.94 & 0.97 \\ 
  2 & -0.06 & 0.96 & 0.92 & 0.95 & 0.98 & 2.73 & 54.68 & 0.95 & 0.97 & 0.99 & 2.73 & 54.68 & 0.95 & 0.97 & 0.98 \\ 
  5 & -0.02 & 0.86 & 0.94 & 0.96 & 0.99 & 1.92 & 44.66 & 0.98 & 0.99 & 1.00 & 1.92 & 44.66 & 0.98 & 0.99 & 1.00 \\ 
   %
   %
   \midrule
   &\multicolumn{15}{c}{$\Delta_n=0.01, n= 5000$}
   \\
0 & -0.05 & 1.01 & 0.91 & 0.95 & 0.99 & -0.38 & 1.26 & 0.84 & 0.90 & 0.94 & -0.35 & 1.27 & 0.84 & 0.90 & 0.94 \\ 
  0.1 & 0.00 & 1.00 & 0.92 & 0.95 & 0.99 & -0.27 & 1.11 & 0.88 & 0.92 & 0.96 & -0.26 & 1.11 & 0.89 & 0.92 & 0.96 \\ 
  0.5 & -0.02 & 1.01 & 0.92 & 0.96 & 0.98 & -0.28 & 0.97 & 0.91 & 0.94 & 0.98 & -0.27 & 0.97 & 0.91 & 0.94 & 0.98 \\ 
  1 & -0.05 & 0.97 & 0.93 & 0.97 & 0.98 & -0.28 & 0.89 & 0.92 & 0.95 & 0.98 & -0.27 & 0.90 & 0.92 & 0.95 & 0.98 \\ 
  2 & -0.02 & 1.00 & 0.92 & 0.95 & 0.98 & 0.82 & 31.61 & 0.95 & 0.98 & 0.99 & 0.82 & 31.61 & 0.95 & 0.97 & 0.99 \\ 
  5 & 0.02 & 0.93 & 0.92 & 0.95 & 0.99 & -0.04 & 0.68 & 0.99 & 0.99 & 1.00 & -0.04 & 0.69 & 0.98 & 0.99 & 1.00 \\ 
   %
   %
   \midrule
   &\multicolumn{15}{c}{$\Delta_n=0.01, n= 10000$}
   \\
 0 & -0.06 & 1.01 & 0.90 & 0.95 & 0.99 & -0.29 & 1.15 & 0.86 & 0.90 & 0.96 & -0.26 & 1.16 & 0.85 & 0.90 & 0.96 \\ 
  0.1 & -0.04 & 1.00 & 0.90 & 0.95 & 0.99 & -0.24 & 1.05 & 0.88 & 0.93 & 0.97 & -0.21 & 1.05 & 0.88 & 0.93 & 0.97 \\ 
  0.5 & -0.03 & 0.98 & 0.92 & 0.96 & 0.99 & -0.20 & 0.89 & 0.92 & 0.96 & 0.99 & -0.18 & 0.90 & 0.92 & 0.96 & 0.99 \\ 
  1 & -0.06 & 0.97 & 0.92 & 0.96 & 0.99 & -0.21 & 0.82 & 0.94 & 0.97 & 0.99 & -0.19 & 0.83 & 0.94 & 0.97 & 0.99 \\ 
  2 & -0.02 & 0.98 & 0.92 & 0.96 & 0.98 & -0.13 & 0.77 & 0.96 & 0.98 & 1.00 & -0.13 & 0.77 & 0.96 & 0.98 & 0.99 \\ 
  5 & 0.04 & 0.96 & 0.92 & 0.96 & 0.98 & -0.02 & 0.70 & 0.99 & 0.99 & 1.00 & -0.02 & 0.71 & 0.99 & 0.99 & 1.00 \\ 
   \bottomrule
\end{tabular}
\caption{ \label{tab:NB-LM-fixedt} Negative Binomial marginal distribution with supGamma trawl function (fixed $t$): We consider three statistics and report their mean, standard deviation and coverage probabilities at levels 90\%, 95\% and 99\% based on 1000 Monte Carlo samples: Infeasible statistic: $T^{IF}(t):=\sqrt{n\Delta_{n}}\left(\hat{a}(t)-a(t)\right)/(\sigma_a^2(t))^{1/2}$; feasible statistic without bias correction: $T^{F}:=\sqrt{n\Delta_{n}}\left(\hat{a}(t)-a(t)\right)/(\hat \sigma_a^2(t))^{1/2}$; Feasible statistic with bias correction: $T^{F}_{\mathrm{bias \, corrected}}:=\sqrt{n\Delta_{n}}\left(\hat{a}(t)-a(t)-\frac{1}{2}\Delta_n \hat a'(t)\right)/(\hat \sigma_a^2(t))^{1/2}$.}
\end{table}
}

%% file: Tables/All_Results_Gamma_LM_fixedt_v4.tex
\afterpage{%
\begin{table}[ht]  \scriptsize
\centering
\setlength{\tabcolsep}{5pt} 
\renewcommand{\arraystretch}{1.1} 
\begin{tabular}{r|rrrrr|rrrrr|rrrrr}
	\toprule
	$t=$	&\multicolumn{5}{c}{Infeasible statistic}
&\multicolumn{5}{c}{Feasible statistic without bias correction}
&	\multicolumn{5}{c}{Feasible statistic with bias correction}
\\	
$i\Delta_n$& \multicolumn{15}{c}{}\\
\midrule
		%
	%
 $t$	& Mean & SD &90\% & 95\% & 99\%
 	& Mean & SD &90\% & 95\% & 99\%
 		& Mean & SD &90\% & 95\% & 99\%
 		 \\ 
	\midrule
	&\multicolumn{15}{c}{$\Delta_n=0.1, n= 2000$}
		\\
0 & -0.15 & 0.98 & 0.91 & 0.95 & 0.99 & -0.50 & 1.26 & 0.84 & 0.89 & 0.94 & -0.33 & 1.26 & 0.85 & 0.90 & 0.94 \\ 
  0.5 & -0.01 & 0.96 & 0.93 & 0.95 & 0.98 & -0.25 & 0.98 & 0.90 & 0.94 & 0.98 & -0.15 & 1.00 & 0.90 & 0.95 & 0.98 \\ 
  2 & 0.04 & 1.01 & 0.91 & 0.95 & 0.98 & -0.19 & 0.92 & 0.92 & 0.96 & 0.99 & -0.16 & 0.95 & 0.92 & 0.96 & 0.98 \\ 
  4 & 0.01 & 1.01 & 0.90 & 0.94 & 0.98 & -0.19 & 0.88 & 0.95 & 0.96 & 0.98 & -0.17 & 0.90 & 0.94 & 0.97 & 0.99 \\ 
  7 & 0.01 & 1.02 & 0.91 & 0.95 & 0.98 & -0.14 & 0.76 & 0.97 & 0.98 & 0.99 & -0.14 & 0.80 & 0.96 & 0.98 & 0.99 \\ 
  10 & 0.01 & 1.02 & 0.92 & 0.95 & 0.97 & -0.12 & 0.74 & 0.97 & 0.98 & 1.00 & -0.13 & 0.79 & 0.96 & 0.98 & 0.99 \\ 
		%
	%
	\midrule
&\multicolumn{15}{c}{$\Delta_n=0.1, n= 5000$}
\\
  \hline
0 & -0.27 & 0.99 & 0.89 & 0.94 & 1.00 & -0.51 & 1.16 & 0.84 & 0.88 & 0.94 & -0.26 & 1.16 & 0.86 & 0.91 & 0.96 \\ 
  0.5 & -0.12 & 1.00 & 0.91 & 0.96 & 0.99 & -0.29 & 0.99 & 0.89 & 0.94 & 0.98 & -0.13 & 1.00 & 0.90 & 0.95 & 0.99 \\ 
  2 & 0.02 & 1.01 & 0.91 & 0.94 & 0.98 & -0.13 & 0.86 & 0.94 & 0.97 & 0.99 & -0.08 & 0.89 & 0.93 & 0.97 & 0.99 \\ 
  4 & 0.00 & 0.97 & 0.92 & 0.95 & 0.98 & -0.12 & 0.78 & 0.95 & 0.98 & 1.00 & -0.10 & 0.81 & 0.96 & 0.98 & 1.00 \\ 
  7 & -0.02 & 1.01 & 0.90 & 0.94 & 0.98 & -0.12 & 0.78 & 0.96 & 0.99 & 1.00 & -0.12 & 0.82 & 0.95 & 0.98 & 1.00 \\ 
  10 & -0.03 & 1.01 & 0.92 & 0.95 & 0.98 & -0.12 & 0.79 & 0.97 & 0.98 & 0.99 & -0.12 & 0.83 & 0.96 & 0.97 & 0.99 \\ 
		%
	%
		\midrule
	&\multicolumn{15}{c}{$\Delta_n=0.1, n= 10000$}
	\\
0 & -0.38 & 1.02 & 0.88 & 0.94 & 0.98 & -0.56 & 1.14 & 0.82 & 0.90 & 0.95 & -0.20 & 1.13 & 0.86 & 0.93 & 0.98 \\ 
  0.5 & -0.22 & 0.98 & 0.90 & 0.95 & 0.99 & -0.33 & 0.96 & 0.89 & 0.94 & 0.97 & -0.12 & 0.97 & 0.91 & 0.96 & 0.99 \\ 
  2 & -0.05 & 0.96 & 0.92 & 0.96 & 0.99 & -0.14 & 0.83 & 0.95 & 0.98 & 1.00 & -0.06 & 0.85 & 0.95 & 0.98 & 1.00 \\ 
  4 & -0.04 & 0.94 & 0.92 & 0.96 & 0.99 & -0.12 & 0.76 & 0.97 & 0.99 & 1.00 & -0.08 & 0.79 & 0.96 & 0.99 & 1.00 \\ 
  7 & -0.01 & 1.05 & 0.89 & 0.94 & 0.98 & -0.10 & 0.80 & 0.97 & 0.98 & 1.00 & -0.09 & 0.84 & 0.96 & 0.98 & 0.99 \\ 
  10 & -0.02 & 0.99 & 0.90 & 0.95 & 0.99 & -0.08 & 0.75 & 0.98 & 0.99 & 1.00 & -0.08 & 0.79 & 0.96 & 0.98 & 1.00 \\ 
	\midrule
	%
	%
	\midrule
	&\multicolumn{15}{c}{$\Delta_n=0.01, n= 2000$}
	\\
0 & -0.03 & 0.94 & 0.95 & 0.96 & 0.98 & -1.17 & 2.35 & 0.72 & 0.76 & 0.82 & -1.16 & 2.35 & 0.72 & 0.76 & 0.82 \\ 
  0.5 & 0.07 & 0.97 & 0.94 & 0.95 & 0.98 & 0.27 & 31.66 & 0.80 & 0.84 & 0.89 & 0.27 & 31.66 & 0.80 & 0.83 & 0.88 \\ 
  2 & -0.02 & 0.98 & 0.94 & 0.96 & 0.98 & 6.30 & 83.40 & 0.83 & 0.86 & 0.92 & 6.30 & 83.40 & 0.83 & 0.86 & 0.92 \\ 
  4 & -0.08 & 0.95 & 0.95 & 0.96 & 0.97 & 4.46 & 70.54 & 0.88 & 0.92 & 0.96 & 4.46 & 70.54 & 0.89 & 0.92 & 0.96 \\ 
  7 & -0.15 & 0.76 & 0.97 & 0.97 & 0.98 & 7.58 & 89.08 & 0.94 & 0.95 & 0.97 & 7.58 & 89.08 & 0.94 & 0.95 & 0.97 \\ 
  10 & -0.19 & 0.63 & 0.97 & 0.98 & 0.99 & 13.66 & 117.47 & 0.97 & 0.97 & 0.98 & 13.66 & 117.47 & 0.97 & 0.97 & 0.98 \\ 
   %
   %
   \midrule
   &\multicolumn{15}{c}{$\Delta_n=0.01, n= 5000$}
   \\
 0 & 0.01 & 0.99 & 0.94 & 0.96 & 0.97 & -0.63 & 1.56 & 0.80 & 0.84 & 0.89 & -0.61 & 1.56 & 0.80 & 0.84 & 0.89 \\ 
  0.5 & 0.11 & 1.06 & 0.92 & 0.94 & 0.97 & -0.38 & 1.21 & 0.86 & 0.89 & 0.94 & -0.37 & 1.22 & 0.86 & 0.89 & 0.94 \\ 
  2 & 0.04 & 1.02 & 0.92 & 0.95 & 0.98 & 1.62 & 44.68 & 0.89 & 0.92 & 0.96 & 1.63 & 44.68 & 0.89 & 0.92 & 0.96 \\ 
  4 & 0.05 & 1.05 & 0.91 & 0.94 & 0.96 & 5.68 & 77.22 & 0.92 & 0.95 & 0.97 & 5.68 & 77.22 & 0.92 & 0.95 & 0.97 \\ 
  7 & -0.00 & 1.02 & 0.92 & 0.95 & 0.97 & 11.75 & 108.86 & 0.95 & 0.96 & 0.98 & 11.75 & 108.86 & 0.95 & 0.96 & 0.98 \\ 
  10 & -0.01 & 0.98 & 0.92 & 0.94 & 0.98 & 3.77 & 63.11 & 0.95 & 0.97 & 0.98 & 3.77 & 63.11 & 0.95 & 0.97 & 0.98 \\ 
%
%
   \midrule
   &\multicolumn{15}{c}{$\Delta_n=0.01, n= 10000$}
   \\
 0 & 0.03 & 1.00 & 0.92 & 0.95 & 0.98 & -0.41 & 1.34 & 0.85 & 0.89 & 0.93 & -0.39 & 1.34 & 0.85 & 0.89 & 0.93 \\ 
  0.5 & 0.09 & 1.00 & 0.93 & 0.95 & 0.97 & -0.23 & 1.04 & 0.90 & 0.94 & 0.97 & -0.22 & 1.04 & 0.90 & 0.94 & 0.97 \\ 
  2 & 0.02 & 0.96 & 0.94 & 0.96 & 0.98 & -0.24 & 0.88 & 0.93 & 0.96 & 0.98 & -0.24 & 0.88 & 0.93 & 0.95 & 0.98 \\ 
  4 & 0.03 & 1.05 & 0.92 & 0.95 & 0.98 & 1.79 & 44.67 & 0.94 & 0.97 & 0.99 & 1.79 & 44.67 & 0.94 & 0.97 & 0.99 \\ 
  7 & 0.04 & 0.99 & 0.93 & 0.96 & 0.98 & 5.85 & 77.20 & 0.97 & 0.98 & 0.99 & 5.85 & 77.20 & 0.97 & 0.98 & 0.99 \\ 
  10 & 0.03 & 0.99 & 0.92 & 0.95 & 0.97 & 1.85 & 44.67 & 0.96 & 0.98 & 0.99 & 1.85 & 44.67 & 0.96 & 0.98 & 0.99 \\ 
   \bottomrule
\end{tabular}
\caption{ \label{tab:Gamma-LM-fixedt} Gamma marginal distribution with supGamma trawl function: We consider three statistics and report their mean, standard deviation and coverage probabilities at levels 90\%, 95\% and 99\% based on 1000 Monte Carlo samples: Infeasible statistic: $T^{IF}(t):=\sqrt{n\Delta_{n}}\left(\hat{a}(t)-a(t)\right)/(\sigma_a^2(t))^{1/2}$; feasible statistic without bias correction: $T^{F}:=\sqrt{n\Delta_{n}}\left(\hat{a}(t)-a(t)\right)/(\hat \sigma_a^2(t))^{1/2}$; Feasible statistic with bias correction: $T^{F}_{\mathrm{bias \, corrected}}:=\sqrt{n\Delta_{n}}\left(\hat{a}(t)-a(t)-\frac{1}{2}\Delta_n \hat a'(t)\right)/(\hat \sigma_a^2(t))^{1/2}$.}
\end{table}

}

%% file: Tables/All_Results_Gaussian_LM_fixedt_v4.tex
\afterpage{%
	\begin{table}[ht]  \scriptsize
\centering
\setlength{\tabcolsep}{5pt} 
\renewcommand{\arraystretch}{1.1}
\begin{tabular}{r|rrrrr|rrrrr|rrrrr}
	\toprule
		$t=$	&\multicolumn{5}{c}{Infeasible statistic}
	&\multicolumn{5}{c}{Feasible statistic without bias correction}
	&	\multicolumn{5}{c}{Feasible statistic with bias correction}
	\\
 $i\Delta_n$& \multicolumn{15}{c}{}\\
	\midrule
		%
	%
 $t$	& Mean & SD &90\% & 95\% & 99\%
 	& Mean & SD &90\% & 95\% & 99\%
 		& Mean & SD &90\% & 95\% & 99\%
 		 \\ 
	\midrule
	&\multicolumn{15}{c}{$\Delta_n=0.1, n= 2000$}
		\\
0 &  &  &  &  &  & -0.72 & 0.61 & 0.93 & 0.97 & 0.99 & -0.10 & 0.74 & 0.97 & 0.99 & 1.00 \\ 
  0.5 & -0.12 & 1.07 & 0.87 & 0.93 & 0.99 & -0.12 & 0.69 & 0.98 & 0.99 & 1.00 & 0.04 & 0.73 & 0.98 & 0.99 & 1.00 \\ 
  2 & -0.04 & 1.01 & 0.90 & 0.94 & 0.99 & -0.08 & 0.70 & 0.98 & 0.99 & 1.00 & -0.05 & 0.74 & 0.97 & 0.99 & 1.00 \\ 
  4 & 0.03 & 1.03 & 0.89 & 0.94 & 0.99 & -0.03 & 0.70 & 0.98 & 1.00 & 1.00 & -0.01 & 0.76 & 0.97 & 0.99 & 1.00 \\ 
  7 & -0.03 & 0.99 & 0.90 & 0.95 & 0.99 & -0.05 & 0.67 & 0.99 & 1.00 & 1.00 & -0.05 & 0.71 & 0.98 & 0.99 & 1.00 \\ 
  10 & 0.02 & 1.03 & 0.88 & 0.94 & 0.99 & -0.01 & 0.71 & 0.98 & 0.99 & 1.00 & -0.00 & 0.76 & 0.98 & 0.99 & 1.00 \\ 
		%
	%
	\midrule
&\multicolumn{15}{c}{$\Delta_n=0.1, n= 5000$}
\\
0 &  &  &  &  &  & -1.10 & 0.59 & 0.82 & 0.92 & 0.99 & -0.12 & 0.74 & 0.97 & 0.99 & 1.00 \\ 
  0.5 & -0.29 & 1.02 & 0.88 & 0.94 & 0.98 & -0.21 & 0.67 & 0.97 & 0.99 & 1.00 & 0.04 & 0.72 & 0.98 & 0.99 & 1.00 \\ 
  2 & -0.10 & 1.02 & 0.89 & 0.95 & 0.99 & -0.10 & 0.71 & 0.98 & 0.99 & 1.00 & -0.04 & 0.76 & 0.97 & 0.99 & 1.00 \\ 
  4 & 0.01 & 0.99 & 0.90 & 0.95 & 0.99 & -0.02 & 0.69 & 0.98 & 1.00 & 1.00 & 0.00 & 0.74 & 0.98 & 0.99 & 1.00 \\ 
  7 & -0.03 & 0.98 & 0.90 & 0.96 & 0.99 & -0.04 & 0.68 & 0.98 & 0.99 & 1.00 & -0.03 & 0.73 & 0.98 & 0.99 & 1.00 \\ 
  10 & 0.01 & 1.01 & 0.90 & 0.95 & 0.99 & -0.00 & 0.71 & 0.99 & 1.00 & 1.00 & 0.00 & 0.75 & 0.97 & 0.99 & 1.00 \\ 
		%
	%
		\midrule
	&\multicolumn{15}{c}{$\Delta_n=0.1, n= 10000$}
	\\
0 &  &  &  &  & & -1.55 & 0.62 & 0.57 & 0.75 & 0.95 & -0.14 & 0.74 & 0.97 & 0.99 & 1.00 \\ 
  0.5 & -0.55 & 1.00 & 0.86 & 0.92 & 0.97 & -0.38 & 0.66 & 0.97 & 0.98 & 1.00 & -0.02 & 0.70 & 0.98 & 0.99 & 1.00 \\ 
  2 & -0.16 & 1.02 & 0.88 & 0.94 & 0.98 & -0.14 & 0.71 & 0.97 & 0.99 & 1.00 & -0.04 & 0.76 & 0.97 & 0.99 & 1.00 \\ 
  4 & -0.02 & 1.03 & 0.90 & 0.94 & 0.99 & -0.04 & 0.72 & 0.98 & 0.99 & 1.00 & -0.00 & 0.77 & 0.97 & 0.99 & 1.00 \\ 
  7 & -0.04 & 0.98 & 0.91 & 0.95 & 0.99 & -0.04 & 0.69 & 0.98 & 1.00 & 1.00 & -0.02 & 0.74 & 0.97 & 0.99 & 1.00 \\ 
  10 & 0.00 & 1.02 & 0.89 & 0.95 & 0.99 & -0.01 & 0.71 & 0.98 & 0.99 & 1.00 & -0.00 & 0.76 & 0.97 & 0.99 & 1.00 \\ 
	\midrule
	%
	%
	\midrule
	&\multicolumn{15}{c}{$\Delta_n=0.01, n= 2000$}
	\\
0 &  &  &  &  &  & -0.06 & 0.58 & 0.99 & 1.00 & 1.00 & 0.01 & 0.70 & 0.99 & 0.99 & 1.00 \\ 
  0.5 & 0.31 & 1.16 & 0.87 & 0.91 & 0.95 & 0.05 & 0.73 & 0.97 & 0.99 & 1.00 & 0.06 & 0.73 & 0.97 & 0.99 & 1.00 \\ 
  2 & -0.04 & 1.00 & 0.91 & 0.94 & 0.98 & -0.16 & 0.68 & 0.98 & 1.00 & 1.00 & -0.16 & 0.69 & 0.98 & 1.00 & 1.00 \\ 
  4 & -0.19 & 0.92 & 0.93 & 0.96 & 0.99 & -0.22 & 0.62 & 0.99 & 1.00 & 1.00 & -0.21 & 0.63 & 0.99 & 1.00 & 1.00 \\ 
  7 & -0.18 & 0.85 & 0.94 & 0.97 & 0.99 & -0.18 & 0.57 & 1.00 & 1.00 & 1.00 & -0.18 & 0.58 & 1.00 & 1.00 & 1.00 \\ 
  10 & -0.19 & 0.69 & 0.97 & 0.99 & 1.00 & -0.15 & 0.49 & 1.00 & 1.00 & 1.00 & -0.15 & 0.50 & 1.00 & 1.00 & 1.00 \\ 
   %
   %
   \midrule
   &\multicolumn{15}{c}{$\Delta_n=0.01, n= 5000$}
   \\
0 &  &  &  &  &  & -0.08 & 0.57 & 1.00 & 1.00 & 1.00 & 0.04 & 0.70 & 0.98 & 1.00 & 1.00 \\ 
  0.5 & 0.24 & 1.04 & 0.88 & 0.93 & 0.98 & 0.08 & 0.69 & 0.98 & 0.99 & 1.00 & 0.09 & 0.70 & 0.98 & 0.99 & 1.00 \\ 
  2 & 0.07 & 1.03 & 0.90 & 0.94 & 0.98 & -0.06 & 0.71 & 0.98 & 0.99 & 1.00 & -0.05 & 0.72 & 0.97 & 0.99 & 1.00 \\ 
  4 & 0.01 & 0.97 & 0.91 & 0.96 & 0.99 & -0.05 & 0.66 & 0.99 & 1.00 & 1.00 & -0.05 & 0.66 & 0.99 & 1.00 & 1.00 \\ 
  7 & -0.06 & 0.99 & 0.91 & 0.94 & 0.99 & -0.09 & 0.66 & 0.99 & 1.00 & 1.00 & -0.09 & 0.66 & 0.99 & 1.00 & 1.00 \\ 
  10 & -0.03 & 0.95 & 0.92 & 0.96 & 0.99 & -0.05 & 0.63 & 1.00 & 1.00 & 1.00 & -0.05 & 0.64 & 1.00 & 1.00 & 1.00 \\ 
   %
   %
   \midrule
   &\multicolumn{15}{c}{$\Delta_n=0.01, n= 10000$}
   \\
 0 &  &  &  &  &  & -0.14 & 0.59 & 0.99 & 1.00 & 1.00 & 0.02 & 0.72 & 0.98 & 0.99 & 1.00 \\ 
  0.5 & 0.17 & 1.04 & 0.88 & 0.94 & 0.99 & 0.05 & 0.71 & 0.98 & 0.99 & 1.00 & 0.06 & 0.72 & 0.98 & 0.99 & 1.00 \\ 
  2 & 0.07 & 1.02 & 0.90 & 0.94 & 0.99 & -0.03 & 0.71 & 0.98 & 0.99 & 1.00 & -0.02 & 0.72 & 0.98 & 0.99 & 1.00 \\ 
  4 & 0.02 & 1.00 & 0.90 & 0.95 & 0.99 & -0.04 & 0.70 & 0.98 & 0.99 & 1.00 & -0.04 & 0.70 & 0.98 & 0.99 & 1.00 \\ 
  7 & -0.02 & 1.02 & 0.90 & 0.94 & 0.98 & -0.05 & 0.68 & 0.98 & 0.99 & 1.00 & -0.06 & 0.69 & 0.98 & 1.00 & 1.00 \\ 
  10 & 0.01 & 0.98 & 0.91 & 0.96 & 0.99 & -0.01 & 0.67 & 0.99 & 1.00 & 1.00 & -0.02 & 0.68 & 0.99 & 1.00 & 1.00 \\ 
   \bottomrule
\end{tabular}
\caption{ \label{tab:Gaussian-LM-fixedt} Gaussian marginal distribution with supGamma trawl function (fixed $t$): We consider three statistics and report their mean, standard deviation and coverage probabilities at levels 90\%, 95\% and 99\% based on 1000 Monte Carlo samples: Infeasible statistic: $T^{IF}(t):=\sqrt{n\Delta_{n}}\left(\hat{a}(t)-a(t)\right)/(\sigma_a^2(t))^{1/2}$; feasible statistic without bias correction: $T^{F}:=\sqrt{n\Delta_{n}}\left(\hat{a}(t)-a(t)\right)/(\hat \sigma_a^2(t))^{1/2}$; Feasible statistic with bias correction: $T^{F}_{\mathrm{bias \, corrected}}:=\sqrt{n\Delta_{n}}\left(\hat{a}(t)-a(t)-\frac{1}{2}\Delta_n \hat a'(t)\right)/(\hat \sigma_a^2(t))^{1/2}$.}
\end{table}
}

%% file: Tables/Slices-Exp-Combined.tex
\afterpage{%
	\begin{table}[ht]\scriptsize
\centering
\setlength{\tabcolsep}{3pt} 
\renewcommand{\arraystretch}{1.1}
\begin{tabular}{l|lll|lll|lll|lll|lll}
		\toprule
			$h=i\Delta_n$ for &  \multicolumn{3}{c}{$\widehat{\mathrm{Leb}(A)}$}  & \multicolumn{3}{c}{$\widehat{\mathrm{Leb}(A\cap A_h)}$} &
			\multicolumn{3}{c}{$\widehat{\mathrm{Leb}(A \setminus A_h)}$} &
			\multicolumn{3}{c}{$\widehat{\mathrm{Leb}(A\cap A_h)}/\widehat{\mathrm{Leb}(A)}$} &
				\multicolumn{3}{c}{$\widehat{\mathrm{Leb}(A \setminus A_h)}/\widehat{\mathrm{Leb}(A)}$}
			\\
		$i$&  Mean & Bias & SD & Mean & Bias & SD  & Mean & Bias & SD 
		&  Mean & Bias & SD
		&  Mean & Bias & SD \\
		\midrule
		&\multicolumn{15}{c}{Trawl-function based estimation without bias correction}\\ \midrule
		&\multicolumn{15}{c}{Negative binomial marginal distribution}\\
 1 & 0.997 & -0.003 & 0.123 & 0.902 & -0.003 & 0.118 & 0.095 & 0 & 0.008 & 0.904 & -0.001 & 0.009 & 0.096 & 0.001 & 0.009 \\ 
  2 & 0.997 & -0.003 & 0.123 & 0.816 & -0.003 & 0.113 & 0.182 & 0 & 0.015 & 0.816 & -0.002 & 0.016 & 0.184 & 0.002 & 0.016 \\ 
  5 & 0.997 & -0.003 & 0.123 & 0.603 & -0.004 & 0.1 & 0.395 & 0.001 & 0.036 & 0.602 & -0.005 & 0.033 & 0.398 & 0.005 & 0.033 \\ 
  10 & 0.997 & -0.003 & 0.123 & 0.363 & -0.005 & 0.08 & 0.634 & 0.002 & 0.065 & 0.361 & -0.007 & 0.045 & 0.639 & 0.007 & 0.045 \\ 
  \midrule
  &\multicolumn{15}{c}{Gamma marginal distribution}\\
 1 & 0.99 & -0.01 & 0.202 & 0.896 & -0.009 & 0.193 & 0.095 & 0 & 0.014 & 0.903 & -0.002 & 0.013 & 0.097 & 0.002 & 0.013 \\ 
  2 & 0.99 & -0.01 & 0.202 & 0.809 & -0.009 & 0.184 & 0.181 & -0.001 & 0.026 & 0.814 & -0.005 & 0.025 & 0.186 & 0.005 & 0.025 \\ 
  5 & 0.99 & -0.01 & 0.202 & 0.598 & -0.008 & 0.159 & 0.392 & -0.001 & 0.06 & 0.598 & -0.009 & 0.047 & 0.402 & 0.009 & 0.047 \\ 
  10 & 0.99 & -0.01 & 0.202 & 0.36 & -0.007 & 0.123 & 0.63 & -0.002 & 0.105 & 0.357 & -0.011 & 0.061 & 0.643 & 0.011 & 0.061 \\ 
  \midrule
   &\multicolumn{15}{c}{Gaussian marginal distribution}\\
   1 & 0.998 & -0.002 & 0.063 & 0.903 & -0.002 & 0.063 & 0.095 & 0 & 0.002 & 0.904 & -0.001 & 0.006 & 0.096 & 0.001 & 0.006 \\ 
  2 & 0.998 & -0.002 & 0.063 & 0.817 & -0.002 & 0.062 & 0.181 & 0 & 0.004 & 0.818 & -0.001 & 0.011 & 0.182 & 0.001 & 0.011 \\ 
  5 & 0.998 & -0.002 & 0.063 & 0.604 & -0.003 & 0.058 & 0.394 & 0.001 & 0.015 & 0.604 & -0.003 & 0.023 & 0.396 & 0.003 & 0.023 \\ 
  10 & 0.998 & -0.002 & 0.063 & 0.364 & -0.004 & 0.053 & 0.634 & 0.002 & 0.031 & 0.363 & -0.004 & 0.035 & 0.637 & 0.004 & 0.035 \\ 
   \midrule	
   &\multicolumn{15}{c}{Trawl-function based estimation with bias correction}\\ \midrule
   &\multicolumn{15}{c}{Negative binomial marginal distribution}\\
  1 & 1.045 & 0.045 & 0.126 & 0.945 & 0.04 & 0.121 & 0.1 & 0.005 & 0.008 & 0.904 & -0.001 & 0.009 & 0.096 & 0.001 & 0.009 \\ 
  2 & 1.045 & 0.045 & 0.126 & 0.855 & 0.036 & 0.116 & 0.19 & 0.009 & 0.016 & 0.817 & -0.002 & 0.016 & 0.183 & 0.002 & 0.016 \\ 
  5 & 1.045 & 0.045 & 0.126 & 0.631 & 0.025 & 0.102 & 0.413 & 0.02 & 0.037 & 0.602 & -0.005 & 0.032 & 0.398 & 0.005 & 0.032 \\ 
  10 & 1.045 & 0.045 & 0.126 & 0.38 & 0.012 & 0.082 & 0.664 & 0.032 & 0.066 & 0.361 & -0.007 & 0.044 & 0.639 & 0.007 & 0.044 \\ 
     \midrule
   &\multicolumn{15}{c}{Gamma marginal distribution}\\
  1 & 1.037 & 0.037 & 0.207 & 0.938 & 0.034 & 0.198 & 0.099 & 0.004 & 0.014 & 0.903 & -0.002 & 0.013 & 0.097 & 0.002 & 0.013 \\ 
  2 & 1.037 & 0.037 & 0.207 & 0.848 & 0.029 & 0.189 & 0.189 & 0.008 & 0.027 & 0.814 & -0.004 & 0.025 & 0.186 & 0.004 & 0.025 \\ 
  5 & 1.037 & 0.037 & 0.207 & 0.627 & 0.02 & 0.163 & 0.411 & 0.017 & 0.061 & 0.598 & -0.008 & 0.046 & 0.402 & 0.008 & 0.046 \\ 
  10 & 1.037 & 0.037 & 0.207 & 0.378 & 0.01 & 0.126 & 0.66 & 0.028 & 0.108 & 0.357 & -0.011 & 0.06 & 0.643 & 0.011 & 0.06 \\ 
     \midrule
   &\multicolumn{15}{c}{Gaussian marginal distribution}\\
  1 & 1.046 & 0.046 & 0.064 & 0.946 & 0.041 & 0.064 & 0.1 & 0.005 & 0.002 & 0.904 & -0.001 & 0.006 & 0.096 & 0.001 & 0.006 \\ 
  2 & 1.046 & 0.046 & 0.064 & 0.856 & 0.037 & 0.063 & 0.19 & 0.009 & 0.004 & 0.818 & -0.001 & 0.011 & 0.182 & 0.001 & 0.011 \\ 
  5 & 1.046 & 0.046 & 0.064 & 0.633 & 0.026 & 0.059 & 0.413 & 0.02 & 0.014 & 0.604 & -0.003 & 0.022 & 0.396 & 0.003 & 0.022 \\ 
  10 & 1.046 & 0.046 & 0.064 & 0.382 & 0.014 & 0.053 & 0.664 & 0.032 & 0.03 & 0.364 & -0.004 & 0.033 & 0.636 & 0.004 & 0.033 \\ 
    \midrule	
   &\multicolumn{15}{c}{ACF-based estimation}\\ \midrule
&\multicolumn{15}{c}{Negative binomial marginal distribution}\\
1 & 0.998 & -0.002 & 0.123 & 0.902 & -0.003 & 0.118 & 0.095 & 0 & 0.008 & 0.903 & -0.001 & 0.009 & 0.097 & 0.001 & 0.009 \\ 
  2 & 0.998 & -0.002 & 0.123 & 0.816 & -0.003 & 0.113 & 0.182 & 0.001 & 0.015 & 0.816 & -0.002 & 0.016 & 0.184 & 0.002 & 0.016 \\ 
  5 & 0.998 & -0.002 & 0.123 & 0.603 & -0.004 & 0.1 & 0.395 & 0.002 & 0.036 & 0.601 & -0.005 & 0.033 & 0.399 & 0.005 & 0.033 \\ 
  10 & 0.998 & -0.002 & 0.123 & 0.363 & -0.005 & 0.08 & 0.635 & 0.003 & 0.065 & 0.361 & -0.007 & 0.045 & 0.639 & 0.007 & 0.045 \\ 
  \midrule
  &\multicolumn{15}{c}{Gamma marginal distribution}\\
 1 & 0.991 & -0.009 & 0.202 & 0.896 & -0.009 & 0.193 & 0.095 & 0 & 0.014 & 0.902 & -0.003 & 0.013 & 0.098 & 0.003 & 0.013 \\ 
  2 & 0.991 & -0.009 & 0.202 & 0.81 & -0.009 & 0.184 & 0.181 & 0 & 0.026 & 0.814 & -0.005 & 0.025 & 0.186 & 0.005 & 0.025 \\ 
  5 & 0.991 & -0.009 & 0.202 & 0.598 & -0.008 & 0.159 & 0.392 & -0.001 & 0.06 & 0.598 & -0.009 & 0.047 & 0.402 & 0.009 & 0.047 \\ 
  10 & 0.991 & -0.009 & 0.202 & 0.36 & -0.007 & 0.123 & 0.63 & -0.002 & 0.105 & 0.357 & -0.011 & 0.061 & 0.643 & 0.011 & 0.061 \\ 
    \midrule
   &\multicolumn{15}{c}{Gaussian marginal distribution}\\
   1 & 0.999 & -0.001 & 0.063 & 0.903 & -0.002 & 0.063 & 0.095 & 0 & 0.002 & 0.904 & -0.001 & 0.006 & 0.096 & 0.001 & 0.006 \\ 
  2 & 0.999 & -0.001 & 0.063 & 0.817 & -0.002 & 0.062 & 0.182 & 0 & 0.005 & 0.817 & -0.001 & 0.011 & 0.183 & 0.001 & 0.011 \\ 
  5 & 0.999 & -0.001 & 0.063 & 0.604 & -0.003 & 0.058 & 0.395 & 0.001 & 0.015 & 0.604 & -0.003 & 0.023 & 0.396 & 0.003 & 0.023 \\ 
  10 & 0.999 & -0.001 & 0.063 & 0.364 & -0.004 & 0.053 & 0.634 & 0.002 & 0.031 & 0.363 & -0.005 & 0.035 & 0.637 & 0.005 & 0.035 \\    
     \bottomrule
\end{tabular}
\caption{\label{tab:slices-exp} Slice estimation for the  exponential trawl with $\lambda =1$ for the case when $\Delta_n=0.1$ and $n=5000$ for 1000 Monte Carlo runs.}
\end{table}
}

%% file: Tables/Slices-LM-Combined.tex
\afterpage{\begin{table}[ht]\scriptsize
	\centering
    \setlength{\tabcolsep}{3pt} 
\renewcommand{\arraystretch}{1.1}
	\begin{tabular}{l|lll|lll|lll|lll|lll}
		\toprule
		$h=i\Delta_n$ for &  \multicolumn{3}{c}{$\widehat{\mathrm{Leb}(A)}$}  & \multicolumn{3}{c}{$\widehat{\mathrm{Leb}(A\cap A_h)}$} &
		\multicolumn{3}{c}{$\widehat{\mathrm{Leb}(A \setminus A_h)}$} &
		\multicolumn{3}{c}{$\widehat{\mathrm{Leb}(A\cap A_h)}/\widehat{\mathrm{Leb}(A)}$} &
		\multicolumn{3}{c}{$\widehat{\mathrm{Leb}(A \setminus A_h)}/\widehat{\mathrm{Leb}(A)}$}
		\\
		$i$&  Mean & Bias & SD & Mean & Bias & SD  & Mean & Bias & SD 
		&  Mean & Bias & SD
		&  Mean & Bias & SD \\
		\midrule
		&\multicolumn{15}{c}{Trawl-function based estimation without bias correction}\\ \midrule
		&\multicolumn{15}{c}{Negative binomial marginal distribution}\\
1 & 3.411 & -0.589 & 0.901 & 3.314 & -0.589 & 0.899 & 0.097 & 0.001 & 0.008 & 0.97 & -0.006 & 0.007 & 0.03 & 0.006 & 0.007 \\ 
  2 & 3.411 & -0.589 & 0.901 & 3.224 & -0.59 & 0.896 & 0.188 & 0.002 & 0.015 & 0.942 & -0.011 & 0.013 & 0.058 & 0.011 & 0.013 \\ 
  5 & 3.411 & -0.589 & 0.901 & 2.985 & -0.593 & 0.89 & 0.426 & 0.004 & 0.038 & 0.868 & -0.026 & 0.029 & 0.132 & 0.026 & 0.029 \\ 
  10 & 3.411 & -0.589 & 0.901 & 2.67 & -0.596 & 0.878 & 0.742 & 0.008 & 0.073 & 0.771 & -0.045 & 0.049 & 0.229 & 0.045 & 0.049 \\ 
  \midrule
  &\multicolumn{15}{c}{Gamma marginal distribution}\\
 1 & 3.354 & -0.646 & 1.274 & 3.258 & -0.645 & 1.27 & 0.096 & 0 & 0.014 & 0.969 & -0.007 & 0.009 & 0.031 & 0.007 & 0.009 \\ 
  2 & 3.354 & -0.646 & 1.274 & 3.168 & -0.646 & 1.265 & 0.187 & 0 & 0.027 & 0.939 & -0.014 & 0.017 & 0.061 & 0.014 & 0.017 \\ 
  5 & 3.354 & -0.646 & 1.274 & 2.93 & -0.648 & 1.251 & 0.425 & 0.002 & 0.063 & 0.862 & -0.032 & 0.037 & 0.138 & 0.032 & 0.037 \\ 
  10 & 3.354 & -0.646 & 1.274 & 2.616 & -0.65 & 1.23 & 0.739 & 0.005 & 0.117 & 0.761 & -0.056 & 0.063 & 0.239 & 0.056 & 0.063 \\  
   \midrule
   &\multicolumn{15}{c}{Gaussian marginal distribution}\\
  1 & 3.376 & -0.624 & 0.606 & 3.28 & -0.624 & 0.606 & 0.096 & 0 & 0.002 & 0.971 & -0.005 & 0.005 & 0.029 & 0.005 & 0.005 \\ 
  2 & 3.376 & -0.624 & 0.606 & 3.19 & -0.624 & 0.605 & 0.187 & 0.001 & 0.004 & 0.943 & -0.011 & 0.01 & 0.057 & 0.011 & 0.01 \\ 
  5 & 3.376 & -0.624 & 0.606 & 2.951 & -0.626 & 0.605 & 0.425 & 0.003 & 0.015 & 0.87 & -0.024 & 0.022 & 0.13 & 0.024 & 0.022 \\ 
  10 & 3.376 & -0.624 & 0.606 & 2.636 & -0.63 & 0.603 & 0.74 & 0.006 & 0.036 & 0.774 & -0.042 & 0.039 & 0.226 & 0.042 & 0.039 \\ 
   \midrule	
  &\multicolumn{15}{c}{Trawl-function based estimation with bias correction}\\ \midrule
  &\multicolumn{15}{c}{Negative binomial marginal distribution}\\
 1 & 3.46 & -0.54 & 0.902 & 3.359 & -0.544 & 0.9 & 0.1 & 0.004 & 0.008 & 0.969 & -0.006 & 0.007 & 0.031 & 0.006 & 0.007 \\ 
  2 & 3.46 & -0.54 & 0.902 & 3.265 & -0.548 & 0.897 & 0.194 & 0.008 & 0.016 & 0.941 & -0.013 & 0.013 & 0.059 & 0.013 & 0.013 \\ 
  5 & 3.46 & -0.54 & 0.902 & 3.019 & -0.558 & 0.891 & 0.44 & 0.018 & 0.038 & 0.866 & -0.028 & 0.029 & 0.134 & 0.028 & 0.029 \\ 
  10 & 3.46 & -0.54 & 0.902 & 2.696 & -0.57 & 0.879 & 0.764 & 0.03 & 0.073 & 0.768 & -0.048 & 0.049 & 0.232 & 0.048 & 0.049 \\ 
   \midrule
  &\multicolumn{15}{c}{Gamma marginal distribution}\\
 1 & 3.402 & -0.598 & 1.277 & 3.303 & -0.601 & 1.272 & 0.1 & 0.003 & 0.014 & 0.968 & -0.008 & 0.009 & 0.032 & 0.008 & 0.009 \\ 
  2 & 3.402 & -0.598 & 1.277 & 3.209 & -0.605 & 1.267 & 0.193 & 0.007 & 0.027 & 0.938 & -0.015 & 0.017 & 0.062 & 0.015 & 0.017 \\ 
  5 & 3.402 & -0.598 & 1.277 & 2.964 & -0.614 & 1.253 & 0.438 & 0.016 & 0.064 & 0.86 & -0.035 & 0.038 & 0.14 & 0.035 & 0.038 \\ 
  10 & 3.402 & -0.598 & 1.277 & 2.642 & -0.624 & 1.232 & 0.76 & 0.026 & 0.118 & 0.758 & -0.059 & 0.063 & 0.242 & 0.059 & 0.063 \\ 
   \midrule
  &\multicolumn{15}{c}{Gaussian marginal distribution}\\
 1 & 3.424 & -0.576 & 0.606 & 3.325 & -0.579 & 0.606 & 0.1 & 0.003 & 0.002 & 0.97 & -0.006 & 0.005 & 0.03 & 0.006 & 0.005 \\ 
  2 & 3.424 & -0.576 & 0.606 & 3.231 & -0.583 & 0.606 & 0.193 & 0.007 & 0.004 & 0.942 & -0.012 & 0.01 & 0.058 & 0.012 & 0.01 \\ 
  5 & 3.424 & -0.576 & 0.606 & 2.986 & -0.592 & 0.605 & 0.439 & 0.016 & 0.014 & 0.868 & -0.026 & 0.022 & 0.132 & 0.026 & 0.022 \\ 
  10 & 3.424 & -0.576 & 0.606 & 2.663 & -0.603 & 0.603 & 0.762 & 0.028 & 0.034 & 0.771 & -0.045 & 0.039 & 0.229 & 0.045 & 0.039 \\ 
		\midrule
		&\multicolumn{15}{c}{ACF-based estimation}\\ \midrule
		&\multicolumn{15}{c}{Negative binomial marginal distribution}\\
1 & 3.413 & -0.587 & 0.901 & 3.315 & -0.588 & 0.899 & 0.098 & 0.001 & 0.008 & 0.97 & -0.006 & 0.007 & 0.03 & 0.006 & 0.007 \\ 
  2 & 3.413 & -0.587 & 0.901 & 3.224 & -0.589 & 0.896 & 0.188 & 0.002 & 0.015 & 0.942 & -0.012 & 0.013 & 0.058 & 0.012 & 0.013 \\ 
  5 & 3.413 & -0.587 & 0.901 & 2.986 & -0.592 & 0.89 & 0.427 & 0.005 & 0.038 & 0.868 & -0.026 & 0.029 & 0.132 & 0.026 & 0.029 \\ 
  10 & 3.413 & -0.587 & 0.901 & 2.67 & -0.596 & 0.878 & 0.743 & 0.009 & 0.073 & 0.771 & -0.045 & 0.049 & 0.229 & 0.045 & 0.049 \\ 
  \midrule
  &\multicolumn{15}{c}{Gamma marginal distribution}\\
 1 & 3.356 & -0.644 & 1.275 & 3.259 & -0.645 & 1.27 & 0.097 & 0.001 & 0.014 & 0.968 & -0.008 & 0.009 & 0.032 & 0.008 & 0.009 \\ 
  2 & 3.356 & -0.644 & 1.275 & 3.168 & -0.645 & 1.265 & 0.187 & 0.001 & 0.027 & 0.939 & -0.014 & 0.017 & 0.061 & 0.014 & 0.017 \\ 
  5 & 3.356 & -0.644 & 1.275 & 2.93 & -0.647 & 1.252 & 0.425 & 0.003 & 0.063 & 0.862 & -0.033 & 0.037 & 0.138 & 0.033 & 0.037 \\ 
  10 & 3.356 & -0.644 & 1.275 & 2.616 & -0.65 & 1.23 & 0.74 & 0.006 & 0.117 & 0.76 & -0.056 & 0.063 & 0.24 & 0.056 & 0.063 \\  
    \midrule
   &\multicolumn{15}{c}{Gaussian marginal distribution}\\
 1 & 3.378 & -0.622 & 0.606 & 3.281 & -0.623 & 0.606 & 0.097 & 0.001 & 0.002 & 0.97 & -0.006 & 0.005 & 0.03 & 0.006 & 0.005 \\ 
  2 & 3.378 & -0.622 & 0.606 & 3.19 & -0.624 & 0.606 & 0.188 & 0.001 & 0.005 & 0.943 & -0.011 & 0.01 & 0.057 & 0.011 & 0.01 \\ 
  5 & 3.378 & -0.622 & 0.606 & 2.952 & -0.626 & 0.605 & 0.426 & 0.004 & 0.015 & 0.87 & -0.024 & 0.022 & 0.13 & 0.024 & 0.022 \\ 
  10 & 3.378 & -0.622 & 0.606 & 2.637 & -0.629 & 0.603 & 0.741 & 0.007 & 0.036 & 0.774 & -0.042 & 0.039 & 0.226 & 0.042 & 0.039 \\ 
      \bottomrule
\end{tabular}
\caption{\label{tab:slices-lm} Slice estimation for the  supGamma trawl with $\overline{\alpha} =2, H=1.5$ for the case when $\Delta_n=0.1$ and $n=5000$ for 1000 Monte Carlo runs.}
\end{table}
}